\documentclass[11pt, reqno]{amsart}
\usepackage{fullpage}
\usepackage{thmtools}
\usepackage{hhline}
\usepackage{thm-restate}
\usepackage{subfiles}

\newif\ifHideFoot
\HideFoottrue  
\HideFootfalse   

\newif\ifAppendix
\Appendixtrue

\usepackage{enumitem, amssymb,amsmath,amsthm,amscd,mathrsfs,graphicx, color}
\usepackage[cmtip,all,matrix,arrow,tips,curve]{xy}
\usepackage{xr-hyper}
\usepackage{longtable}
\usepackage{hyperref}
\usepackage{cleveref}
\usepackage{caption}
\usepackage{subcaption}
\usepackage[usenames,dvipsnames]{xcolor}
\usepackage{xypic}
\hypersetup{colorlinks=true,citecolor=ForestGreen,linkcolor=Maroon,urlcolor=NavyBlue}
\usepackage{mathtools}
\usepackage{mathpazo}
\usepackage{tikz}
\usepackage{array}
\usepackage{tikz-cd}
\usepackage{circuitikz}
\usepackage{float}
\setlist[enumerate]{leftmargin=8.5mm}

\numberwithin{equation}{section}
\newtheorem{teo}{Theorem}[section]
\newtheorem{pro}[teo]{Proposition}
\newtheorem{lem}[teo]{Lemma}
\newtheorem{cor}[teo]{Corollary}

\newtheorem*{teo*}{Theorem}
\newtheorem*{pro*}{Proposition}
\newtheorem*{dfn*}{Definition}

\theoremstyle{definition}
\newtheorem{dfn}[teo]{Definition}

\newtheorem{conv}[teo]{Convention}
\newtheorem{proc}[teo]{Procedure}

\theoremstyle{remark}
\newtheorem{rem}[teo]{Remark}

\ifHideFoot

\newcommand{\Yano}[1]{}
\newcommand{\Jon}[1]{}

\else

\newcommand{\marg}[1]{\normalsize{{
			\color{red}\footnote{{\color{blue}#1}}}{\marginpar[\vskip
			-.25cm{\color{Maroon}\hfill\thefootnote$\implies$}]{\vskip
				-.2cm{\color{Maroon}$\impliedby$\tiny\thefootnote}}}}}
\newcommand{\Yano}[1]{\marg{(Yano) #1}}
\newcommand{\Jon}[1]{\marg{(Jon) #1}}

\fi

\author{Jon Kim}
\address{University of Colorado, Department of Mathematics,
	Boulder, CO 80309, USA }
\email{joki5923@colorado.edu}

\title{KSBA moduli spaces of cubic surfaces with a marked line}

\subjclass{14D20 (primary), 14J10, 14J26, 14J15 (secondary)}
\keywords{moduli space, wall-crossing, compactification, cubic surface, stable pair}

\begin{document}
	\begin{abstract}
		The moduli space of cubic surfaces and its compactifications are classical and date back to the mid-nineteenth century. Recently, Schock described compactifications of moduli spaces of fully marked cubic surfaces with their 27 lines via Koll\'ar--Shepherd-Barron--Alexeev (KSBA) weighted stable pairs where the 27 lines are uniformly weighted. Furthermore, he provided an explicit finite wall-and-chamber decomposition of the weight domain, together with explicit descriptions of the weighted stable pairs parameterized by the moduli spaces in each chamber. We extend this work by describing compactifications of moduli spaces of cubic surfaces with a marked line via KSBA stable pairs with nonuniform weights, in which the marked line is weighted differently from the other 26 lines. In particular, we provide an explicit finite wall-and-chamber decomposition of the weight domain, yielding new KSBA coarse moduli spaces that have not previously been studied. We also give explicit descriptions of these nonuniform weighted stable pairs parameterized by the moduli spaces in each chamber. 
	\end{abstract}
	
	\maketitle
	
	\section{Introduction}\label{S:intro}
The moduli space $Y$ of marked cubic surfaces is one of the most classical moduli spaces in algebraic geometry, whose study dates back to the mid-nineteenth century, beginning with Cayley and Salmon's investigation of smooth cubic surfaces and their $27$ lines \cite{cayley1849triple}. There is a natural action of the Weyl group $W(E_6)$ of the root system $E_6$ on $Y$, obtained by permuting the marking of the lines. Furthermore, the geometry of marked cubic surfaces, the moduli space $Y$, and its compactifications is intimately connected with the combinatorics of the $E_6$ root system (see \cite{manin1986cubic}). In the KSBA setting, if one weights all $27$ lines uniformly by a rational number $c \in (\frac{1}{9}, 1] \cap \mathbb Q$, Schock  \cite{schock_moduli_2024} constructed an explicit wall-and-chamber decomposition of the weight domain yielding five walls (see \Cref{S:detectingwalls}) at $c = \frac{1}{6}, \frac{1}{4}, \frac{1}{3}, \frac{1}{2}$, and $\frac{2}{3}$, with birational wall-crossing morphisms between the KSBA moduli spaces in different chambers. 

We extend Schock's work by considering moduli spaces of marked cubic surfaces with non-uniform rational weights $(b, c) \in (\frac{1}{9}, 1]^2\cap \mathbb Q^2$ satisfying $b \geq c$, where one line is weighted by $b$ and the remaining $26$ lines are uniformly weighted by $c$. While KSBA wall crossing  \cite{ascher_wall_2023, meng2023mmplocallystablefamilies} for pairs with a uniform weight has been studied extensively, explicit computations in higher-dimensional weight domains remain relatively rare \cite{alexeev2015moduli, alexeev2008weighted, gallardo2025explicit}. In addition to contributing further examples of KSBA wall-crossings and explicit degenerations of weighted marked cubic surfaces, the case considered here provides a concrete two-parameter example arising naturally from the geometry of cubic surfaces. Allowing a separate weight for a distinguished line introduces new wall crossings corresponding to interactions between the distinguished line and the singularities of the cubic surface, giving rise to a moduli space of cubic surfaces with a marked line (\Cref{S:ksba}). In particular, we show that the normalization of the KSBA compactification $\overline{Y}_\ell$ coincides with the toroidal compactification $\overline{\mathbb{B}_4/ \Gamma_\ell}$ (see \cite{doran2004moduli} and \cite{dolgachev2005complex}). The birational geometry of $\overline{\mathbb{B}_4/ \Gamma_\ell}$ was studied extensively by Casalaina-Martin--Grushesky--Hulek \cite{casalaina-martin_birational_2024}. The following is our main result:

	\begin{teo}\label{teo:main}
		The following is a wall-and-chamber decomposition of the weight domain $(b, c) \in \mathrm{Amp}$ (\Cref{eqn:ampcone}) with $b \ge c$. There are precisely 20 chambers where the \textbf{bold lines}
		\begin{align*}
			c = \frac{b}{3}, \hspace{.2in} c = \frac{b}{4}, \hspace{.2in} c = \frac{1}{4} \hspace{.1in} (\frac{1}{2} < b \leq 1), \hspace{.2in} \frac{1}{9} < b=c \leq \frac{1}{6}, \hspace{.2in} \frac{1}{6} < b=c \leq \frac{1}{4}, \hspace{.1in} \text{ and } \hspace{.1in} \frac{1}{4} < b=c \leq \frac{1}{3}
		\end{align*}
		are chambers themselves.
		\begin{center}
			\begin{tikzpicture}[scale=14]
				\draw[very thick,-latex] (1/9, 1/9) -- (1.03, 1/9) 
				node[right]{$b$};
				
				\draw[very thick,-latex] (1/9, 1/9) -- (1/9, 1.03)
				node[left]{$c$};
				
				\draw[solid] (1/9,1/9) -- (1,1);
				\node at (1/9, 1/9) [below] {\tiny $(\frac{1}{9}, \frac{1}{9})$};
				\node at (1, 1) [left] {$(1, 1)$};
				
				\draw[thick] (1, 1/9) -- (1,1);
				\node at (1, 1/9) [below] {$(1, \frac{1}{9})$};
				
				\draw[thick] (2/3, 2/3) -- (1, 2/3) node[right] {\tiny{$c=\frac{2}{3}$}};
				\node[circle, fill, inner sep=1pt] at (2/3, 2/3) {};
				\node at (2/3, 2/3) [left] {\tiny$(\frac{2}{3}, \frac{2}{3})$};
				\node at (8/9, 7/9) {$\overline{\mathcal{Y}}_{(\frac{2}{3}, 1]}$};
				
				\draw[thick] (1/2, 1/2) -- (1, 1/2);
				\node[circle, fill, inner sep=1pt] at (1/2, 1/2) {};
				\node at (1/2, 1/2) [left] {\tiny $(\frac{1}{2}, \frac{1}{2})$};
				\node at (0.7, 0.56) {$\overline{\mathcal{Y}}_{(\frac{1}{2}, \frac{2}{3}]}^{\leq}$};
				
				\node at (0.9, 0.63) {$\overline{\mathcal{Y}}_{(\frac{1}{2}, \frac{2}{3}]}^{>}$};
				\draw[solid] (2/3, 2/3) -- (1, 1/2) node[right] {\color{black} \tiny{$c = \frac{1}{2}$}};
				\node[rotate=-26.6] at (5/6, 57/96) {\tiny wall due to $\ell$ containing Eckardt point(s)};
				
				\draw[thick] (1/2, 1/2) -- (4/5, 1/5);
				\node[rotate=-45] at (7/12, 5/12) [above] {\tiny $c = -b+1$};
				\node[align=center] at (5/6, 0.42) {$\overline{\mathcal{Y}}_{(\frac{1}{3}, \frac{1}{2}]}^>$};
				
				\draw[thick] (1/4, 1/4) -- (7/13, 2/13);
				\node[circle, fill, inner sep=1pt] at (1/4, 1/4) {};
				\node[rotate=-18.43] at (3/8, 24/120) [above] {\tiny$c=-\frac{b}{3} + \frac{1}{3}$};
				\node at (7/13, 2/13) [below] {\tiny{$(\frac{7}{13}, \frac{2}{13})$}};
				\node at (1/4, 1/4) [left] {\tiny$(\frac{1}{4}, \frac{1}{4})$};
				\node at (1.77/6, 4.8/24) {$\overline{\mathcal{Y}}_{(\frac{1}{6}, -\frac{b}{3} + \frac{1}{3}]}$};
				
				\draw[thick] (1/3, 1/3) -- (1,1/3);
				\node at (1, 1/3) [right] {\tiny$c=\frac{1}{3}$};
				\node at (1/3, 1/3) [left] {\tiny$(\frac{1}{3}, \frac{1}{3})$};
				\node[align=center] at (5.2/6, 7/24) {$\overline{\mathcal{Y}}_{(\frac{1}{4}, \frac{1}{3}]}^>$};
				\node[align=center] at (1/2, 4.7/12) {$\overline{\mathcal{Y}}_{(\frac{1}{3}, \frac{1}{2}]}^\leq$};
				\node[align=center] at (1/2, 7/24) {$\overline{\mathcal{Y}}_{(\frac{1}{4}, \frac{1}{3}]}^\leq$};
				
				\draw[thick] (1/4, 1/4) -- (3/4,1/4);
				\node at (1/4, 1/4) [left] {\tiny$(\frac{1}{4}, \frac{1}{4})$};
				
				\draw[thick] (1/6, 1/6) -- (2/3, 1/6);
				\node[circle, fill, inner sep=1pt] at (1/6, 1/6) {};
				\node at (7/36, 7/36) [left] {\tiny$(\frac{1}{6}, \frac{1}{6}$)};
				\node at (2/3, 1/6) [below] {\tiny{$(\frac{2}{3}, \frac{1}{6})$}};
				\draw[->] (7/24, 1/11) -- (1/4, 10/72);
				\node[circle, fill, inner sep=1.5pt] at (1/4, 10/72) [above] {};
				
				\draw[very thick] (1/4, 1/4) -- (1/3, 1/3);
				\draw[->] (1/11, 1/3) -- (7/24, 7/24);
				\node at (1/10, 1/3) [left] {$\overline{\mathcal{Y}}_{(\frac{1}{4}, \frac{1}{3}]}^=$};		
				\node[circle, fill, inner sep=1pt] at (1/4, 1/4) {};
				\node[circle, fill, inner sep=1pt] at (1/3, 1/3) {};
				
				\draw[very thick] (1/6, 1/6) -- (1/4, 1/4);
				\draw[->] (1/11, 1/4) -- (5/24, 5/24);
				\node at (1/10, 1/4) [left] {$\overline{\mathcal{Y}}_{(\frac{1}{6}, \frac{1}{4}]}^=$};
				\node[circle, fill, inner sep=1pt] at (1/6, 1/6) {};
				%
				\draw[thick] (1/6, 1/6) -- (2/3, 1/6);
				\node at (7/36, 7/36) [left] {\tiny$(\frac{1}{6}, \frac{1}{6}$)};
				\draw[->] (7/24, 1/11) -- (1/4, 10/72);
				\node[circle, fill, inner sep=1pt] at (1/4, 10/72) [above] {};
				\node[align=center] at (7/24, 1/20) {$\overline{\mathcal{Y}}_{(\frac{b}{10} + \frac{1}{10}, \frac{1}{6}]}^\leq$};
				%
				\draw[thick] (1/4, 1/4) -- (3/4, 1/4);
				\draw[very thick] (3/4, 1/4) -- (1, 1/4);
				\node[align=center] at (6.5/13, 2.5/12) {$\overline{\mathcal{Y}}_{(-\frac{b}{3} + \frac{1}{3}, \frac{1}{4}]}$};
				\node at (1, 1/4) [right] {\tiny{$c=\frac{1}{4}$}};
				%
				\draw[thick] (4/5, 1/5) -- (1, 1/5);
				%
				\draw[very thick] (1/2, 1/6) -- (3/4, 1/4);
				\node[rotate=18.43] at (2/3, 17/81) [above] {\tiny$c=\frac{b}{3}$};
				%
				\draw[very thick] (2/3, 1/6) -- (4/5, 1/5);
				\node at (2/3, 1/6) [below] {\tiny{$(\frac{2}{3}, \frac{1}{6})$}};
				\node[rotate=14.04] at (22/30, 21/120) [above] {\tiny $c=\frac{b}{4}$};
				
				\draw[->] (1.1, 1/6) -- (5.2/6, 10.5/48);
				\node[circle, fill, inner sep=1.5pt] at (5.15/6, 10.5/48) [above] {};
				\node[align=left] at (1.16, 1/6) {$\overline{\mathcal{Y}}_{(\frac{1}{5}, \frac{1}{4})}^>$};
				
				\draw[->] (1.05, .85/3) -- (8.5/13, 1/5);
				\node[circle, fill, inner sep=1.5pt] at (8.5/13, 1/5) [left] {};
				\node[align=center] at (1.1, .85/3) {$\overline{\mathcal{Y}}_{(\frac{b}{4}, \frac{b}{3})}$};
				
				\draw[->] (7/13, 1/11) -- (7.3/13, 9.4/60);
				\node[circle, fill, inner sep=1.5pt] at (7.3/13, 9.4/60) [above] {};
				\node[align=center] at (7/13, 1/20) {$\overline{\mathcal{Y}}_{(\frac{b}{10} + \frac{1}{10}\frac{1}{6}]}^>$};
				
				\draw[->] (10/13, 1/11) -- (10.5/13, 1.1/6);
				\node[circle, fill, inner sep=1.5pt] at (10.5/13, 1.1/6) [above] {};
				\node[align=center] at (10/13, 1/20) {$\overline{\mathcal{Y}}_{(\frac{b}{10} + \frac{1}{10}, \frac{1}{5}]}$};
				
				\draw[very thick] (1/9, 1/9) -- (1/6, 1/6);
				\draw[->] (1/11, 1/6) -- (5/36, 5/36);
				\node at (1/10, 1/6) [left] {$\overline{\mathcal{Y}}_{(\frac{1}{9}, \frac{1}{6}]}^=$};
				
				\draw[dashed] (1/9, 1/9) -- (1, 1/5);
				\node[rotate=5.71] at (5/6, 11/60) [below] {\tiny$c=b/10 + 1/10$};
				\node at (1, 1/5) [right] {\tiny{$c=\frac{1}{5}$}};
				
			\end{tikzpicture}
		\end{center}
	\end{teo}
	
	The following table gives a detailed description of the $(b, c)$-chamber region for each moduli stack in the wall-and-chamber decomposition in \Cref{teo:main}. The subscript in the notation for the moduli spaces below  refers to bounds on $c$, whereas the superscript indicates that $b$ and $c$ are further constrained by one or more further linear inequalities. Recall that we always have $\frac{1}{9} < b \le 1$.  
	\begin{center}
		\begin{tabular}{ | m{1.5cm} | m {5.5cm} | } 
			\hline
			Moduli & $(b, c)$ chamber region with $b \geq c$\\
			\hline
			$\overline{\mathcal{Y}}_{(\frac{2}{3}, 1]}$& $\frac{2}{3} < c \leq 1$ \\ 
			\hline
			$\overline{\mathcal{Y}}_{(\frac{1}{2}, \frac{2}{3}]}^>$ & $\frac{1}{2} < c \leq \frac{2}{3}$, $c > -\frac{b}{2}+1$ \\ 
			\hline
			$\overline{\mathcal{Y}}_{(\frac{1}{2}, \frac{2}{3}]}^\leq$ & $\frac{1}{2} < c \leq \frac{2}{3}$, $c \leq -\frac{b}{2}+1$ \\ 
			\hline
			$\overline{\mathcal{Y}}_{(\frac{1}{3}, \frac{1}{2}]}^>$ & $\frac{1}{3} < c \leq \frac{1}{2}$, $c > -b+1$ \\ 
			\hline
			$\overline{\mathcal{Y}}_{(\frac{1}{3}, \frac{1}{2}]}^\leq$ & $\frac{1}{3} < c \leq \frac{1}{2}$, $c \leq -b+1$\\
			\hline
			$\overline{\mathcal{Y}}_{(\frac{1}{4}, \frac{1}{3}]}^>$ &$\frac{1}{4} < c \leq \frac{1}{3}$, $c > -b+1$\\
			\hline
			$\overline{\mathcal{Y}}_{(\frac{1}{4}, \frac{1}{3}]}^\leq$ &$\frac{1}{4} < c \leq \frac{1}{3}$, $c \leq -b+1$, $b \neq c$\\
			\hline
			$\overline{\mathcal{Y}}_{(\frac{1}{4}, \frac{1}{3}]}^=$ & $\frac{1}{4} < b=c \leq \frac{1}{3}$\\
			\hline
			$\overline{\mathcal{Y}}_{c = \frac{1}{4}}$ & $c = \frac{1}{4}$, $\frac{3}{4} < b \leq 1$\\
			\hline
			$\overline{\mathcal{Y}}_{(\frac{1}{5}, \frac{1}{4})}^>$ & $\frac{1}{5} < c < \frac{1}{4}$, $c > - b +1$\\
			\hline
		\end{tabular}
		\hspace{.5cm}
		\begin{tabular}{ | m{1.5cm} | m {5.5cm} | } 
			\hline
			Moduli & $(b, c)$ chamber region with $b \geq c$\\
			\hline
			$\overline{\mathcal{Y}}_{(\frac{b}{10} + \frac{1}{10}, \frac{1}{5}]}$ & $\frac{b}{10} + \frac{1}{10} < c \leq \frac{1}{5}$, $c < \frac{b}{4}$\\
			\hline
			$\overline{\mathcal{Y}}_{c=\frac{b}{4}}$& $c=\frac{b}{4}$, $\frac{2}{3} < b \leq \frac{4}{5}$ \\ 
			\hline
			$\overline{\mathcal{Y}}_{(\frac{b}{4}, \frac{b}{3})}$ & $\frac{b}{4} < c < \frac{b}{3}$, $\frac{1}{6} < c \leq -b+1$\\
			\hline
			$\overline{\mathcal{Y}}_{c = \frac{b}{3}}$ & $c = \frac{b}{3}$, $\frac{1}{2} < c \leq \frac{3}{4}$\\
			\hline
			$\overline{\mathcal{Y}}_{(-\frac{b}{3} + \frac{1}{3}, \frac{1}{4}]}$ & $-\frac{b}{3} + \frac{1}{3} < c \leq \frac{1}{4}$, $c > \frac{b}{3}$\\
			\hline
			$\overline{\mathcal{Y}}_{(\frac{1}{6}, -\frac{b}{3} + \frac{1}{3}]}$ & $\frac{1}{6} < c \leq - \frac{b}{3} + \frac{1}{3}$, $b \neq c$\\
			\hline
			$\overline{\mathcal{Y}}_{(\frac{1}{6}, \frac{1}{4}]}^=$ & $\frac{1}{6} < b =c \leq \frac{1}{4}$\\
			\hline
			$\overline{\mathcal{Y}}_{(\frac{b}{10} + \frac{1}{10}, \frac{1}{6}]}^>$ & $\frac{b}{10} + \frac{1}{10} < c \leq \frac{1}{6}$, $c > - \frac{b}{3} + \frac{1}{3}$\\
			\hline
			$\overline{\mathcal{Y}}_{(\frac{b}{10} + \frac{1}{10}, \frac{1}{6}]}^\leq$ & $\frac{b}{10} + \frac{1}{10} < c \leq \frac{1}{6}$, $c \leq - \frac{b}{3} + \frac{1}{3}$, $b \neq c$\\
			\hline
			$\overline{\mathcal{Y}}_{(\frac{1}{9}, \frac{1}{6}]}^=$ & $\frac{1}{9} < b = c \leq \frac{1}{6}$\\
			\hline
		\end{tabular}
	\end{center}
	
	\begin{teo}\label{cor:mainmorphisms}
		For the wall-and-chamber decomposition of \Cref{teo:main}, we obtain wall-crossing morphisms of the moduli stacks, which induce the following wall-crossing morphisms of coarse moduli spaces
\[\begin{tikzcd}
	{\overline{Y}_{(\frac{1}{9},\frac{1}{6}]}^=} & {\overline{Y}_{(\frac{1}{6},\frac{1}{4}]}^=} && {\overline{Y}_{(\frac{1}{4},\frac{1}{3}]}^=} & {\overline{Y}_{(\frac{1}{3},\frac{1}{2}]}^\leq} & {\overline{Y}_{(\frac{1}{2},\frac{2}{3}]}^\leq} & {\overline{Y}_{(\frac{2}{3},1]}} \\
	{\overline{Y}_{(\frac{b}{10}+\frac{1}{10},\frac{1}{6}]}^\leq} & {\overline{Y}_{(\frac{1}{6}, -\frac{b}{3}+\frac{1}{3}]}} && {\overline{Y}_{(\frac{1}{4},\frac{1}{3}]}^\leq} & {\overline{Y}_{(\frac{1}{3},\frac{1}{2}]}^>} & {\overline{Y}_{(\frac{1}{2},\frac{2}{3}]}^>} \\
	{\overline{Y}_{(\frac{b}{10}+\frac{1}{10},\frac{1}{6}]}^>} & {\overline{Y}_{(-\frac{b}{3}+\frac{1}{3}, \frac{1}{4}]}} && {\overline{Y}_{(\frac{1}{4},\frac{1}{3}]}^>} \\
	& {\overline{Y}_{c=\frac{b}{3}}} & {\overline{Y}_{c=\frac{1}{4}}} \\
	& {\overline{Y}_{(\frac{b}{4}, \frac{b}{3})}} & {\overline{Y}_{(\frac{1}{5}, \frac{1}{4})}^>} \\
	& {\overline{Y}_{c=\frac{b}{4}}} \\
	& {\overline{Y}_{(\frac{b}{10} + \frac{1}{10}, \frac{1}{5})}}
	\arrow[from=1-2, to=1-1]
	\arrow["\times"{description}, from=1-4, to=1-2]
	\arrow[from=1-5, to=1-4]
	\arrow[from=1-5, to=2-4]
	\arrow[from=1-6, to=1-5]
	\arrow[from=1-6, to=2-5]
	\arrow[from=1-7, to=1-6]
	\arrow[from=1-7, to=2-6]
	\arrow[from=2-1, to=1-1]
	\arrow[from=2-2, to=1-2]
	\arrow[from=2-2, to=2-1]
	\arrow[from=2-4, to=1-4]
	\arrow["\times"{description}, from=2-4, to=3-2]
	\arrow[from=2-5, to=1-5]
	\arrow[from=2-5, to=2-4]
	\arrow[from=2-5, to=3-4]
	\arrow[from=2-6, to=1-6]
	\arrow[from=2-6, to=2-5]
	\arrow["\times"{description}, from=3-1, to=2-1]
	\arrow["\times"{description}, from=3-2, to=2-1]
	\arrow["\times"{description}, from=3-2, to=2-2]
	\arrow[from=3-4, to=2-4]
	\arrow["\times"{description}, from=3-4, to=4-2]
	\arrow["\times"{description}, from=3-4, to=4-3]
	\arrow[from=4-2, to=3-2]
	\arrow[from=4-3, to=5-3]
	\arrow[from=5-2, to=3-1]
	\arrow[from=5-2, to=4-2]
	\arrow[from=5-3, to=5-2]
	\arrow[from=5-3, to=6-2]
	\arrow[from=5-3, to=7-2]
	\arrow[from=6-2, to=5-2]
	\arrow[from=7-2, to=6-2]
\end{tikzcd}\]
		where each morphism is an isomorphism except for the ones labeled by an $\times$.
	\end{teo}

In order to describe the birational morphisms among the three non-isomorphic (coarse) moduli spaces, we recall from \cite{casalaina-martin_birational_2024} some of the geometry of $\overline{\mathbb{B}_4/\Gamma_\ell}$. Recall the boundary divisors $D_{A_1, \ell}^{\text{in}}$ and $D_{A_1, \ell}^{\text{out}}$ of $\overline{\mathbb{B}_4/ \Gamma_\ell}$ given by the closure of the locus of cubic surfaces with a unique node and the marked line $\ell$ containing and not containing the node, respectively. As a convention, we denote by $D_{2A_1, \ell}^{\text{in}}$ as the self-intersection of $D_{A_1, \ell}^{\text{in}}$. Analogously, for $i = 2, 3, 4$, we denote by $D_{iA_1, \ell}^{\text{out}}$ as the self-intersection of $i$ $D_{A_1, \ell}^{\text{out}}$ divisors. The following table gives a list of the $A_1$ boundary strata of $\overline{\mathbb{B}_4/\Gamma_\ell}$.
\begin{center}
\begin{tabular}{ | m{1.1cm} | m {7cm} | } 
	\hline
	Codim & Boundary strata\\
	\hline
	\hline
	1 & $D_{A_1, \ell}^{\text{in}}$ and $D_{A_1, \ell}^{\text{out}}$\\
	\hline
	2 & $D_{2A_1, \ell}^{\text{in}}$, $D_{A_1, \ell}^{\text{in}} \cap D_{A_1, \ell}^{\text{out}}$, and $D_{2A_1, \ell}^{\text{out}}$\\
	\hline
	3 & $D_{2A_1, \ell}^{\text{in}} \cap D_{A_1, \ell}^{\text{out}}$, $D_{A_1, \ell}^{\text{out}} \cap D_{2A_1, \ell}^{\text{out}}$, and $D_{3A_1, \ell}^{\text{out}}$\\
	\hline
	4 & $D_{2A_1, \ell}^{\text{in}} \cap D_{2A_1, \ell}^{\text{out}}$ and $D_{4A_1, \ell}^{\text{out}}$\\
	\hline
\end{tabular}
\end{center}
	
	\begin{teo}\label{cor:mainmoduli}
		In the wall-and-chamber decomposition of \Cref{teo:main}, there are three non-isomorphic coarse moduli spaces with birational morphisms
		\begin{align*}
			\widetilde{Y}_\ell \to \ddot Y_\ell \to \overline{Y}_\ell
		\end{align*}
		where the spaces and morphisms can be described as follows:
		\begin{enumerate}
			\item The morphism $\ddot Y_\ell \to \overline{Y}_\ell$ is the blow up along the strict transforms of the intersections of $D_{2A_1, \ell}^{in}$ with other $A_1$ boundary strata, described above, in increasing order of dimension.
			\item The morphism $\widetilde{Y}_\ell \to \ddot{Y}_\ell$ is the blow up along the strict transforms of the intersections of the remaining $A_1$ boundary strata, in increasing order of dimension.
		\end{enumerate}
		The composition $\widetilde{Y}_\ell \to \overline{Y}_\ell$ is the blow up along the strict transforms of the intersections of $A_1$ boundary strata, in increasing order of dimension.
	\end{teo}
	
	\subsection{Outline} In \Cref{S:moduliproblems} and \Cref{S:ksba} we introduce the moduli problems we consider and their KSBA compactifications.
	
	In \Cref{S:univfam}, we give a summary of Hacking--Keel--Tevelev's \cite{hacking_stable_2009} construction of the weight 1 (fully) marked compactification $\overline{Y}_{(1, \dots, 1)}$ and its universal family $\ddot{Y}(E_7) \to \overline{Y}_{(1, \dots 1)}$. We also summarize Schock's \cite{schock_moduli_2024} explicit computation of the general fibers of $\ddot{Y}(E_7) \to \overline{Y}_{(1, \dots, 1)}$ over each boundary stratum of $\overline{Y}_{(1, \dots 1)}$. 
	
	Sections \ref{S:modulichange} and \ref{S:gluing} establish technical lemmas needed for the proofs of the main results. In particular, we show that the naive approach to constructing stable replacements at each wall-crossing, by working on each irreducible component separately, and then gluing the irreducible components obtained by taking the appropriate contractions, is indeed the same as the log canonical model under certain base-point free conditions. 
	
	In \Cref{S:wallandchamberbytype}, we organize the wall-and-chamber decomposition of \Cref{teo:main} by considering boundary strata in $\overline{Y}_{(1, \dots, 1)}$ determined by the number of singularities, as well as the number of singularities the marked line $\ell$ goes through.
	
	Finally, in \Cref{S:proofofchambers}, we run the log minimal model program on the fibers $(S, (b, c)B)$ via the induced forgetful morphism from the universal family over $\overline{Y}_{(1, \dots, 1)}$ to the universal family over $\overline{Y}_{(1, 1)}$ as one decreases $(b, c)$ from $(1, 1)$ toward the minimal weight $(1/9 + \epsilon, 1/9 + \epsilon)$ for a fixed marked line $\ell$. This will give us the wall-and-chamber decomposition as well as the corresponding descriptions of the stable models parameterized by the moduli spaces in each chamber for this case of marked line $\ell$, and we repeat this process for every possible case.
	
	\subsection*{Acknowledgements} I would first like to thank my advisor, Sebastian Casalaina-Martin, for introducing and guiding me through this project. I would also like to thank Patricio Gallardo, Kaden Saucedo, Luca Schaffler, and Nolan Schock, for their very helpful conversations, feedback, and encouragement.
	
	\section{The smooth moduli problems and existing compactifications}\label{S:moduliproblems}
In this section, we will establish the moduli problems of smooth cubic surfaces that are unmarked, marked, or marked by a single line $\ell$, together with a brief survey of their compactifications. Namely, let $\mathcal{Y}_u$, $\mathcal{Y}_m$, and $\mathcal{Y}_\ell$ denote the moduli stacks of smooth cubic surfaces that are unmarked, fully marked, and singly marked, respectively. Similarly, let $Y_u$, $Y_m$, and $Y_\ell$ denote the corresponding coarse moduli spaces of smooth cubic surfaces that are unmarked, fully marked, and singly marked. We also call $\mathcal{Y}_\ell$ (resp. $Y_\ell$) the moduli stack (resp. moduli space) of smooth cubic surfaces with a marked line $\ell$.

\subsection{The moduli problem $\mathcal{Y}_u$}
The objects of this stack are smooth proper families $X \to T$, where $T$ is a reduced scheme over $\mathbb{C}$, such that for every geometric point $t \in T$, the fiber $X_t = X \times_T \mathrm{Spec}k(t)$ is a smooth cubic surface. Naturally, the morphisms are given by pullbacks. We also have a natural forgetful morphism $\mathcal{Y}_u \to \mathfrak{Sch}/\mathbb{C}$, where $\mathfrak{Sch}/\mathbb{C}$ is the stack of schemes over $\mathbb{C}$.

\subsection{The moduli problem $\mathcal{Y}_m$}
To define this moduli problem, we first clarify what we mean by a marked cubic surface. Following the discussion in \cite[\S~2]{dolgachev2005complex}, given a cubic surface $S$, there are two notions of markings: \textit{geometric} and \textit{lattice markings}.

\begin{dfn}(Geometric marking)
	Given a cubic surface, its minimal resolution $S$ admits a birational morphism $\pi : S \to \mathbb{P}^2$ whose factorization $\pi = \pi_6 \circ \cdots \circ \pi_1$ is called a \textit{geometric marking}. Two geometric markings $S = S_0 \to S_1 \to \cdots \to S_6 = \mathbb{P}^2$ and $S' = S_0' \to S_1' \to \cdots \to S_6' = \mathbb{P}^2$ are \textit{isomorphic} if there exist isomorphisms $\phi_i : S_i \to S_i'$ for each $i = 0, \dots, 6$, such that $\pi'_{i+1} \circ \phi_i = \phi_{i+1} \circ \pi_{i+1}$ for $i = 0, \dots, 5$.
\end{dfn}

\begin{dfn}(Lattice marking)\label{dfn:latticemarking}
	Let $I_{1, 6}$ denote the standard odd unimodular hyperbolic lattice of signature $(1, 6)$ with $k = -3 \epsilon_0 + \epsilon_1 + \cdots + \epsilon_6$. A \textit{lattice marking} for a cubic surface $S$ is given by an isometry $\phi : I_{1, 6} \to \mathrm{Pic}(S)$ such that $\phi(k) = K_S$. In particular, the restriction of $\phi$ to $k^\perp$ is an isometry $k^\perp \to K_S^\perp$. Using the identification of the root lattice $E_6$ with the sublattice $k^\perp \subset \mathrm{Pic}(S)$, its Weyl group $W(E_6)$ is the subgroup of the orthogonal group of $\mathrm{Pic}(S)$ generated by reflections in the classes of the $(-2)$ curves on $S$. Equivalently, it is the subgroup of automorphisms of $I_{1, 6}$ preserving $k$ and the intersection pairing (see \cite[Theorem 23.9]{manin1986cubic}). Two lattice markings $\phi, \phi'$ are \textit{equivalent} if there exists a $\sigma \in W(E_6)$ such that $\phi = \sigma \circ \phi$.  
\end{dfn}

Given a geometric marking $(S, \pi)$, the ordered sequence of blow ups $\pi = \pi_6 \circ \cdots \circ \pi_1$ defines an ordered basis $(e_0, e_1, \dots, e_6)$ for $\mathrm{Pic}(S)$ where $e_0$ is the divisor class of the pullback of the hyperplane class in $\mathbb{P}^2$ and $e_i$ for $1 \leq i \leq 6$ is the exceptional divisor class obtained by blowing up the $i$-th point in $\mathbb{P}^2$. The divisor classes in $S$ corresponding to the 27 lines have unique representations as $\mathbb{Z}$-linear combinations of these basis elements, e.g. $e_i, e_0 - e_i - e_j $, etc. Furthermore, this isomorphism class of a geometric marking defines an equivalence class of a lattice marking by defining 
\begin{align*}
	\phi : I_{1, 6} \to \mathrm{Pic}(S) \hspace{.2in} \phi(\epsilon_i) = e_i,
\end{align*}
and there is a natural bijection between the isomorphism classes of geometric markings and equivalence classes of lattice markings on $S$ (see \cite[Prop.~2.4]{dolgachev2005complex}, \cite[Thm.~4.6]{looijenga1981rational})

We now define the moduli problem $\mathcal{Y}_m$ as follows. Its objects are families
 $$\pi : (X, \mathrm{Pic}_{X/T}) \to T$$
where each geometric fiber is the pair $(X_t, \mathrm{Pic}(X_t))$ composed of a cubic surface $X_t$ and its Picard group with its intersection pairing. In particular, note that $\mathrm{Pic}_{X/T} \cong R^2 \pi_* \underline{\mathbb{Z}}$ is a locally constant sheaf and we have the intersection pairing
\begin{align*}
	R^2 \pi_* \underline{\mathbb{Z}} \times R^2 \pi_* \underline{\mathbb{Z}} \to R^4 \pi_* \underline{\mathbb{Z}}.
\end{align*}
Equivalently, the objects of this moduli problem are families $(X, L_1, \dots, L_{27}) \to T$ where each fiber is made up of a cubic surface $X_t$ and a labeling of its 27 lines $L_{1, t}, \dots, L_{27, t}$ (including multiplicity) with their intersections.

\subsection{The moduli problem $\mathcal{Y}_\ell$}
The objects of this stack are given by $(X, \ell)\to T$ where each fiber $(X_t,  \ell_t)$ is a pair of a cubic surface $X_t$ and a marked line $\ell_t$. To specify the marked line $\ell$, by our discussion of $\mathcal{Y}_m$ above, we first take a full marking $(X, L_1, \dots, L_{27})$ of $X$ and choose $\ell = L_1$. This choice of marked line is dependent on our full marking, so we have a chain of natural forgetful morphisms
\begin{align*}
	\mathcal{Y}_m \to \mathcal{Y}_\ell \to \mathcal{Y}_u
\end{align*}

\subsection{Existing compactifications of the moduli of marked cubic surfaces}\label{S:existingcpt}
To place our results in context, we will do a brief literature survey of the compactifications of the moduli space of unmarked, marked, and singly marked cubic surfaces.

\subsubsection{The GIT compactification}
Recall the GIT compactification of the moduli space of unmarked cubic surfaces. We refer the reader to \cite{oxbury_introduction_2003} for an introduction to GIT and particularly Section 7.2. for details on the moduli space of  cubic surfaces. 
\begin{align*}
	\overline{Y}_u^{\mathrm{GIT}} := |\mathcal{O}_{\mathbb{P}^3}(3)| /\!\!/_{\mathcal{O}_{\mathbb{P}^{19}}(1)} \mathrm{SL}(4, \mathbb{C})
\end{align*}
By explicit computations dating back to Cayley and Salmon \cite{salmon1912treatise}, there are homogeneous polynomials $I_8, I_{16}, I_{24}, I_{32}, I_{40}$
of degrees 8, 16, 24, 32, 40, respectively, generating the ring of $\mathrm{SL}(4, \mathbb{C})$ invariant homogeneous polynomials on $|\mathcal{O}_{\mathbb{P}^3}(3)|$, so we get an alternative description of the moduli space
\begin{align*}
	\overline{Y}_u^{\mathrm{GIT}} = \mathbb{P}(8, 16, 24, 32, 40) = \mathbb{P}(1, 2, 3, 4, 5)
\end{align*}
as a weighted projective space. The (semi)stable locus consists of cubic surfaces with at worst $A_1$ or $A_2$ singularities with a single isolated boundary point corresponding to the GIT polystable $3A_2$ cubic surface given by $x_0 x_1 x_2 + x_3^3 = 0$.

Following \cite[\S~2.8]{dolgachev2005complex}, we can modify this GIT quotient to construct the GIT compactification of $Y_m$. Let $\mathbb{C}(Y_m)$ be the field of rational functions of $Y_m$, which is an extension of the field of rational functions $\mathbb{C}(\overline{Y}_u^{\mathrm{GIT}})$ with Galois group $W(E_6)$. Define the compactification $\overline{Y}_m^\mathrm{GIT}$ to be the normalization of $\overline{Y}_u^\mathrm{GIT}$ in $\mathbb{C}(Y_m)$. By construction, $W(E_6)$ acts on $\overline{Y}_m^\mathrm{GIT}$ and the quotient by this action recovers $\overline{Y}_u^\mathrm{GIT}$. Thus, we have a natural $51840 : 1$ cover
\begin{align*}
	\overline{Y}_m^{\mathrm{GIT}} \to \overline{Y}_u^{\mathrm{GIT}}
\end{align*}
with 40 singular points parameterizing different markings of the unique GIT polystable $3A_2$ cubic surface. The 40 singular points are identified with the vertex of the cone over the Segre embedding of $(\mathbb{P}^1)^3$ into $\mathbb{P}^7$ \cite[\S 2.12]{dolgachev2005complex}.

\subsubsection{Naruki's cross-ratio variety}\label{S:naruki}
Related to the GIT compactification of $Y_m$,  we have Naruki's compactification $\overline{N}$ and his family of cubic surfaces over it introduced in \cite{naruki_cross_1982}. Namely, from the cover
\begin{align*}
	\overline{Y}_m^{\mathrm{GIT}} \to \overline{Y}_u^{\mathrm{GIT}},
\end{align*}
the Naruki compactification $\overline{N}$ is then obtained by blowing up the 40 singular points
\begin{align*}
	\overline{N} \to \overline{Y}_m^{\mathrm{GIT}} \to \overline{Y}_u^{\mathrm{GIT}}.
\end{align*} 
The Naruki compactification is smooth, and its boundary $\overline{N} - Y_m$ is normal crossing and decomposes into the union of 76 irreducible divisors. 36 of these divisors are associated to marked cubic surfaces with $A_1$ singularities called \textit{Type $A_1$} \cite{naruki_cross_1982}[Prop.~11.1]. Cubic surfaces with multiple $A_1$ singularities define higher-codimension strata called \textit{Type} $2A_1$, $3A_1$, or $4A_1$. The remaining 40 divisors are pairwise disjoint and are isomorphic to $(\mathbb{P}^1)^3$; these are called \textit{Type $3A_2$} divisors \cite[Prop.~11.2]{naruki_cross_1982}. 

Furthermore, there is a family of cubic surfaces over $\overline{N}$ constructed by Naruki in \cite{naruki_cross_1982} by modifying the family of cubic surfaces by Cayley in \cite[pg.~446]{cayley1849triple}. In particular,  let $(X: Y : Z : W)$ denote the coordinates of $\mathbb{P}^3$ and $(\lambda, \mu, \nu, \rho)$ denote the coordinates of the torus $T = (\mathbb{C})^4$. Then define the family $\mathcal{S} \subseteq \mathbb{P}^3 \times T$ over $T$ of cubic surfaces defined by the following equation found in \cite[Eq.~(5.1)]{naruki_cross_1982}:
\begin{align*}
	\rho W &\Big (\lambda X^2 + \mu Y^2 + \nu Z^2 + (\rho -1)^2 (\lambda \mu \nu \rho - 1)^2 W^2 + (\mu \nu + 1) Y Z + (\lambda \nu + 1) X Z + (\lambda \mu + 1) XY\\
	& - (\rho - 1) (\lambda \mu \nu \rho - 1) W ((\lambda + 1) X + (\mu + 1) Y + (\nu + 1) Z)\Big ) + XYZ = 0.
\end{align*}
The family $\mathcal{S} \to T$ extends to a family of cubic surfaces $\overline{\mathcal{S}} \to \overline{N}$ by \cite[Prop.~12.2]{naruki_cross_1982}.  The fibers over a general point in each boundary strata is given in \Cref{tab:narukifibers}.

\begin{table}
	\centering
	\caption{
		 This table is copied from \cite[Table 1]{gallardo_geometric_2021}. Fibers of the family $\overline{S} \to \overline{N}$ over a general point in each boundary strata of $\overline{N}$ that also shows the degenerations of the 27 lines. Solid lines have multiplicity $1$, dashed lines have multiplicity $2$, and dotted lines have multiplicity $4$. For the fibers over the boundary of Type $i A_1$ for $i = 1, 2, 3, 4$, lines with multiplicity $1$ are not drawn. Similar conventions are used in \Cref{conv:descpresent}.
	}
	\label{tab:narukifibers}
	\renewcommand{\arraystretch}{1.4}
	\begin{tabular}{|>{\centering\arraybackslash}m{7.8cm}|>{\centering\arraybackslash}m{7.8cm}|}
		\hline
		Type $A_1$
		&
		Type $3A_2$
		\\
		\hline
		\begin{tikzpicture}[scale=0.5]
			
			\draw[dashed,line width=.5pt] (0,0) -- (-1.4,.5);
			\draw[dashed,line width=.5pt] (0,0) -- (-1,1);
			\draw[dashed,line width=.5pt] (0,0) -- (-.7,1.5);
			\draw[dashed,line width=.5pt] (0,0) -- (1.4,.5);
			\draw[dashed,line width=.5pt] (0,0) -- (1,1);
			\draw[dashed,line width=.5pt] (0,0) -- (.7,1.5);
			
			\fill (0,0) circle (5pt);
			
			\node at (0,-0.7) {$A_1$};

		\end{tikzpicture}
		&
		\begin{tikzpicture}[scale=0.5]
			
			\draw[line width=1pt] (0,0) -- (0,4);
			\draw[line width=1pt] (0,0) -- (-4,-3);
			\draw[line width=1pt] (0,0) -- (4,-3);
			
			\draw[line width=.5pt] (0,1) -- (1,-3/4);
			\draw[line width=.5pt] (0,1) -- (2,-3/2);
			\draw[line width=.5pt] (0,1) -- (3,-9/4);
			\draw[line width=.5pt] (0,2) -- (1,-3/4);
			\draw[line width=.5pt] (0,2) -- (2,-3/2);
			\draw[line width=.5pt] (0,2) -- (3,-9/4);
			\draw[line width=.5pt] (0,3) -- (1,-3/4);
			\draw[line width=.5pt] (0,3) -- (2,-3/2);
			\draw[line width=.5pt] (0,3) -- (3,-9/4);
			
			\draw[line width=.5pt] (0,1) -- (-1,-3/4);
			\draw[line width=.5pt] (0,1) -- (-2,-3/2);
			\draw[line width=.5pt] (0,1) -- (-3,-9/4);
			\draw[line width=.5pt] (0,2) -- (-1,-3/4);
			\draw[line width=.5pt] (0,2) -- (-2,-3/2);
			\draw[line width=.5pt] (0,2) -- (-3,-9/4);
			\draw[line width=.5pt] (0,3) -- (-1,-3/4);
			\draw[line width=.5pt] (0,3) -- (-2,-3/2);
			\draw[line width=.5pt] (0,3) -- (-3,-9/4);
			
			\draw[line width=.5pt] (-1,-3/4) -- (1,-3/4);
			\draw[line width=.5pt] (-1,-3/4) -- (2,-3/2);
			\draw[line width=.5pt] (-1,-3/4) -- (3,-9/4);
			\draw[line width=.5pt] (-2,-3/2) -- (1,-3/4);
			\draw[line width=.5pt] (-2,-3/2) -- (2,-3/2);
			\draw[line width=.5pt] (-2,-3/2) -- (3,-9/4);
			\draw[line width=.5pt] (-3,-9/4) -- (1,-3/4);
			\draw[line width=.5pt] (-3,-9/4) -- (2,-3/2);
			\draw[line width=.5pt] (-3,-9/4) -- (3,-9/4);

		\end{tikzpicture}
		\\
		\hline
		\hline
		Type $2A_1$
		&
		Type $A_1 \cap 3A_2$
		\\
		\hline
		\begin{tikzpicture}[scale=0.5]
			
			\draw[dotted,line width=.5pt] (-3,0) -- (3,0);
			
			\draw[dashed,line width=.5pt] (0+1.5,0) -- (-1+1.5,1);
			\draw[dashed,line width=.5pt] (0+1.5,0) -- (-.7+1.5,1.5);
			\draw[dashed,line width=.5pt] (0+1.5,0) -- (1+1.5,1);
			\draw[dashed,line width=.5pt] (0+1.5,0) -- (.7+1.5,1.5);
			
			\draw[dashed,line width=.5pt] (0-1.5,0) -- (-1-1.5,1);
			\draw[dashed,line width=.5pt] (0-1.5,0) -- (-.7-1.5,1.5);
			\draw[dashed,line width=.5pt] (0-1.5,0) -- (1-1.5,1);
			\draw[dashed,line width=.5pt] (0-1.5,0) -- (.7-1.5,1.5);
			
			\fill (1.5,0) circle (5pt);
			\fill (-1.5,0) circle (5pt);
			
			\node at (1.5,-0.7) {$A_1$};
			\node at (-1.5,-0.7) {$A_1$};
			
		\end{tikzpicture}
		&
		\begin{tikzpicture}[scale=0.5]
			
			\draw[line width=1pt] (0,0) -- (0,4);
			\draw[line width=1pt] (0,0) -- (-4,-3);
			\draw[line width=1pt] (0,0) -- (4,-3);
			
			\draw[line width=.5pt] (0,1) -- (1,-3/4);
			\draw[line width=.5pt] (0,1) -- (2,-3/2);
			\draw[line width=.5pt] (0,1) -- (3,-9/4);
			\draw[line width=.5pt] (0,2) -- (1,-3/4);
			\draw[line width=.5pt] (0,2) -- (2,-3/2);
			\draw[line width=.5pt] (0,2) -- (3,-9/4);
			\draw[line width=.5pt] (0,3) -- (1,-3/4);
			\draw[line width=.5pt] (0,3) -- (2,-3/2);
			\draw[line width=.5pt] (0,3) -- (3,-9/4);
			
			\draw[line width=.5pt] (0,1) -- (-1,-3/4);
			\draw[dashed,line width=.5pt] (0,1) -- (-3,-9/4);
			\draw[line width=.5pt] (0,2) -- (-1,-3/4);
			\draw[dashed,line width=.5pt] (0,2) -- (-3,-9/4);
			\draw[line width=.5pt] (0,3) -- (-1,-3/4);
			\draw[dashed,line width=.5pt] (0,3) -- (-3,-9/4);
			
			\draw[line width=.5pt] (-1,-3/4) -- (1,-3/4);
			\draw[line width=.5pt] (-1,-3/4) -- (2,-3/2);
			\draw[line width=.5pt] (-1,-3/4) -- (3,-9/4);
			\draw[dashed,line width=.5pt] (-3,-9/4) -- (1,-3/4);
			\draw[dashed,line width=.5pt] (-3,-9/4) -- (2,-3/2);
			\draw[dashed,line width=.5pt] (-3,-9/4) -- (3,-9/4);

		\end{tikzpicture}
		\\
		\hline
		\hline
		Type $3A_1$
		&
		Type $2A_1 \cap 3A_2$
		\\
		\hline
		\begin{tikzpicture}[scale=0.5]
			
			\draw[dotted,line width=.5pt] (0,4) -- (-2,0);
			\draw[dotted,line width=.5pt] (0,4) -- (2,0);
			\draw[dotted,line width=.5pt] (-2,0) -- (2,0);
			
			\draw[dashed,line width=.5pt] (0,4) -- (-.3,2.5);
			\draw[dashed,line width=.5pt] (0,4) -- (.3,2.5);
			
			\draw[dashed,line width=.5pt] (-2,0) -- (-.7,.5);
			\draw[dashed,line width=.5pt] (-2,0) -- (-1,1);
			
			\draw[dashed,line width=.5pt] (2,0) -- (.7,.5);
			\draw[dashed,line width=.5pt] (2,0) -- (1,1);
			
			\fill (-2,0) circle (5pt);
			\fill (2,0) circle (5pt);
			\fill (0,4) circle (5pt);
			
			\node at (-2.7,0) {$A_1$};
			\node at (2.7,0) {$A_1$};
			\node at (0,4.5) {$A_1$};
			
		\end{tikzpicture}
		&
		\begin{tikzpicture}[scale=0.5]
			
			\draw[line width=1pt] (0,0) -- (0,4);
			\draw[line width=1pt] (0,0) -- (-4,-3);
			\draw[line width=1pt] (0,0) -- (4,-3);
			
			\draw[line width=.5pt] (0,1) -- (1,-3/4);
			\draw[dashed,line width=.5pt] (0,1) -- (3,-9/4);
			\draw[line width=.5pt] (0,2) -- (1,-3/4);
			\draw[dashed,line width=.5pt] (0,2) -- (3,-9/4);
			\draw[line width=.5pt] (0,3) -- (1,-3/4);
			\draw[dashed,line width=.5pt] (0,3) -- (3,-9/4);
			
			\draw[line width=.5pt] (0,1) -- (-1,-3/4);
			\draw[dashed,line width=.5pt] (0,1) -- (-3,-9/4);
			\draw[line width=.5pt] (0,2) -- (-1,-3/4);
			\draw[dashed,line width=.5pt] (0,2) -- (-3,-9/4);
			\draw[line width=.5pt] (0,3) -- (-1,-3/4);
			\draw[dashed,line width=.5pt] (0,3) -- (-3,-9/4);
			
			\draw[line width=.5pt] (-1,-3/4) -- (1,-3/4);
			\draw[dashed,line width=.5pt] (-1,-3/4) -- (3,-9/4);
			\draw[dashed,line width=.5pt] (-3,-9/4) -- (1,-3/4);
			\draw[dotted,line width=.5pt] (-3,-9/4) -- (3,-9/4);
			
		\end{tikzpicture}
		\\
		\hline
		\hline
		Type $4A_1$
		&
		Type $3A_1 \cap 3A_2$
		\\
		\hline
		\begin{tikzpicture}[scale=0.5]
			
			\draw[dotted,line width=0.5 pt] (-3,0) -- (3,0);
			\draw[dotted,line width=0.5 pt] (-3,3) -- (3,3);	
			\draw[dotted,line width=0.5 pt] (-2,-1) -- (-2,4);
			\draw[dotted,line width=0.5 pt] (2,-1) -- (2,4);	
			\draw[dotted,line width=0.5 pt] (-2,0) -- (2,3);	
			\draw[dotted,line width=0.5 pt] (2,0) -- (-2,3);	
			
			\fill (-2,0) circle (5pt);
			\fill (2,0) circle (5pt);
			\fill (-2,3) circle (5pt);
			\fill (2,3) circle (5pt);
			
			\node at (-2.7,-0.5) {$A_1$};
			\node at (2.7,-0.5) {$A_1$};
			\node at (2.7,3.5) {$A_1$};
			\node at (-2.7,3.5) {$A_1$};
			
		\end{tikzpicture}
		&
		\begin{tikzpicture}[scale=0.5]
			
			\draw[line width=.5pt] (-1,-3/4) -- (0,1);
			\draw[line width=1pt] (0,0) -- (0,4);
			\draw[line width=1pt] (0,0) -- (-4,-3);
			\draw[line width=1pt] (0,0) -- (4,-3);
			
			\draw[line width=.5pt] (1,-3/4) -- (0,1);
			\draw[dashed,line width=.5pt] (0,1) -- (3,-9/4);
			\draw[dashed,line width=.5pt] (0,3) -- (1,-3/4);
			\draw[dotted,line width=.5pt] (0,3) -- (3,-9/4);
			
			\draw[dashed,line width=.5pt] (0,1) -- (-3,-9/4);
			\draw[dashed,line width=.5pt] (0,3) -- (-1,-3/4);
			\draw[dotted,line width=.5pt] (0,3) -- (-3,-9/4);
			
			\draw[line width=.5pt] (-1,-3/4) -- (1,-3/4);
			\draw[dashed,line width=.5pt] (-1,-3/4) -- (3,-9/4);
			\draw[dashed,line width=.5pt] (-3,-9/4) -- (1,-3/4);
			\draw[dotted,line width=.5pt] (-3,-9/4) -- (3,-9/4);
			
		\end{tikzpicture}
		\\
		\hline
	\end{tabular}
\end{table}

\subsubsection{The Satake-Baily-Borel and toroidal compactification}\label{SS:toroidalcpt}
For a Hodge-theoretic compactification of $Y_m$, $Y_\ell$, and $Y_u$, recall the work of \cite{allcock_complex_2002}, where the authors describe ball quotients $\mathbb{B}_4/\Gamma_m$, $\mathbb{B}_4/\Gamma_\ell$, and $\mathbb{B}_4/\Gamma$ for a suitable projective unitary group $\Gamma_m$, $\Gamma_\ell$, and $\Gamma$.  The Satake--Baily--Borel compactification is obtained as the $\mathrm{Proj}$ of the graded ring of automorphic forms and its toroidal compactification 
$\overline{\mathbb{B}_4/\Gamma_m} \to \overline{\mathbb{B}_4/\Gamma_m}^{\mathrm{BB}}$ is obtained by blowing up the singular cuspidal points (similarly for the singly marked and unmarked ball quotients). Recall from \cite{doran2004moduli, dolgachev2005complex, casalaina-martin_birational_2024} that $\overline{\mathbb{B}_4/\Gamma_\ell}$ sits in between the toroidal compactifications of the ball quotient of marked cubic surfaces $\overline{\mathbb{B}_4\Gamma_m}$ and unmarked cubic surfaces $\overline{\mathbb{B}_4/\Gamma}$ via forgetful morphisms
\begin{align*}
	\overline{\mathbb{B}_4/\Gamma_m} \to \overline{\mathbb{B}_4/\Gamma_\ell} \xrightarrow{\ell} \overline{\mathbb{B}_4/\Gamma}.
\end{align*}
Let $D$ denote the closure of the discriminant divisor of cubic surfaces with nodal $A_1$-singularities in the Satake--Baily--Borel compactification $\overline{\mathbb{B}_4/\Gamma}^{BB} \cong \overline{Y}_u^{GIT}$ given in \cite{allcock_complex_2002}. Let $D_{A_1}$ denote the strict transform of $D$ via the standard birational morphism $\overline{\mathbb{B}_4/ \Gamma_m} \to \overline{\mathbb{B}_4/\Gamma_m}^{\mathrm{BB}} \cong \overline{Y}_u^{GIT}$ While $D_{A_1}$ is irreducible, its preimage in $ \overline{\mathbb{B}_4/\Gamma_\ell}$ by the cover $\ell: \overline{\mathbb{B}_4/\Gamma_\ell} \to \overline{\mathbb{B}_4/\Gamma}$ decomposes into two irreducible components
\begin{align*}
	D_{A_1, \ell} := \ell^{-1}(D_{A_1}) = D_{A_1, \ell}^{\text{in}} \cup D_{A_1, \ell}^{\text{out}} \subseteq \overline{\mathbb{B}_4/\Gamma_\ell}
\end{align*}
The component $D_{A_1, \ell}^{\text{in}}$ is the closure of the locus of cubic surfaces with a unique node $P$ and the marked line $\ell$ containing $P$, and similarly, $D_{A_1, \ell}^{\text{out}}$ is the closure of the locus of nodal cubics with $\ell$ not containing the node. The preimage of $D_{A_1}$ in $\overline{\mathbb{B}_4/\Gamma_m}$ consists of 36 components, permuted by the action of $W(E_6)$. Similarly, $D_{3A_2, \ell}$ and $D_{3A_2, m}$ denote the preimages of the irreducible toric boundary of $\overline{\mathbb{B}_4/\Gamma}$ in $\overline{\mathbb{B}_4/\Gamma_\ell}$ and $\overline{\mathbb{B}_4/\Gamma_m}$, respectively. While $D_{3A_2, \ell}$ is irreducible, $D_{3A_2, m}$ has 40 irreducible components permuted by $W(E_6)$. From this boundary description of $\overline{\mathbb{B}_4/ \Gamma_m}$, it is natural to ask whether the birational map $\overline{N} \dasharrow \overline{\mathbb{B}_4/\Gamma_m}$ fitting in the diagram
\[\begin{tikzcd}
	{\overline{N}} & {\overline{\mathbb{B}_4/\Gamma_m}} \\
	{\overline{Y}^{\mathrm{GIT}}} & {\overline{\mathbb{B}_4/\Gamma_m}^{\mathrm{BB}}}
	\arrow[dashed, from=1-1, to=1-2]
	\arrow[from=1-1, to=2-1]
	\arrow[from=1-2, to=2-2]
	\arrow["\sim", from=2-1, to=2-2]
\end{tikzcd}\]
extends to an isomorphism on $\mathbb{B}_4/\Gamma_m$. This is confirmed in \cite{gallardo_geometric_2021} and we give a quick proof using a technique learned from \cite{casalaina2021complete}.

\begin{teo}\cite[Thm.~1.4]{gallardo_geometric_2021}\label{teo:GKStoroidal}
	The Naruki compactification $\overline{N}$ of the moduli space $Y_m$ of marked cubic surfaces is isomorphic to the toroidal compactification $\overline{\mathbb{B}_4/\Gamma_m}$ of its ball quotient. In other words, the birational map $\overline{N} \dashrightarrow \overline{\mathbb{B}_4/\Gamma_m}$ above extends to an isomorphism.
\end{teo}
\begin{proof}
	Since $\overline{N}$ has normal crossing boundary, by the Borel extension theorem \cite[Thm.~A]{borel1972some}, the birational map extends to a morphism $\overline{N} \to \overline{\mathbb{B}_4/\Gamma_m}$. Now since $\dim \overline{N} = \dim \overline{\mathbb{B}_4/\Gamma_m}$, the morphism must either be an isomorphism, divisorial contraction, or a small contraction. It cannot be a small contraction as $\overline{\mathbb{B}_4/\Gamma_m}$ is $\mathbb{Q}$-factorial. On the other hand, it cannot be a divisorial contraction since both $\overline{N}$ and $\overline{\mathbb B_4/\Gamma_m}$ have the same number of boundary divisors.
\end{proof}

	\section{The moduli space of KSBA stable pairs}	\label{S:ksba}
In this section, we recall the necessary background on the moduli of KSBA stable pairs and construct the KSBA moduli stack of $(b, c)$-weighted stable marked cubic surfaces together with its coarse moduli space (\Cref{def:(bc)modulistack}). For a full treatment of KSBA theory, we refer the reader to \cite{kollar_families_2023}. Furthermore, see \cite{kollar_birational_1998, kollar_singularities_2013} for the singularities of the minimal model program.

\subsection{KSBA moduli stack of marked pairs with rational coefficients}
Here we give a brief recollection of some of the main results in  \cite[\S~8.2]{kollar_families_2023} on the KSBA stability of marked pairs with rational coefficients and their moduli stack.

\begin{dfn}[Marked pair]\cite[Dfn.~8.3]{kollar_families_2023}
	A \textit{marking} of an effective Weil $\mathbb{Q}$-divisor $B$ is a way of writing $B = \sum a_i D_i$, where $D_i$ are effective $\mathbb{Z}$-divisors and $0 < a_i \in \mathbb{Q}$. We call $\textbf{a} = (a_1, \dots, a_n)$ the \textit{coefficient vector}. A \textit{marked pair} is a pair $(X, B)$, plus a marking $B = \sum a_i D_i$.
\end{dfn}

\begin{rem}
	Under this definition, the pairs $(\mathbb{A}^1, (x=0))$ and $(\mathbb{A}^1, \frac{1}{2}(x^2=0))$ are different as \textit{marked pairs}.
\end{rem}

\begin{dfn}[Family of marked pairs]\cite[Dfn.~8.4]{kollar_families_2023}
	For a fixed rational coefficient vector $\textbf{a} = (a_1, \dots, a_n)$, a \textit{family of marked pairs} with coefficient vector $\textbf{a}$ is a morphism $f: (X, B) \to T$ over some base scheme $T$ that is flat with demi-normal fibers where $B$ is an effective, relative Mumford $\mathbb{Q}$-divisor \cite[Def.~4.68]{kollar_families_2023}, and a marking $B = \sum a_i D_i$ where each $D_i$ is an effective, relative, Mumford $\mathbb{Z}$-divisor \cite[(4.68)]{kollar_families_2023}.
\end{dfn}

\begin{dfn}[Semi-log canonical (slc)]\cite[(8.6) and Dfn.~11.37]{kollar_families_2023}\label{def:KSBApair}
	Let $X$ be a variety (over $\mathbb{C}$) and $B = \sum a_i B_i$ a $\mathbb{Q}$-divisor on $X$ with each $a_i \in (0, 1] \cap \mathbb{Q}$. The marked pair $(X, B)$ is called \textit{semi-log canonical (slc)} if the following conditions hold:
	\begin{enumerate}
		\item $X$ is demi-normal (codimension 1 points are at worst nodal);
		\item If $\nu : X^\nu \to X$ is the normalization with conductors $D \subseteq X$ and $D^\nu \subseteq X^\nu$, then the support of $D$ does not contain any irreducible component of $B$;
		\item $K_X + D$ is $\mathbb{Q}$-Cartier;
		\item The pair $(X^\nu, D^\nu + \nu_*^{-1} B)$ is log canonical, where $\nu_*^{-1} B$ denotes the strict transform of $B$.
	\end{enumerate}
	A marked pair $(X, B)$ is called \textit{KSBA stable} if $X$ is projective and $K_X + B$ is ample.
\end{dfn}

\begin{dfn}[KSBA stable]\cite[(8.7)]{kollar_families_2023}\label{dfn:KSBAfamily}
	A family of marked pairs $(X, B = \sum a_i B_i) \to T$ is \textit{KSBA stable} if
	\begin{enumerate}
		\item $f : X \to T$ is flat, finite type, and pure dimensional;
		\item The $B_i$ are K-flat \cite[Def.~7.1]{kollar_families_2023} families of relative, Mumford, $\mathbb{Z}$-divisors;
		\item The fibers $(X_t, B_t)$ are slc;\label{dfn:KSBAfamilyslc}
		\item $\omega_{X/S}^{[m]}(mB - D)$ is a flat family of divisorial sheaves, provided $\mathrm{lcd}(a_1, \dots, a_n)$ divides $m$ and $D = \sum_{j \in J} B_i$ with $a_j = 1$ for $j \in J$;
		\item $f$ is proper and $K_{X/S} + B$ is $f$-ample.
	\end{enumerate}
\end{dfn}
The main theorem of \cite{kollar_families_2023} is the following:

\begin{teo}\cite[Thm.~8.1]{kollar_families_2023}\label{teo:ksbastack}
	KSBA stability of marked pairs with rational coefficients is a good moduli theory \cite[(6.10)]{kollar_families_2023}. In other words, for a fixed positive rational number $v$, a positive integer $d$, and a rational coefficient vector $\textbf{a}$, there is a proper DM stack $\overline{\mathcal{M}}_{\textbf{a}, d, v}^{\mathrm{KSBA}}$ that, for base scheme $T$, represents the moduli problem of KSBA stable families $f : (X,  \sum a_i B_i) \to T$ with fibers of dimension $d$ and volume $v$. Furthermore, it has a coarse moduli space $\overline{M}_{\textbf{a}, d, v}^{\mathrm{KSBA}}$, which is projective over $T$. 
\end{teo}

\subsection{KSBA moduli spaces of cubic surfaces}
Considering boundary divisors separately or together (e.g., the $27$ lines on a cubic surface as one divisor or as $27$ separate divisors) gives slight variations on the moduli problem. We will want to compare these moduli spaces, and we fix some terminology for this here. Following the notation in \Cref{teo:ksbastack}, in our setting, we fix $d = 2$ and for rational coefficient vector $\textbf{a} = (a_1, \dots, a_n)$, let 
$$v = v(\textbf{a}) = (K_S + \sum a_i B_i)^2$$
be constant. We are concerned with the three moduli stacks $\overline{\mathcal{M}}_{\alpha, 2, v(\alpha)}^{KSBA}$, $\overline{\mathcal{M}}_{(b, c), 2, v(b, c)}^{KSBA}$, and $\overline{\mathcal{M}}_{\textbf{c}, 2, v(\textbf{c})}^{KSBA}$ for rational coefficient vectors $\alpha \in \mathbb{Q}$, $(b, c)\in \mathbb Q^2$, and $\textbf{c} = (c_1, \dots, c_{27})\in \mathbb Q^{27}$, respectively. These are DM stacks that parametrizes KSBA stable families of surfaces
\begin{itemize}
	\item $(X, \alpha B) \to T$ with constant volume $v = v(\alpha)$,
	\item $(X, bB_1 + cB_2) \to T$ with constant volume $v = v(b, c)$, and
	\item $(X, \sum c_i B_i) \to T$ with constant volume $v = v(\textbf{c})$,
\end{itemize}
respectively. When $(c_1,\dots,c_{27})=(b,c,\dots, c)$, i.e., when $c_2=c_3=\cdots = c_{27}$, we have a forgetful morphism 
\begin{align}\label{eqn:Forget-KSBAbc}
	\overline{\mathcal{M}}_{\textbf{c}, 2, v(\textbf{c})}^{KSBA} \to \overline{\mathcal{M}}_{(b, c), 2, v(b, c)}^{KSBA},
	\end{align}
	sending a family $(X, \sum c_i B_i) \to T$ to $(X, b B_1 + c(B_2 + \cdots + B_{27})) \to T$. When $(b,c)=(\alpha,\alpha)$, i.e., when $b=c=\alpha$, we have a forgetful morphism 
	\begin{align}\label{eqn:Forget-KSBAalpha}
	\overline{\mathcal{M}}_{(b, c), 2, v(b, c)}^{KSBA}  \to 
\overline{\mathcal{M}}_{\alpha, 2, v(\alpha)}^{KSBA}
\end{align}
sending $(X, bB_1 + cB_2) \to T$ to $(X, \alpha(B_1 + \cdots + B_{27})) \to T$. We define the following moduli stacks of smooth KSBA stable cubic surfaces 
\begin{align*}
	{\mathcal{M}}(\alpha) &\hookrightarrow \overline{\mathcal{M}}_{\alpha, 2, v(\alpha)}^{KSBA}\\
	{\mathcal{M}}_\ell (b, c) &\hookrightarrow \overline{\mathcal{M}}_{(b, c), 2, v(b, c)}^{KSBA}\\
	{\mathcal{M}}(\textbf{c}) &\hookrightarrow \overline{\mathcal{M}}_{\textbf{c}, 2, v(\textbf{c})}^{KSBA}
\end{align*}
parameterizing KSBA stable families of $(b, c)$-weighted smooth marked cubic surfaces:
\begin{itemize}
	\item $(S, \alpha (L_1 + \cdots + L_{27}))$,
	\item $(S, b\ell + c(L_2 + \cdots + L_{27}))$, and
	\item $(S, \sum c_i L_i)$,
\end{itemize}
respectively. Furthermore, we have isomorphisms of stacks 
\begin{align*}
	\mathcal{Y}_u &\stackrel{\sim}{\to} \mathcal{M}(\alpha) \\
	\mathcal{Y}_\ell &\stackrel{\sim}{\to} \mathcal{M}_\ell (b, c)\\
	\mathcal{Y}_m &\stackrel{\sim}{\to} \mathcal{M}(\textbf{c})
\end{align*}
of the moduli stacks of unmarked, singly marked, and fully marked smooth cubic surfaces, respectively. For instance, given a smooth cubic surface $S$ corresponding to a point of $\mathcal Y_u$, we assign the KSBA surface $(S,\alpha (L_1+\cdots+L_{27}))$. 
\begin{rem}
	The moduli stack $\mathcal{Y}_\ell = \mathrm{Fano}_{\mathcal{X}/\mathcal{M}}$ is the relative Fano stack of the universal cubic surface $\mathcal{X}$ over the moduli stack of smooth cubic surfaces $\mathcal{M}$.
\end{rem}

\begin{dfn}[KSBA moduli stacks of stable cubic surfaces]\label{def:(bc)modulistack}
	Let 
	\begin{align*}
		\overline{\mathcal{Y}}_\alpha^{\mathrm{KSBA}} = \overline{\mathcal{Y}}_\alpha
	\end{align*}
	denote the closure of $\mathcal{Y}_u$ in the KSBA moduli stack of unmarked pairs $\overline{\mathcal{M}}(\alpha)  \subseteq \overline{\mathcal{M}}_{\alpha, 2, v(\alpha)}^{KSBA}$. Similarly, let $\overline{\mathcal{Y}}_{(b, c)}^{\mathrm{KSBA}} = \overline{\mathcal{Y}}_{(b, c)}$ and $\overline{\mathcal{Y}}_\textbf{c}^{\mathrm{KSBA}} = \overline{\mathcal{Y}}_\textbf{c}$
	denote the closure of $\mathcal{Y}_\ell$ and $\mathcal{Y}_m$ in the KSBA moduli stacks of singly-marked and (fully) marked pairs, respectively. These moduli stacks admit a projective coarse moduli space (see \Cref{teo:wallcrossing}) and we write $\overline{Y}_u^{\mathrm{KSBA}} = \overline{Y}_u$, $\overline{Y}_{(b, c)}^{\mathrm{KSBA}} = \overline{Y}_{(b, c)}$, and $\overline{Y}_\textbf{c}^{\mathrm{KSBA}} = \overline{Y}_\textbf{c}$ to denote the \textit{normalization} of the KSBA coarse moduli spaces of $\overline{\mathcal{Y}}_\alpha, \overline{\mathcal{Y}}_{(b, c)},$ and $\overline{\mathcal{Y}}_\textbf{c}$, respectively.
\end{dfn}

\begin{rem}\label{rem:KSBAforgetfulmaps}
	Similarly to \ref{eqn:Forget-KSBAbc} and \ref{eqn:Forget-KSBAalpha},  when $(c_1 , \dots, c_{27}) = (b, c, \dots, c)$, we have a forgetful morphism
	\begin{align*}
		\overline{Y}_\textbf{c}  \to \overline{Y}_{(b, c)}.
	\end{align*}
	Furthermore, when $(b ,c) = (\alpha, \alpha)$, we have a forgetful morphism
	\begin{align*}
		\overline{Y}_{(b, c)} \to \overline{Y}_\alpha
	\end{align*}
\end{rem}

\begin{dfn}[KSBA moduli space of cubic surfaces with a marked line]\label{def:markedlinemodulispace}
	For $\epsilon >0$ a small rational number,  let $\overline{Y}_{\ell} := \overline{Y}_{(\frac{1}{9} + \epsilon, \frac{1}{9} + \epsilon)}$ be the normalization of the KSBA compactification of $(\frac{1}{9} + \epsilon, \frac{1}{9} + \epsilon)$-weighted stable marked cubic surfaces. We call $\overline{Y}_\ell$ the \textit{KSBA (coarse) moduli space of cubic surfaces with a marked line $\ell$}.
\end{dfn}

Similarly to \Cref{teo:GKStoroidal}, we give a quick proof that the (normalization of) the KSBA compactifications $\overline{Y}_{\ell}$ and $\overline{Y}_{\frac{1}{9} + \epsilon}$ is isomorphic to its toroidal compactifications. Recall by \Cref{rem:KSBAforgetfulmaps}, we have forgetful morphisms analogous to the ball quotient case.
\begin{align*}
	\overline{Y}_{(\frac{1}{9} + \epsilon, \dots, \frac{1}{9} + \epsilon)} \to \overline{Y}_\ell := \overline{Y}_{(\frac{1}{9} + \epsilon, \frac{1}{9} + \epsilon)} \to \overline{Y}_{\frac{1}{9} + \epsilon}
\end{align*}
In \cite[Thm.~1.5]{gallardo_geometric_2021}, the authors showed Naruki's cross-ratio variety $\overline{N}$ (\Cref{S:naruki}) is isomorphic to $\overline{Y}_{(\frac{1}{9} + \epsilon, \dots, \frac{1}{9} + \epsilon)}$. Thus, by \Cref{teo:GKStoroidal}, we get the isomorphism $\overline{Y}_{(\frac{1}{9} + \epsilon, \dots, \frac{1}{9} + \epsilon)} \cong \overline{\mathbb{B}_4/\Gamma_m}$, giving us the diagram
\[\begin{tikzcd}
	& {\overline{Y}_{(\frac{1}{9} + \epsilon, \dots, \frac{1}{9} + \epsilon)}} &&& {\overline{\mathbb{B}_4/\Gamma_m}} \\
	{\overline{Y}_{\ell}} &&& {\overline{\mathbb{B}_4/\Gamma_\ell}} \\
	\\
	& {\overline{Y}_{\frac{1}{9} + \epsilon}} &&& {\overline{\mathbb{B}_4/\Gamma}}
	\arrow["{\sim}", from=1-2, to=1-5]
	\arrow[ from=1-2, to=2-1]
	\arrow[{pos=0.7}, from=1-2, to=4-2]
	\arrow["{/W(D_5)}", from=1-5, to=2-4]
	\arrow["{/W(E_6)}"{pos=0.7}, from=1-5, to=4-5]
	\arrow["{\text{birational}}"{pos=0.6}, dashed, from=2-1, to=2-4]
	\arrow[from=2-1, to=4-2]
	\arrow["{27:1}", from=2-4, to=4-5]
	\arrow["{\text{birational}}"{pos=0.6}, dashed,  from=4-2, to=4-5]
\end{tikzcd}\] 
where the vertical arrows are forgetful maps.

\begin{cor}\label{teo:GKS+epsilon}
	The KSBA moduli space of cubic surfaces with a marked line $\overline{Y}_{\ell}$ is isomorphic to the toroidal compactification of the moduli space of cubic surfaces with a marked line $\overline{\mathbb{B}_4/\Gamma_\ell}$.
\end{cor}

\begin{proof}
	We follow the same proof technique as in \Cref{teo:GKStoroidal} using the Borel extension theorem \cite[Thm.~A]{borel1972some}. Since $\overline{Y}_{(\frac{1}{9} + \epsilon, \dots, \frac{1}{9} + \epsilon)} \cong \overline{N}$ and $\overline{N}$ has normal crossing boundary, by the Borel extension theorem, the birational map $\overline{Y}_{\ell} \dasharrow \overline{\mathbb{B}_4/\Gamma_\ell}$ extends to a morphism. Since $\dim \overline{Y}_{\ell} = \dim \overline{\mathbb{B}_4/\Gamma_\ell}$, this morphism can either be an isomorphism, divisorial contraction, or a small contraction. It cannot be a small contraction as $\overline{\mathbb{B}_4/\Gamma_\ell}$ is $\mathbb{Q}$-factorial and it cannot be a divisorial contraction since both have boundary divisors given by the same quotient of the boundary divisors of $\overline{N}$.
\end{proof}

\begin{cor}
	The KSBA moduli space of unmarked cubic surfaces $\overline{Y}_{\frac{1}{9} + \epsilon}$ is isomorphic to the toroidal compactification of the moduli space of unmarked cubic surfaces $\overline{\mathbb{B}_4/\Gamma}$.
\end{cor}
\begin{proof}
	The proof is similar to \Cref{teo:GKStoroidal} and \Cref{teo:GKS+epsilon}.
\end{proof}

\begin{rem}
	Note that \Cref{teo:GKS+epsilon} shows us that the cubic surfaces parameterized at the boundary of $\overline{Y}_\ell \cong \overline{\mathbb{B}_4/\Gamma_\ell}$ are precisely the cubic surfaces parameterized at the boundary of Naruki's cross-ratio variety $\overline{N}$ (see \Cref{tab:narukifibers}) with a choice of a marked line $\ell$. The cubic surfaces at the Naruki $\overline{N}$ boundary will be in a different boundary stratum in $\overline{Y}_\ell$ depending on whether or not $\ell$ passes through a node. In particular, a general cubic surface parameterized in the boundary divisor $D_{A_1, \ell} \subseteq \overline{Y}_\ell$ (see \Cref{SS:toroidalcpt} or \cite{casalaina-martin_birational_2024}) will either be in $D_{A_1, \ell}^{\text{in}}$ or $D_{A_1, \ell}^{\text{out}}$ depending on whether $\ell$ contains the node or not.
\end{rem}

\subsection{Positivity of the log canonical divisors for (generalized) Del Pezzo surfaces}
As ampleness of the log canonical bundle is one of the KSBA stability conditions, we compute the ample cones of (generalized) Del Pezzo surfaces that will be needed in the proof of \Cref{teo:main} found in \Cref{proof:E2A1_mult4}. We will use the main results of \cite{derenthal2008nef} to compute the ample cone of the log canonical divisor $K_S + (b, c)B$ for degree $2$ and $3$ Del Pezzo surfaces with $k = 1, 2, 3, 4$ $A_1$ singularities.

To begin, recall for a smooth Del Pezzo surface $X$ of degree $d \leq 8$, the cone of curves $\overline{\mathrm{NE}}(X)$ is generated by its $(-1)$ curves. Indeed, let $C$ be an irreducible rational curve with $C^2 < 0$. By the genus formula,
\begin{align*}
	2 &= (K_X \cdot C) + C^2,
\end{align*}
and since $-K_X$ is ample, it must be that $(K_X \cdot C) < 0$ as well. Thus, it must be that $C^2 = -1$, and since a smooth Del Pezzo surface has finitely many $(-1)$-curves, we can label them as $C_1, \dots, C_{n_d}$ where $n_d$ depends on the degree of $X$. So by Mori's cone theorem,
\begin{align*}
	\overline{\mathrm{NE}}(X) &= \sum_1^{n_d} \mathbb{R}_{\geq 0} [C_i],
\end{align*}
and we get
\begin{align*}
	\mathrm{Nef}(X) &= \{C \in N^1(X) : (C \cdot C_i) \geq 0 \text{ for } 1 \leq i \leq n_d\}. 
\end{align*}
Thus, for $(b, c)$-weighted smooth marked cubic surfaces $(S, (b, c)B)$, we get
\begin{equation}\label{eqn:ampcone}
	\mathrm{Amp}(S, (b, c)B) = \mathrm{Amp} = \{(b, c) \in ((1/9, 1] \cap \mathbb{Q})^2 \mid  c > b/10 + 1/10\}. 
\end{equation}

For singular cubics, we start with some notation: Denote for $k = 1, 2, 3,$ or $4$, $S_{kA_1}$ as a cubic surface with $k$ $A_1$-singularities and let $\widetilde{S}_{kA_1}$ denote its simultaneous resolution. For $k = 1$, we will just write $S_{A_1}$ and $\widetilde{S}_{A_1}$. We will compute the nef cones of the log canonical divisors of a smooth cubic surface $S$, as well as for the singular cubic surfaces $\widetilde{S}_{k A_1}$. By \cite[Prop.~3.3, 3.5]{tu_semi-stable_2005} $\widetilde{S}_{kA_1}$ can be obtained as a blow up of $\mathbb{P}^2$ where the six points are now in special position shown in \Cref{tab:specialposcubics}.
\begin{table}[h]
	\centering
	\begin{tabular}{| c | c |}
			\hline
			$k$ & Special position\\
			\hline
			\hline
			1 & six points lie on a conic \\
			2 & exactly two pairwise intersecting sets of three colinear points\\
			3 & exactly three pairwise intersecting sets of three colinear points\\
			4 & exactly four pairwise intersecting sets of three colinear points\\
			\hline
		\end{tabular}
	\caption{The left column indicates the number of $A_1$ singularities while the right column describes a special position of the six points $p_1, \dots, p_6$ to obtain $\widetilde{S}_{k A_1}$ after blowing them up.}
	\label{tab:specialposcubics}
\end{table}
For each $\widetilde{S}_{kA_1}$, its Picard group is given by $\mathbb{Z}\langle h, e_1, \dots, e_6 \rangle$, where $h$ is the class of a pullback of a line in $\mathbb{P}^2$ and $e_1, \dots, e_6$ are the classes of the exceptional divisors.

In the language of \cite[\S~1]{derenthal2008nef}, each $\widetilde{S}_{kA_1}$ is a generalized Del Pezzo surface. A \textit{generalized Del Pezzo surface} is a smooth projective rational surface $Y$ on which $-K_Y$ is big and nef. However, since we are defined over $\mathbb{C}$, $Y$ must be one of the following: $\mathbb{P}^2$, $\mathbb{F}_0$, $\mathbb{F}_2$, or a surface obtained from $\mathbb{P}^2$ by a sequence of blowing up at up to 8 points, possibly infinitely near, each not lying on any $(-2)$-curve. One of the main results of \cite{derenthal2008nef} is that the cone of curves of these generalized Del Pezzo surfaces are generated by the $(-1)$ and $(-2)$-curves.

\begin{teo}\cite[Cor.~3.12]{derenthal2008nef}\label{teo:gendelpezzonefcone}
	Let $\Gamma \subseteq N^1(Y)_\mathbb{R}$ be the cone spanned by the set of $(-2)$-curves on $Y$. Then if $X$ is a smooth Del Pezzo surface with $\deg(X) = \deg(Y) \leq 7$,
	\begin{align*}
		\mathrm{Nef}(Y) = \mathrm{Nef}(X) \cap \Gamma^\lor.
	\end{align*}
\end{teo}
Using \Cref{teo:gendelpezzonefcone} above and the following table listing the possible $(-1)$ and $(-2)$-curves for $\widetilde{S}_{kA_1}$, we get \Cref{teo:nefcones} giving the ample cone for the (generalized) Del Pezzo surfaces needed in the proof of \Cref{teo:main} in \Cref{proof:E2A1_mult4}.
\begin{table}[h]
	\centering
	\begin{tabular}{| c | c |c|}
		\hline
		$k$ & $(-1)$-curves & $(-2)$-curves\\
		\hline
		\hline
		$1$ & $e_1, \dots, e_6$, $h - e_i - e_j$ for $1 \leq i < j \leq 6$ & $2h - e_1 - \cdots e_6$ \\
		$k = 2,3, 4$ & $h - e_i - e_j$ & $e_1, \dots, e_6$, $h - e_i-e_j-e_l$\\
		&  for $\{p_i, p_j\}$ in general position & for $\{p_i, p_j, p_l\}$ of three colinear points\\
		\hline
	\end{tabular}
	\caption{The $(-1)$ and $(-2)$ curves for $\widetilde{S}_{kA_1}$ for $k =1,2,3,4$.}
	\label{tab:(-1)and(-2)curves}
\end{table}

\begin{teo}\label{teo:nefcones}
	Under the special position of the six points $p_1, \dots, p_6$ described in \Cref{tab:specialposcubics}, the ample cones of a smooth cubic $S$ and the simultaneous resolution of a cubic surface with $k = 1, 2, 3,$ or $4$ $A_1$-singularities $\widetilde{S}_{k A_1}$ are given by divisors $L = ah -  \sum_1^6 b_i e_i$, such that
	\begin{table}[H]
			\centering
			\begin{tabular}{| c | c |}
					\hline
					Type & Conditions\\
					\hline
					\hline
					$S$ & $b_1, \dots b_6 > 0$, $a  \geq b_i + b_j$ where $1 \leq i < j \leq 6$, and $2a > b_{i_1} + \cdots + b_{i_5}$\\
					$\widetilde{S}_{A_1}$ & $b_1, \dots, b_6 > 0$, $a > b_i + b_j$ where $1 \leq i < j \leq 6$, and $2a > b_1 + \cdots + b_6$\\
					$\widetilde{S}_{kA_1}$ for $k = 2, 3, 4$ & $b_1, \dots, b_6 > 0$, $a > b_i + b_j + b_l$ for the $k$ pairwise intersecting sets\\
					& of three colinear points,  and $a > b_i + b_j$ for the remaining points\\
					\hline
				\end{tabular}
				\label{tab:nefconditions}
		\end{table}
	Similarly, if $p_7$ is a seventh point in general position, then the ample cones of the blow up of $S$ and $\widetilde{S}_{k A_1}$ at $p_7$ with exceptional divisor $e_7$ ((generalized) degree $2$ Del Pezzo surfaces) are given by divisors $L = ah - \sum_1^7 b_i e_i$ satisfying the conditions of the table above with the extra conditions that $b_7 > 0$ and $a > b_i + b_7$ for $1 \leq i \leq 6$.
\end{teo}
	
	\section{The universal family $\ddot{\pi} : \ddot Y(E_7) \to \overline{Y}_{(1, \dots, 1)}$ and its fibers}\label{S:univfam}
In this section, we first summarize the construction of the universal family $\ddot Y(E_7) \to \overline{Y}_{(1, \dots, 1)}$ due to Hacking--Keel--Tevelev \cite[\S~10]{hacking_stable_2009}. We also summarize Schock's computation of the fibers in \cite[\S~7]{schock_moduli_2024}. By determining the explicit fibers over $\overline{Y}_{(1, \dots, 1)}$, we also obtain the fibers over $\overline{Y}_{(1, 1)}$.

\subsection{Log canonical compactifications of $Y_m$ and $Y(E_7)$}
First, we recall the moduli space of marked quadric del Pezzo surfaces $Y(E_7)$ analogous to \Cref{dfn:latticemarking} for marked cubic surfaces. 
\begin{dfn}[Marking of quadric del Pezzo]\label{dfn:markedquadric}
	Let $I_{1, 7} = \langle 1 \rangle \oplus \langle -1 \rangle^7$ be the standard odd unimodular hyperbolic lattice of signature $(1, 7)$ with standard orthonormal basis $\langle e_0, \dots, e_7 \rangle$. Furthermore, let $k = -3e_0 + e_1 + \cdots + e_7$. Then a \textit{marking} of a smooth quadric del Pezzo surface $S$ is an isometry
	\begin{align*}
		\phi : I_{1, 7} \to \mathrm{Pic}(S)
	\end{align*}
	sending $k$ to the canonical divisor $K_S$ of $S$. Let $Y(E_7)$ denote the moduli space of marked quadric del Pezzo surfaces.
\end{dfn}

\begin{rem}
	We could have applied the same convention as in \Cref{dfn:markedquadric} and denoted $Y_m$ as $Y(E_6)$. This is the convention used in \cite{schock_moduli_2024} and \cite{hacking_stable_2009}. However, we choose not to follow this convention as this article is mainly concerned with cubic surfaces, while the articles referenced are concerned with del Pezzo surfaces in general.
\end{rem}
 
For $n = 6, 7$, there are $W(E_n)$-equivariant compactifications of $Y_m$ and $Y(E_7)$ done by Naruki \cite{naruki_cross_1982}, Sekiguchi \cite{sekiguchi_cross_2000}, and Hacking--Keel--Tevelev \cite{hacking_stable_2009}. Namely, for a smooth variety $X$, let $X \subseteq \overline{X}$ denote a compactification of $X$ with normal crossing boundary $B = \overline{X} - X$. Then for $n \geq 0$, the vector space
\begin{align*}
	H^0(\overline{X}, n(K_{\overline{X}} + B))
\end{align*}
only depends on $X$. The \textit{log canonical ring} of $X$ is the associated graded ring
\begin{align*}
	R(\overline{X}, K_{\overline{X}} + B) := \bigoplus_{n \geq 0} H^0(\overline{X}, n(K_{\overline{X}} + B)),
\end{align*}
and if it is finitely generated, the induced rational map
\begin{align*}
	X \dashrightarrow \overline{X}^{\mathrm{lc}} := \mathrm{Proj} R(\overline{X}, K_{\overline{X}} + B) 
\end{align*}
is called the \textit{log canonical model}. We call $X$ \textit{log-minimal} if there is an $n \geq 0$ sufficiently divisible such that the rational map
\begin{align*}
	X \dashrightarrow \mathbb{P}(H^0(\overline{X}, n(K_{\overline{X}} + B))^\lor)
\end{align*}
is an embedding. In this case, it is conjectured that the log canonical ring is finitely generated and that $\overline{X}^\mathrm{lc}$ is the \textit{log canonical compactification} of $X$  (see \cite[2.4]{birkar2010existence}). The following result by Hacking--Keel--Tevelev shows that the varieties $Y_m$ and $Y(E_7)$ are log-minimal and admit log canonical compactifications. 

\begin{teo}\cite[Thm.~1.2]{hacking_stable_2009}
	The varieties $Y_m$ and $Y(E_7)$ are log-minimal. The log canonical compactifications $\overline{Y}_m^\mathrm{lc}$ and $\overline{Y}(E_7)^\mathrm{lc}$ are smooth and their boundary is a union of smooth divisors with normal crossings. Furthermore, the compactifications are, respectively, $W(E_6)$ and $W(E_7)$-equivariant. \qed
\end{teo}

We have nice geometric interpretations of these log canonical compactifications, but let us first establish some conventions regarding root subsystems of $E_6$ and $E_7$.

\begin{dfn}[Abstract simplicial complex $\mathcal{R}(E_n)$]\cite[1.14]{hacking_stable_2009}\label{dfn:simplicalcomplex}
	For $n = 6,7$, let $\mathcal{R}(E_n)$ be the abstract simplicial complex whose vertices are in bijections with the root subsystems of $E_6$ of the form $A_1$ and $3A_2$ for $n = 6$ or $A_1$, $A_2$, $2A_3$, and $A_7$ for $n = 7$. A collection of vertices forms a simplex if the corresponding root subsystems are either pairwise orthogonal or disjoint. In addition, for $E_7$, remove the $7$-simplices formed by $7$-tuples of pairwise orthogonal $A_1$ root subsystems.
\end{dfn}

\begin{conv}[Boundary divisors of $\overline{Y}_m^\mathrm{lc}$ and $\overline{Y}(E_7)^\mathrm{lc}$]
\label{conv:simplicalcomplex}
	We follow the notation in \cite[2.6]{schock_moduli_2024}. For a vertex $\Theta$ of $\mathcal{R}(E_n)$, we write $D_\Theta$ for the corresponding boundary divisor of $\overline{Y}_m^\mathrm{lc}$ and $\overline{Y}(E_7)^\mathrm{lc}$. 
\end{conv}

\begin{teo}[{\cite{hacking_stable_2009}}]
\label{teo:lcboundarydiv}
	The boundary complex of $\overline{Y}_m^\mathrm{lc}$ and $\overline{Y}(E_7)^\mathrm{lc}$ is the simplicial complex $\mathcal{R}(E_6)$ and $\mathcal{R}(E_7)$, respectively. We also have the following descriptions.
	\begin{enumerate}
		\item $\overline{Y}_m^\mathrm{lc}$ is Naruki's cross-ratio variety \cite{naruki_cross_1982} (see \Cref{S:naruki}). It has a total of 76 irreducible boundary divisors. 36 boundary divisors $D_{A_1}$ isomorphic to $\overline{M}_{0, 6}$ (Type $A_1$ in \Cref{tab:narukifibers}) and 40 boundary divisors $D_{3A_2}$ isomorphic to $(\overline{M}_{0, 4})^3$ (Type $3A_2$ in \Cref{tab:narukifibers}).
		\item $\overline{Y}(E_7)^\mathrm{lc}$ is Sekiguchi's cross-ratio variety \cite{sekiguchi_cross_2000}. It has a total of 1065 boundary divisors.
		\begin{enumerate}
			\item 63 boundary divisors $D_{A_1}$, see \cite{sekiguchi_cross_2000} or \cite[2.2]{schock_moduli_2024} for a description.
			\item 336 boundary divisors $D_{A_2}$ isomorphic to $\overline{M}_{0, 4} \times \overline{M}_{0, 7}$.
			\item 630 boundary divisors $D_{2A_3}$ isomorphic to $\overline{M}_{0, 5} \times \overline{M}_{0, 5} \times \overline{M}_{0, 4}$.
			\item 36 boundary divisors $D_{A_7}$ isomorphic to $\overline{M}_{0, 8}$.
		\end{enumerate}
	\end{enumerate}\qed
\end{teo}

The following result describes the intersections of boundary divisors on $\overline{Y}_m^\mathrm{lc}$. For the explicit subsystems of $E_6$ and $E_7$ corresponding to these boundary divisors, please refer to \cite[App.~C]{schock_moduli_2024}.

\begin{pro}[{\cite{sekiguchi_cross_2000}, \cite[Prop.~9.17]{hacking_stable_2009},\cite[Prop.~2.7]{schock_moduli_2024}, \cite{naruki_cross_1982}}]\label{pro:boundaryint}
Given divisors $D_{\Theta}$ and $D_{\Theta'}$ corresponding to vertices $\Theta$ and $\Theta'$ (see \Cref{dfn:simplicalcomplex} and \Cref{conv:simplicalcomplex}), the intersections $D_\Theta\cap D_{\Theta'}$ are described below:  
	\begin{enumerate}
		\item On $\overline{Y}_m^\mathrm{lc}$, one has the following types of intersections among boundary divisors.
		\begin{enumerate}
			\item For vertices $A_1$ and $A_1'$ (of type $A_1$), the corresponding divisors $D_{A_1}$ and $D_{A_1'}$ intersect if and only if $A_1 \perp A_1'$. In this case, the intersection
			is isomorphic to $\overline{M}_{0,6} \cong D_{2,4}$ on $D_{A_1} \cong \overline{M}_{0,6}$.
			\item $D_{A_1}$ and $D_{3A_2}$ intersect if and only if $A_1 \subset A_2^3$. In this case, the intersection is isomorphic to $\overline{M}_{0,4} \times \overline{M}_{0,4} \cong D_{3,3}$ on $D_{A_1} \cong \overline{M}_{0,6}$. This is a divisor of the form $p \times \mathbb{P`}^1 \times \mathbb{P}^1$, $\mathbb{P}^1 \times p \times \mathbb{P}^1$, or $\mathbb{P}^1 \times \mathbb{P}^1 \times p$, $p=0$, $1$, or $\infty$, on $D_{3A_2} \cong (\mathbb{P}^1)^3$.
			\item No two $D_{3A_2}$ divisors intersect.
		\end{enumerate}
		\item On $\overline{Y}(E_7)^\mathrm{lc}$, one has the following types of intersections among boundary divisors.
		\begin{enumerate}
			\item $D_{A_1}$ and $D_{A_1'}$ intersect if and only if $A_1 \perp A_1'$. In this case, the intersection is isomorphic to a smooth projective fourfold called $\overline{Z}(D_4)$, described in \cite{sekiguchi_cross_2000}, \cite[2.2]{schock_moduli_2024}.
			\item $D_{A_1}$ and $D_{A_2}$ intersect if and only if either $A_1 \subset A_2$, or $A_1 \perp A_2$. In the former case, the intersection is isomorphic to $p \times \overline{M}_{0,7}$ on $D_{A_2} \cong \overline{M}_{0,4} \times \overline{M}_{0,7}$ (where $p=0,1,\infty$ is a boundary divisor of $\overline{M}_{0,4} \cong \mathbb{P}^1$). In the latter case, the intersection is isomorphic to $\overline{M}_{0,4} \times \overline{M}_{0,6} \cong \overline{M}_{0,4} \times D_{2,5} \subset D_{A_2}$.
			\item $D_{A_1}$ and $D_{2A_3}$ intersect if and only if $A_1 \subset 2A_3$ or $A_1 \perp 2A_3$. In the former case, the intersection is isomorphic to $\overline{M}_{0,4} \times \overline{M}_{0,5} \times \overline{M}_{0,4} \cong D_{2,3} \times \overline{M}_{0,5} \times \overline{M}_{0,4}$ or $\overline{M}_{0,5} \times \overline{M}_{0,4} \times \overline{M}_{0,4} \cong \overline{M}_{0,5} \times D_{2,3} \times \overline{M}_{0,4}$ as a divisor on $D_{2A_3} \cong \overline{M}_{0,5} \times \overline{M}_{0,5} \times \overline{M}_{0,4}$. In the latter case, the intersection is isomorphic to $\overline{M}_{0,5} \times \overline{M}_{0,5} \times p$ on $D_{2A_3}$, where $p=0,1,\infty$ is a boundary divisor of $\overline{M}_{0,4} \cong \mathbb{P}^1$.
			\item $D_{A_1}$ and $D_{A_7}$ intersect if and only if $A_1 \subset A_7$. In this case, the intersection is isomorphic to $\overline{M}_{0,7} \cong D_{2,6}$ on $D_{A_7} \cong \overline{M}_{0,8}$.
			\item $D_{A_2}$ and $D(A_2')$ intersect if and only if $A_2 \perp A_2'$. In this case, the intersection is isomorphic to $\overline{M}_{0,4} \times \overline{M}_{0,4} \times \overline{M}_{0,5} \cong \overline{M}_{0,4} \times D_{3,4}$ on $D_{A_2} \cong \overline{M}_{0,4} \times \overline{M}_{0,7}$.
			\item $D_{A_2}$ and $D_{2A_3}$ intersect if and only if $A_2 \subset 2A_3$. In this case, the	intersection is isomorphic to $\overline{M}_{0,4} \times \overline{M}_{0,4} \times \overline{M}_{0,5} \cong \overline{M}_{0,4} \times D_{3,4}$ on $D_{A_2} \cong \overline{M}_{0,4} \times \overline{M}_{0,7}$, and to $\overline{M}_{0,4} \times \overline{M}_{0,5} \times \overline{M}_{0,4} \cong D_{2,3} \times \overline{M}_{0,5} \times \overline{M}_{0,4}$ or $\overline{M}_{0,5} \times \overline{M}_{0,4} \times \overline{M}_{0,4} \cong \overline{M}_{0,5} \times D_{2,3} \times \overline{M}_{0,4}$ on $D_{2A_3} \cong \overline{M}_{0,5} \times \overline{M}_{0,5} \times \overline{M}_{0,4}$.
			\item $D_{A_2}$ and $D_{A_7}$ intersect if and only if $A_2 \subset A_7$. In this case, the intersection is isomorphic to $\overline{M}_{0,4} \times \overline{M}_{0,6} \cong \overline{M}_{0,4} \times D_{2,5}$ on $D_{A_2} \cong \overline{M}_{0,4} \times \overline{M}_{0,7}$, and $D_{3,5}$ on $D_{A_7} \cong \overline{M}_{0,8}$.
			\item No two $D_{2A_3}$ divisors intersect.
			\item $D_{2A_3}$ and $D_{A_7}$ intersect if and only if $2A_3 \subset A_7$. In this case, the intersection is isomorphic to $\overline{M}_{0,5} \times \overline{M}_{0,5} \times p$ as a divisor on $D_{2A_3} = \overline{M}_{0,5} \times \overline{M}_{0,5} \times \overline{M}_{0,4}$ (where $p=0,1,\infty$ is a boundary divisor of $\overline{M}_{0,4}$), and to $D_{4,4}$ on $D_{A_7} \cong \overline{M}_{0,8}$.
			\item No two $D_{A_7}$ divisors intersect.
		\end{enumerate}
	\end{enumerate}
\end{pro}\qed

Let $Y(E_7) \to Y_m$ denote the  natural forgetful morphism obtained by contracting a $(-1)$-curve on a quadric del Pezzo surface. As we know the explicit boundary divisors of the log canonical compactifications, Hacking---Keel--Tevelev were able to show that this morphism extends to the log canonical compactifications that can be described explicitly.

\begin{teo}\cite[Prop.~10.9]{hacking_stable_2009}
	The morphism $Y(E_7) \to Y_m$ obtained by contracting a $(-1)$-curve extends to a morphism of log canonical compactifications $\pi : \overline{Y}(E_7)^\mathrm{lc} \to \overline{Y}_m^\mathrm{lc}$. The restriction of $\pi$ to each boundary divisor of $\overline{Y}(E_7)^\mathrm{lc}$ is described in \Cref{tab:piboundary}, where each morphism $\pi_I : \overline{M}_{0, n} \to \overline{M}_{0, m}$ denotes the canonical fibration dropping points not labeled by $I$. \qed
\end{teo}

\begin{table}[htpb]
		\centering
		\begin{tabular}{| c | c | c | c |}
			\hline
			$D_\Theta \subseteq \overline{Y}(E_7)^{\mathrm{lc}}$ & $\pi(D_\Theta) \subseteq \overline{Y}_m^{\mathrm{lc}} $ & Condition & $\pi\vert_{D(\Theta)}$ \\
			\hline
			\hline
			$D_{A_1}$ & $D_{A_1}$ & $A_1 \subseteq E_6$ & see \cite[Lem.~4.10]{schock_moduli_2024}\\
			$D_{A_1}$ & $\overline{Y}_m^\mathrm{lc}$ & $A_1 \not\subseteq E_6$ & \\
			\hline
			$D_{A_2}$ & $D_{A_1}$ & $A_2 \cap E_6 = A_1$ & $\overline{M}_{0,4} \times \overline{M}_{0,7} \to pt \times \overline{M}_{0,6}$ \\
			{$D_{A_2}$} & {$D_{3A_2}$} & $A_2 \subseteq E_6$,  & {$\overline{M}_{0,4} \times
				\overline{M}_{0,7} \to \overline{M}_{0,4} \times (\overline{M}_{0,4} \times \overline{M}_{0,4})$} \\
			& & $3A_2 = A_2 + A_2^{\perp}$ & \\
			\hline
			{$D_{2A_3}$} & {$D_{3A_2}$} & $2A_3 \cap E_6 = A_2^2$,  & {$\overline{M}_{0,5}
				\times \overline{M}_{0,5} \times \overline{M}_{0,4} \to \overline{M}_{0,4} \times \overline{M}_{0,4} \times \overline{M}_{0,4}$} \\
			& & $3A_2 = 2A_2 + (2A_2)^{\perp}$ & \\
			$D_{2A_3}$ & $D_{A_1} \cap D_{A_1}$ & $2A_3 \cap E_6 = A_3 + 2A_1$ & $\overline{M}_{0,5} \times
			\overline{M}_{0,5} \times \overline{M}_{0,4} \to \overline{M}_{0,5}$ \\
			\hline
			$D_{A_7}$ & $D_{A_1}$ & $A_7 \cap E_6 = A_1 + A_5$ & $\overline{M}_{0,8} \to \overline{M}_{0,6}$ \\
			\hline
		\end{tabular}
		\caption{The morphism $\pi : \overline{Y}(E_7)^{lc} \to \overline{Y}_m^{lc}$. This table is copied from \cite[Table 1]{schock_moduli_2024}.}
		\label{tab:piboundary}
\end{table}
Note that from the explicit description of $\pi$ to each boundary divisor of $\overline{Y}(E_7)^\mathrm{lc}$ in \Cref{tab:piboundary}, the morphism $\pi : \overline{Y}(E_7)^\mathrm{lc} \to \overline{Y}_m^\mathrm{lc}$ is not flat along the union of the $D_{2A_3} \mapsto D_{A_1} \cap D_{A_1}$ divisors. We will call these $D_{2A_3}$ divisors, \textit{non-flat $2A_3$ divisors} following \cite[Dfn.~10.10]{hacking_stable_2009}.  Fortunately, there is a canonical flattening procedure of the morphism $\pi$.

\subsection{Flattening of $\pi : \overline{Y}(E_7)^\mathrm{lc} \to \overline{Y}_m^\mathrm{lc}$}
The key step of flattening the morphism $\pi$ is to realize the log canonical compactification as a \textit{tropical compactification}, obtained as the closure  in a toric variety $X(\mathcal{F}(E_n))$ associated to a fan $\mathcal{F}(E_n)$ whose underlying topological simplicial complex is $\mathcal{R}(E_n)$ of \Cref{dfn:simplicalcomplex} (see \cite[Thm~1.16]{hacking_stable_2009}). The dominant morphism $Y(E_7) \to Y_m$ induces a surjection of fans $\mathcal{F}(E_7) \to \mathcal{F}(E_6)$ by \cite[Thm.~1.17]{hacking_stable_2009}, and the corresponding morphism of toric varieties pulls back to the morphism $\pi : \overline{Y}(E_7)^{\mathrm{lc}} \to \overline{Y}_m^\mathrm{lc}$:
\[\begin{tikzcd}
	{\overline{Y}(E_7)^{\mathrm{lc}}} & {X(\mathcal{F}(E_7))} \\
	{\overline{Y}_m^\mathrm{lc}} & {X(\mathcal{F}(E_6))}
	\arrow[from=1-1, to=1-2]
	\arrow["\pi"', from=1-1, to=2-1]
	\arrow[from=1-2, to=2-2]
	\arrow[from=2-1, to=2-2]
\end{tikzcd}\]
where the horizontal maps are closed embeddings, i.e. $\pi$ is entirely determined by the morphism of toric varieties $X(\mathcal{F}(E_7)) \to X(\mathcal{F}(E_6))$. Now there is a canonical combinatorial procedure for flattening a morphism of toric varieties.

\begin{pro}\cite[Section 10]{hacking_stable_2009}\label{pro:flattening}
	Let $p : \mathcal{F}(E_7) \to \mathcal{F}(E_6)$ be the morphism of fans induced by $\pi: \overline{Y}(E_7)^\mathrm{lc} \to \overline{Y}_m^\mathrm{lc}$. Let $\widetilde{\mathcal{F}}(E_6)$ be the refinement of $\mathcal{F}(E_6)$ obtained by taking the barycentric subdivision of cones formed by 4-tuples of pairwise orthogonal $A_1$ subsystems, as well as the corresponding minimal subdivisions of the cones formed by $3A_1$ subsystems and an $A_2^3$ subsystem. Let $\widetilde{\mathcal{F}}(E_7)$ be the corresponding refinement of $\mathcal{F}(E_7)$ given by
	\begin{align*}
		\widetilde{\mathcal{F}}(E_7) = \{p^{-1}(\gamma) \cap \sigma \mid \gamma \in \widetilde{\mathcal{F}}(E_6), \sigma \in \mathcal{F}(E_7)\}.
	\end{align*}
	Let $\widetilde{Y}_m$ and $\widetilde{Y}(E_7)$ be the toroidal modifications of $\overline{Y}_m^\mathrm{lc}$ and $\overline{Y}(E_7)^\mathrm{lc}$, respectively, corresponding to the refinement $\widetilde{\mathcal{F}}(E_n)$ of $\mathcal{F}(E_n)$. Then the induced morphism $\widetilde{\pi} : \widetilde{Y}(E_7) \to \widetilde{Y}_m$ is flat with reduced fibers. \qed
\end{pro}

\begin{pro}\cite[Thm.~1.1, 6.5]{schock_moduli_2024}
	The toroidal modification $\widetilde{Y}_m$ is the (normalization of the) KSBA moduli space $\overline{Y}_{(1, \dots, 1)}$ of (fully) marked cubic surfaces with rational coefficient vector $(1, \dots, 1)$. \qed
\end{pro}

Now while the flat morphism $\widetilde{\pi}  : \widetilde{Y}(E_7) \to \overline{Y}_{(1, \dots, 1)}$ seems like a good candidate for the universal family of weight 1 stable marked cubic surfaces, we still have the presence of Eckardt points (also called a triple point) on the fibers. In particular, if $(S, \sum c L_i)$ is a uniformly $c$-weighted marked smooth cubic surface with an Eckardt point $p$, then by \Cref{lem:linesA1lccriteria}, the pair is log canonical if and only if $c \leq \frac{2}{3}$. Thus, the presence of Eckardt points prevents $\widetilde{\pi}$ from being the universal family of weight 1 stable marked cubic surfaces as it fails the slc condition of \Cref{dfn:KSBAfamily}.

\subsection{Schock's explicit description of the fibers of $\widetilde{\pi} : \widetilde{Y}(E_7) \to \overline{Y}_{(1, \dots, 1)}$}\label{S:schockfiberdesc}
To handle the Eckardt point obstruction of  $\widetilde{\pi} :  \widetilde{Y}(E_7) \to \overline{Y}_{(1, \dots, 1)}$ being the universal family of weight 1 stable marked cubic surfaces, we must achieve two objectives:
\begin{enumerate}
	\item Describe the possible configurations of Eckardt points on the fibers of $\widetilde{\pi}$;\label{eck:desc}
	\item Explain the blow up of these Eckardt points to obtain the universal family of weight 1 stable marked cubic surfaces. \label{eck:blowup}
\end{enumerate}
In this subsection and the next, we will address \Cref{eck:desc} by summarizing Schock's explicit computation of the fibers of $\widetilde{\pi} :  \widetilde{Y}(E_7) \to \overline{Y}_{(1, \dots, 1)}$ found in \cite[\S~4]{schock_moduli_2024} and the Eckardt point configurations for each fiber found in \cite[\S~5]{schock_moduli_2024}. Recall that $\overline{Y}_{(1, \dots, 1)}$ is the blow up of $\overline{Y}_m^\mathrm{lc} \cong \overline{N}$ along the (strict transforms of the) intersections of $A_1$ divisors, in increasing order of dimension. With that in mind, we introduce the following convention of \textit{types of boundary divisors} for $\overline{Y}_{(1, \dots, 1)}$ that follows \cite[Ntn.~4.1]{schock_moduli_2024} and \cite{ren2016tropicalization}.

\begin{conv}[Boundary strata types]\label{conv:types}
	The boundary divisors of \textit{type $\widetilde{D}_{A_1}$} on $\overline{Y}_{(1, \dots, 1)}$ are the strict transforms of the $A_1$ divisors along the blow up $\overline{Y}_{(1, \dots, 1)}\to \overline{Y}_m^{\mathrm{lc}} \cong \overline{N}$. Similarly, the boundary divisors of \textit{type} $\widetilde{D}_{3A_2}$ on $\overline{Y}_{(1, \dots, 1)}$ are the strict transforms of the $3A_2$ divisors in $\overline{Y}_m^\mathrm{lc}$. For $i = 2, 3, 4$ the boundary divisors of \textit{type} $E_{iA_1}$ on $\overline{Y}_{(1, \dots, 1)}$ are the divisors in the exceptional locus of $\overline{Y}_{(1, \dots, 1)}\to \overline{Y}_m^{\mathrm{lc}}$ that lie over the strata given by the intersections of $i$ of the  $A_1$ divisors in $\overline{Y}_m^\mathrm{lc}$. Higher-codimension boundary strata are denoted by intersections. For example, the curves of type $$\widetilde{D}_{A_1} \cap \widetilde{E}_{2A_1} \cap \widetilde{D}_{3A_2}$$ are the curves obtained as intersections of boundary divisors of types $\widetilde{D}_{A_1}$, $\widetilde{E}_{2A_1}$, and $\widetilde{D}_{3A_2}$. Let $T$ be a boundary stratum of $\overline{Y}_{(1, \dots 1)}$, labeled as above; e.g., $T= \widetilde{D}_{A_1} \cap \widetilde{E}_{2A_1} \cap \widetilde{D}_{3A_2}$. We say a surface $(S, B)$ is of \textit{type} $T$ if it is obtained as the fiber of $\widetilde{\pi} : \widetilde{Y}(E_7) \to \overline{Y}_{(1, \dots, 1)}$ over a general point of the stratum $T$.
\end{conv}

\begin{rem}
	Note that \cite[Ntn.~4.1]{schock_moduli_2024} and \cite{ren2016tropicalization} use a different convention to denote the boundary types. Namely, they denote boundary divisors of type $\widetilde D_{A_1}$ by $a_1$ and $\widetilde E_{iA_1}$ by  $a_i$ for $i =  2, 3, 4$.  We prefer our notation, as we find it more descriptive. 
	\end{rem}
	
\begin{rem}\label{rem:line-choice-strata}
	For each boundary stratum of $\overline{Y}_{(1, \dots, 1)}$, the associated strata in $\overline{Y}_{(1, 1)}$ under the forgetful morphism $\overline{Y}_{(1, \dots, 1)} \to \overline{Y}_{(1, 1)}$ can be reducible. The different components correspond to different choices of the marked line $\ell$ (i.e., $L_1$) (see \Cref{S:ksba}).
\end{rem}

\begin{dfn}[Horizontal divisors and lines]\cite[Dfn.~10.18]{hacking_stable_2009}
	The \textit{horizontal} $A_1$ divisors on $\widetilde{Y}(E_7)$ are the strict transforms of the $A_1$ divisors on $\overline{Y}(E_7)^\mathrm{lc}$ which surject onto $\overline{Y}_m^\mathrm{lc}$ under the morphism $\pi : \overline{Y}(E_7)^\mathrm{lc} \to \overline{Y}_m^\mathrm{lc}$. In particular, by \Cref{tab:piboundary}, the horizontal $A_1$ divisors are the strict transforms of $D_{A_1} \subseteq \overline{Y}(E_7)^\mathrm{lc}$ such that $A_1 \not\subseteq E_6$; under the natural embedding of $E_6$ in $E_7$, there are $27$ horizontal $A_1$ divisors (see \cite[Ntn.~2.6, 4.3]{schock_moduli_2024} for more details). We denote by $B_{\widetilde{\pi}}$ the sum of the horizontal $A_1$ divisors on $\widetilde{Y}(E_7)$. 	
	The \textit{lines} on a fiber of $\widetilde{\pi} :  \widetilde{Y}(E_7) \to \overline{Y}_{(1, \dots, 1)}$ are given by the intersections of the fiber with the horizontal $A_1$ divisors.
\end{dfn}

\begin{teo}\cite[Thm.~4.4]{schock_moduli_2024}\label{teo:fullweightfibers}
	The fibers of the morphism $\widetilde{\pi} : (\widetilde{Y}(E_7), B_{\widetilde{\pi}}) \to \overline{Y}_{(1, \dots, 1)}$ over a general point of $\overline{Y}_{(1, \dots, 1)}$ is a smooth (fully) marked cubic surface. For a given boundary stratum $T$ of $\overline{Y}_{(1, \dots, 1)}$, the fiber of $\widetilde{\pi}$ over a general point of $T$ is a reducible surface as described in \cite[Figs.~4, 9, 10, 15, 16, 21, 22, 24, 28, 35, 36, 43, and 44]{schock_moduli_2024}. \qed
\end{teo}

\begin{proc}[Schock's strategy in computing the fibers of $\widetilde{\pi} : (\widetilde{Y}(E_7), B_{\widetilde{\pi}}) \to \overline{Y}_{(1, \dots, 1)}$]
	As the above theorem is proved in \cite{schock_moduli_2024}, we will not include its proof, but instead provide a summary of the general strategy. See \cite[Exs.~4.13, 4.14, 4.15, 4.16, and 4.17]{schock_moduli_2024} for concrete examples of applying this strategy.
	\begin{enumerate}
		\item By the intersection theory of boundary divisors of $\overline{Y}_m^\mathrm{lc} \cong \overline{N}$ in \Cref{pro:boundaryint}, the combinatorics of $E_6$ root system in \cite[App.~C]{schock_moduli_2024}, and the explicit blow up description of $\overline{Y}_{(1, \dots, 1)} \to \overline{Y}_m^\mathrm{lc}$, the possible boundary strata of $\overline{Y}_{(1, \dots, 1)}$ are listed below in \Cref{tab:bddstrata}. Recall the intersection of $i$ $D_{A_1}$ divisors in $\overline{Y}_m^\mathrm{lc} \cong \overline{N}$ is the codimension $i$ boundary stratum parameterizing cubic surfaces with $i$ $A_1$ singularities, i.e. the Type $iA_1$ surfaces in \Cref{tab:narukifibers}. On the other hand, the boundary divisor $D_{3A_2}$ parameterizes the Type $3A_2$ surfaces in \Cref{tab:narukifibers}.
		\item Pick a boundary stratum $T$ of codimension $k = 1, 2, 3, 4$ from above. Then either $T$ is outside or intersects the union of the non-flat $2A_3$ divisors (see \Cref{tab:piboundary}). Note that this boundary stratum $T$ is a chosen $W(E_6)$-representative, but this is sufficient as the boundary strata for a given type are transitively permuted by $W(E_6)$.
		\item \label{step:flatdivisors} If $T$ is outside the union of the non-flat $2A_3$ divisors, $\widetilde{\pi}$ is the pullback of $\pi: \overline{Y}(E_7)^\mathrm{lc} \to \overline{Y}_m^\mathrm{lc}$ (\Cref{pro:flattening}), so we can compute the fibers of $\pi$ instead. So let $(S, B)$ be a fiber of $\pi$ over a general point of $T$ and $U$ a $k$-codimensional boundary stratum of $\overline{Y}(E_7)^\mathrm{lc}$ dominating $T$. Then the general fiber of $\pi|_U : U \to T$ is a surface $S'$ yielding an irreducible component of $(S, B)$. Concretely, the possible choices of $U$ can be found using \Cref{tab:piboundary} and the combinatorics of the $E_7$ root subsystems found in \cite[App.~C]{schock_moduli_2024}. Then the irreducible component $S'$ is computed using the lemmas of \cite[\S 4.1]{schock_moduli_2024} where it explicitly describes the fibers of $\pi|_U$ for each restriction described in \Cref{tab:piboundary}. On the other hand, the lines $B'$ of $S'$ are obtained by the intersections $U\cap D_{A_1}$ of $U$ with the horizontal $A_1$ divisors $D_{A_1}$ on $\widetilde{Y}(E_7)$ (see \Cref{pro:boundaryint}).
		\item \label{step:non-flatdivisor} If $T$ intersects the union of the non-flat $2A_3$ divisors, \cite[Proof of Prop.~10.23]{hacking_stable_2009} or \cite[Lem.~4.12]{schock_moduli_2024} describes the restriction of $\widetilde{\pi}$ to the strict transforms of these divisors, yielding the irreducible components. The lines are obtained similarly.
	\end{enumerate}
	Applying this strategy, we obtain the surfaces described in \cite[Figs.~4, 9, 10, 15, 16, 21, 22, 24, 28, 35, 36, 43, and 44]{schock_moduli_2024} as listed in \Cref{tab:bddstrata}.
\end{proc}

\begin{table}
	\begin{tabular}{ | m{5cm} | m {5cm} | m{5cm} |}
		\hline
		Boundary strata in $\overline{Y}_{(1, \dots 1)}$ & Corresponding figure(s) in \cite{schock_moduli_2024} & Corresponding strata in $\overline{Y}_m^\mathrm{lc}$\\
		\hline
		$\widetilde{D}_{A_1}$ & Figure 4 & $D_{A_1}$\\
		\hline
		$\widetilde{E}_{2A_1}$& Figures 9 and 10 & $D_{A_1} \cap D_{A_1}$\\
		$\widetilde{D}_{A_1} \cap \widetilde{E}_{2A_1}$ & &\\ 
		\hline
		$\widetilde{E}_{3A_1}$& Figures 15 and 16 & $D_{A_1} \cap D_{A_1} \cap D_{A_1}$\\
		$\widetilde{D}_{A_1} \cap E_{3A_1}$& &\\
		$\widetilde{E}_{2A_1} \cap E_{3A_1}$& &\\
		$\widetilde{D}_{A_1} \cap \widetilde{E}_{2A_1} \cap \widetilde{E}_{3A_1}$ & &\\
		\hline
		$\widetilde{E}_{4A_1}$ & Figures 21 and 22 & $D_{A_1} \cap D_{A_1} \cap D_{A_1} \cap D_{A_1}$\\
		$\widetilde{D}_{A_1} \cap \widetilde{E}_{4A_1}$ & &\\
		$\widetilde{E}_{2A_1} \cap \widetilde{E}_{4A_1}$ & &\\
		$\widetilde{E}_{3A_1} \cap \widetilde{E}_{4A_1}$ & &\\
		$\widetilde{D}_{A_1} \cap \widetilde{E}_{2A_1} \cap \widetilde{E}_{4A_1}$ & &\\
		$\widetilde{D}_{A_1} \cap \widetilde{E}_{3A_1} \cap \widetilde{E}_{4A_1}$ & &\\
		$\widetilde{E}_{2A_1} \cap \widetilde{E}_{3A_1} \cap \widetilde{E}_{4A_1}$ & &\\
		$\widetilde{D}_{A_1} \cap E_{2A_1} \cap E_{3A_1} \cap E_{4A_1}$ & &\\
		\hline
		$\widetilde{D}_{3A_2}$ & Figure 24 & $D_{3A_2}$\\
		\hline
		$\widetilde{D}_{A_1} \cap \widetilde{D}_{3A_2}$ & Figure 28 & $D_{A_1} \cap D_{3A_2}$\\
		\hline
		$\widetilde{E}_{2A_1} \cap \widetilde{D}_{3A_2}$ & Figures 35 and 36 & $D_{A_1} \cap D_{A_1} \cap D_{3A_2}$\\
		$\widetilde{D}_{A_1} \cap \widetilde{E}_{2A_1} \cap \widetilde{D}_{3A_2}$ & &\\
		\hline
		$\widetilde{E}_{3A_1} \cap \widetilde{D}_{3A_2}$ & Figures 43 and 44 & $D_{A_1} \cap D_{A_1} \cap D_{A_1} \cap D_{3A_2}$\\
		$\widetilde{D}_{A_1} \cap E_{3A_1} \cap \widetilde{D}_{3A_2}$ & &\\
		$\widetilde{E}_{2A_1} \cap E_{3A_1} \cap \widetilde{D}_{3A_2}$& &\\
		$\widetilde{D}_{A_1} \cap \widetilde{E}_{2A_1} \cap \widetilde{E}_{3A_1} \cap \widetilde{D}_{3A_2}$ & &\\ 
		\hline
	\end{tabular}
	\caption{The most-left column lists all possible boundary strata in $\overline{Y}_{(1, \dots, 1)}$. The center column gives the corresponding figures in \cite{schock_moduli_2024} explicitly describing the fiber of $\widetilde{\pi} : (\widetilde{Y}(E_7), B_{\widetilde{\pi}}) \to \overline{Y}_{(1, \dots, 1)}$ over a general point in each boundary strata in $\overline{Y}_{(1, \dots, 1)}$. The most-right column lists the strata in $\overline{Y}_m^{\text{lc}}$ that the corresponding boundary strata in $\widetilde{Y}_{(1, \dots, 1)}$ maps to via the blow-down $\overline{Y}_{(1, \dots, 1)} \to \overline{Y}_{m}^{\text{lc}}$.}
	\label{tab:bddstrata}
\end{table}

\subsection{Configuration of Eckardt points}
With the fibers of $\widetilde{\pi}$ computed, we can now describe the possible configurations of Eckardt points on the fibers. From  \cite[Figs.~4, 9, 10, 15, 16, 21, 22, 24, 28, 35, 36, 43, and 44]{schock_moduli_2024},  we will need to describe the configuration of Eckardt points on the following irreducible surfaces:
\begin{enumerate}
	\item $S$ a cubic surface that is either smooth, has only $A_1$ singularities, or has exactly one singularity of type $A_2$;
	\item The blow up of $\mathbb{P}^2$ at five points $p_1, \dots, p_5$, such that $p_1, \dots, p_4$ are in general position and $p_5$ is the intersection point of the line through $p_1$ and $p_2$ with the line through $p_3$ and $p_4$.
\end{enumerate}

The following results regarding the configuration of Eckardt points are stated in \cite{tu_semi-stable_2005}, \cite{cools2011singular}, and \cite{schock_moduli_2024}, but we will restate it here for the reader's convenience. 

\begin{lem}\cite[Prop.~5.3]{tu_semi-stable_2005}\cite[Lem.~4.8]{cools2011singular}
	Let $S$ be a cubic surface that has only $A_1$ singularities, or has exactly one singularity, of type $A_2$. Then any Eckardt point of $S$ is the specialization of an Eckardt point on a smooth cubic surface. \qed
\end{lem}

\begin{rem}
	Note that the components isomorphic to $\mathrm{Bl}_6 \mathbb{P}^2$ are precisely the minimal resolutions of a cubic surface that has exactly one singularity of type $A_2$ as the six points blown up are in two sets of three colinear points.
\end{rem}

\begin{pro}\cite[Prop.~5.9]{schock_moduli_2024}\label{pro:cubiceck}
	Let $S$ be a cubic surface that is either smooth, has only $A_1$ singularities, or has exactly one singularity, of type $A_2$.
	\begin{enumerate}
		\item If $S$ is smooth, then $S$ has either $0, 1, 2, 3, 4, 6, 9, 10$, or $18$ Eckardt points.
		\item If $S$ has one $A_1$ singularity, then $S$ has either $0, 1, 2, 3, 4$, or $6$ Eckardt points.
		\item If $S$ has two $A_1$ singularities, then $S$ has either $0$ or $1$ Eckardt points.
		\item If $S$ has three $A_1$ singularities, then $S$ has either $0$ or $1$ Eckardt points.
		\item If $S$ has four $A_1$ singularities, then $S$ has no Eckardt points.
		\item If $S$ has one $A_2$ singularity, then $S$ has either $0, 1, 2$, or $3$ Eckardt points.
	\end{enumerate}\qed
\end{pro}

\begin{pro}\cite[Prop.~5.11]{schock_moduli_2024}\label{pro:bl5p2eck}
	Let $S$ be the blow up of $\mathbb{P}^2$ at five points $p_1, \dots, p_5$ such that $p_1, \dots, p_4$ are in general position and $p_5$ is the intersection point of the line through $p_1$ and $p_2$ with the line through $p_3$ and $p_4$. Consider the collection of lines on $S$ consisting of the four lines $\ell_{13}$, $\ell_{14}$, $\ell_{23}$, $\ell_{24}$, where $\ell_{ij}$ is the strict  transform of the line through $p_i$ and $p_j$, and a line $\ell$ the strict transform of a line passing through $p_5$, but not through  any of $p_1, \dots, p_4$. Then $S$ has two potential Eckardt points--the triple intersection $\ell \cap \ell_{13} \cap \ell_{24}$, and the triple intersection $\ell \cap \ell_{23} \cap \ell_{14}$. It is not possible to have both of these Eckardt points. Thus, $S$ has either $0$ or $1$ Eckardt points. \qed
\end{pro}

We can now describe the possible configurations of Eckardt points on the fibers of $\widetilde{\pi}$, completing \Cref{eck:desc} of \Cref{S:schockfiberdesc}.

\begin{teo}\cite[Thm.~5.13]{schock_moduli_2024}\label{teo:eckconfig}
	The possible configurations of Eckardt points on irreducible components of fibers of $\widetilde{\pi} :  (\widetilde{Y}(E_7), B_{\widetilde{\pi}}) \to \overline{Y}_{(1, \dots, 1)}$ of types
	\begin{align*}
		\widetilde{D}_{A_1}, \hspace{.1in} \widetilde{E}_{2A_1} , \hspace{.1in} \widetilde{E}_{3A_1}, \hspace{.1in} \widetilde{E}_{4A_1}, \hspace{.1in} \widetilde{D}_{3A_2}, \hspace{.1in} \widetilde{D}_{A_1} \cap \widetilde{D}_{3A_2}, \hspace{.1in}  \widetilde{E}_{2A_1} \cap \widetilde{D}_{3A_2}, \text{ and } \widetilde{E}_{3A_2} \cap \widetilde{D}_{3A_2}.
	\end{align*} 
	are listed in
	\ifAppendix
	Tables \ref{tab:DA1eck} to \ref{tab:E3A1D3A2eck} (also found in \cite[Tables~2 to 9]{schock_moduli_2024}). 
	\fi
	The remaining fibers of $\widetilde{\pi}$ are obtained as degenerations of the listed fibers, where some of the components isomorphic $\mathrm{Bl}_5 \mathbb{P}^2$ as in \Cref{pro:bl5p2eck} degenerate into four irreducible components, as described in \cite[Ex.~4.16]{schock_moduli_2024}. No Eckardt points can occur on these degenerate components. 
\end{teo}
\begin{proof}
	This follows directly from Propositions \ref{pro:cubiceck} and \ref{pro:bl5p2eck}, and Schock's explicit description of the fibers of $\widetilde{\pi}$ found in \cite[Figs.~4, 9, 10, 15, 16, 21, 22, 24, 28, 35, 36, 43, and 44]{schock_moduli_2024}.
\end{proof}

\subsection{Blow up of Eckardt points} 
In this subsection, we will explore \Cref{eck:blowup} of \Cref{S:schockfiberdesc} by explaining the blow up process of Eckardt points of the fibers of $\widetilde{\pi} : \widetilde{Y}(E_7) \to \overline{Y}_{(1, \dots, 1)}$ and obtain the universal family of weight 1 stable marked cubic surfaces as in \cite[Thm.~10.31]{hacking_stable_2009}. Let us first be precise in what we mean by an Eckardt point of a fiber of $\widetilde{\pi}$. This is also summarized in \cite[\S~5]{schock_moduli_2024}.

\begin{dfn}[Eckardt locus]\cite[Definition 10.21]{hacking_stable_2009}
	The \textit{Eckardt locus} of $\widetilde{Y}(E_7)$ is the union of the intersections of three horizontal $A_1$ divisors of $\widetilde{\pi} : \widetilde{Y}(E_7) \to \overline{Y}_{(1, \dots, 1)}$. The \textit{Eckardt points} on a fiber $(S, B)$ of $\widetilde{\pi} : (\widetilde{Y}(E_7), B_{\widetilde{\pi}}) \to \overline{Y}_{(1, \dots, 1)}$ are the intersection points of $(S, B)$ with the Eckardt locus of $\widetilde{Y}(E_7)$.
\end{dfn}

It was shown in \cite[Prop~10.23]{hacking_stable_2009} that any Eckardt points on a fiber $(S, B)$ of $\widetilde{\pi}$ lie in the smooth locus of $S$. This allows for the following result for obtaining the universal family of weight 1 stable marked cubic surfaces over $\overline{Y}_{(1, \dots, 1)}$.

\begin{teo}\cite[Thm~10.31]{hacking_stable_2009}\label{teo:eckblowup}
	Let $\ddot{Y}(E_7)$ denote the blow up of $\widetilde{Y}(E_7)$ along the Eckardt locus, and let
	\begin{align*}
		\ddot{\pi} : (\ddot Y(E_7), B_{\ddot \pi}) \to \overline{Y}_{(1, \dots, 1)}
	\end{align*}
	be the induced morphism to $\overline{Y}_{(1, \dots, 1)}$, where $B_{\ddot \pi}$ is the strict transform of $B_{\widetilde{\pi}}$. Then $\ddot\pi$ is a flat family of KSBA stable pairs, giving the universal family of weight 1 stable marked cubic surfaces over the moduli space $\overline{Y}_{(1, \dots, 1)}$. Each fiber $(\ddot S, \ddot B)$ of $\ddot\pi$ is obtained from a fiber $(S, B)$ of $\widetilde{\pi}$ by blowing up each Eckardt point on $(S, B)$ and attaching to the exceptional divisor a copy of $\mathbb{P}^2$ glued along a general line. The pullbacks of the three lines passing through the Eckardt point on $(S, B)$ intersect the $\mathbb{P}^2$ in three lines in general position, as pictured below (copied from \cite[Fig.~1]{schock_moduli_2024})
	\begin{center}
		\includegraphics[scale=0.3]{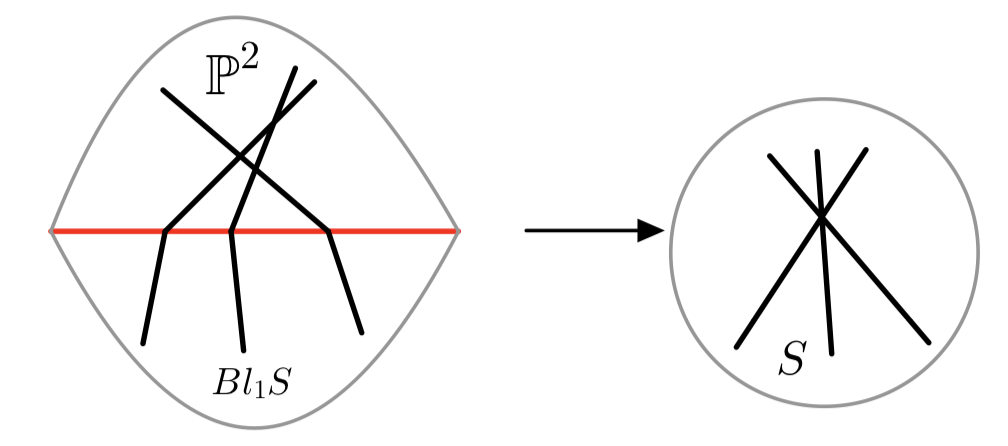}
	\end{center}\qed
\end{teo}

	\section{Determining change in moduli}\label{S:modulichange}
In this section and next, we will establish some useful results that will be needed to prove  \Cref{teo:main} in \Cref{S:proofofchambers} along with the explicit descriptions of the surfaces parameterized at the boundary. In this section, we establish a useful lemma that will help us determine when the coarse moduli space underlying its moduli stack is isomorphic. The following is a special case of main result in \cite{ascher_wall_2023}.

\begin{teo}{\cite[Thm.~1.1]{ascher_wall_2023} }\label{teo:wallcrossing}
	For $(b, c) \in \mathrm{Amp}$ (see \Cref{eqn:ampcone}), 
	there exists a normal projective variety $\overline{Y}_{(b, c)}$ (see \Cref{def:(bc)modulistack}) and a flat family $(\mathcal{S}_{(b, c)}, (b, c)\mathcal{B})$ compactifying the moduli space $Y_{(b, c)}$ of smooth $(b, c)$-weighted stable marked cubic surfaces. Furthermore, there is a finite rational polyhedral wall-and-chamber decomposition of the weight domain $\mathrm{Amp}$ (\Cref{eqn:ampcone}) such that the following hold.
	\begin{enumerate}[label=(\arabic*)]
		\item \label{teo:wallcrossing1}For $(b, c)$ and $(b' ,c')$ contained in the same chamber, there are canonical isomorphisms
		\[
		\begin{tikzpicture}[ampersand replacement=\&, 
			every node/.style={anchor=center},
			xscale=3, yscale=2]
			
			\matrix (m) [matrix of math nodes, row sep=2em, column sep=3em] {
				\mathcal{S}_{(b, c)} \& \mathcal{S}_{(b', c')} \\
				\overline{Y}_{(b, c)} \& \overline{Y}_{(b', c')} \\
			};
			
			\draw[->] (m-1-1) -- node[above] {$\sim$} (m-1-2);
			\draw[->] (m-2-1) -- node[below] {$\sim$} (m-2-2);
			
			\draw[->] (m-1-1) -- (m-2-1);
			\draw[->] (m-1-2) -- (m-2-2);
			
		\end{tikzpicture}
		\]
		\item \label{teo:wallcrossing2}For $(b, c)$ and $(b, c')$ in different chambers with $b \geq b'$ and $c \geq c'$, then there is a canonical functorial birational wall-crossing morphism 
		\begin{align*}
			\rho_{(b', c')}^{(b, c)} : \overline{Y}_{(b, c)} \to \overline{Y}_{(b' ,c')}
		\end{align*}
		induced by a canonical birational map $\mathcal{S}_{(b, c)} \dashrightarrow \mathcal{S}_{(b', c')}$ sending a stable pair $(S, (b, c)B)$ to the stable model of the pair $(S, (b', c')B)$. Furthermore, for any third weight vector $(b'', c'')$ with $b' \geq b''$ and $c' \geq c''$, we have 
		\begin{align*}
			\rho_{(b'', c'')}^{(b', c')} \circ \rho_{(b', c')}^{(b, c)} = \rho_{(b'', c'')}^{(b, c)}. 
		\end{align*}
	\end{enumerate}
	\qed
\end{teo}

Using this theorem, we can deduce a case of wall-crossing where the underlying moduli space does not change.

\begin{lem}\label{lem:isomoduli}
	Suppose we are in the following scenario locally in the wall-and-chamber decomposition:
	\begin{center}
		\begin{tikzpicture}[scale=2]
			\draw[dashed] (0, 0) -- (2, 0);
			\draw[dashed] (2, 0) -- (2, 1);
			\draw[dashed] (2, 1) -- (0, 1);
			\draw[dashed] (0, 1) -- (0, 0);
			\draw[thick] (0, 0) -- (2, 1);
			\node at (1/2, 1.5/2) {Chamber $A$};
			\node at (3/2, .5/2) {Chamber $B$};
		\end{tikzpicture}
	\end{center}
	where chamber $A$ includes the wall of positive slope, i.e. we are still in chamber $A$ at the wall. Denote by $\overline{Y}_A$ the coarse moduli space $\overline{Y}_{(b, c)}$ for $(b, c)$ in chamber $A$ and similarly for $\overline{Y}_B$. Then the coarse moduli spaces $\overline{Y}_A \cong \overline{Y}_B$ are isomorphic (even though, by assumption, the moduli stacks are different).
\end{lem}

\begin{proof}
	First, recall that from \Cref{teo:wallcrossing}\ref{teo:wallcrossing1}, we have that for all $(b,c)$ in chamber $A$ that the coarse moduli spaces $\overline Y_{(b,c)}$ are isomorphic, so that it makes sense to write $\overline Y_A$, and similarly for $\overline Y_B$.
	Now for the isomorphism, without loss of generality, suppose the wall is given by $c = mb$ for some rational $m > 0$. Then take a point $(b, mb)$ on the wall in chamber $A$, so that for some small rational $0 < \epsilon \ll 1$ we have the point $(b, mb-\epsilon)$ (i.e., moving down by $\epsilon$) is in chamber $B$. Then by   \Cref{teo:wallcrossing}\ref{teo:wallcrossing2}, there is a birational wall-crossing morphism $\rho' : \overline{Y}_A \to \overline{Y}_B$. On the other hand, by now moving left by $\frac{1}{m}\epsilon$, we have the point $(b - \frac{1}{m} \epsilon, mb - \epsilon)$ on the wall in chamber $A$, and  \Cref{teo:wallcrossing}\ref{teo:wallcrossing2}  gives us a birational wall-crossing morphism $\rho'' : \overline{Y}_B \to \overline{Y}_A$. Composing $\rho'$ and $\rho''$, we get a canonical isomorphism
	\begin{align*}
		\rho = \rho'' \circ \rho' : \overline{Y}_A \to \overline{Y}_B \to \overline{Y}_A.
	\end{align*}
	It remains to show that $\rho''$ is finite. Since $\rho''$ is proper, to show $\rho''$ is finite, it suffices to show that $\rho''$ is quasi-finite. Suppose for contradiction there is a point $x \in \overline{Y}_A$ such that $(\rho'')^{-1}(x)$ has positive dimension. Using the fact that $\rho$ is an isomorphism, there is a unique $y \in \overline{Y}_A$ such that $\rho(y) = (\rho'' \circ \rho')(y) = x$. However, we get that
	\begin{align*}
		\rho^{-1}(x) &= (\rho')^{-1} \circ (\rho'')^{-1}(x)
	\end{align*}
	would be positive dimensional, a contradiction. Hence $\rho''$ is finite and since $\rho''$ is a finite birational morphism between two normal proper varieties, by Zariski's main theorem, $\rho''$ must be an isomorphism. Thus $\overline{Y}_B \cong \overline{Y}_A$ and the desired result is obtained.
\end{proof}
	
	\section{Lemmas regarding gluing and the log canonical model}\label{S:gluing}
In this section, we discuss a method to obtain log canonical models of reducible surfaces by gluing log canonical models of the components. This is well-understood in the literature, but we discuss here in detail a special case that appears in our situation. Informally, this says that our ``naive" approach at each wall-crossing of checking ampleness on each irreducible component of $X$ and then gluing the irreducible components after taking the appropriate contractions is indeed the same as the log canonical model under certain base-point free conditions. In particular, our goal is to prove the following result.

\begin{pro}[\Cref{cor:gluing}]
	Let $X$ be a complex projective variety with irreducible components  $X_1, \dots, X_n$ with normalization $\overline{X}_1, \dots, \overline{X}_n$, let $\overline B_i = \overline B|_{\overline X_i}$ be the restriction of $\overline B$ to $\overline X_i$ and similarly, let $\overline{D} = \sum_i \overline D_i = \sum_i \overline D|_{\overline X_i}$ be the restriction of the conductor. Suppose for all $1 \leq i \leq n$ the pair $(\overline X_i,\overline B_i+\overline D_i)$ is klt. Further suppose
	\begin{enumerate}
		\item $\omega_{\overline X_i}^{m-1}((m-1)\overline B_i + (m-2)\overline D_i)$ is big and nef, 
		\item $\omega_{\overline X_i}^m (m(\overline B_i + \overline D_i))$ is semi-ample, and 
		\item $|\omega_{\overline X_i}^{m} (m(\overline B_i + \overline D_i))|_{\overline D_i}|$ is basepoint-free.
	\end{enumerate}
	Then the log canonical model of $(X, B)$ is obtained by taking the log canonical models of each of the $(\overline X_i, \overline B_i + \overline D_i)$ and gluing them along the conductors $\overline D_i$ under a generically fixed-point-free Galois involution $\tau : (\overline D^n, \operatorname{Diff}_{\overline D^n} \overline B) \to (\overline D^n, \operatorname{Diff}_{\overline D^n} \overline B)$ of the normalization $\overline{D}^n \to \overline{D}$ with the different $\operatorname{Diff}_{\overline D^n} \overline B$.
\end{pro}

In our setting of a $(b, c)$-weighted marked cubic surface $(X, B)$, as we decrease our $(b,c)$-weights and reach a wall where the ampleness condition fails, we can apply the proposition above along with the following theorem:

\begin{teo}\cite[Thm.~5.13]{kollar_singularities_2013}, \cite[Thm.~11.38]{kollar_families_2023}\label{teo:kollar5.13}
	Normalization gives a one-to-one correspondence between proper slc pairs $(X, B)$ such that $K_X + B$ is ample and proper lc pairs $(\overline X, \overline B + \overline D)$ along with a generically fixed point free involution $\tau$ of $(\overline D^n, \operatorname{Diff}_{\overline D^n}\overline B)$ such that $K_{\overline X} + \overline B + \overline D$ is ample.
\end{teo}

This will allow us to take a birational replacement $\overline X_i'$ for each irreducible component $\overline X_i$ via an associated contraction morphism $\overline X_i \to \overline X_i'$. Then after an appropriate gluing of each $\overline X_i'$, we then check if we get a stable replacement of $(X, B)$. We continue this process with the stable replacement of $(X, B)$ until we reach our minimal weight of $(b, c) = (\frac{1}{9} + \epsilon, \frac{1}{9} + \epsilon)$. 

\subsection{The different and the Poincar\'e residue map} 
We start by reviewing some of the standard setup which can also be found in \cite[\S~4]{kollar_singularities_2013} or \cite[\S~11]{kollar_families_2023}, regarding the different and the Poincar\'e residue map.  This is important for the process of gluing components of reducible surfaces.  In the case we are in, where all of the surfaces have at worst isolated $A_1$ singularities, and the curves of attachment are all smooth (in fact smooth rational curves), this is somewhat simpler, but we review how to make some of the simplifications here. Let $S$ be a regular surface and $B \subseteq S$ a regular curve (as in \cite[(1.5)]{kollar_singularities_2013}, see \cite[Dfn~4.2]{kollar_singularities_2013} or \cite[Dfn~11.14]{kollar_families_2023} for a higher dimensional generalization). Recall the classical adjunction formula states that $K_S \sim (K_S + B)|_B$. If $X$ is only normal, then each singularity of $(S, B)$ leads to a correction term, called \textit{the different}. 

\begin{dfn}[The different]\cite[Dfn~2.34]{kollar_singularities_2013}
	Let $S$ be a normal surface and $B \subseteq S$ a reduced curve with normalization $\overline{B} \to B$. Let $B'$ be a $\mathbb{Q}$-Cartier divisor that has no components in common with $B$. Let $f : T \to S$ be a log resolution of $(S, B)$ with exceptional curves $C_i$ and $B_T, B'_T$ the strict transforms of $B, B'$, respectively. Then there is a unique $\Delta(S, T, B + B') := \sum d_j C_j$ such that
	\begin{align*}
		(\Delta(S, T, B + B') \cdot C_i) = -((K_T + B_T + B_T') \cdot C_i) \hspace{.2in} \text{for all } i.
	\end{align*}
	Then \textit{the different} is defined as
	\begin{align*}
		\operatorname{Diff}_{\overline{B}}(B') := (B'_T + \Delta(S, T , B + B'))|_{B_T}.
	\end{align*}
\end{dfn}

\begin{rem}
	In terms of $\mathbb{Q}$-divisors, the different $\operatorname{Diff}_{\overline{B}}(B')$ is the divisor such that
	\begin{align*}
		(K_S + B + B')|_{\overline{B}} \sim_{\mathbb{Q}} K_{\overline{B}} + \operatorname{Diff}_{\overline{B}}(B').
	\end{align*}
	In our setting, $S$ will have at worst $A_1$ singularities and $B$ is smooth, so the classical adjunction formula gives us
	\begin{align*}
		\operatorname{Diff}_{B}(B') \sim_{\mathbb{Q}} B'|_{B}.
	\end{align*}
\end{rem}

\begin{dfn}[Poincar\'e residue map]\cite[Dfn.~4.1]{kollar_singularities_2013}\cite[Dfn~11.13]{kollar_families_2023} \label{def:PResM}
  	If $X$ is a regular scheme, we view $\omega_X$ and $\omega_D$ as dualizing sheaves and we get the usual short exact sequence
	\begin{align*}
		0 \to \omega_X \to \omega_X(D) \xrightarrow{\mathcal{R}_{X \to D}} \omega_D \to 0
	\end{align*}
	where $\mathcal{R}_{X \to D} : \omega_X(D) \to \omega_D$ is the Poincar\'e residue map. As a slight abuse of terminology, we also refer to the induced map on global sections
	\begin{align*}
		\mathcal{R}_{X \to D} : H^0(X, \omega_X(D)) \to H^0(D, \omega_D)
	\end{align*}
	as the \textit{Poincar\'e residue map}. This also generalizes to pairs $(X, B)$ where $B$ is a $\mathbb{Q}$-divisor with no components in common with $D$ (see \cite[(1.5)]{kollar_singularities_2013}), to get
	\begin{align*}
		\mathcal{R}_{X \to D} : \omega_X(\Delta + D) \to \omega_S(\operatorname{Diff}_D(B))
	\end{align*}
	and
	\begin{align*}
		\mathcal{R}_{X \to D} : H^0(X, \omega_X(B + D)) \to H^0(S, \omega_S(\operatorname{Diff}_{\overline{D}}(B))).
	\end{align*}
\end{dfn}

The following result tells us when pluri-log canonical sections on a divisor descends to a pluri-log canonincal section of the entire surface.

\begin{pro}\cite[Prop.~5.8]{kollar_singularities_2013}\label{pro:kollar5.8}
	A section $\phi$ of $\omega_{\overline X}^m(m(\overline B + \overline D))$ descends to a section of $\omega_X^{[m]}(mB)$ if and only if $\mathcal{R}^m_{\overline X \to \overline D}(\phi)$ is $\tau$-invariant if $m$ is even and $\tau$-anti-invariant if $m$ is odd, where $\mathcal{R}^m_{\overline X \to \overline D}$ is the Poincar\'e residue map. \qed
\end{pro}

\subsection{Gluing lemmas and main strategy}
We now prove how one can glue sections of log canonical rings on reducible surfaces to obtain stable models. The key points are that the irreducible components are rational surfaces, and have well-understood nef cones. This allows us to have the ''naive" approach to constructing stable replacements as we reduce the coefficients on the boundary divisors.

\subsubsection{Main strategy}
The set-up for us will be a reducible surface $X$ with irreducible components $X_1, \dots, X_n$ where each $X_i$ has at worst isolated $A_1$ singularities. In fact, recall that we have a pair $(X,(b, c)B)$ where $(b, c) \in ((1/9, 1] \cap \mathbb{Q})^2$
\begin{align*}
	(b, c)B &= b \ell + \sum_2^{27} c L_i
\end{align*}
is the sum of the 27 lines (counting multiplicity) where our marked line $\ell$ has weight $b$ and the other 26 lines have uniform weight $c$. In our situation, we will start with weight $(b, c)$ in a chamber of $\operatorname{Amp} \subset ((\frac{1}{9}, 1] \cap \mathbb{Q})^2$ where the pair is KSBA stable. We then start decreasing the weight along some ray until we hit a wall $(b_0, c_0)$ where the pair $(X, (b_0, c_0)B)$ is still slc and the log canonical form $K_X + (b_0, c_0)B$ is semi-ample and $\mathbb{Q}$-factorial, but is no longer ample. Then by Theorem \ref{teo:kollar5.13}, there must be some irreducible component(s) of the normalization $(\overline{X}_i, (b_0, c_0) \overline{B}_i)$ of the normalization $(\overline{X}, (b, c)\overline{B})$ where $K_{\overline{X}_i} + (b_0, c_0) \overline{B}_i + \overline{D}_i$ is no longer ample, where $\overline{D}_i$ is the restriction of the conductor subscheme to $\overline{X}_i$. For each irreducible component, we obtain a candidate an irreducible surface via an associated contraction morphism $(\overline{X}_i, (b_0, c_0) \overline{B}_i) \to (\overline{X}_i', (b_0, c_0)\overline{B}_i')$. In other words, since in our situation, the gluing curves $\overline{D}_i$ are always smooth, we have a natural candidate $(X', (b_0, c_0)B')$ for the stable replacement of $(X, (b_0, c_0)B)$ given by gluing the $(\overline{X}_i', (b_0, c_0)\overline{B}_i')$ along the $D_i$.

However, $(X', (b_0, c_0)B')$ is not necessarily the stable replacement of $(X, (b_0, c_0)B)$. The issue is that the stable replacement is obtained from the $\operatorname{Proj}$ of the log canonical ring of the entire surface, and there can be issues when gluing sections of the log canonical ring on each component. In principle, the log canonical form $K_{\overline{X}_i} + \overline{B}_i + \overline{D}_i$ of the irreducible components $(\overline{X}_i, (b_0, c_0)\overline{B}_i)$ might have basepoints along $\overline{D}_i$, the volume condition might fail (see \Cref{E2A1_mult4_c=b3} for an example) and/or the $\mathbb{Q}$-factorial condition might fail. Thus we would like conditions guaranteeing that there are enough log canonical sections on each irreducible component to glue and give log canonical sections on the entire surface. This allows us to conclude in these situations that the stable replacement is indeed the ``naive" candidate, obtained by gluing the stable replacements of each component; this is the goal of \Cref{cor:gluing} below. 

\begin{rem}
	As a small aside, we note that in the cases where the conditions of Lemmas \ref{lem:nobspts} and \ref{lem:gluing} does not hold, one can often do a series of blow ups and blow-downs to obtain new surfaces where the ``naive" approach works.
\end{rem}

\subsubsection{Gluing lemmas used to execute the main strategy}

Let us formalize the discussion above. Following \cite[\S 5]{kollar_singularities_2013}, let $X$ be a complex projective demi-normal scheme with normaliztion $\pi : \overline X \to X$ and conductors $D \subseteq X$, $\overline D \subseteq \overline X$, and let $\tau: (\overline D^n, \operatorname{Diff}_{\overline D^n} \overline B) \to (\overline D^n, \operatorname{Diff}_{\overline D^n} \overline B)$ be a generically fixed point free (fixed point set does not cointain any irreducible components of $\overline D^n$) Galois involution of the normalization $\overline D^n \to \overline D$. Since we work in characteristic $0$, the Galois condition is automatic.  Furthermore, let $B$ be an effective $\mathbb{Q}$-divisor called the \textit{boundary divisor} whose support does not contain any irreducible component of $D$ and $\overline B$ the divisorial part of $\pi^{-1}(D)$. 

We use the following terminology going forward: we assume  $X$ is reducible with complex projective irreducible components  $X_1, \dots, X_n$ with normalization $\overline{X}_1, \dots, \overline{X}_n$, $\overline B_i = \overline B|_{\overline X_i}$ as the restriction of $\overline B$ to $\overline X_i$ and similarly, $\overline D_i = \overline D|_{\overline X_i}$. We further make the assumption that each pair $(\overline{X_i}, \overline{B_i} + \overline{D_i})$ is at worst klt. We now state our lemmas regarding our log canonical linear systems, that essentially state that if the log canonical divisors with conductors are big and nef on each component, and are base-point free along the conductors, then the ``naive" strategy outlined above will work.

\begin{lem}\label{lem:nobspts}
	Fix a sufficiently divisble $m \geq 0$. Suppose for all $1 \leq i \leq n$ the pair $(\overline X_i,\overline B_i+\overline D_i)$ is klt. Further suppose
	\begin{itemize}
		\item $\omega_{\overline X_i}^{m-1}((m-1)\overline B_i + (m-2)\overline D_i)$ is big and nef, and
		\item $|\omega_{\overline X_i}^{m} (m(\overline B_i + \overline D_i))|_{\overline D_i}|$ is basepoint-free.
	\end{itemize}
	Then $|\omega_X^{[m]}(mB)|$ is basepoint-free along $D$; i.e., has no basepoints at points in the support of $D$.
\end{lem}
\begin{proof}
	Consider the short exact sequence obtained by twisting the standard exact sequence for the divisor $\overline D_i$ on $\overline X_i$ with  $\omega_{\overline X_i}^{m}(m\overline B_i +(m-1) \overline D_i)$; we get 
	\begin{align*}
		0 \to \omega_{\overline X_i}^{m}(m \overline B_i + (m-1)\overline D_i) \to \omega_{\overline X_i}^{m}(m(\overline B_i + \overline D_i)) \to \omega_{\overline X_i}^{m} (m(\overline B_i + \overline D_i))|_{\overline D_i} \to 0.
	\end{align*}
	Now $\omega_{\overline X_i}^{m-1}((m-1)\overline B_i + (m-2)\overline D_i))$ is big and nef by assumption, so applying the generalized Kodaira vanishing theorem for klt pairs \cite[Thm.~10.37]{kollar_singularities_2013}, we get
	\begin{align*}
		H^1(\overline X_i, \omega_{\overline X_i}^{m}(m \overline B_i + (m-1)\overline D_i)) = 0.
	\end{align*}
	So taking global sections gives us that the Poincar\'e residue map
	\begin{align*}
		\mathcal{R}^m_{\overline X_i \to \overline D_i} : H^0(\overline X_i, \omega_{\overline X_i}^{m}(m(\overline B_i + \overline D_i))) \twoheadrightarrow  H^0(\overline D_i, \omega_{\overline X_i}^{m} (m(\overline B_i + \overline D_i))|_{\overline D_i})
	\end{align*}
	is surjective. Now since $|\omega_{\overline X_i}^{m} (m(\overline B_i + \overline D_i))|_{\overline D_i}|$ has no basepoints by assumption, applying Proposition \ref{pro:kollar5.8}, it must be that $|\omega_X^{[m]} (mB)|$ has no basepoints along $D$.
\end{proof}

We will use this lemma in conjunction with the following lemma:

\begin{lem}\label{lem:gluing} Suppose we are in the situation of  \Cref{pro:kollar5.8}, so that we have 
	\begin{align*}
		\mathrm{Proj} R(X, B) &= \mathrm{Proj} R(\overline{X}, \overline B + \overline D)^\tau
	\end{align*}
	where
	\begin{align*}
		R(\overline{X}, \overline B + \overline D)^\tau &= \bigoplus_{m \geq 0} H^0(\overline X, \omega_{\overline X}^{m}(m(\overline B + \overline D)))^\tau
	\end{align*}
	are the sections whose image under the Poincar\'e residue map is $\tau$-invariant (up to a sign). If for all $1 \leq i \leq n$ we have $\omega_{\overline X_i}^m (m(\overline B_i + \overline D_i))$ is semi-ample and $|\omega_X^{[m]}(mB)|$ has no basepoints along $D$ for $m \geq 0$ sufficiently divisible, then the log canonical model of $(X, B)$ is obtained by gluing the log canonical models of each $(\overline X_i, \overline B_i + \overline D_i)$ along the conductors $\overline D_i$ under $\tau$.
\end{lem}
\begin{proof}
	From  \Cref{pro:kollar5.8} we get that for any $m \geq 0$ sufficiently divisible,  
	$$H^0(\overline X, \omega_{\overline X}^m (m (\overline B + \overline D)))^\tau$$
	is given by
	\begin{align*}
		\left\{(\phi_i) \in \bigoplus_{1 \leq i \leq n} H^0(\overline X_i , \omega_{\overline X_i}^m (m(\overline B_i + \overline D_i))) :  \mathcal{R}_{\overline X_i \to \overline D_i}^m (\phi_i) = \pm \tau^* \mathcal{R}_{\overline X_j \to \overline D_j}^m (\phi_j)\right\}.
	\end{align*}
	Now if $|\omega_X^m(mB)|$ has no basepoints along $D$, each $|\omega_{\overline X_i}^m(m(\overline B_i + \overline D_i))|$ must have no basepoints along $\overline D_i$. Therefore, taking the $\mathrm{Proj}$ of $R(\overline X, \overline B + \overline D)^\tau$, will result in gluing the log canonical models of each $(\overline X_i , \overline B_i + \overline D_i)$ along the conductors $\overline D_i$ under $\tau$.
\end{proof}

We thus get our main result for this section:

\begin{cor}\label{cor:gluing}
	Let $X$ be a complex projective variety with irreducible components  $X_1, \dots, X_n$ with normalization $\overline{X}_1, \dots, \overline{X}_n$, let $\overline B_i = \overline B|_{\overline X_i}$ be the restriction of $\overline B$ to $\overline X_i$ and similarly, let $\overline D_i = \overline D|_{\overline X_i}$.  Suppose for all $1 \leq i \leq n$ the pair $(\overline X_i,\overline B_i+\overline D_i)$ is klt. Further suppose
	\begin{enumerate}
		\item \label{cor:gluing1} $\omega_{\overline X_i}^{m-1}((m-1)\overline B_i + (m-2)\overline D_i)$ is big and nef, 
		\item  \label{cor:gluing2} $\omega_{\overline X_i}^m (m(\overline B_i + \overline D_i))$ is semi-ample, and 
		\item \label{cor:gluing3}  $|\omega_{\overline X_i}^{m} (m(\overline B_i + \overline D_i))|_{\overline D_i}|$ is basepoint-free.
	\end{enumerate}
	Then the log canonical model of $(X, B)$ is obtained by gluing the log canonical models of each $(\overline X_i, \overline B_i + \overline D_i)$ along the conductors $\overline D_i$ under $\tau$.
\end{cor}

\begin{proof}
	This follows immediately from \Cref{lem:nobspts} and \Cref{lem:gluing}.
\end{proof}

	\section{The wall-and-chamber decomposition and the surfaces at the boundary}\label{S:wallandchamberbytype}
To obtain the full wall-and-chamber decomposition as described in \Cref{teo:main}, we will describe the wall-and-chamber decomposition organized by surface type (see \Cref{conv:types}):
\begin{align*}
	\widetilde{D}_{A_1}, \hspace{.1in} \widetilde{E}_{2A_1} , \hspace{.1in} \widetilde{E}_{3A_1}, \hspace{.1in} \widetilde{E}_{4A_1}, \hspace{.1in} \widetilde{D}_{3A_2}, \hspace{.1in} \widetilde{D}_{A_1} \cap \widetilde{D}_{3A_2}, \hspace{.1in} \widetilde{D}_{3A_2} \cap \widetilde{E}_{2A_1}, \hspace{.1in} \widetilde{D}_{3A_2} \cap \widetilde{E}_{3A_1},
\end{align*} 
and their degenerations. We further stratify these loci, and separate the wall-and-chamber decomposition further by these cases, depending on whether the marked line $\ell$ passes through zero, one, or two nodes on the surface (see \Cref{tab:narukifibers}). Following \cite{schock_moduli_2024}, we will extend the notation of \Cref{conv:types}, which  describes a stratification  of $\overline{Y}_{(1, \dots, 1)}$ as well as the fibers over those strata, to the space  $\overline{Y}_{(1, 1)}$. We will say a weighted $(b, c)$-weighted stable marked cubic surface for \textit{any} weight $(b, c) \in \mathrm{Amp}$ is of type $T$ if its stable replacement for weights $\frac{1}{2} < c \leq \frac{2}{3}$ is of type $T$ as described in \Cref{conv:types}. Any other type of $(b, c)$-weighted stable marked cubic surfaces is described as a degeneration of one of these types above. 

For brevity, we will only present the wall-and-chamber decomposition for surfaces of type $\widetilde{E}_{2A_1}$ where the marked line $\ell$ passes through exactly two nodes in the corresponding boundary divisor of $\overline{Y}_\ell$. Recall the surfaces of type $\widetilde{E}_{2A_1}$ correspond to cubic surfaces with exactly two $A_1$ singularities as shown in \Cref{tab:narukifibers}. The reason why we showcase this example is because this is the simplest case where the following occur:
\begin{enumerate}
	\item There are chambers and surfaces not seen before in \cite{schock_moduli_2024}. 
	\item Unlike in \cite{schock_moduli_2024}, some walls are obtained by failure of the constant volume condition, instead of failure of ampleness or the slc condition.
	\item This case detects the unique new change  in coarse moduli spaces that occurs  when crossing the wall $c = -\frac{b}{3} + \frac{1}{3}$.
\end{enumerate}
\ifAppendix
The remaining cases are similar and shown in Appendix \ref{app:remainingchambers}. 

\fi

\subsection{Type $\widetilde{E}_{2A_1}$: $\ell$ passes through exactly two nodes}
\begin{pro}\label{E2A1_mult4}
	For type $\widetilde{E}_{2A_1}$ $(b,c)$-weighted stable marked cubic surfaces $(S, (b,c)B)$ where $\ell$ passes through exactly two nodes in the corresponding $\overline{Y}_{\ell}$ boundary, there are 11 chambers. The line $c = \frac{b}{3}$ is its own chamber. Additionally, crossing the wall $c = -\frac{b}{3} + \frac{1}{3}$ introduces type $\widetilde{D}_{A_1} \cap \widetilde{E}_{2A_1}$ surfaces as degenerations of type $\widetilde{E}_{2A_1}$. 
	\ifAppendix
	The wall-and-chamber of this type is found in \Cref{app:E2A1_mult4degen}.
	\fi
	\begin{center}
		\begin{tikzpicture}[scale=14]
			\draw[very thick,-latex] (1/9, 1/9) -- (1.03, 1/9) 
			node[right]{$b$};
			
			\draw[very thick,-latex] (1/9, 1/9) -- (1/9, 1.03) 
			node[left]{$c$};
			
			\draw[solid] (1/9,1/9) -- (1,1);
			\node at (1/9, 1/9) [below] {\tiny $(\frac{1}{9}, \frac{1}{9})$};
			\node at (1, 1) [left] {$(1, 1)$};
			
			\draw[thick] (1, 1/9) -- (1,1);
			\node at (1, 1/9) [below] {$(1, \frac{1}{9})$};
			
			\draw[thick] (2/3, 2/3) -- (1, 2/3) node[right] {\tiny{$c=\frac{2}{3}$}};
			\node[circle, fill, inner sep=1pt] at (2/3, 2/3) {};
			\node at (2/3, 2/3) [left] {\tiny$(\frac{2}{3}, \frac{2}{3})$};
			\node at (8/9, 7/9) {Sch $\frac{2}{3} < c \leq 1$};
			\node at (0.9, 0.63) {$-\frac{b}{2} + 1 < c \leq \frac{2}{3}$};
			
			\draw[solid] (2/3, 2/3) -- (1, 1/2) node[right] {\color{black} \tiny{$c = \frac{1}{2}$}};
			\node[rotate=-26.6] at (5/6, 57/96) {\tiny wall due to $\ell$ containing Eckardt point(s)};
			
			\draw[thick] (1/2, 1/2) -- (1, 1/2);
			\node[circle, fill, inner sep=1pt] at (1/2, 1/2) {};
			\node at (1/2, 1/2) [left] {\tiny $(\frac{1}{2}, \frac{1}{2})$};
			\node at (0.7, 0.56) {Sch $\frac{1}{2} < c \leq \frac{2}{3}$};
			
			\draw[thick] (1/2, 1/2) -- (3/4, 1/4);
			\node[rotate=-45] at (7/12, 5/12) [above] {\tiny $c = -b+1$};
			\node[align=center] at (5/6, 0.42) {\(\frac{1}{4} < c \leq \frac{1}{2}\) \\ \(c > -b+1 \)};
			\node[align=center] at (1/2, 1/3) {Sch $\frac{1}{4} < c \leq \frac{1}{2}$};
			
			\draw[thick] (1/4, 1/4) -- (7/13, 2/13);
			\node[circle, fill, inner sep=1pt] at (1/4, 1/4) {};
			\node[rotate=-18.43] at (3/8, 24/120) [above] {\tiny$c=-\frac{b}{3} + \frac{1}{3}$};
			\node at (7/13, 2/13) [below] {\tiny{$(\frac{7}{13}, \frac{2}{13})$}};
			\node at (1/4, 1/4) [left] {\tiny$(\frac{1}{4}, \frac{1}{4})$};
			\node at (1.75/6, 4.55/24) {Sch $\frac{1}{6} < c \leq \frac{1}{4}$};
			
			\draw[thick] (3/4, 1/4) -- (1, 1/4) node[right] {\tiny $c = \frac{1}{4}$};
			\node[align=left] at (5.2/6, 10.5/48) {\tiny \(\frac{b}{10} + \frac{1}{10} < c \leq \frac{1}{4}\) \\ \tiny \(\frac{1}{6} < c < \frac{b}{3}\)};
			
			\draw[very thick] (1/2, 1/6) -- (3/4, 1/4);
			\node[rotate=18.43] at (2/3, 17/81) [above] {\tiny$c=\frac{b}{3}$};
			
			\draw[thick] (1/6, 1/6) -- (2/3, 1/6);
			\node[circle, fill, inner sep=1pt] at (1/6, 1/6) {};
			\node at (7/36, 7/36) [left] {\tiny$(\frac{1}{6}, \frac{1}{6}$)};
			\node at (2/3, 1/6) [below] {\tiny{$(\frac{2}{3}, \frac{1}{6})$}};
			\draw[->] (7/24, 1/11) -- (1/4, 10/72);
			\node[circle, fill, inner sep=1.5pt] at (1/4, 10/72) [above] {};
			\node[align=center] at (7/24, 1/20) {\(\frac{b}{10} + \frac{1}{10} < c \leq \frac{1}{6}\) \\ \(c \leq - \frac{b}{3} + \frac{1}{3}, b \neq c\)};
			\draw[->] (7/13, 1/11) -- (7.3/13, 9.4/60);
			\node[circle, fill, inner sep=1.5pt] at (7.3/13, 9.4/60) [above] {};
			\node[align=center] at (7/13, 1/20) {\(\frac{b}{10} + \frac{1}{10} < c \leq \frac{1}{6}\) \\ \(c > -\frac{b}{3} + \frac{1}{3}\)};
			
			\draw[very thick] (1/9, 1/9) -- (1/6, 1/6);
			\draw[->] (1/11, 1/6) -- (5/36, 5/36);
			\node at (1/10, 1/6) [left] {Sch $\frac{1}{9} < c \leq \frac{1}{6}$};
			
			\draw[dashed] (1/9, 1/9) -- (1, 1/5);
			\node[rotate=5.71] at (5/6, 11/60) [below] {\tiny$c=b/10 + 1/10$};
			\node at (1, 1/5) [right] {\tiny{$c=\frac{1}{5}$}};
			
		\end{tikzpicture}
	\end{center}
\end{pro}

\begin{table}[H]
	\centering
	\label{tab:E2A1eck}
	\begin{tabular}{| c | c | c | c |}
		\hline
		Label & Surface & \# & Eckardt points \\
		\hline
		\hline
		1 & $\widetilde S_{2A_1}$ & 1 & 0 or 1 \\
		2 & $\mathrm{Bl}_5\mathbb{F}_0$ & 2 & 0 \\
		3 & $\mathrm{Bl}_5\mathbb{F}^2$ & 1 & 0 or 1 \\
		4 & $\mathbb{F}_0$ & 4 & 0 \\
		\hline
	\end{tabular}
	\caption{This table (taken from \cite[Table~3]{schock_moduli_2024}) gives the possible numbers of Eckardt points on each component of the surface in \Cref{E2A1_mult4_12c23}. The first two columns are the numbered types of irreducible components of the surfaces in the chamber Sch $1/2 < c \leq 2/3$ of type $\widetilde{E}_{2A_1}$ pictured in \Cref{E2A1_mult4_12c23}, or \cite[Fig.~9]{schock_moduli_2024}. The third column tells us the number of each corresponding component found in the surface. Finally, the last row gives the possible numbers of Eckardt points on each component.  For the component of type 1, $\widetilde{S}_{2A_1}$ refers to the minimal resolution of a cubic surface with two $A_1$ singularities. For the components of type 2, $\mathrm{Bl}_5\mathbb{F}_0$ refers to the blowup of $\mathbb{F}_0 \cong \mathbb{P}^1 \times \mathbb{P}^1$ at 5 points on the diagonal. For the component of type 3, $\mathrm{Bl}_5\mathbb{P}^2$ refers to the special blowup of $\mathbb{P}^2$ at 5 points as in \Cref{pro:bl5p2eck}.}
\end{table}

\subsection{Explicit description of the $(b, c)$-weighted stable marked cubic surfaces}\label{S:descriptions}
In this subsection, we give an explicit description of the $(b, c)$-weighted stable marked cubic surfaces of type $\widetilde{E}_{2A_1}$ of \Cref{E2A1_mult4}. Again, as this article is already lengthy and since the computations are similar, we leave out the descriptions of other types. For the explicit descriptions of other types, please see the supplementary file \href{https://drive.google.com/file/d/1nQBGKhehcd8xiDVru9otsMT-grKk8n28/view?usp=sharing}{``Fiber descriptions for surface types"} on \href{https://www.jon-kim.net/research}{https://www.jon-kim.net/research}. Many of the descriptions are inspired by \cite{gallardo_geometric_2021} and \cite{schock_moduli_2024}. Furthermore, by \Cref{teo:eckblowup}, it suffices to provide an explicit description up to $(b, c) = (\frac{2}{3}, \frac{2}{3})$ (exclusive).

\begin{conv}[Conventions of the descriptions]\label{conv:descpresent}
	For the explicit descriptions of the $(b, c)$- weighted stable marked cubic surfaces, we follow the conventions in \cite{schock_moduli_2024} and \cite{gallardo_geometric_2021}: We denote the lines of multiplicities 1, 2, and 4, by solid, dashed, and dotted black lines, respectively. If a line has higher multiplicity, we will denote it by a solid line with a label that explicitly states the multiplicity. For example, the $(-2)$-exceptional curve $e_1$ of \Cref{E2A1_mult4_19c16bneqc} is labeled by $6c \cdot e_1$ to indicate that  it has a multiplicity of 6 with weight $c$. We denote the gluing curves of irreducible components by red lines, usually solid, but there will be some cases where we draw them as dashed red lines in order to avoid confusion about the number of irreducible components. When there is more than one isomorphic copy of a component with different boundary, we will distinguish them by placing a number label, e.g. $1 : \mathbb{P}^2$ and $2 : \mathbb{P}^2$. The $b$-weighted (marked) line $\ell$ will be drawn in green. If a dashed line is in green, then the line has weight $b + c$. Similarly, if a dotted line is green, then the line has weight $b + 3c$. If $\ell$ is present in a component, we will decorate the number label with a prime, e.g. $1' : \mathbb{P}^2$. When many components are glued together where confusion might arise, we indicate this in the figures with subfigures and matching highlighted colors of the gluing lines.
\end{conv}
\newpage
\subsection{Type $\widetilde{E}_{2A_1}$ where $\ell$ passes through exactly two nodes}
\noindent\null\par
\begin{figure}[H]
		\includegraphics[scale=0.14]{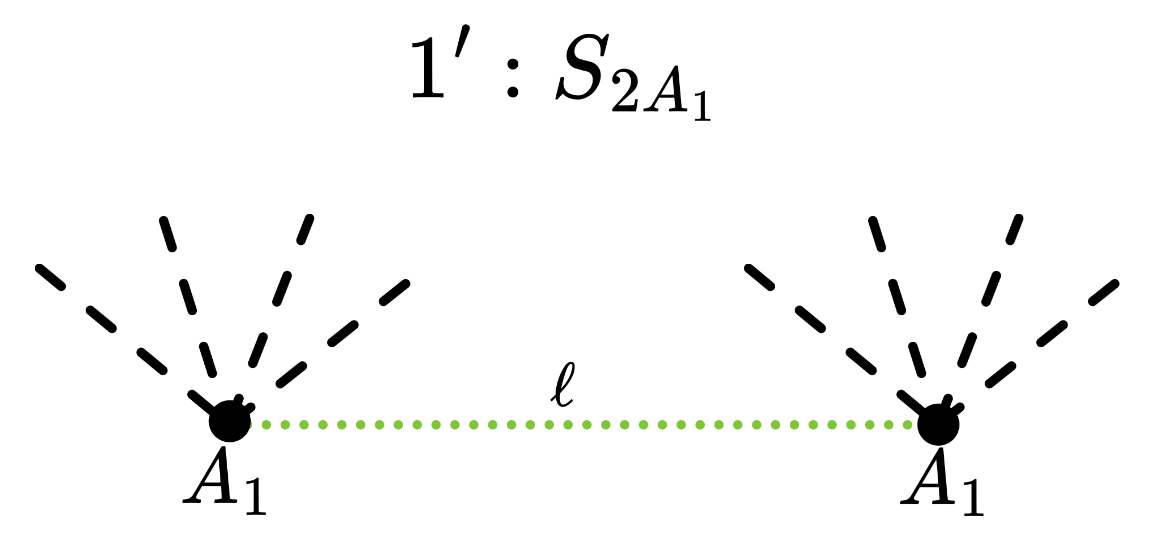}
		\caption{A type $\widetilde{E}_{2A_1}$ surface (see \Cref{conv:types}) $(S_0, (b, c)B_0)$ with $\ell$ passing through exactly two nodes in the chamber Sch $\frac{1}{9} < c \leq \frac{1}{6}$ in \Cref{E2A1_mult4}. There are also seven multiplicity 1 lines (not shown).}
		\label{E2A1_mult4_naruki}
\end{figure}
	
\begin{table}[H]
		\centering
		\label{tab:E2A1_mult4_naruki}
		\begin{tabular}{| c | c | c | c| c|}
			\hline
			Label & Surface & \# & Boundary lines $B$ & Gluing lines $\Delta$\\
			\hline
			\hline
			$1'$ & $\widetilde{S}_{2A_1}$ & 1 & mult 1: $M_1, \dots, M_7$ & None\\
			& & & mult 2: $L_1, \dots, L_8$ & \\
			& & & mult 4: $\ell$ & \\
			\hline
		\end{tabular}
		\caption{Explicit description of the boundary lines and gluing lines of the irreducible surface $1': S_{2A_1}$ of \Cref{E2A1_mult4_naruki}. The following is the intersection theory of the lines: $\ell$ has self intersection $\ell^2 = 0$, has intersection $(\ell \cdot \sum_1^8 L_i) = 4$ with the multiplicity 2 lines, and has intersection $(\ell \cdot \sum_1^7 M_j) = 1$ with the multiplicity 1 lines. On the other hand, the multiplicity 2 lines has self intersection $(\sum_1^8 L_i)^2 = 16$ with itself and has intersection $(\sum_1^8 L_i \cdot \sum_1^7 M_j) = 24$ with the multiplicity 1 lines. Finally, each multiplicity 1 line intersects three other multiplicity 1 line transversely.}
\end{table}

\begin{figure}[H]
		\includegraphics[scale=0.2]{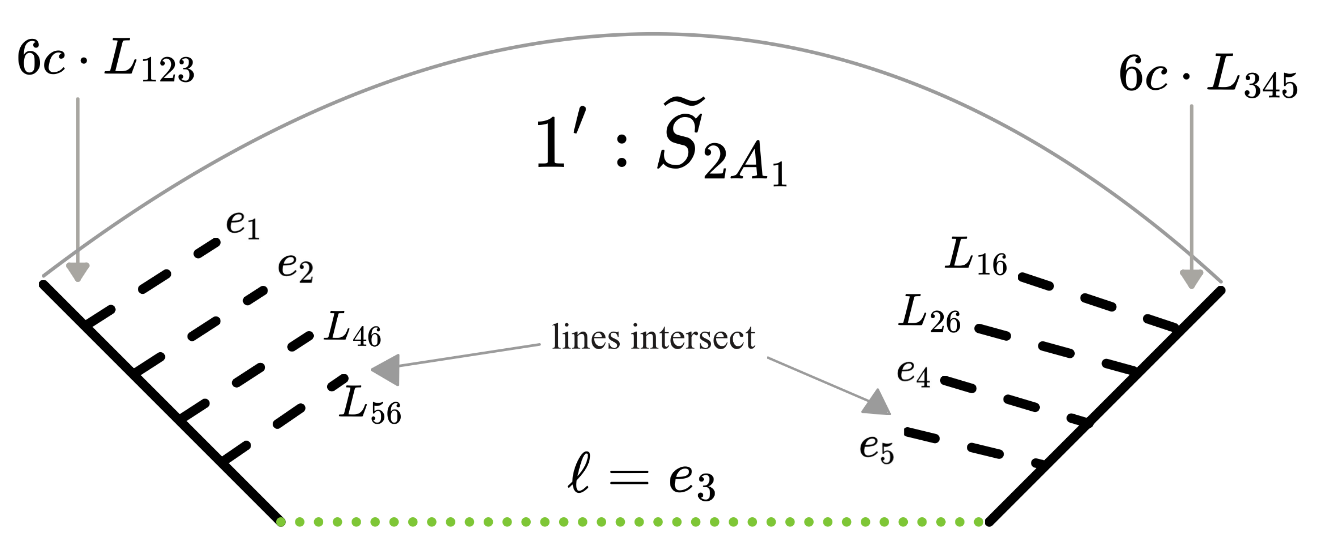}
		\caption{A type $\widetilde{E}_{2A_1}$ surface $(S_1, (b, c) B_1)$ with $\ell$ passing through exactly two nodes in the chamber $\frac{b}{10} + \frac{1}{10} < c \leq \frac{1}{6}, b \leq -\frac{b}{3} + \frac{1}{3}, b \neq c$ in \Cref{E2A1_mult4}, obtained as the stable replacement for these weights of the surface $(S_0, (b, c) B_0)$ of \Cref{E2A1_mult4_naruki}. This is obtained by taking the simultaneous resolution of the two $A_1$ singularities where each $(-2)$-exceptional curve has multiplicity 6. There are also seven multiplicity 1 lines (not shown)}
		\label{E2A1_mult4_19c16bneqc}
\end{figure}
	
\begin{table}[H]
		\centering
		\label{tab:E2A1_mult4_19c16bneqc}
		\begin{tabular}{| c | c | c | c| c|}
			\hline
			Label & Surface & \# & Boundary lines $B$ & Gluing lines $\Delta$\\
			\hline
			\hline
			$1'$ & $\widetilde{S}_{2A_1}$ & 1 & mult 1: $L_{14}, L_{15}, L_{24}, L_{25}, L_{36}, C, e_6$ & None\\
			 & & & mult 2: $e_1, e_2, L_{46}, L_{56}, e_4, e_5, L_{16}, L_{26}$ & \\
			 & & & mult 4: $\ell = e_3$ & \\
			 & & & mult 6: $L_{123}, L_{345}$ & \\
			\hline
		\end{tabular}
		\caption{Explicit description of the boundary lines and gluing lines of the irreducible surface $1': \widetilde{S}_{2A_1}$ of \Cref{E2A1_mult4_19c16bneqc}. We identify $\widetilde{S}_{2A_1}$ as $\mathrm{Bl}_6 \mathbb{P}^2$ where the six points $p_1, \dots, p_6$ are in special position where the sets of points $\{p_1, p_2, p_3\}$ and $\{p_3, p_4, p_5\}$ are colinear (see \Cref{tab:specialposcubics}). Let $L_{ij} - h - e_i - e_j$ denote the strict transform of the line passing through points $p_i$ and $p_j$, where $h$ denotes the class of the pullback of the hyperplane divisor and $e_i$ denotes the class of the exceptional divisor of the blow-up of $p_i$. Similarly, let $L_{ijk} = h - e_i - e_j - e_k$ denote the strict transform of the line passing through points $p_i, p_j, p_k$. Finally, let $C = 2h - e_1 - e_2 - e_4 - e_5 - e_6$ denote the strict transform of the unique conic passing through the five points $p_1, p_2, p_4, p_5, p_6$ in general position.}
\end{table}
	
\begin{figure}[H]
		\includegraphics[scale=0.19]{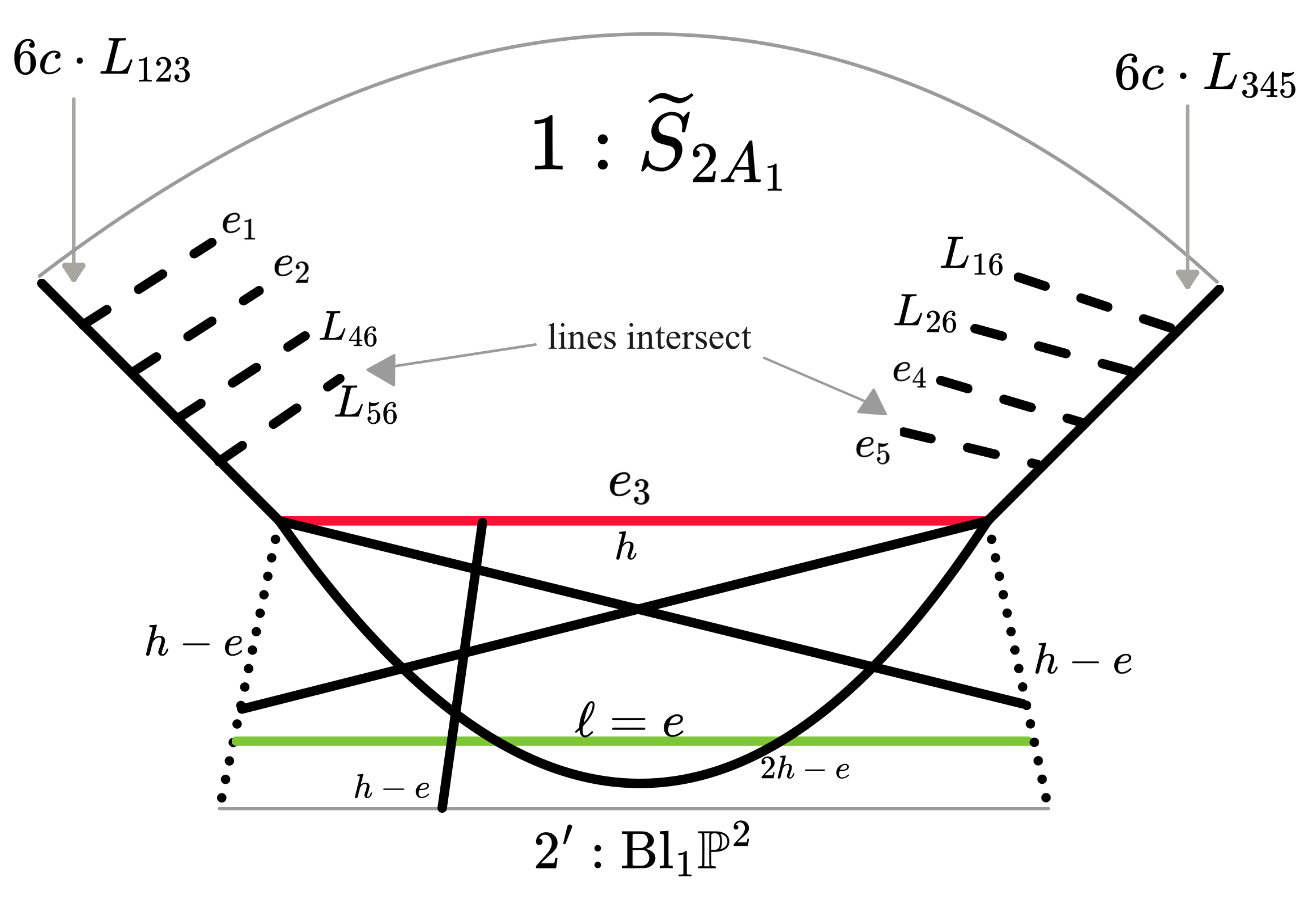}
		\caption{A type $\widetilde{E}_{2A_1}$ surface $(S_2, (b, c)B_2)$ with $\ell$ passing through exactly two nodes in the chamber $\frac{b}{10} + \frac{1}{10} < c \leq \frac{1}{6}, c > -\frac{b}{3} + \frac{1}{3}$ in \Cref{E2A1_mult4}, obtained as the stable replacement for these weights of the surface $(S_1, (b, c)B_1)$ of \Cref{E2A1_mult4_19c16bneqc}. This is obtained by blowing up $\ell$ and attaching to the exceptional divisor a $2': \operatorname{Bl}_1 \mathbb{P}^2$ component, as shown. The four multiplicity 1 lines consisting of $e$ (shown as $\ell$), two copies of $h$, and $2h - e$ on $2' : \mathrm{Bl}_1 \mathbb{P}^2$  come from $\ell$. The remaining multiplicity 1 line comes from a multiplicity 1 line in $1 : \widetilde{S}_{2A_1}$. Each multiplicity 6 $(-2)$-exceptional curve in $1 : \widetilde{S}_{2A_1}$ splits into a passing through exactly two nodes line $h - e$ and the multiplicity 1 lines $h$ and $2h-e$.}
		\label{E2A1_mult4_b313c16}
\end{figure}
	
\begin{table}[H]
		\centering
		\label{tab:E2A1_mult4_b313c16}
		\begin{tabular}{| c | c | c | c| c|}
			\hline
			Label & Surface & \# & Boundary lines $B$ & Gluing lines $\Delta$\\
			\hline
			\hline
			$1$ & $\widetilde{S}_{2A_1}$ & 1 & mult 1: $L_{14}, L_{15}, L_{24}, L_{25}, L_{36}, C, e_6$ & $e_3$\\
			& & & mult 2: $e_1, e_2, L_{46}, L_{56}, e_4, e_5, L_{16}, L_{26}$ & \\
			& & & mult 6: $L_{123}, L_{345}$ & \\
			\hline
			$2'$ & $\mathrm{Bl}_1 \mathbb{P}^2$ & 1 & mult 1: $\ell = e, 2h, h -e, 2h-e$ & $h$\\
			& & & mult 4: $2 (h - e)$ & \\
			\hline
		\end{tabular}
		\caption{Explicit description of the boundary lines and gluing lines of the irreducible components of the surface of \Cref{E2A1_mult4_b313c16}. We identify $1: \widetilde{S}_{2A_1}$ as $\mathrm{Bl}_6 \mathbb{P}^2$ (see \Cref{tab:E2A1_mult4_19c16bneqc} above). For $2: \mathrm{Bl}_1 \mathbb{P}^2$, let $h$ denote the class of the pullback of the hyperplane divisor and $e$ denote the class of the exceptional divisor of the blow-up.}
\end{table}
	
\begin{figure}[H]
		\includegraphics[scale=0.2]{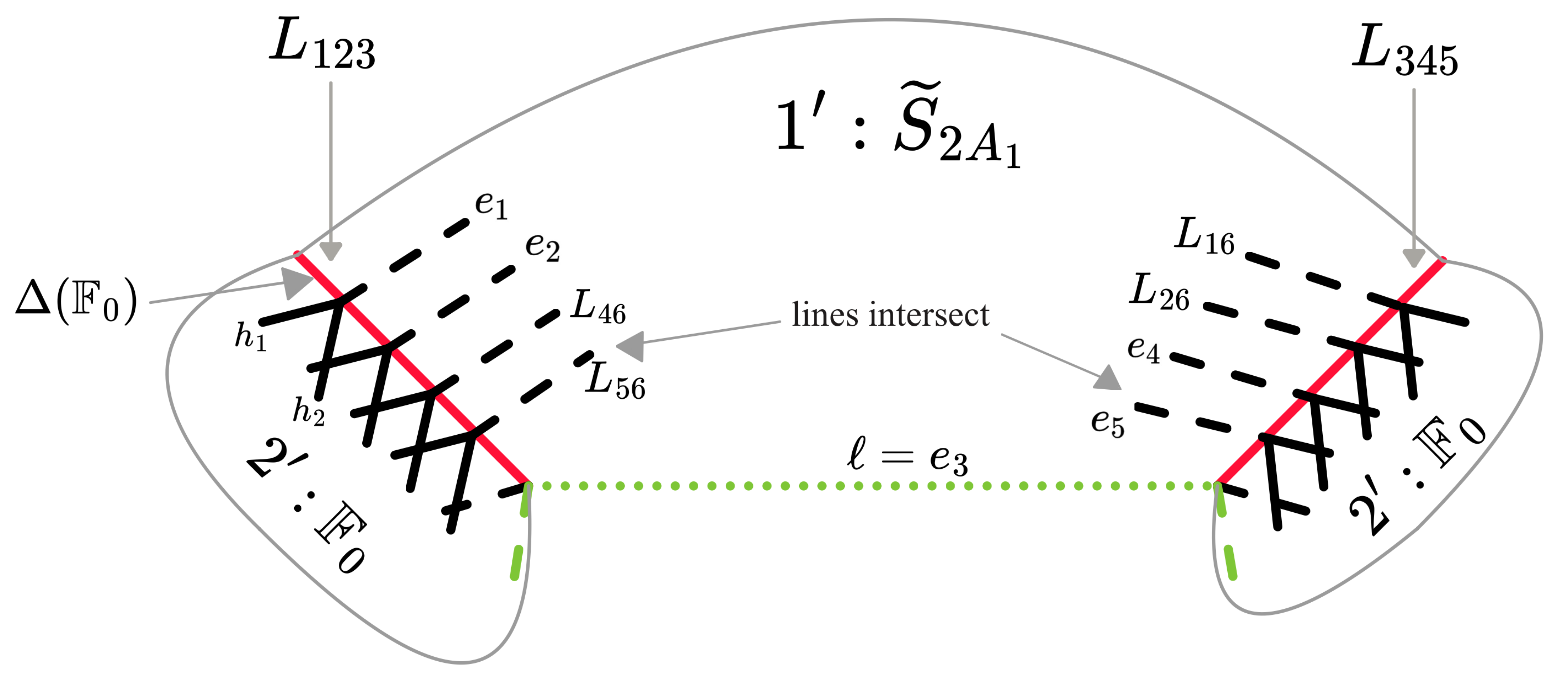}
		\caption{A type $\widetilde{E}_{2A_1}$ surface $(S_3, (b, c)B_3)$ with $\ell$ passing through exactly two nodes in the chamber Sch $\frac{1}{6} < c \leq \frac{1}{4}$ in \Cref{E2A1_mult4} obtained as the stable replacement for these weights of the surface $(S_1, (b, c)B_1)$ of \Cref{E2A1_mult4_19c16bneqc}.}
		\label{E2A1_mult4_16c14}
\end{figure}
	
\begin{table}[H]
		\centering
		\label{tab:E2A1_mult4_16c14}
		\begin{tabular}{| c | c | c | c| c|}
			\hline
			Label & Surface & \# & Boundary lines $B$ & Gluing lines $\Delta$\\
			\hline
			\hline
			$1'$ & $\widetilde{S}_{2A_1}$ & 1 & mult 1: $L_{14}, L_{15}, L_{24}, L_{25}, L_{36}, C, e_6$ & $L_{123}, L_{345}$\\
			& & & mult 2: $e_1, e_2, L_{46}, L_{56}, e_4, e_5, L_{16}, L_{26}$ & \\
			& & & mult 4: $\ell = e_3$ & \\
			\hline
			$2'$ & $\mathbb{F}_0$ & 2 & mult 1: $4h_1, 4h_2$ &$\Delta(\mathbb{F}_0) = h_1 + h_2$\\
			& & & mult 2: $\ell = h_1, h_2$ &\\
			\hline
		\end{tabular}
		\caption{Explicit description of the boundary lines and gluing lines of the irreducible components of the surface of \Cref{E2A1_mult4_16c14}. We identify $\widetilde{S}_{2A_1}$ as $\mathrm{Bl}_6 \mathbb{P}^2$ (see \Cref{tab:E2A1_mult4_19c16bneqc} above). For $2' : \mathbb{F}_0$, let $h_1, h_2$ denote the classes of the two rulings.}
\end{table}
	
\begin{figure}[H]
		\includegraphics[scale=0.17]{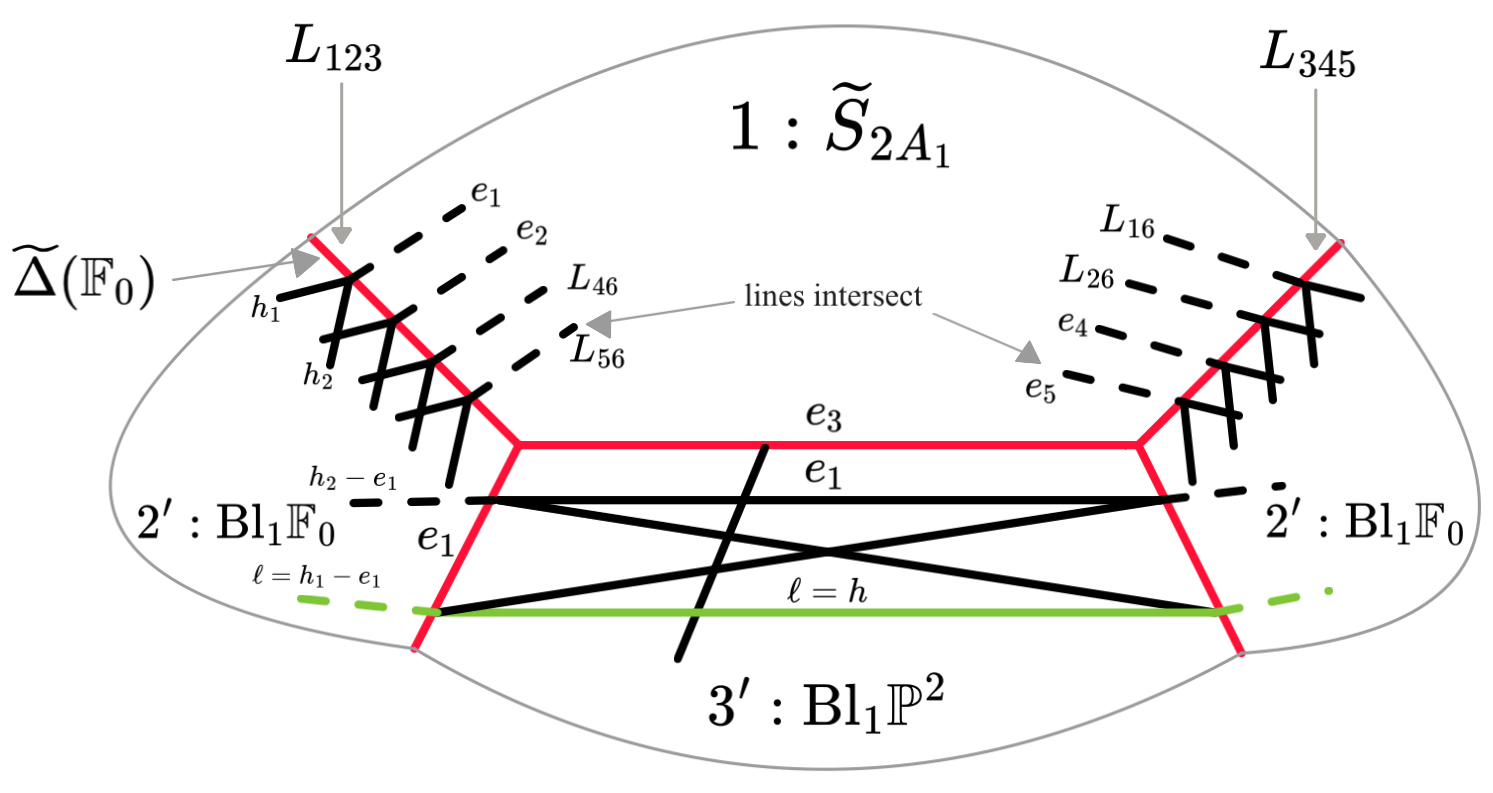}
		\caption{A type $\widetilde{E}_{2A_1}$ surface $(S_4, (b,c)B_4)$ with $\ell$ passing through exactly two nodes in the chamber Sch $\frac{1}{4} < c \leq \frac{1}{2}$ in \Cref{E2A1_mult4} obtained as the stable replacement for these weights of the surface ($S_3, (b, c)B_3)$ of \Cref{E2A1_mult4_16c14}.}
		\label{E2A1_mult4_14c12}
\end{figure}
	
\begin{table}[H]
		\centering
		\label{tab:E2A1_mult4_14c12}
		\begin{tabular}{| c | c | c | c| c|}
			\hline
			Label & Surface & \# & Boundary lines $B$ & Gluing lines $\Delta$\\
			\hline
			\hline
			$1$ & $\widetilde{S}_{2A_1}$ & 1 & mult 1: $L_{14}, L_{15}, L_{24}, L_{25}, L_{36}, C, e_6$ & $L_{123}, L_{345}, e_3$\\
			& & & mult 2: $e_1, e_2, L_{46}, L_{56}, e_4, e_5, L_{16}, L_{26}$ & \\
			\hline
			$2'$ & $\mathrm{Bl}_1 \mathbb{F}_0$ & 2 & mult 1: $4h_1, 4h_2$ &$\widetilde{\Delta}(\mathbb{F}_0) = h_1 + h_2 - e_1$\\
			& & & mult 2: $\ell = h_1 - e_1, h_2 - e_1$ & $e_1$\\
			\hline
			$3'$ & $\mathrm{Bl}_1 \mathbb{P}^2$ & 1 & mult 1: $\ell = h$, $3h, h-e_1$& $2(h-e_1), e_1$\\
			\hline
		\end{tabular}
		\caption{Explicit description of the boundary lines and gluing lines of the irreducible components of the surface of \Cref{E2A1_mult4_14c12}. We identify $\widetilde{S}_{2A_1}$ as $\mathrm{Bl}_6 \mathbb{P}^2$ (see \Cref{tab:E2A1_mult4_19c16bneqc} above). For $2' :\mathrm{Bl}_1 \mathbb{F}_0$, let $h_1, h_2$ denote the pullback of the classes of the two rulings and $e_1$ denote the exceptional divisor of the blow-up. Similarly for $3' : \mathrm{Bl}_1 \mathbb{P}^2$, let $h$ denote the pullback of the classes of the two rulings and $e_1$ denote the exceptional divisor of the blow-up.}
\end{table}
	
\begin{figure}[H]
		\includegraphics[scale=0.17]{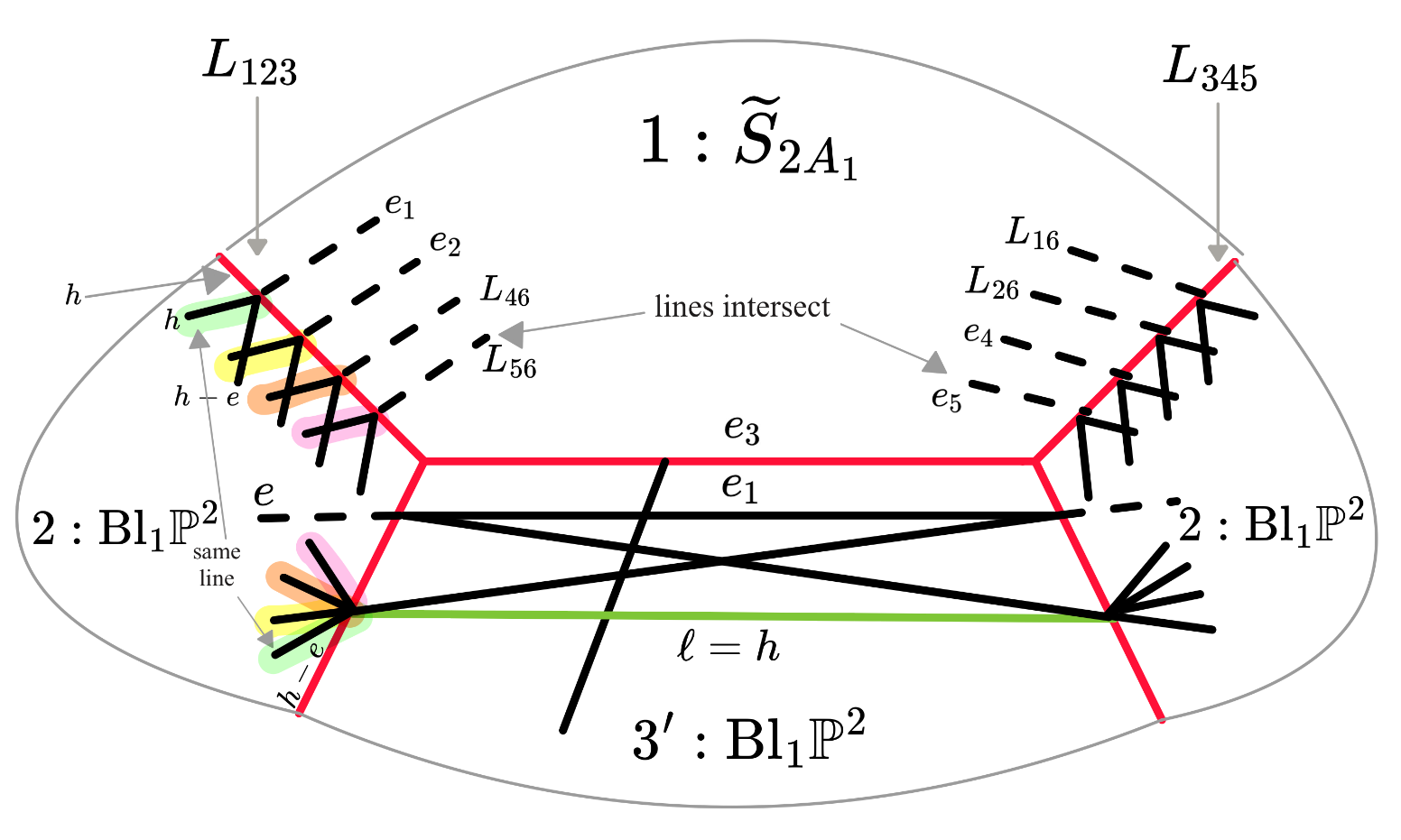}
		\caption{A type $\widetilde{E}_{2A_1}$ surface $(S_5, (b, c)B_5)$ with $\ell$ passing through exactly two nodes in the chamber $c = \frac{b}{3}$ in \Cref{E2A1_mult4} obtained as the stable replacement for these weights of the surface $(S_4, (b, c)B_4)$ of \Cref{E2A1_mult4_14c12}. This is obtained by contracting the multiplicity 2 line corresponding to $\ell$ in each $2' : \mathrm{Bl}_1 \mathbb{F}_0$ bringing along the four multiplicity 1 lines, highlighted above. }
		\label{E2A1_mult4_c=b3}
	\end{figure}
	\begin{table}[H]
	\centering
	\label{tab:E2A1_mult4_c=b3}
	\begin{tabular}{| c | c | c | c| c|}
		\hline
		Label & Surface & \# & Boundary lines $B$ & Gluing lines $\Delta$\\
		\hline
		\hline
		$1$ & $\widetilde{S}_{2A_1}$ & 1 & mult 1: $L_{14}, L_{15}, L_{24}, L_{25}, L_{36}, C, e_6$ & $L_{123}, L_{345}, e_3$\\
		& & & mult 2: $e_1, e_2, L_{46}, L_{56}, e_4, e_5, L_{16}, L_{26}$ & \\
		\hline
		$2'$ & $\mathrm{Bl}_1 \mathbb{P}^2$ & 2 & mult 1: $4h, 4(h-e)$ &$h, h-e$\\
		& & & mult 2: $e$ & \\
		\hline
		$3'$ & $\mathbb{P}^2$ & 1 & mult 1: $\ell = h$, $3h$, $h-e_1$ & $2(h-e_1), e_1$\\
		\hline 
	\end{tabular}
	\caption{Explicit description of the boundary lines and gluing lines of the irreducible components of the surface of \Cref{E2A1_mult4_c=b3}.}
\end{table}

\begin{figure}[H]
		\includegraphics[scale=0.18]{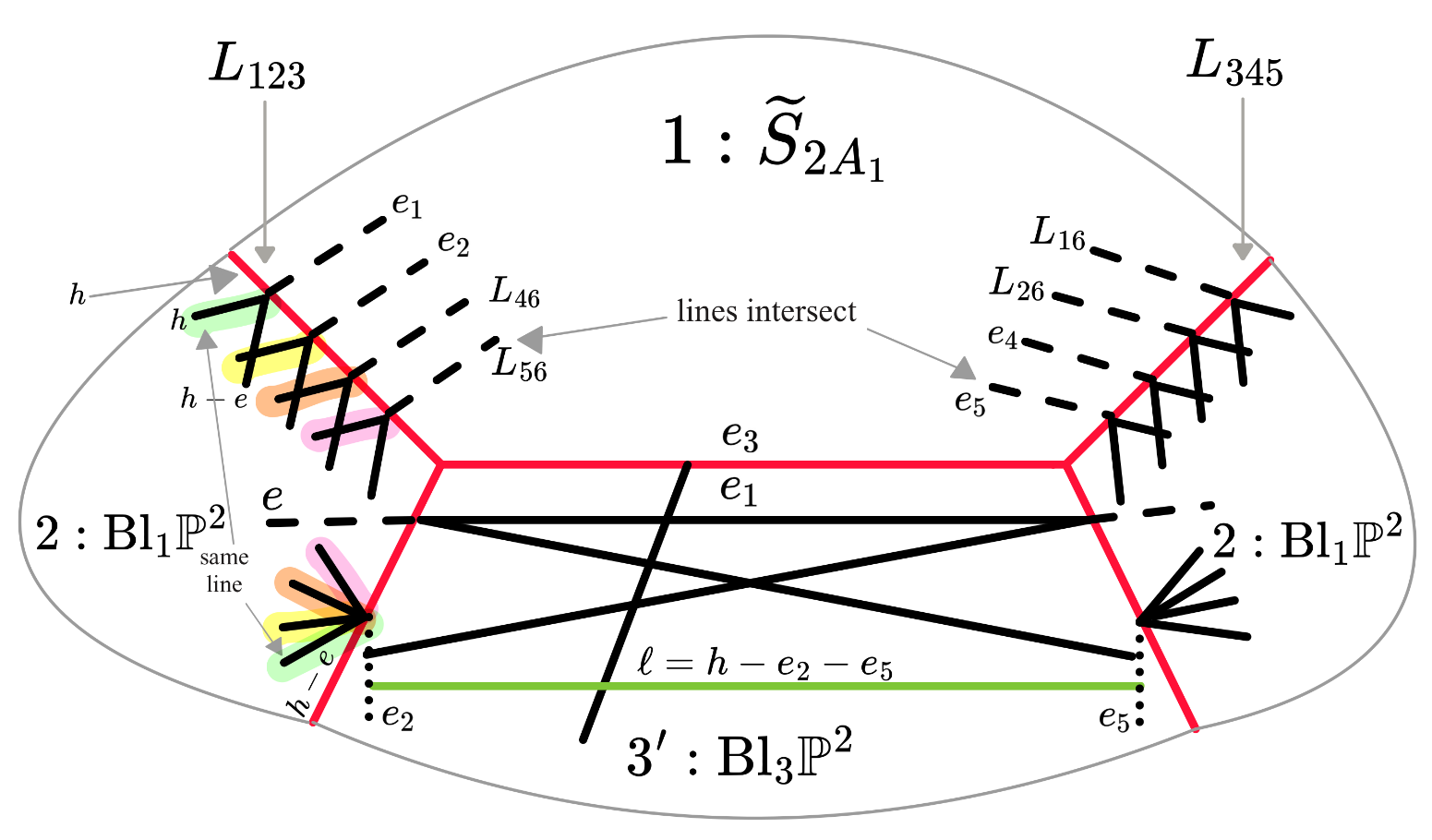}
		\caption{A type $\widetilde{E}_{2A_1}$ surface $(S_6, (b, c)B_6)$ with $\ell$ passing through exactly two nodes in the chamber $\frac{b}{10} + \frac{1}{10} < c \leq \frac{1}{4}, \frac{1}{6} < c < \frac{b}{3}$ in \Cref{E2A1_mult4} obtained as the stable replacement for these weights of the surfaces $(S_2, (b, c) B_2)$ and  $(S_5, (b, c)B_5)$ of \Cref{E2A1_mult4_b313c16} and \Cref{E2A1_mult4_c=b3}, respectively. This is obtained from $(S_5, (b, c)B_5)$ by blowing up the points of intersection of $\ell$ with the gluing line $h - e_1$ in $3': \mathrm{Bl}_1 \mathbb{P}^2$ with the exceptional curve having multiplicity 4.}
		\label{E2A1_mult4_b10110c14}
	\end{figure}
	\begin{table}[H]
	\centering
	\label{tab:E2A1_mult4_b10110c14}
	\begin{tabular}{| c | c | c | c| c|}
		\hline
		Label & Surface & \# & Boundary lines $B$ & Gluing lines $\Delta$\\
		\hline
		\hline
		$1$ & $\widetilde{S}_{2A_1}$ & 1 & mult 1: $L_{14}, L_{15}, L_{24}, L_{25}, L_{36}, C, e_6$ & $L_{123}, L_{345}, e_3$\\
		& & & mult 2: $e_1, e_2, L_{46}, L_{56}, e_4, e_5, L_{16}, L_{26}$ & \\
		\hline
		$2'$ & $\mathrm{Bl}_1 \mathbb{P}^2$ & 2 & mult 1: $4h, 4(h-e)$ &$h, h-e$\\
		& & & mult 2: $e$ & \\
		\hline
		$3'$ & $\mathrm{Bl}_3 \mathbb{P}^2$ & 1 & mult 1: $\ell = h-e_2-e_5$, $h$, $h- e_1$ & $h-e_1 - e_2, h-e_1 - e_5$\\
		& & & $h- e_2$, $h-e_5$ & \\
		& & & mult 4: $e_2, e_5$ & $e_1$\\
		\hline 
	\end{tabular}
	\caption{Explicit description of the boundary lines and gluing lines of the irreducible components of the surface of \Cref{E2A1_mult4_b10110c14}. The component $3' \mathrm{Bl}-3 \mathbb{P}^2$ is obtained by blowing up three points where the classes of the exceptional divisors are denoted by $e_1, e_2, e_5$.}
\end{table}	

\begin{figure}[H]
		\begin{subfigure}{0.8\textwidth}
			\centering
			\includegraphics[scale=0.18]{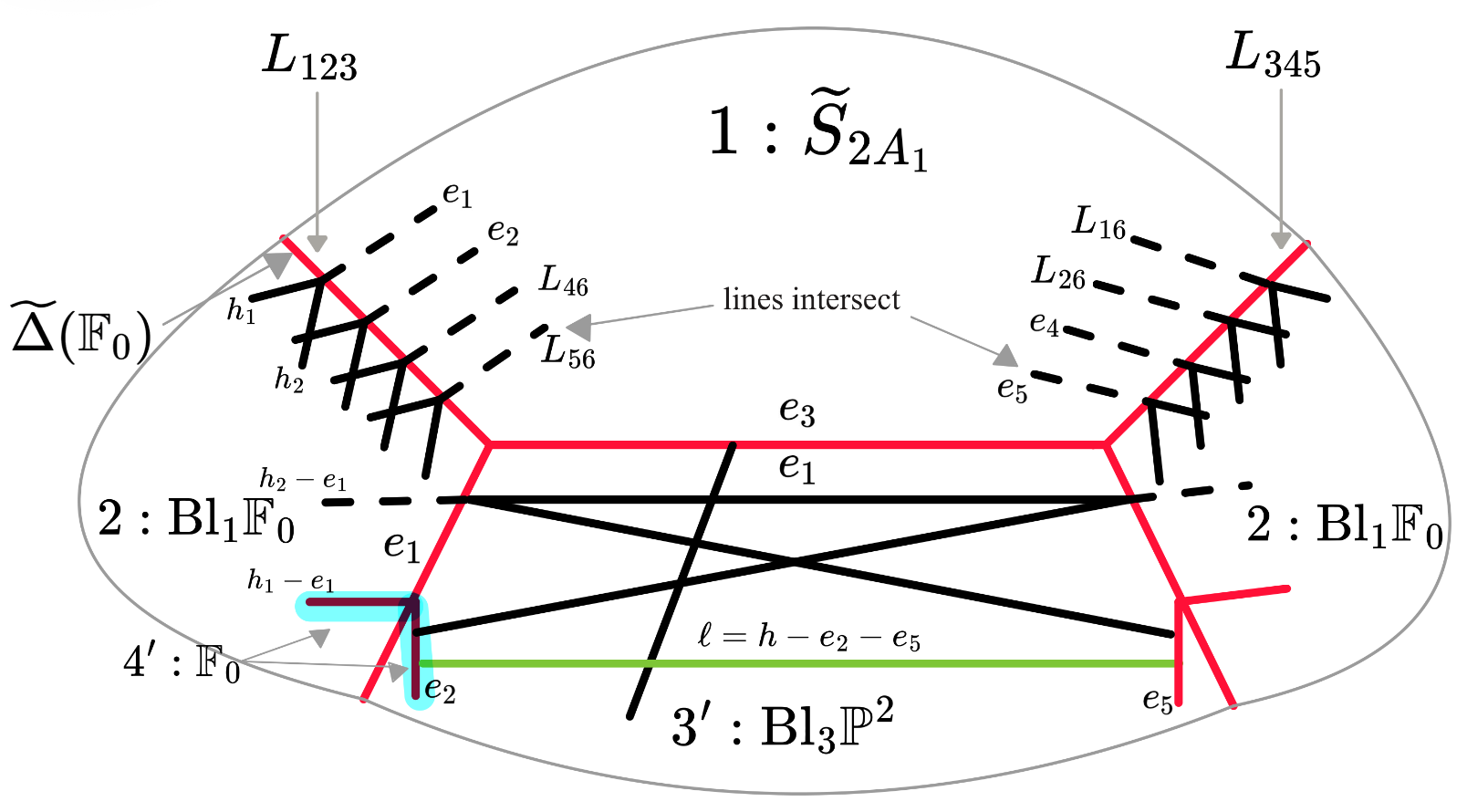}
			\caption{A type $\widetilde{E}_{2A_1}$ surface $(S_7, (b, c)B_7)$ with $\ell$ passing through exactly two nodes in the chamber $\frac{1}{4} < c \leq \frac{1}{2}$, $c > -b +1$.}
		\end{subfigure}
		\begin{subfigure}{0.8\textwidth}
			\centering
			\includegraphics[scale=0.22]{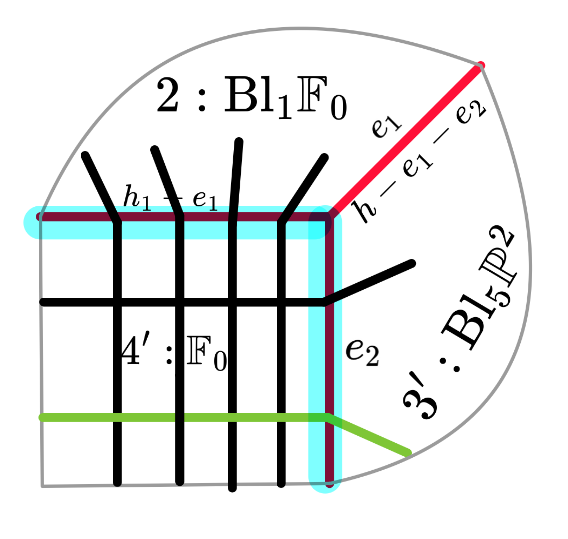}
			\caption{Another view of the gluing of $4': \mathbb{F}_0$ to $2 : \mathrm{Bl}_1 \mathbb{F}_0$ and $3' : \mathrm{Bl}_3 \mathbb{P}^2$, highlighted in blue.}
		\end{subfigure}
		\caption{A type $\widetilde{E}_{2A_1}$ surface $(S_7, (b, c)B_7)$ with $\ell$ passing through exactly two nodes in the chamber $\frac{1}{4} < c \leq \frac{1}{2}$, $c > -b +1$ of \Cref{E2A1_mult4}, obtained as the stable replacement for these weights of the surfaces $(S_4, (b, c)B_4)$,  $(S_5, (b, c)B_5)$, and $(S_6, (b, c)B_6)$ of \Cref{E2A1_mult4_14c12}, \Cref{E2A1_mult4_c=b3}, and \Cref{E2A1_mult4_b10110c14}, respectively. This is  obtained roughly from $(S_4, (b, c)B_4)$ by blowing up $\ell$ on each $2' : \mathrm{Bl}_1 \mathbb{F}_0$ and attaching to the exceptional fiber a copy of $4' : \mathbb{F}_0$.}
		\label{E2A1_mult4_-b1c12}
\end{figure}

\begin{table}[H]
	\centering
	\label{tab:E2A1_mult4_-b1c12}
	\begin{tabular}{| c | c | c | c| c|}
		\hline
		Label & Surface & \# & Boundary lines $B$ & Gluing lines $\Delta$\\
		\hline
		\hline
		$1$ & $\widetilde{S}_{2A_1}$ & 1 & mult 1: $L_{14}, L_{15}, L_{24}, L_{25}, L_{36}, C, e_6$ & $L_{123}, L_{345}, e_3$\\
		& & & mult 2: $e_1, e_2, L_{46}, L_{56}, e_4, e_5, L_{16}, L_{26}$ & \\
		\hline
		$2$ & $\mathrm{Bl}_1 \mathbb{F}_0$ & 2 & mult 1: $4h_1, 4h_2$ &$\widetilde{\Delta}(\mathbb{F}_0) = h_1 + h_2 - e_1$\\
		& & & mult 2: $h_2 - e_1$ & $h_1 - e_1, e_1$\\
		\hline
		$3'$ & $\mathrm{Bl}_3 \mathbb{P}^2$ & 1 & mult 1: $\ell = h - e_2 - e_5$, $h$, $h-e_1$ & $h-e_1 - e_2$, $h - e_1 - e_5$\\
		& & & $h- e_2$, $h-e_5$ & $e_1, e_2, e_5$\\
		\hline
		$4'$ & $\mathbb{F}_0$ & 2 & mult 1: $\ell = h_1, h_1, 4h_2$ & $h_1, h_2$\\
		\hline
	\end{tabular}
	\caption{Explicit description of the boundary lines and gluing lines of the irreducible components of the surface of \Cref{E2A1_mult4_-b1c12}.}
\end{table}
	
\begin{figure}[H]
		\begin{subfigure}{0.8\textwidth}
			\centering
			\includegraphics[scale=0.16]{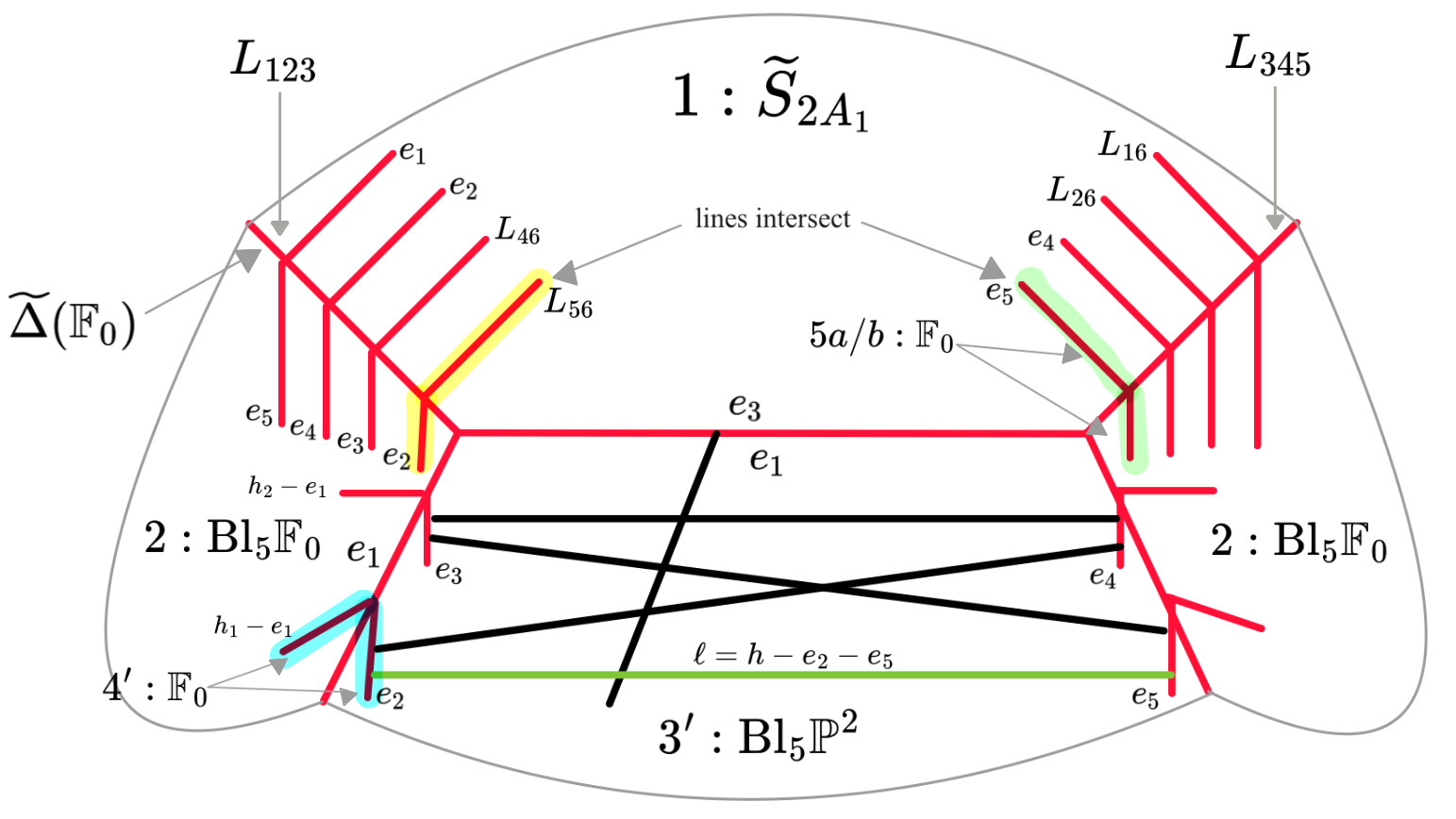}
			\caption{A type $\widetilde{E}_{2A_1}$ surface $(S_8, (b, c)B_8)$ with $\ell$ passing through exactly two nodes in the chamber Sch $\frac{1}{2} < c \leq \frac{2}{3}$.}
		\end{subfigure}
		\begin{subfigure}{0.45\textwidth}
			\centering
			\includegraphics[scale=0.13]{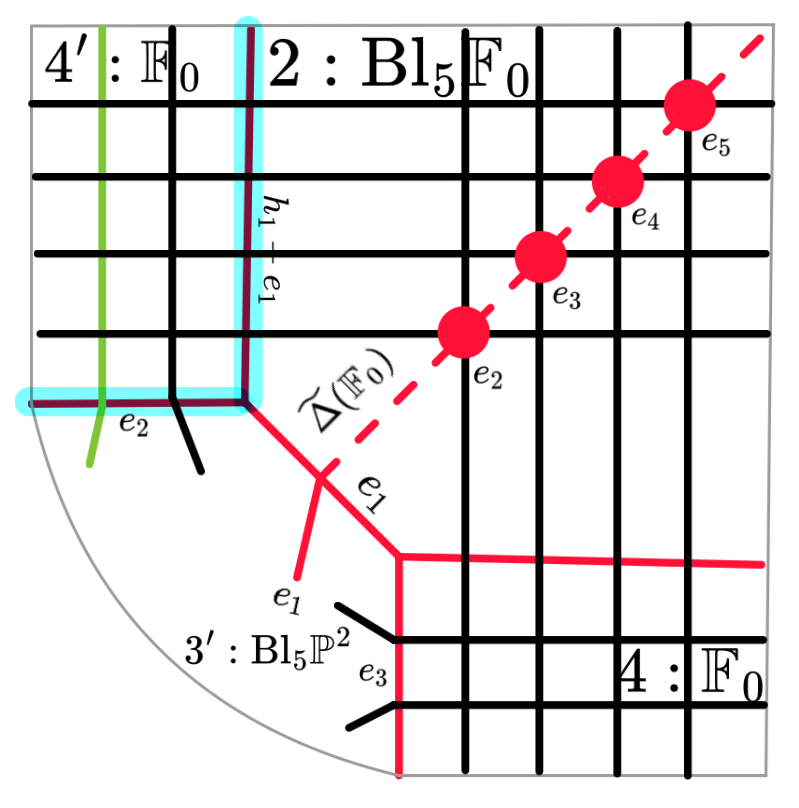}
			\caption{Another view of the gluing of the type 4 components to the type 3' component and type 2 components. The red dashed line indicates a gluing line, not a line of multiplicity 2 to avoid confusion about the number of irreducible components.}
		\end{subfigure}
		\begin{subfigure}{0.45\textwidth}
			\centering
			\includegraphics[scale=0.13]{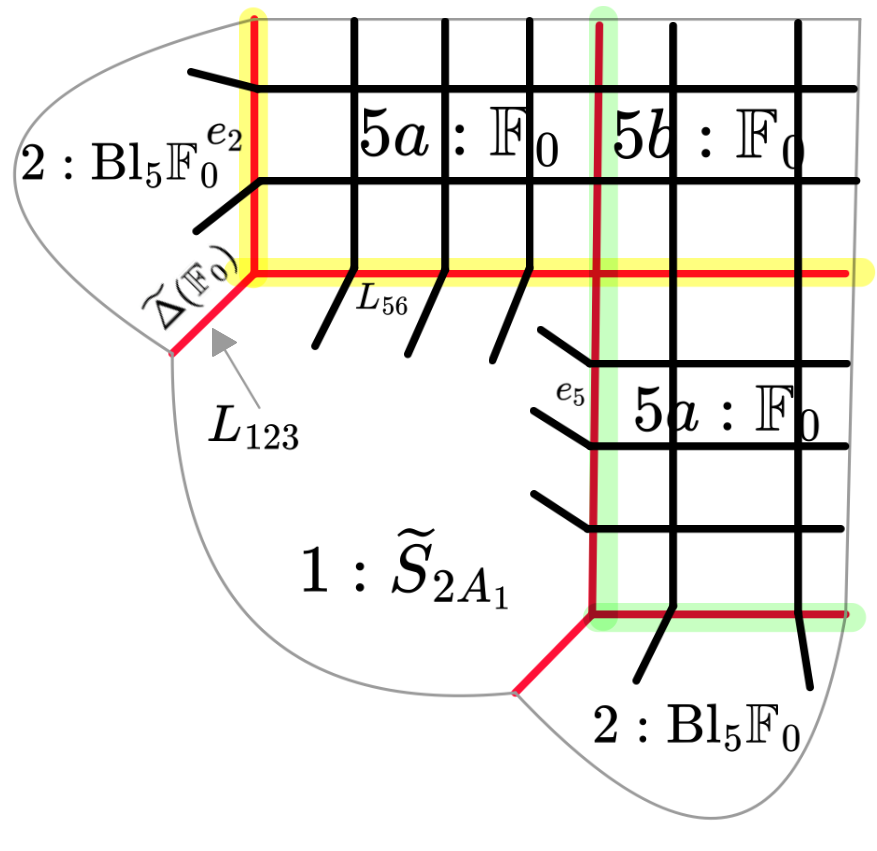}
			\caption{Intersections of a pair of type 5a components and the gluing of a type 5b component isomorphic to $\mathbb{F}_0$ along their intersection.}
		\end{subfigure}
		\caption{A type $\widetilde{E}_{2A_1}$ surface $(S_8, (b, c)B_8)$ with $\ell$ passing through exactly two nodes in the chamber Sch $\frac{1}{2} < c \leq \frac{2}{3}$, obtained as the stable replacement for these weights of the surfaces $(S_4, (b, c)B_4)$ and $(S_7, (b, c)B_7)$ of \Cref{E2A1_mult4_14c12} and  \Cref{E2A1_mult4_-b1c12}, respectively.}
		\label{E2A1_mult4_12c23}
\end{figure}

\begin{table}[H]
		\centering
		\label{tab:E2A1_mult4_12c23}
		\begin{tabular}{| c | c | c | c| c|}
			\hline
			Label & Surface & \# & Boundary lines $B$ & Gluing lines $\Delta$\\
			\hline
			\hline
			$1$ & $\widetilde{S}_{2A_1}$ & 1 & mult 1: $L_{14}, L_{15}, L_{24}, L_{25}, L_{36}, C, e_6$ & $e_1, e_2, L_{46}, L_{56}, e_4, e_5, L_{16}, L_{26}$\\
			& & & & $L_{123}, L_{345}, e_3$\\
			\hline
			$2$ & $\mathrm{Bl}_5 \mathbb{F}_0$ & 2 & mult 1: $h_1 - e_i$, $h_2 - e_i$ for $2 \leq i \leq 5$ &$\widetilde{\Delta}(\mathbb{F}_0) = h_1 + h_2 - \sum_1^5 e_i$\\
			& & & & $e_j$ for $1 \leq j\leq 5$\\
			& & & & $h_1 - e_1$, $h_2 - e_1$\\
			\hline
			$3'$ & $\mathrm{Bl}_5 \mathbb{P}^2$ & 1 & mult 1: $\ell = h - e_2 - e_5$, $h-e_2 - e_4$ & $h-e_1 - e_2 - e_3$, $h - e_1 -e_4- e_5$,\\
			& & & $h- e_3-e_4$, $h-e_3 - e_5$, $h-e_1$ & $e_j$ for $1 \leq j \leq 5$\\
			\hline
			$4$ & $\mathbb{F}_0$ & 2& mult 1: $2h_1, 4h_2$ & $h_1, h_2$\\
			\hline
			$4'$ & $\mathbb{F}_0$ & 2 & mult 1: $\ell = h_1, h_1, 4h_2$ & $h_1, h_2$\\
			\hline
			$5a$ & $\mathbb{F}_0$ & 8 & mult 1: $2h_1, 3h_2$ & $h_1, 2h_2$\\
			\hline
			$5b$ & $\mathbb{F}_0$ & 4 & mult 1: $2h_1, 2h_2$ & $h_1, h_2$\\
			\hline
		\end{tabular}
		\caption{Explicit description of the boundary lines and gluing lines of the irreducible components of the surface of \Cref{E2A1_mult4_12c23}.}
\end{table}

	\newpage
	\section{Proofs of main results}\label{S:proofofchambers}
In this section, we prove Theorems \ref{teo:main}, \ref{cor:mainmorphisms}, and \ref{cor:mainmoduli}. The proofs follow a procedure similar to that in \cite[\S 7]{schock_moduli_2024}. In particular, to prove \Cref{teo:main}, we carry out the following:

\begin{proc}[Procedure of proving \Cref{teo:main}]\label{conv:descproof}
	\noindent The basic strategy is to start at the top of a region and move left, and then down. In other words, it is a right-to-left ``scanning" approach. The new regions are polytopes dictated by the linear inequalities for ampleness and volume. Typically the new region (chamber) has at most two new edges (walls), with one edge having negative slope and the other edge being horizontal (zero slope). In a few cases, the new polytope could have three or four edges with positive slope.  Then the strategy is to start on the right, and move down. In each step, moving either right-to-left or downward, there will be a minimal $c$-value. This is considered a \textit{floor}, and we will compute all the polytope chambers above this floor before finally crossing the floor and repeating this process.
	
	In more detail, the strategy is as follows: 
	\begin{enumerate}
		\item Choose a boundary strata type for strata on $\widetilde{Y}_{(1, \dots, 1)}$ as described in \Cref{conv:types}, and then choose an irreducible component of the strata, which further determines the  marked line $\ell$ on a general surface parameterized by the chosen boundary stratum, as described in   \Cref{rem:line-choice-strata}.
		\item Fix a general surface $(S, B)$ parameterized by the chosen boundary strata type in step (1), for the universal family $\ddot{\pi} : (\ddot{Y}(E_7), B_{\ddot{\pi}}) \to \overline{Y}_{(1, \dots, 1)}$ starting with $(b_0, c_0, \dots, c_0) = (1, 1, \dots, 1)$.
		\item For each irreducible component $S'$ of $S$ and general rational weight coefficients $(b, c)$, compute the log canonical divisor
		\begin{align*}
			K_{S'} + (b,c)B|_{S'} + \Delta_{S'}
		\end{align*}
		where $\Delta_{S'}$ is the gluing locus of $S'$ to $S$.
		\item We can now identify the walls for the chamber containing $(b_0,c_0,\dots,c_0)$: It turns out that the  walls (see \Cref{S:detectingwalls} below for more details) will be given by lines (intersected with the polytope given by $b \geq c$ in the weight domain $\mathrm{Amp}$).
		Only three types of lines occur as walls:
		\begin{enumerate}
			\item Lines having strictly negative slope,
			\item zero slope, or
			\item strictly positive slope.
		\end{enumerate} 
		These lines are determined by when the log canonical divisor fails to be ample, the pair $(S, (b, c)B)$ fails to be slc, or that the volume fails to satisfy
		\begin{align*}
			\mathrm{vol}(K_S + (b,c)B) &= -b^2 + 20bc - 2b + 224c^2 - 52c + 3.
		\end{align*}
		Note that to compute the volume, we will compute the volume $(K_{S'} + B|_{S'} + \Delta_{S'})^2$ for each irreducible component $S' \subset S$.
		\item Once all walls are identified, let the \textit{new known region} be the region in $\mathrm{Amp}$ defined by 
		\begin{align*}
			\textit{new known region } = \{b \geq c\} \cap \{(b, c) \text{ strictly within walls identified}\}
		\end{align*}
		Any \textit{new known region} in our scenario will take the form of one the following regions in $\mathrm{Amp}$ with $b \geq c$ (drawn not to scale): 
		\begin{center}
			\begin{tikzpicture}[scale=2]
				\draw[thick] (-4, 0) -- (-3, 1);
				
				\draw[thick] (-3, 0.5) -- (-2, 0.5);
				
				\draw[solid] (-2, 0) -- (-1, 0);
				\draw[solid] (-1, 0) -- (-1, 1);
				\draw[thick] (-2, 0) -- (-1, 1);
				\node[rotate=45] at (-1.5, .61) {$b = c$};
				
				\draw[thick] (-0.9, 0) -- (0, 1);
				\node[rotate=45] at (-0.5, .61) {$b = c$};
				\draw[solid] (-0.9, 0) -- (0.5, 0);
				\draw[solid] (0.5, 0) -- (0, 1);
				
				\draw[solid] (0.5, 1) -- (1.5, 0);
				\draw[thick] (1.5, 0) -- (1.5, 1);
				\draw[solid] (1.5, 1) -- (0.5, 1);
				\node[rotate=-90] at (1.6, 0.5) {$b= 1$};
				
				\draw[solid] (1.75, 0) -- (2.75, 1);
				\draw[solid] (2.75, 1) -- (3.5, 1);
				\draw[solid] (3.5, 1) -- (1.75, 0);
				
				\draw[solid] (-1.5, -1.1) -- (-0.5, -0.1);
				\draw[solid] (-0.5, -0.1) -- (-0.5, -0.1);
				\draw[solid] (-0.5, -0.1) -- (-0.5, -1.1);
				\draw[solid] (-0.5, -1.1) -- (-1.5, -1.1);
				
				\draw[solid] (-0.2, -0.1) -- (0.8, -1.1);
				\draw[solid] (-0.2, -0.1) -- (1.3, -0.1);
				\draw[solid] (1.3, -0.1) -- (1.3, -1.1);
				\draw[solid] (0.8, -1.1) -- (1.3, -1.1);
				\node[rotate=-90] at (1.45, -0.55) {$b=1$};
				
				\draw[solid] (-3, -1.1) -- (-2, -0.1);
				\draw[solid] (-2, -0.1) -- (-1, -0.1);
				\draw[dashed] (-1, -0.1) -- (-3, -1.1);
				\node[rotate=32] at (-2, -0.75) {$c = \frac{b}{10} + \frac{1}{10}$};
				
				\draw[solid] (-4.5, -0.1) -- (-4.1, -1.1);
				\draw[solid] (-4.5, -0.1) -- (-2.75, -0.1);
				\draw[dashed] (-4.1, -1.1) -- (-2.75, -0.1);
				\node[rotate = 35] at (-3.4, -0.75) {$c = \frac{b}{10} + \frac{1}{10}$};
				
				\draw[solid] (1.6, -1.1) -- (2.0, -1.1);
				\draw[dashed] (2.0, -1.1) -- (3.5, -0.1);
				\node[rotate=32] at (3, -0.6) {$c = \frac{b}{10} + \frac{1}{10}$};
				\draw[solid] (1.6, -1.1) -- (2.5, -0.1);
				\draw[solid] (2.5, -0.1) -- (3.5, -0.1);
			\end{tikzpicture}
		\end{center}
		In each of the forms, when the \textit{new known region} is one-dimensional, the walls will be determined by the failure of the volume condition. When the \textit{new known region} is two-dimensional, the horizontal top walls arise when the slc condtion fails while the remaining edges are either the boundary of $\{b = c\} \cap \mathrm{Amp}$ or walls that arise when the ample condition and/or the volume condition fails. Furthermore, initially set both values \textit{old floor $c$-value} and \textit{new floor $c$-value} as $c = 1$. Now, there are three cases:
		\begin{enumerate}
			\item If the \textit{new known region} is not a horizontal line as shown above and the minimum $c$-value in the \textit{new known region} is strictly smaller than the \textit{new floor $c$-value}, then redeclare the original \textit{new floor $c$-value} as the \textit{old floor $c$-value} and this new minimum $c$-value as the \textit{new floor $c$-value}. 
			\item If the \textit{new known region} is not a horizontal line and the minimum $c$-value in the \textit{new known region} is the same as the \textit{new floor $c$-value}, then do nothing.
			\item If the \textit{new known region} is a horizontal line $c = a$ for $\alpha \leq b \leq \beta$, then redeclare the original \textit{new floor $c$-value} as the \textit{old floor $c$-value} and the value $c = a - \eta$ for $0 < \eta \ll 1$ small as the \textit{new floor $c$-value}. Furthermore, set the weights $(b, c) = (b, a - \eta)$ with $\alpha \leq b \leq \beta$ and skip step (6), i.e. move onto step (7).
		\end{enumerate}
		\item We have another three cases:
		\begin{enumerate}
			\item If the \textit{new floor $c$-value} was redeclared in the previous step and the \textit{new known region} \textbf{does not} contain the boundary $b= 1$ of $\mathrm{Amp}$, then go back to the region containing the \textit{old floor $c$-value}. Then decrease the $c$-weight of $(b, c)$ until it equals the \textit{old floor $c$-value}. 
			\item If the \textit{new floor $c$-value} was \textbf{not} redeclared in the previous step and the \textit{new known region} \textbf{does} contain the boundary $b = c$ of $\mathrm{Amp}$, then set $(b, c)$ back to the region (chamber) where the \textit{new floor $c$-value} was declared. Then decrease the $c$-weight of $(b, c)$ until it equals the \textit{new floor $c$-value}.
			\item Otherwise, let $(b, c)$ be at the left-most wall of the \textit{new known region}.
		\end{enumerate}
		\item For the weight chosen $(b, c)$ in the previous step, compute the log canonical model for each irreducible component and, if \Cref{cor:gluing} applies, glue the modified components back together to get the stable replacement model. If \Cref{cor:gluing} does not apply, perform a sequence of birational modifications until we can apply \Cref{cor:gluing} to get the stable replacement model for the weight $(b, c)$ at the wall.
		\item There are three cases:
		\begin{enumerate}
			\item If the \textit{new known region} does not contain the point $(\frac{1}{9} + \epsilon, \frac{1}{9} + \epsilon)$, then repeat steps (2) - (7) with $(S, B)$ being the stable replacement model computed in step (7), $(b_0, c_0) = (b, c)$ chosen in step (6).
			\item If the \textit{new known region} contains the point $(\frac{1}{9} + \epsilon, \frac{1}{9} + \epsilon)$, continue onto step (9).
		\end{enumerate}
		\item Start over from step (1) with a new boundary divisor type and/or choice of $\ell$; repeat this process until we exhaust all boundary divisor types and choices of $\ell$.
	\end{enumerate}
\end{proc}

\begin{rem}
A priori, this procedure will only tell us that the surfaces obtained as stable replacements at a wall will appear in the boundary for the KSBA moduli space in the respective chamber, and it is possible that there are surfaces in the boundary not detected. However, this is not the case. Indeed, suppose that $(S, B)$ is a cubic surface at the boundary of the KSBA moduli space for weights $(b, c)$. Then there must be some one-parameter family of smooth cubic surfaces degenerating to it in the central fiber. Now, after increasing the weights $(b, c)$, we will eventually hit a wall where we will have already computed how to replace the central fiber by a stable replacement. This allows us to work inductively. The base case is the weight $1$ surfaces given by \cite[Thm.~4.4]{schock_moduli_2024} (see also \Cref{teo:fullweightfibers}); the stable replacement must be one of these weight $1$ surfaces. Inductively, after applying \Cref{conv:descproof} starting from weight $1$, we see the surface $(S, B)$ must be a stable replacement for some weights $(b, c)$ that we have computed in \Cref{conv:descproof}. This argument also shows that this procedure gives all the walls in our wall-and-chamber decomposition.
\end{rem}

As the proof is very similar for each boundary type and choice of $\ell$, we follow \Cref{S:wallandchamberbytype} and only provide a proof of \Cref{E2A1_mult4} following this procedure. In other words, in step (1), we choose the boundary divisor type $\widetilde{E}_{2A_1}$ and choose $\ell$ to be the line passing through exactly two nodes in the corresponding boundary divisor of $\overline{Y}_\ell$. We further suppose we are in the substrata where the surface has exactly one copy of $E : \mathbb{P}^2$ glued along an exceptional divisor obtained by blowing up Eckardt points as in \Cref{teo:eckblowup} \textbf{not} containing $\ell$. In particular, we are in the modification of the substratum where $\ell$ does not contain an Eckardt point. Following \Cref{conv:descproof} in this case, we will compute the chambers in \Cref{E2A1_mult4} in the following order:
\begin{enumerate}
	\item Sch $\frac{2}{3} < c \leq 1$
	\item Union of $-\frac{b}{2} + 1 < c \leq \frac{2}{3}$ and Sch $\frac{1}{2} < c \leq \frac{2}{3}$ (since $\ell$ does not intersect $E : \mathbb{P}^2$, c.f. \Cref{rem:elleckcase})
	\item $\frac{1}{4} < c \leq \frac{1}{2}$, $c > -b+1$
	\item Sch $\frac{1}{4} < c \leq \frac{1}{2}$
	\item $\frac{b}{10} + \frac{1}{10} < c \leq \frac{1}{4}$, $\frac{1}{6} < c < \frac{b}{3}$
	\item $c = \frac{b}{3}$
	\item Sch $\frac{1}{6} < c \leq \frac{1}{4}$
	\item $\frac{b}{10} + \frac{1}{10} < c \leq \frac{1}{6}$, $c > -\frac{b}{3} + \frac{1}{3}$
	\item $\frac{b}{10} < c \leq \frac{1}{6}$, $c \leq -\frac{b}{3} + \frac{1}{3}$, $b \neq c$
	\item Sch $\frac{1}{9} < c \leq \frac{1}{6}$
\end{enumerate}

\subsection{Detecting walls}\label{S:detectingwalls}
Let $(S, (b, c)B)$ be a $(b, c)$-weighted stable marked cubic surface for some fixed weight $(b, c) \in \mathrm{Amp}$. By the definition of being KSBA stable (\Cref{def:KSBApair}), the pair $(S, (b, c)B)$ is slc and $K_S + (b, c)B$ is ample. Furthermore, $(S, (b, c)B)$ must also satisfy the constant volume condition, where in this case,
\begin{align}\label{eq:constantvol}
	\mathrm{vol}(K_S + (b,c)B) &= -b^2 + 20bc - 2b + 224c^2 - 52c + 3.
\end{align}
Also recall that the restriction of $K_S + (b, c)B$ to an irreducible component $S'$ of $S$ is given by
\begin{align*}
	(K_S + (b, c)B)|_{S'} &=  K_{S'} + (b, c)B' + \Delta
\end{align*}
where $B'$ is the restriction of $B$ to $S'$ and $\Delta$ is the gluing locus of $S'$ to the other irreducible components of $S$. Walls occur at values $(b', c') \in \mathrm{Amp}$ for which stability may fail for one of the the following reasons: 
\begin{enumerate}
	\item $(S, (b', c')B)$ is no longer slc, \label{stable:slc}
	\item $K_{S'} + (b', c') B' + \Delta$ is nef and semi-ample but not ample for some irreducible component $S'$ of $S$, or \label{stable:ample}
	\item $(S, (b', c')B)$ fails the constant volume condition \ref{eq:constantvol}. \label{stable:vol}
\end{enumerate}
Note that we do not need to worry about the condition that $K_S + (b', c')B$ is $\mathbb{Q}$-Cartier as all of our irreducible components are Gorenstein $\mathbb{Q}$-factorial surfaces glued along smooth curves. To check for condition (1), we can use the following lemmas:
\begin{lem}\label{lem:lineslccriteria}
	Let $D = \sum_1^n \alpha_i L_i$ be a sum of distinct lines $L_1, \dots, L_n$ with weights $\alpha_1, \dots, \alpha_n \in \mathbb{Q} \cap (0, 1]$ intersecting at a point in $X = \mathbb{A}^2$. Then the pair $(X, D)$ is log canonical if and only if $\sum_1^n \alpha_i \leq 2$.
\end{lem}
\begin{proof}
	Without loss of generality, we can assume the lines intersect at $0 \in \mathbb{A}^2$. Blowing up the origin $\pi : \tilde{X} \to X$ with exceptional divisor $e$ provides a log-resolution of $(X, D)$. For each $L_i$, we get $\pi^* L_i = \widetilde{L}_i + e$ where $\widetilde{L}_i$ denotes the strict transform of $L_i$; denote $\widetilde{D} = \sum_1^n \alpha_i \widetilde{L}_i$. Then for some $\alpha \in \mathbb{Q}$, we obtain $K_{\widetilde{X}} + \widetilde{D} = \pi^*(K_X + D) + \alpha e$. Since $K_X = 0$, we have that $K_{\widetilde{X}} = e$, and so $\alpha = 1 - \sum_1^n \alpha_i$. Thus, the pair $(X, D)$ is log canonical if and only if $\sum_1^n \alpha_i \leq 2$.
\end{proof}

\begin{lem}\label{lem:linesA1lccriteria}
	Let $X \subseteq \mathbb{P}^2$ be a quadric cone and let $D = \sum_1^n \alpha_i L_i$ be a sum of distinct lines $L_1, \dots, L_n$ with weights $\alpha_1, \dots, \alpha_n \in \mathbb{Q} \cap (0, 1]$ passing through the $A_1$ singularity. Then the pair $(X, D)$ is log canonical if and only if $\sum_1^n \alpha_i \leq 2$.
\end{lem}
\begin{proof}
	Blowing up the $A_1$ singularity $\pi : \widetilde{X} \to X$ with exceptional divisor $e$ provides a log-resolution of $(X, D)$. For each $L_i$, we get $\pi^* L_i = \widetilde{L}_i + c e$ where $\widetilde{L}_i$ denotes the strict transform of $L_i$ and $c \in \mathbb{Q}$. To find $c$, we can intersect both sides with $e$, to get $0 = 1 - 2c e$. Thus, $c = \frac{1}{2}$. Then for some $\alpha \in \mathbb{Q}$, we obtain $K_{\widetilde{X}} + \widetilde{D} = \pi^*(K_X + D) + \alpha e$ where $\widetilde{D} = \sum_1^n \alpha_i \widetilde{L}_i$.  Since resolving an $A_1$ singularity is a crepant resolution, we have $\pi^* K_{\widetilde{X}} = K_X$, and so $\alpha = -\sum_1^n \frac{\alpha_i}{2}$. Thus, the pair $(X, D)$ is log canonicalcanonical if and only if $\sum_1^n \alpha_i \leq 2$.
\end{proof}

\subsection{Proofs of \Cref{E2A1_mult4} and Theorems \ref{teo:main}, \ref{cor:mainmorphisms}, and \ref{cor:mainmoduli}}
We now carry out \Cref{conv:descproof} explicitly for \Cref{E2A1_mult4}. 

\begin{proof}[Proof of \Cref{E2A1_mult4}]\label{proof:E2A1_mult4}\hfill\\
We identify each chamber by finding all possible walls, i.e. values $(b' ,c') \in \mathrm{Amp}$ for which
\begin{enumerate}
		\item the pair $(S, (b' ,c')B)$ is no longer slc,
		\item $K_{S'} + (b', c')B' + \Delta$ is nef (and semi-ample) but not ample for some irreducible component $S'$ of $S$, $B'$ being the restriction of $B$ to $S'$, and $\Delta$ being the gluing locus of $S'$ to the other irreducible components of $S$, or
		\item the pair $(S, (b', c')B)$ fails the constant volume condition \ref{eq:constantvol}.
\end{enumerate}
For each pair $(S, (b' , c')B)$, we will check each condition on the irreducible components given by the explicit descriptions found in \Cref{S:descriptions}. In particular, we will check condition (1) by using the explicit descriptions for each irreducible component and Lemmas \ref{lem:lineslccriteria} and \ref{lem:linesA1lccriteria}. Condition (2) will be checked by an explicit computation of the log canonical divisor $K_{S'} + (b', c')B' + \Delta$ for each irreducible component $S'$ of $S$. Finally, condition (3) will be checked by computing the volume $(K_{S'} + (b', c')B' + \Delta)^2$ for each irreducible component $S'$ and checking if their sum satisfies \ref{eq:constantvol}.

\textbf{Step 0: Obtaining \Cref{E2A1_mult4_12c23}}
Let $b, c \leq 1$ and recall by \Cref{teo:eckblowup} that the surfaces $(\ddot{S}, \ddot{B})$ obtained as the fiber of $\ddot{\pi} :  (\ddot{Y}(E_7), B_{\ddot{\pi}}) \to \overline{Y}_{(1, \dots, 1)}$ is obtained by taking a fiber of $(S, B)$ of $\widetilde{\pi}$ by blowing up each Eckardt point on $(S, B)$ and attaching to the exceptional divisor a copy of $\mathbb{P}^2$ glued along a general line (see \Cref{teo:eckblowup} for more details). Thus, in this case, the surface $(\ddot{S}, \ddot{B})$ will be obtained by doing this modification to the surface $(S, B)$ of \Cref{E2A1_mult4_12c23}. By \Cref{tab:E2A1eck}, we will assume the surface has one Eckardt point in general position in the component labeled $1 : \widetilde{S}_{2A_1}$ given by the intersection of three multiplicity 1 lines. From this description, the surface $(\ddot{S}, \ddot{B})$ will have $21$ irreducible components: One component of Type $1$ isomorphic to $\mathrm{Bl}_1 \widetilde{S}_{2A_1}$ (blow up the Eckardt point in the simultaneous resolution of a cubic surface with two $A_1$ singularities). Two components of Type $2$ isomorphic to $\mathrm{Bl}_5 \mathbb{F}_0$. One component isomorphic of Type $3'$ isomorphic to $\mathrm{Bl}_1 (\mathrm{Bl}_5 \mathbb{P}^2)$ (where the extra blow up is from the Eckardt point containing $\ell$). Two components of Type $4$ isomorphic to $\mathbb{F}_0$. Two components of Type $4'$ isomorphic to $\mathbb{F}_0$. Eight components of Type $5a$ isomorphic to $\mathbb{F}_0$. Four components of Type $5b$ isomorphic to $\mathbb{F}_0$. One component of type $E$ isomorphic to $\mathbb{P}^2$ attached along a general line to the exceptional divisor of the blow up $1: \mathrm{Bl}_1 \widetilde{S}_{2A_1}$.

Let us first show that $(\ddot{S}, (b, c)\ddot{B})$ satisfies condition (1). By the explicit description of the lines on each irreducible component in \Cref{E2A1_mult4_12c23}, \Cref{tab:E2A1_mult4_12c23}, and \Cref{teo:eckblowup} we will have at most two lines intersecting at a point. Therefore, by \Cref{lem:lineslccriteria}, for each irreducible component $\ddot{S}' \subseteq \ddot{S}$ with gluing divisor $\ddot{\Delta}$, the pair $(\ddot{S}', \ddot{B}|_{\ddot{S}'} + \ddot{\Delta})$ will be log canonical. Now to show conditions (2) and (3), let us compute the log canonical divisor for each irreducible component.

\textbf{Type $1$ : $\mathrm{Bl}_1 \widetilde{S}_{2A_1}$ (one copy)}  Let us first identify $\widetilde{S}_{2A_1}$ as $\mathrm{Bl}_6 \mathbb{P}^2$ where the six points are in special position where there are exactly two pairwise intersecting sets of three colinear points (see \Cref{tab:specialposcubics}). So say we are blowing up the points $p_1, \dots, p_6$ where the set of points $\{p_1, p_2, p_3\}$ and $\{p_3, p_4, p_5\}$ are colinear. For some notation, let 
$$L_{ij} - h - e_i - e_j$$
denote the strict transform of the line passing through points $p_i$ and $p_j$, where $h$ denotes the class of the pullback of the hyperplane divisor and $e_i$ denotes the class of the exceptional divisor of the blow up of $p_i$. Similarly, let 
$$L_{ijk} = h - e_i - e_j - e_k$$
denote the strict transform of the line passing through points $p_i, p_j, p_k$. Furthermore, let 
$$C = 2h - e_1 - e_2 - e_4 - e_5 - e_6$$
denote the strict transform of the unique conic passing through the five points $p_1, p_2, p_4, p_5, p_6$ in general position. Then by \Cref{tab:E2A1_mult4_12c23}, the gluing divisor $\Delta_{\widetilde{S}_{2A_1}}$ is then given by
\begin{align*}
	\Delta_{\widetilde{S}_{2A_1}} &= L_{123} + L_{345} + (e_1 + e_2 + L_{46} + L_{56}) + (e_4 + e_5 + L_{16} + L_{26}) + e_3,
\end{align*}
and the boundary $B_{\widetilde{S}_{2A_1}}$ is given by
\begin{align*}
	B_{\widetilde{S}_{2A_1}} = B|_{\widetilde{S}_{2A_1}} = L_{14} + L_{15} + L_{24} + L_{25} + L_{36} + C + e_6.
\end{align*}
Snce $\ell$ ($b$-weighted line) is not contained in $\widetilde{S}_{2A_1}$, we have 
\begin{align*}
	(b, c)B_{\widetilde{S}_{2A_1}} = cB_{\widetilde{S}_{2A_1}} = c( L_{14} + L_{15} + L_{24} + L_{25} + L_{36} + C + e_6).
\end{align*}
We now blow up the Eckardt point $p_7$ in general position giving us $\mathrm{Bl}_1 \widetilde{S}_{2A_1}$ with exceptional divisor $e_7$. Then since we are attaching the component $E : \mathbb{P}^2$ along $e_7$ and none of the lines in $\Delta_{\widetilde{S}_{2A_1}}$ intersects $p_7$, the gluing divisor $\Delta_{\mathrm{Bl}_1 \widetilde{S}_{2A_1}}$ of $\mathrm{Bl}_1 \widetilde{S}_{2A_1}$ is given by
\begin{align*}
	\Delta_{\mathrm{Bl}_1 \widetilde{S}_{2A_1}} = L_{123} + L_{345} + (e_1 + e_2 + L_{46} + L_{56}) + (e_4 + e_5 + L_{16} + L_{26}) + e_3 + e_7.
\end{align*}
On the other hand, by the proof of \Cref{pro:cubiceck}(3) found in \cite[Prop.~5.9]{schock_moduli_2024}, we can suppose that that the three multiplicity 1 lines $L_{14}, L_{25}$, and $L_{36}$ intersect $p_7$ to give the Eckardt point. Then 
\begin{align*}
	(b, c)B_{\mathrm{Bl}_1 \widetilde{S}_{2A_1}} = c(\widetilde{L}_{14} + L_{15} + L_{24} + \widetilde{L}_{25} + \widetilde{L}_{36} + C + e_6)
\end{align*}
where $\widetilde{L}_{14}$, $\widetilde{L}_{25}$, and $\widetilde{L}_{36}$ are the strict transforms of $L_{14}, L_{25}$, and $L_{36}$, respectively. Furthermore, since the other lines do not intersect $p_7$, we slightly abuse notation and denote them the same we did in $\widetilde{S}_{2A_1}$. Thus, the log canonical divisor $K_{\mathrm{Bl}_1 \widetilde{S}_{2A_1}} + cB_{\mathrm{Bl}_1 \widetilde{S}_{2A_1}} + \Delta_{\mathrm{Bl}_1 \widetilde{S}_{2A_1}}$ is given by
\begin{align*}
	K_{\mathrm{Bl}_1 \widetilde{S}_{2A_1}} + cB_{\mathrm{Bl}_1 \widetilde{S}_{2A_1}} + \Delta_{\mathrm{Bl}_1 \widetilde{S}_{2A_1}} &= (7c+3)h - 3ce_1 - 3ce_2 - ce_3 - 3ce_4 - 3ce_5 - (c+3)e_6 - (3c - 2)e_7.
\end{align*}
We claim the log canonical divisor is ample for $c > \frac{2}{3}$. Indeed, since the Eckardt point was in general position, by \Cref{teo:nefcones}, the ample cone of $\mathrm{Bl}_1 \widetilde{S}_{A_1}$ is given by divisors $L = ah - \sum_1^7 b_i e_i$ such that $b_1, \dots, b_7 > 0$, $a > b_1 + b_2 + b_3, b_3 + b_4 + b_5$, and $a> b_i + b_j$ for the remaining points $p_i, p_j$ in general position. Applying this criterion to $K_{\mathrm{Bl}_1 \widetilde{S}_{2A_1}} + cB_{\mathrm{Bl}_1 \widetilde{S}_{2A_1}} + \Delta_{\mathrm{Bl}_1 \widetilde{S}_{2A_1}}$, we see that we will fail ampleness first along $e_7$ when $c = \frac{2}{3}$. On the other hand, in regards to condition (3) we have the volume of the log canonical divisor in given by
\begin{align}\label{eqn:Bl1S2A1_23c1}
	(K_{\mathrm{Bl}_1 \widetilde{S}_{2A_1}} + cB_{\mathrm{Bl}_1 \widetilde{S}_{2A_1}} + \Delta_{\mathrm{Bl}_1 \widetilde{S}_{2A_1}})^2 &= 2c^2 + 42c - 4.
\end{align}

\textbf{Type 2: $\mathrm{Bl}_5 \mathbb{F}_0$ (two copies)}
Since there are no Eckardt points in this component, we can use the description of the component $2 : \mathrm{Bl}_2 \mathbb{F}_0$ in \Cref{E2A1_mult4_12c23}. From \Cref{tab:E2A1_mult4_12c23}, we have the gluing divisor $\Delta_{\mathrm{Bl}_5 \mathbb{F}_0}$ is given by
\begin{align*}
	\Delta_{\mathrm{Bl}_5 \mathbb{F}_0} &= (h_1 - e_1) + (h_2 - e_1) + (e_1 + \cdots + e_5) + (h_1 + h_2 - e_1 - \cdots - e_5).
\end{align*}
where $h_1, h_2$ denote the classes of the two rulings and $e_1, \dots, e_5$ denotes the classes of the exceptional divisors. Now since $\ell$ ($b$-weighted line) is not contained in this $\mathrm{Bl}_5 \mathbb{F}_0$ component, we have
\begin{align*}
	(b, c)B_{\mathrm{Bl}_5 \mathbb{F}_0} = cB_{\mathrm{Bl}_5 \mathbb{F}_0} = 4ch_1 + 4ch_2 - 2c(e_2 + \cdots + e_5).
\end{align*}
Then the log canonical divisor $K_{\mathbb{F}_0} + cB_{\mathbb{F}_0} + \Delta$ is given by
\begin{align*}
	K_{\mathrm{Bl}_5 \mathbb{F}_0} + cB_{\mathrm{Bl}_5 \mathbb{F}_0} + \Delta_{\mathrm{Bl}_5 \mathbb{F}_0} &= -2h_1 - 2h_2 + (e_1 + \cdots + e_5) + cB_{\mathrm{Bl}_5 \mathbb{F}_0} + \Delta_{\mathrm{Bl}_5 \mathbb{F}_0}\\
	&= 4ch_1 + 4ch_2 - e_1 - (2c-1)(e_2 + \cdots + e_5).
\end{align*}
In regards to condition (2), we can see that $K_{\mathrm{Bl}_5 \mathbb{F}_0} + cB_{\mathrm{Bl}_5 \mathbb{F}_0} + \Delta$ is ample for $c > \frac{1}{2}$. Indeed, suppose $c > \frac{1}{2}$ and let $L = a_1h_1 + a_2h_2 - \sum_1^5 m_i e_i$ be the class of an effective irreducible divisor. Then it must be that $a_1, a_2 \geq 0$ and when $a_1 = a_2 = 0$, exactly one $m_k = -1$ while $m_j = 0$ for $j \neq k$. Intersecting with $K_{\mathrm{Bl}_5 \mathbb{F}_0} + cB_{\mathrm{Bl}_5 \mathbb{F}_0} + \Delta_{\mathrm{Bl}_5 \mathbb{F}_0}$, we obtain
\begin{align*}
	(K_{\mathrm{Bl}_5 \mathbb{F}_0} + cB_{\mathrm{Bl}_5 \mathbb{F}_0} + \Delta_{\mathrm{Bl}_5 \mathbb{F}_0} \cdot L) &= 4c(a_1 + a_2) - m_1 - (2c-1)\sum_2^5 m_i.
\end{align*}
Now using the assumptions that $c > \frac{1}{2}$ and $a_1, a_2 \geq 0$,
\begin{align*}
	4c(a_1 + a_2) - m_1 - (2c-1)\sum_2^5 m_i &> 2(a_1 + a_2) - m_1 - (2c-1)\sum_2^5 m_i\\
	&\geq -m_1 - (2c-1) \sum_2^5 m_i.
\end{align*}
Now the last inequality is an equality if and only if $a_1 = a_2 = 0$. In this case, exactly one $m_k = -1$ while $m_j = 0$ for $j \neq k$, thus
\begin{align*}
	-m_1 - (2c-1) \sum_2^5 m_i = \begin{cases}
		1 & k = 1\\
		2c-1 & \text{otherwise}
	\end{cases}.
\end{align*}
Since $c > \frac{1}{2}$, we have $-m_1 - (2c-1) \sum_2^5 m_i > 0$ and
\begin{align*}
	(K_{\mathrm{Bl}_5 \mathbb{F}_0} + cB_{\mathrm{Bl}_5 \mathbb{F}_0} + \Delta_{\mathrm{Bl}_5 \mathbb{F}_0} \cdot L) > 0.
\end{align*}
Therefore, the log canonical divisor $K_{\mathrm{Bl}_5 \mathbb{F}_0} + cB_{\mathrm{Bl}_5 \mathbb{F}_0} + \Delta_{\mathrm{Bl}_5 \mathbb{F}_0}$ fails to be ample first when $c = \frac{1}{2}$ as
\begin{align*}
	(K_{\mathrm{Bl}_5 \mathbb{F}_0} + cB_{\mathrm{Bl}_5 \mathbb{F}_0} + \Delta \cdot e_i) = 2c-1
\end{align*}
for each $i = 2, \dots, 5$. So when $c = \frac{1}{2}$, the log canonical model of components of Type $2 : \mathrm{Bl}_5 \mathbb{F}_0$ will be given by a blow-down morphism $\mathrm{Bl}_5 \mathbb{F}_0 \to \mathrm{Bl}_1 \mathbb{F}_0$ where $e_2, \dots, e_5$ contracts to a smooth point. On the other hand, in regards to condition (3), the volume of the log canonical divisor is given by
\begin{align}\label{eqn:Bl5F0_12c23}
	(K_{\mathrm{Bl}_5 \mathbb{F}_0} + cB_{\mathrm{Bl}_5 \mathbb{F}_0} + \Delta)^2 &= 16c^2 + 16c - 5.
\end{align}

\textbf{Type $3'$: $\mathrm{Bl}_5 \mathbb{P}^2$ (one copy)}
From the description of the component $3' : \mathrm{Bl}_5 \mathbb{P}^2$ in \Cref{E2A1_mult4_12c23}, we have the gluing divisor $\Delta_{\mathrm{Bl}_5 \mathbb{P}^2}$ is given by
\begin{align*}
	\Delta_{\mathrm{Bl}_5 \mathbb{P}^2} &= (h - e_1 - e_2 - e_3) + (h - e_1 - e_4 - e_5) + (e_1 + \cdots + e_5)
\end{align*}
where $h$ denotes the class of the pullback of the hyperplane divisor in $\mathbb{P}^2$ and $e_1, \dots, e_5$ denotes the classes of the exceptional divisors. We see that $\ell$ ($b$-weighted line) is indeed contained in this component, so
\begin{align*}
	(b, c)B_{\mathrm{Bl}_5 \mathbb{P}^2} &=  (b + 4c)h - (b+c)(e_2 + e_4) - 2c(e_3 + e_5) - ce_1.
\end{align*}
Then the log canonical divisor $K_{\mathrm{Bl}_5 \mathbb{P}^2} + (b, c)B_{\mathrm{Bl}_5 \mathbb{P}^2} + \Delta_{\mathrm{Bl}_5 \mathbb{P}^2}$ is given by
\begin{align*}
	K_{\mathrm{Bl}_5 \mathbb{P}^2} + (b, c)B_{\mathrm{Bl}_5 \mathbb{P}^2} + \Delta_{\mathrm{Bl}_5 \mathbb{P}^2} &= -3h + (e_1 + \cdots + e_5) + (b, c)B_{\mathrm{Bl}_5 \mathbb{P}^2} + \Delta_{\mathrm{Bl}_5 \mathbb{P}^2}\\
	&= (b + 4c -1) h - ce_1 - (b+ c - 1)(e_2 + e_4) - (2c-1)(e_3 + e_5).
\end{align*}
In regards to condition (2), we claim that $K_{\mathrm{Bl}_5 \mathbb{P}^2} + (b, c)B_{\mathrm{Bl}_5 \mathbb{P}^2} + \Delta_{\mathrm{Bl}_5 \mathbb{P}^2}$ is ample for $c > \frac{1}{2}$. Indeed, let $L = ah - \sum_1^5 m_i e_i$ be the class of an effective irreducible divisor. Similar to the argument for Type $2$ above, we have $a \geq 0$ and when $a = 0$, exactly one $m_k = -1$ while $m_j = 0$ for $j \neq k$. Intersecting with $K_{\mathrm{Bl}_5 \mathbb{P}^2} + (b, c)B_{\mathrm{Bl}_5 \mathbb{P}^2} + \Delta_{\mathrm{Bl}_5 \mathbb{P}^2}$, we obtain
\begin{align*}
	(K_{\mathrm{Bl}_5 \mathbb{P}^2} + (b, c)B_{\mathrm{Bl}_5 \mathbb{P}^2} + \Delta_{\mathrm{Bl}_5 \mathbb{P}^2} \cdot L) &= (b + 4c - 1)a - cm_1 - (b + c - 1)(m_2 + m_4) - (2c-1)(m_3 + m_5).
\end{align*}
Using the assumptions that $c > \frac{1}{2}$ and $a \geq 0$,
\begin{align*}
	&(b + 4c - 1)a - cm_1 - (b + c - 1)(m_2 + m_4) - (2c-1)(m_3 + m_5)\\
	&> (b +1)a - cm_1 - (b+c-1)(m_2 + m_4) - (2c-1)(m_3 + m_5)\\
	&\geq - cm_1 - (b+c-1)(m_2 + m_4) - (2c-1)(m_3 + m_5).
\end{align*}
Now the last inequality is an equality if and only if $a = 0$. In this case, exactly one $m_k = -1$ while $m_j = 0$ for $j \neq k$, thus
\begin{align*}
	- cm_1 - (b+c-1)(m_2 + m_4) - (2c-1)(m_3 + m_5) &= \begin{cases}
		c & k = 1\\
		b+c-1 & k = 2, 4\\
		2c-1 & k = 3, 5
	\end{cases},
\end{align*}
Since $b \geq c > \frac{1}{2}$, we have $- cm_1 - (b+c-1)(m_2 + m_4) - (2c-1)(m_3 + m_5) > 0$ and
\begin{align*}
	K_{\mathrm{Bl}_5 \mathbb{P}^2} + (b, c)B_{\mathrm{Bl}_5 \mathbb{P}^2} + \Delta_{\mathrm{Bl}_5 \mathbb{P}^2} \cdot L) > 0.
\end{align*}
Therefore, the log canonical divisor $K_{\mathrm{Bl}_5 \mathbb{P}^2} + (b, c)B_{\mathrm{Bl}_5 \mathbb{P}^2} + \Delta_{\mathrm{Bl}_5 \mathbb{P}^2}$ fails to be ample first when $c = \frac{1}{2}$ as
\begin{align*}
	(K_{\mathrm{Bl}_5 \mathbb{P}^2} + (b, c)B_{\mathrm{Bl}_5 \mathbb{P}^2} + \Delta_{\mathrm{Bl}_5 \mathbb{P}^2} \cdot e_i) = 2c-1
\end{align*}
for $i = 3, 5$. So when $c = \frac{1}{2}$, the Type $3': \mathrm{Bl}_5 \mathbb{P}^2$ component will get contracted via the blow-down morphism $\mathrm{Bl}_5 \mathbb{P}^2 \to \mathrm{Bl}_3 \mathbb{P}^2$ where $e_3, e_5$ contracts to a smooth point. On the other hand, in regards to condition (3), the volume of the log canonical divisor is given by
\begin{align}\label{eqn:Bl5P2_12c23}
	(K_{\mathrm{Bl}_5 \mathbb{P}^2} + (b, c)B_{\mathrm{Bl}_5 \mathbb{P}^2} + \Delta_{\mathrm{Bl}_5 \mathbb{P}^2})^2 &= -b^2 +4bc +2b + 5c^2 + 4c - 3.
\end{align}

\textbf{Type $4$: $\mathbb{F}_0$ (two copies)}
From the description of the component $4 : \mathbb{F}_0$ in \Cref{E2A1_mult4_12c23} and \Cref{tab:E2A1_mult4_12c23}, we have the gluing divisor $\Delta_{\mathbb{F}_0}$ is given by
\begin{align*}
	\Delta_{\mathbb{F}_0} = h_1 + h_2
\end{align*}
where $h_1, h_2$ denotes the classes of the two rulings. Now since $\ell$ ($b$-weighted line) is not contained in this $\mathbb{F}_0$ component, we have
\begin{align*}
	(b, c)B_{\mathbb{F}_0}  = cB_{\mathbb{F}_0} = 2ch_1 + 4ch_2.
\end{align*}
Then the log canonical divisor $K_{\mathbb{F}_0} + cB_{\mathbb{F}_0} + \Delta_{\mathbb{F}_0}$ is given by
\begin{align*}
	K_{\mathbb{F}_0} + cB_{\mathbb{F}_0} + \Delta_{\mathbb{F}_0} &= (2c - 1)h_1 + (4c - 1)h_2.
\end{align*}
In regards to condition (2), we see that $K_{\mathbb{F}_0} + B_{\mathbb{F}_0} + \Delta_{\mathbb{F}_0}$ first fails to be ample when $c = \frac{1}{2}$ as
\begin{align*}
	(K_{\mathbb{F}_0} + cB_{\mathbb{F}_0} + \Delta_{\mathbb{F}_0} \cdot h_2) = 2c-1.
\end{align*}
Thus, when $c = \frac{1}{2}$, components of Type $4 : \mathbb{F}_0$ will get contracted via the projection morphism $\mathbb{F}_0 \to \mathbb{P}^1$ where the ruling $h_2$ contracts. On the other hand, in regards to condition (3), the volume of the log canonical divisor is given by
\begin{align}\label{eqn:4F0_12c23}
	(K_{\mathbb{F}_0} + cB_{\mathbb{F}_0} + \Delta_{\mathbb{F}_0})^2 &= 16c^2 - 12c + 2.
\end{align}

\textbf{Type $4'$: $\mathbb{F}_0$ (two copies)}
From the description of the component $4' : \mathbb{F}_0$ in \Cref{E2A1_mult4_12c23}, we have the gluing divisor $\Delta_{\mathbb{F}_0}$ is given by
\begin{align*}
	\Delta_{\mathbb{F}_0} = h_1 + h_2
\end{align*}
where $h_1, h_2$ denotes the classes of the two rulings. Now since $\ell$ ($b$-weighted line) is contained in this $\mathbb{F}_0$ component, we have
\begin{align*}
	(b, c)B_{\mathbb{F}_0} = (b+c)h_1 + 4ch_2.
\end{align*}
Then the log canonical divisor $K_{\mathbb{F}_0} + (b, c)B_{\mathbb{F}_0} + \Delta_{\mathbb{F}_0}$ is given by
\begin{align*}
	K_{\mathbb{F}_0} + (b,c)B_{\mathbb{F}_0} + \Delta_{\mathbb{F}_0} &= (b+c- 1)h_1 + (4c - 1)h_2.
\end{align*}
In regards to condition (2), we see that $K_{\mathbb{F}_0} + B_{\mathbb{F}_0} + \Delta_{\mathbb{F}_0}$ first fails to be ample when $c = -b+1$ as
\begin{align*}
	(K_{\mathbb{F}_0} + (b,c)B_{\mathbb{F}_0} + \Delta_{\mathbb{F}_0} \cdot h_2) = b+c-1.
\end{align*}
Thus, when $c = -b + 1$, components of Type $4' : \mathbb{F}_0$ will get contracted via the projection morphism $\mathbb{F}_0 \to \mathbb{P}^1$ where the ruling $h_2$ contracts. On the other hand, in regards to condition (3), the volume of the log canonical divisor is given by
\begin{align}\label{eqn:4'F0_12c23}
	(K_{\mathbb{F}_0} + (b,c)B_{\mathbb{F}_0} + \Delta_{\mathbb{F}_0})^2 &= 8bc -2b+8c^2 -10c + 2.
\end{align}

\textbf{Type $5a$: $\mathbb{F}_0$ (eight copies)}
From the description of the component $5a : \mathbb{F}_0$ in \Cref{E2A1_mult4_12c23}, we have the gluing divisor $\Delta_{\mathbb{F}_0}$ is given by
\begin{align*}
	\Delta_{\mathbb{F}_0} = h_1 + 2h_2
\end{align*}
where $h_1, h_2$ denotes the classes of the two rulings. Now since $\ell$ ($b$-weighted line) is not contained in this $\mathbb{F}_0$ component, we have
\begin{align*}
	(b, c)B_{\mathbb{F}_0}  = cB_{\mathbb{F}_0} = 2ch_1 + 3ch_2.
\end{align*}
Then the log canonical divisor $K_{\mathbb{F}_0} + cB_{\mathbb{F}_0} + \Delta_{\mathbb{F}_0}$ is given by
\begin{align*}
	K_{\mathbb{F}_0} + cB_{\mathbb{F}_0} + \Delta_{\mathbb{F}_0} &= (2c - 1)h_1 + 3ch_2.
\end{align*}
In regards to condition (2), we see that $K_{\mathbb{F}_0} + B_{\mathbb{F}_0} + \Delta_{\mathbb{F}_0}$ first fails to be ample when $c = \frac{1}{2}$ as
\begin{align*}
	(K_{\mathbb{F}_0} + cB_{\mathbb{F}_0} + \Delta_{\mathbb{F}_0} \cdot h_2) = 2c-1.
\end{align*}
Thus, when $c = \frac{1}{2}$, components of Type $5a : \mathbb{F}_0$ will get contracted via the projection morphism $\mathbb{F}_0 \to \mathbb{P}^1$ where the ruling $h_2$ contracts. On the other hand, in regards to condition (3), the volume of the log canonical divisor is given by
\begin{align}\label{eqn:5aF0_12c23}
	(K_{\mathbb{F}_0} + cB_{\mathbb{F}_0} + \Delta_{\mathbb{F}_0})^2 &= 12c^2 - 6c.
\end{align}

\textbf{Type $5b$: $\mathbb{F}_0$ (four copies)}
From the description of the component $5b : \mathbb{F}_0$ in \Cref{E2A1_mult4_12c23}, we have the gluing divisor $\Delta_{\mathbb{F}_0}$ is given by
\begin{align*}
	\Delta_{\mathbb{F}_0} = h_1 + h_2
\end{align*}
where $h_1, h_2$ denotes the classes of the two rulings. Now since $\ell$ ($b$-weighted line) is not contained in this $\mathbb{F}_0$ component, we have
\begin{align*}
	(b, c)B_{\mathbb{F}_0}  = cB_{\mathbb{F}_0} = 2ch_1 + 2ch_2.
\end{align*}
Then the log canonical divisor $K_{\mathbb{F}_0} + cB_{\mathbb{F}_0} + \Delta_{\mathbb{F}_0}$ is given by
\begin{align*}
	K_{\mathbb{F}_0} + cB_{\mathbb{F}_0} + \Delta_{\mathbb{F}_0} &= (2c - 1)h_1 + (2c-1)h_2.
\end{align*}
In regards to condition (2), we see that $K_{\mathbb{F}_0} + B_{\mathbb{F}_0} + \Delta_{\mathbb{F}_0}$ first fails to be ample when $c = \frac{1}{2}$ as
\begin{align*}
	(K_{\mathbb{F}_0} + cB_{\mathbb{F}_0} + \Delta_{\mathbb{F}_0} \cdot h_i) = 2c-1
\end{align*}
for $i = 1, 2$. Thus, when $c = \frac{1}{2}$, components of Type $5a : \mathbb{F}_0$ will get contracted to a point $\mathbb{F}_0 \to \{*\}$. On the other hand, in regards to condition (3), the volume of the log canonical divisor is given by
\begin{align}\label{eqn:5bF0_12c23}
	(K_{\mathbb{F}_0} + cB_{\mathbb{F}_0} + \Delta_{\mathbb{F}_0})^2 &= 8c^2 - 8c +2.
\end{align}

\textbf{Type $E$: $\mathbb{P}^2$ (one copy)}
This component isomorphic to $\mathbb{P}^2$ is glued along a general line to the exceptional divisor of the blow up $1 : \mathrm{Bl}_1 \widetilde{S}_{2A_1}$. Therefore, the gluing divisor is given by $\Delta_{\mathbb{P}^2} = h$ where $h$ denotes the class of the hyperplane divisor. Furthermore, using the explicit description of this component in \Cref{teo:eckblowup}, we have
\begin{align*}
	(b, c)B_{\mathbb{P}^2} = cB_{\mathbb{P}^2} = 3ch.
\end{align*}
Thus, the log canonical-canonical divisor $K_{\mathbb{P}^2} + cB_{\mathbb{P}^2} + \Delta_{\mathbb{P}^2}$ is given by
\begin{align*}
	K_{\mathbb{P}^2} + cB_{\mathbb{P}^2} + \Delta_{\mathbb{P}^2} &= (3c-2)h.
\end{align*}
In regards to condition (2), the log canonical divisor will first fail to be ample when $c = \frac{2}{3}$, where this component will get contracted to a point $\mathbb{P}^2 \to \{*\}$. On the other hand, in regards to condition (3), the volume of the log canonical divisor is given by
\begin{align}\label{eqn:EP2_12c23}
	(K_{\mathbb{P}^2} + cB_{\mathbb{P}^2} + \Delta_{\mathbb{P}^2})^2 &= 9c^2 - 12c + 4.
\end{align}

We have now computed the log canonical divisor and its volume for each irreducible component. Adding up the volumes in each irreducible component given in \ref{eqn:Bl1S2A1_23c1}, \ref{eqn:Bl5F0_12c23}, \ref{eqn:Bl5P2_12c23}, \ref{eqn:4F0_12c23}, \ref{eqn:4'F0_12c23}, \ref{eqn:5aF0_12c23}, \ref{eqn:5bF0_12c23}, and \ref{eqn:EP2_12c23}, we obtain
\begin{align*}
	-b^2 + 20bc -2b + 224c^2 - 52c + 3,
\end{align*}
so the volume condition (condition (3)) is satisfied. Thus, we will first hit a wall at $c = \frac{2}{3}$ and we obatin the chamber labeled Sch $\frac{2}{3} < c \leq 1$ in \Cref{E2A1_mult4}. When $c = \frac{2}{3}$, the component $E : \mathbb{P}^2$ will contract contract to a point and the component $1: \mathrm{Bl}_1 \widetilde{S}_{2A_1}$ will contract via the blow-down $\mathrm{Bl}_1 \widetilde{S}_{2A_1} \to \widetilde{S}_{2A_1}$. Now, one can check that this is compatible with the reamining components by \Cref{cor:gluing}, and we obtain the surface of \Cref{E2A1_mult4_12c23}.

\textbf{Step 1: Obtaining \Cref{E2A1_mult4_-b1c12}.} Let $(S, B)$ be the surface of \Cref{E2A1_mult4_12c23} and fix $(b, c)$ so that $c \leq \frac{2}{3}$. From this description of $(S, B)$, there are $20$ irreducible components: One component of Type $1$ isomorphic to $\widetilde{S}_{2A_1}$ (the simultaneous resolution of a cubic surface with two $A_1$ singularities). Two components of Type $2$ isomorphic to $\mathrm{Bl}_5 \mathbb{F}_0$. One component isomorphic of Type $3'$ isomorphic to $\mathrm{Bl}_5 \mathbb{P}^2$. Two components of Type $4$ isomorphic to $\mathbb{F}_0$. Two components of Type $4'$ isomorphic to $\mathbb{F}_0$. Eight components of Type $5a$ isomorphic to $\mathbb{F}_0$. Four compoments of Type $5b$ isomorphic to $\mathbb{F}_0$.

To show $(S, (b, c)B)$ satisfies condition (1), we can use the explicit the description of the lines in each irreducible component in \Cref{E2A1_mult4_12c23} and \Cref{tab:E2A1_mult4_12c23} guaranteeing us that the at most three lines are intersecting at a point (one Eckardt point in the component $1 : \widetilde{S}_{2A_1}$). Therefore, by \Cref{lem:lineslccriteria}, for each irreducible component $S' \subseteq S$ with gluing divisor $\Delta_{S'}$, the pair $(S' : B|_{S'} + \Delta)$ will be log canonical. Now to show conditions (2) and (3), we follow the same technique as in \textbf{Step 0} above by computing the log canonical divisor for each irreducible component.

\textbf{Type $1$: $\widetilde{S}_{2A_1}$ (one copy)} 
From the computations done in \textbf{Type $1$: $\mathrm{Bl}_1 \widetilde{S}_{2A_1}$} for \textbf{Step $0$} above, the log canonical divisor for this component is given by
\begin{align*}
	K_{\widetilde{S}_{2A_1}} + cB_{\widetilde{S}_{2A_1}} + \Delta_{\widetilde{S}_{2A_1}} &= (7c + 3)h - 3ce_1 - 3ce_2 - ce_3 - 3ce_4 - 3ce_5 - (c+3)e_6.
\end{align*}
By \Cref{teo:nefcones}, the log canonical divisor will stay ample for $c \geq \frac{1}{2}$. On the other hand, in regards to condition (3) we have the volume of the log canonical divisor in given by
\begin{align}\label{eqn:S2A1_12c23}
	(K_{\widetilde{S}_{2A_1}} + cB_{\widetilde{S}_{2A_1}} + \Delta_{\widetilde{S}_{2A_1}})^2 &= 11c^2 + 36c.
\end{align}

\textbf{Type $2$: $\mathrm{Bl}_5 \mathbb{F}_0$ (two copies)}
By the computations done for the same component type in \textbf{Step 0} above, the log canonical divisor $K_{\mathrm{Bl}_5 \mathbb{F}_0} + B_{\mathrm{Bl}_5 \mathbb{F}_0} + \Delta_{\mathrm{Bl}_5 \mathbb{F}_0}$ will first fail to be ample when $c = \frac{1}{2}$. Furthermore, when $c = \frac{1}{2}$, this component will contract by a blow-down $\mathrm{Bl}_5 \mathbb{F}_0 \to \mathrm{Bl}_1 \mathbb{F}_0$ contracting exceptional divisors $e_2, \dots, e_5$ to a smooth point. Furthermore, the volume of the log canonical divisor is given by \Cref{eqn:Bl5F0_12c23}.

\textbf{Type $3'$: $\mathrm{Bl}_5 \mathbb{P}^2$ (one copy)}
Similarly, the computations for this component was done in \textbf{Step 0} above. When $c = \frac{1}{2}$, the component will contract by a blow-down $\mathrm{Bl}_5 \mathbb{P}^2 \to \mathrm{Bl}_3 \mathbb{P}^2$ contracting exceptional divisors $e_3, e_5$. Furthermore, the volume of the log canonical divisor is given by \Cref{eqn:Bl5P2_12c23}.

\textbf{Type $4$: $\mathbb{F}_0$ (two copies)}
The computations for this component was done in \textbf{Step 0} above. When $c = \frac{1}{2}$, the component will contract via the projection $\mathbb{F}_0 \to \mathbb{P}^1$ where the ruling $h_2$ contracts. The volume of its log canonical divisor is given by \Cref{eqn:4F0_12c23}.

\textbf{Type $4'$: $\mathbb{F}_0$ (two copies)}
The computations for this component was done in \textbf{Step 0} above and the log canonical divisor will stay ample for $c > -b+1$. Its volume is given by \Cref{eqn:4'F0_12c23}.

\textbf{Type $5a$: $\mathbb{F}_0$ (eight copies)} By \textbf{Step 0} above, when $c = \frac{1}{2}$, these components will get contracted via the projection $\mathbb{F}_0 \to \mathbb{P}^1$ where the ruling $h_2$ contracts. The volume of its log canonical divisor is given by \Cref{eqn:5aF0_12c23}.

\textbf{Type $5b$: $\mathbb{F}_0$ (four copies)}
By \textbf{Step 0} above, when $c = \frac{1}{2}$, these components will get contracted to a point $\mathbb{F}_0 \to \{*\}$. The volume of its log canonical divisor is given by \Cref{eqn:5bF0_12c23}.

Adding up the volumes in each irreducible component given in \ref{eqn:S2A1_12c23}, \ref{eqn:Bl5F0_12c23}, \ref{eqn:Bl5P2_12c23}, \ref{eqn:4F0_12c23}, \ref{eqn:4'F0_12c23}, \ref{eqn:5aF0_12c23}, \ref{eqn:5bF0_12c23}, we obtain
\begin{align*}
	-b^2 + 20bc -2b + 224c^2 - 52c + 3,
\end{align*}
so the volume condition (condition (3)) is satisfied. Furthermore, we will first hit a wall when $c = \frac{1}{2}$, so we obtain the union of chambers labeled $-\frac{b}{2} + \frac{1}{2} < c \leq \frac{2}{3}$ and Sch $\frac{1}{2} < c \leq \frac{2}{3}$ in \Cref{E2A1_mult4} (we do not have $c = -\frac{b}{2} + 1$ as a wall because $\ell$ does not intersect a copy of $E: \mathbb{P}^2$, c.f. \Cref{rem:elleckcase}). At $c = \frac{1}{2}$, the components of Types $4$ and $5a$ contract via the projection $\mathbb{F}_0 \to \mathbb{P}^1$ and the components of Type $5b$ contract to a point. Now, one can check that this is compatible with the remaining components by \Cref{cor:gluing}, and we obtain the surface of \Cref{E2A1_mult4_-b1c12}.

For the remaining steps of this proof, we will omit details as the procedure and computations are very similar as the ones done explicitly in this step.\\
\textbf{Step 2: Obtaining Figures \ref{E2A1_mult4_14c12} and \ref{E2A1_mult4_b10110c14}.} By abuse of notation, refer to the surface of \Cref{E2A1_mult4_-b1c12} as $(S, B)$ and consider $c \leq \frac{1}{2}$. Restricting $K_S + (b, c)B$ to $4' : \mathbb{F}_0$, we obtain
\begin{align*}
K_{\mathbb F_0} + (b, c)B_{\mathbb F_0}= 	(4c - 1)h_1 + (b+ c - 1)h_2,
\end{align*}
and we can hit a wall at either $c = -b + 1$ or $c = \frac{1}{4}$ giving the chamber labeled $\frac{1}{4} < c \leq \frac{1}{2}$, $c > -b+1$ in \Cref{E2A1_mult4}. Sticking to \Cref{conv:descproof}, we first consider the wall $c = -b+1$ where the components of Type $4'$ will get contracted via the projection $\mathbb{F}_0 \to \mathbb{P}^1$ contracting the $h_1$ ruling. This is compatible with the analogous contractions for the remaining components of $(S, B)$ and we obtain the surface of \Cref{E2A1_mult4_14c12}. On the other hand, when $c = \frac{1}{4}$, the components of type $4'$ will get contracted via the projection $\mathbb{F}_0 \to \mathbb{P}^1$ contracting the $h_2$ ruling, resulting in the surface of \Cref{E2A1_mult4_b10110c14}.\\
\\
\textbf{Step 3: Obtaining Figures \ref{E2A1_mult4_b313c16} and \ref{E2A1_mult4_c=b3}.} By abuse of notation, we refer to surface of \Cref{E2A1_mult4_b10110c14} as $(S, B)$ and consider $c \leq \frac{1}{4}$. Restricting $K_S + (b, c)B$ to $2 : \mathrm{Bl}_1 \mathbb{P}^2$ and $3' : \mathrm{Bl}_3 \mathbb{P}^2$, we obtain
\begin{align*}
	\text{Type $2$} &: (8c-1)h - 2c e_1\\
	\text{Type $3'$} &: (b+ 4c -1)h - ce_1 - (b-3c)(e_2 + e_4),
\end{align*}
and we hit a wall at either $c = \frac{1}{6}$ or $c = \frac{b}{3}$ giving the chamber labeled $\frac{b}{10} + \frac{1}{10} < c \leq \frac{1}{4}$, $\frac{1}{6} < c < \frac{b}{3}$ in \Cref{E2A1_mult4}. At $c = \frac{1}{6}$, the class $h - e_1$ in $2 : \mathbb{Bl}_1 \mathbb{P}^2$ gets contracted via the covering map $\mathrm{Bl}_1 \mathbb{P}^2 \cong \mathbb{F}_1 \to \mathbb{P}^1$. In $3' : \mathrm{Bl}_3 \mathbb{P}^2$, the classes $h - e_1 - e_2$ and $h - e_1 - e_4$ gets contracted via the birational morphism $\mathrm{Bl}_3 \mathbb{P}^2 \to \mathrm{Bl}_1 \mathbb{P}^2$, resulting in the surface of \Cref{E2A1_mult4_b313c16}. At $c = \frac{b}{3}$, the exceptional divisors $e_2$ and $e_4$ in $3' : \mathrm{Bl}_3 \mathbb{P}^2$ will get contracted via the blowdown $\mathrm{Bl}_3 \mathbb{P}^2  \cong \mathrm{Bl}_2 \mathbb{F}_1 \to \mathbb{F}_1 \cong \mathrm{Bl}_1 \mathbb{P}^2$, resulting in the surface of \Cref{E2A1_mult4_c=b3}. By abuse of notation, we now refer to this surface as $(S, B)$ and consider $c = \frac{b}{3}$. Restricting $K_S + (b, c)B$ to each component, we find that the ampleness condition is always satisfied. However, when computing the volume of $(S, (b, c)B)$, we obtain
\begin{align*}
	(K_S + (b, c) B)^2 &= b^2 + 8 b c - 2 b + 242 c^2 - 52 c + 3,
\end{align*}
which equals \Cref{eq:constantvol} if and only if $c = \frac{b}{3}$, so we obtain the chamber labeled $c = \frac{b}{3}$ in \Cref{E2A1_mult4}. Once $c \neq \frac{b}{3}$, the volume condition is not satisfied, so we must make appropriate birational modification to recover \Cref{E2A1_mult4_b10110c14} if $c < \frac{b}{3}$ and \Cref{E2A1_mult4_14c12} if $c > \frac{b}{3}$.
\textbf{Step 4: Obtaining \Cref{E2A1_mult4_16c14}} By abuse of notation, let $(S, B)$ denote the surface of \Cref{E2A1_mult4_14c12}. Restricting $K_S + (b,c)B$ to $3' : \mathrm{Bl}_1 \mathbb{P}^2$, we obtain
\begin{align*}
	(b + 4c-1)h - ce_1,
\end{align*}
and we hit a new wall when $c = -\frac{b}{3} + \frac{1}{3}$, where the class $h - e_1$ contracts via the covering $\mathrm{Bl}_1 \mathbb{P}^2 \cong \mathbb{F}_1 \to \mathbb{P}^1$. This is compatible with the analogous contractions for the remaining components of $(S, B)$ and we obtain the surface of \Cref{E2A1_mult4_16c14}. Letting $(S, B)$ denote this surface, we can restrict $K_S + (b, c)B$ to $2 : \mathbb{F}_0$, and get
\begin{align*}
	(b + 5c - 1)h_1 + (6c-1)h_2.
\end{align*}\\
So we hit a wall when $c = \frac{1}{6}$ giving the chamber labeled Sch $\frac{1}{6} < c \leq \frac{1}{4}$ in \Cref{E2A1_mult4}. At $c= \frac{1}{6}$, we obtain the surface \Cref{E2A1_mult4_19c16bneqc} if $b \neq c$ and the surface \Cref{E2A1_mult4_naruki} if $b = c$.\\
\textbf{Step 5: Obtaining Figures \ref{E2A1_mult4_19c16bneqc} and \ref{E2A1_mult4_naruki}} Now going back to \Cref{E2A1_mult4_b313c16}, denote this surface by $(S, B)$. Restricting $K_S + (b, c)B$ to $3' : \mathrm{Bl}_1 \mathbb{P}^2$, we obtain
\begin{align*}
	(13c - 2)h - (10 c - b - 1)e,
\end{align*}
and we hit a wall when $c = - \frac{b}{3} + \frac{1}{3}$ giving the chamber labeled $\frac{b}{10} + \frac{1}{10} < c \leq \frac{1}{6}$, $c > -\frac{b}{3} + \frac{1}{3}$ in \Cref{E2A1_mult4}. At $c = - \frac{b}{3} + \frac{1}{3}$, the class $h - e$ contracts via $\mathrm{Bl}_1 \mathbb{P}^2 \cong \mathbb{F}_1 \to \mathbb{P}^1$, resulting in the surface of \Cref{E2A1_mult4_19c16bneqc}. Now this surface is given by one irreducible component $(\widetilde{S}_{2A_1}, B)$ with log canonical form $K_{\widetilde{S}_{2A_1}} + B_{\widetilde{S}_{2A_1}}$ given by
\begin{align*}
	(27c - 3)h - (9c-1)(e_1 + e_2 + e_4 + e_5) - (10 c - b - 1) e_3 - (9c - 1)e_6.
\end{align*}
So, we hit a wall when $c = b$ giving the chamber labeled $\frac{b}{10} + \frac{1}{10} < c \leq \frac{1}{6}$, $c \leq -\frac{b}{3} + \frac{1}{3}$, $b \neq c$ in \Cref{E2A1_mult4}. At this wall, the two $(-2)$ curves $h - e_1 - e_2 - e_3$ and $h - e_3 - e_4 - e_5$ contract to $A_1$ singularities via the simultaneous blow-down $\widetilde{S}_{2A_1} \to S_{2A_1}$ resulting in the surface of \Cref{E2A1_mult4_naruki}. 
\end{proof}

\begin{rem}\label{rem:elleckcase}
	If $\ell$ intersected a type $E :\mathbb{P}^2$ component as appearing in \textbf{Step 0}, when we hit the wall $c = \frac{2}{3}$, the resulting surface will instead be the surface of \Cref{E2A1_mult4_12c23} modified with this attached copy of $E : \mathbb{P}^2$ surviving. A similar computation will then show the next wall will be at $c = -\frac{b}{2} + 1$ giving the chamber $-\frac{b}{2} + 1 < c \leq \frac{2}{3}$ in \Cref{E2A1_mult4}. Similarly, at $c = -\frac{b}{2} + 1$, this $E: \mathbb{P}^2$ component will contract to a point to obtain the surface of \Cref{E2A1_mult4_12c23}.
\end{rem}

\begin{proof}[Proof of Theorems \ref{teo:main}, \ref{cor:mainmorphisms}, and \ref{cor:mainmoduli}]
	The full wall-and-chamber decomposition of \Cref{teo:main} is obtained by following \Cref{conv:descproof} very similar to the proof  of \Cref{E2A1_mult4} above. By \cite[Thm.~1.1]{ascher_wall_2023} (see \Cref{teo:wallcrossing}), there are birational wall-crossing morphisms from one chamber to another. Using \Cref{lem:isomoduli} many of the coarse moduli spaces in different chambers of \Cref{teo:main} are isomorphic. For the chambers where \Cref{lem:isomoduli} does not apply, we can use the explicit descriptions of the stable surfaces to see that the coarse moduli spaces do not change except at the walls $c = \frac{1}{4}$ and $c = -\frac{b}{3} + \frac{1}{3}$, resulting in three non-isomorphic coarse moduli spaces.
	\begin{align*}
		\overline{Y}_{(1, 1)} = \widetilde{Y}_\ell \to \ddot{Y}_\ell \to \overline{Y}_\ell = \overline{Y}_{(\frac{1}{9} + \epsilon, \frac{1}{9} + \epsilon)}.
	\end{align*}
	where each birational morphism correspond to contractions of components of weighted stable surfaces resulting in identifications of surfaces of different types. Namely, the birational morphism $\widetilde{Y}_\ell \to \ddot{Y}_\ell$, is obtained as follows: For any surface whose type involves $\widetilde{E}_{i A_1}$, $i = 2, 3, 4$, we contract the components arising from the non-flat $A_3^2$ divisors (\Cref{tab:piboundary}) of $\pi : \overline{Y}(E_7)^{\text{lc}} \to \overline{Y}_m^{\text{lc}}$ that \textbf{do not contain} $\ell$. On the other hand, the birational morphism $\ddot{Y}_\ell \to \overline{Y}_\ell$ is obtained by contracting the remaining non-flat $A_3^2$ divisor components that \textbf{do contain} $\ell$.
\end{proof}
	
	\ifAppendix
	\newpage
	\appendix
	\section{Remaining wall-and-chamber decompositions by type}\label{app:remainingchambers}
In this appendix, we will describe the wall-and-chamber decomposition of the remaining surface types (see \Cref{conv:types}) and their degenerations. As in the wall-and-chamber decomposition for Type $\widetilde{E}_{2A_1}$ in \Cref{S:wallandchamberbytype}, we will separate the wall-and-chamber decomposition by cases depending on whether our marked line $\ell$ passes through no, one, or two nodes in the corresponding boundary divisor of $\overline{Y}_\ell$ for each type.

\subsection{Smooth case}
\begin{pro}\label{smooth}
	For smooth $(b, c)$-weighted stable marked cubic surfaces $(S, (b,c)B)$, there are three chambers.
	\begin{center}
		\begin{tikzpicture}[scale=14]
			\draw[very thick,-latex] (1/9, 1/9) -- (1.03, 1/9) 
			node[right]{$b$};
			
			\draw[very thick,-latex] (1/9, 1/9) -- (1/9, 1.03) 
			node[left]{$c$};
			
			\draw[solid] (1/9,1/9) -- (1,1);
			\node at (1/9, 1/9) [below] {\tiny $(\frac{1}{9}, \frac{1}{9})$};
			\node at (1, 1) [left] {$(1, 1)$};
			
			\draw[thick] (1, 1/9) -- (1,1);
			\node at (1, 1/9) [below] {$(1, \frac{1}{9})$};
			\node at (0.9, 0.63) {$-\frac{b}{2}+1 < c \leq \frac{2}{3}$};
			\draw[thick] (2/3, 2/3) -- (1, 2/3) node[right] {\tiny{$c=\frac{2}{3}$}};
			\node[circle, fill, inner sep=1pt] at (2/3, 2/3) {};
			\node at (2/3, 2/3) [left] {\tiny$(\frac{2}{3}, \frac{2}{3})$};
			\node at (8/9, 7/9) {Sch $\frac{2}{3} < c \leq 1$};
			
			\node at (2/3, 1/3) {Sch $\frac{1}{9} < c \leq \frac{2}{3}$};
			
			\draw[solid] (2/3, 2/3) -- (1, 1/2) node[right] {\color{black} \tiny{$c = \frac{1}{2}$}};
			\node[rotate=-26.6] at (5/6, 57/96) {\tiny wall due to $\ell$ containing Eckardt point(s)};

			\draw[dashed] (1/9, 1/9) -- (1, 1/5);
			\node[rotate=5.71] at (5/6, 11/60) [below] {\tiny$c=b/10 + 1/10$};
			\node at (1, 1/5) [right] {\tiny{$c=\frac{1}{5}$}};
			
		\end{tikzpicture}
	\end{center}
\end{pro}

\begin{rem}\label{rem:smoothdesc}
	Since the smooth case is classic, we will not include visual descriptions of the smooth $(b, c)$-weighted stable marked cubic surfaces in each chamber in \Cref{S:descriptions}. For a quick description of the surfaces in Sch $\frac{2}{3} < c \leq 1$, recall that it is a classical fact that a smooth cubic surface can only have $s = 0, 1, 2, 3, 4, 6, 9,10$ or $18$ Eckardt points (\Cref{pro:cubiceck}). So in the chamber Sch $\frac{2}{3} < c \leq 1$, the $(b, c)$-weighted stable marked cubic surfaces arising from smooth cubic surfaces have $s + 1$ irreducible components. One is isomorphic to the blow-up of the smooth cubic surface at its $s$ Eckardt points and $s$ copies of $\mathbb{P}^2$ attached to the $s$ exceptional fibers where the three lines intersecting the Eckardt point intersect in $\mathbb{P}^2$ in general position (see \Cref{teo:eckblowup}).
\end{rem}

\subsection{Type $\widetilde{D}_{A_1}$}
\subsubsection{\textbf{$\ell$ is in the smooth locus}}

\begin{pro}\label{DA1_mult1}
	For type $\widetilde{D}_{A_1}$ $(b,c)$-weighted stable marked cubic surfaces $(S, (b, c)B)$ where $\ell$ is in the smooth locus, there are five chambers. The number of Eckardt points found in each irreducible component of the surfaces in the chamber Sch $\frac{1}{2} < c \leq \frac{2}{3}$ of this type is listed in \Cref{tab:DA1eck}.
	\begin{center}
		\begin{tikzpicture}[scale=14]
			\draw[very thick,-latex] (1/9, 1/9) -- (1.03, 1/9) 
			node[right]{$b$};
			
			\draw[very thick,-latex] (1/9, 1/9) -- (1/9, 1.03) 
			node[left]{$c$};
			
			\draw[thick] (1/9,1/9) -- (1,1);
			\node at (1/9, 1/9) [below] {\tiny $(\frac{1}{9}, \frac{1}{9})$};
			\node at (1, 1) [left] {$(1, 1)$};
			
			\draw[solid] (1, 1/9) -- (1,1);
			\node at (1, 1/9) [below] {$(1, \frac{1}{9})$};
			
			\draw[thick] (2/3, 2/3) -- (1, 2/3) node[right] {\tiny{$c=\frac{2}{3}$}};
			\node[circle, fill, inner sep=1pt] at (2/3, 2/3) {};
			\node at (2/3, 2/3) [left] {\tiny$(\frac{2}{3}, \frac{2}{3})$};
			\node at (8/9, 7/9) {Sch $\frac{2}{3} < c \leq 1$};
			
			\node at (2/3, 1/3) {Sch $\frac{1}{9} < c \leq \frac{1}{2}$};
			
			\draw[solid] (2/3, 2/3) -- (1, 1/2) node[right] {\color{black} \tiny{$c = \frac{1}{2}$}};
			\node[rotate=-26.6] at (5/6, 57/96) {\tiny wall due to $\ell$ containing Eckardt point(s)};
			
			\draw[thick] (1/2, 1/2) -- (1, 1/2);
			\node[circle, fill, inner sep=1pt] at (1/2, 1/2) {};
			\node at (1/2, 1/2) [left] {\tiny $(\frac{1}{2}, \frac{1}{2})$};
			\node at (0.7, 0.56) {Sch $\frac{1}{2} < c \leq \frac{2}{3}$};
			\node at (0.9, 0.63) {$-\frac{b}{2}+1 < c \leq \frac{2}{3}$};
			
			\draw[thick] (1/6, 1/6) -- (2/3, 1/6);
			\node[circle, fill, inner sep=1pt] at (1/6, 1/6) {};
			\node at (7/36, 7/36) [left] {\tiny$(\frac{1}{6}, \frac{1}{6}$)};
			\node at (2/3, 1/6) [below] {\tiny{$(\frac{2}{3}, \frac{1}{6})$}};
			\draw[->] (7/24, 1/11) -- (1/4, 10/72);
			\node[circle, fill, inner sep=1.5pt] at (1/4, 10/72) [above] {};
			\node[align=center] at (7/24, 1/20) {Sch $\frac{1}{9} < c \leq \frac{1}{6}$};
			
			\draw[dashed] (1/9, 1/9) -- (1, 1/5);
			\node[rotate=5.71] at (5/6, 11/60) [below] {\tiny$c=b/10 + 1/10$};
			\node at (1, 1/5) [right] {\tiny{$c=\frac{1}{5}$}};
			
		\end{tikzpicture}
	\end{center}
\end{pro}

\newpage
\subsubsection{\textbf{$\ell$ passes through exactly one node}}

\begin{pro}\label{DA1_mult2}
	For type $\widetilde{D}_{A_1}$ $(b,c)$-weighted stable marked cubic surfaces $(S, (b,c)B)$ where $\ell$ passes through exactly one node in the corresponding $\overline{Y}_{\ell}$ boundary, there are eight chambers. The line $c = \frac{b}{4}$ is its own chamber.  The number of Eckardt points found in each irreducible component of the surfaces in the chamber Sch $\frac{1}{2} < c \leq \frac{2}{3}$ of this type is listed in \Cref{tab:DA1eck}.
	\begin{center}
		\begin{tikzpicture}[scale=14]
			\draw[very thick,-latex] (1/9, 1/9) -- (1.03, 1/9) 
			node[right]{$b$};
			
			\draw[very thick,-latex] (1/9, 1/9) -- (1/9, 1.03) 
			node[left]{$c$};
			
			\draw[solid] (1/9,1/9) -- (1,1);
			\node at (1/9, 1/9) [below] {\tiny $(\frac{1}{9}, \frac{1}{9})$};
			\node at (1, 1) [left] {$(1, 1)$};
			
			\draw[thick] (1, 1/9) -- (1,1);
			\node at (1, 1/9) [below] {$(1, \frac{1}{9})$};
			
			\draw[thick] (2/3, 2/3) -- (1, 2/3) node[right] {\tiny{$c=\frac{2}{3}$}};
			\node[circle, fill, inner sep=1pt] at (2/3, 2/3) {};
			\node at (2/3, 2/3) [left] {\tiny$(\frac{2}{3}, \frac{2}{3})$};
			\node at (8/9, 7/9) {Sch $\frac{2}{3} < c \leq 1$};
			
			\draw[thick] (1/2, 1/2) -- (1, 1/2);
			\node[circle, fill, inner sep=1pt] at (1/2, 1/2) {};
			\node at (1/2, 1/2) [left] {\tiny $(\frac{1}{2}, \frac{1}{2})$};
			\node at (7.5/9, 0.57) {Sch $\frac{1}{2} < c \leq \frac{2}{3}$};
			
			\draw[thick] (1/2, 1/2) -- (4/5, 1/5);
			\node[rotate=-45] at (7/12, 5/12) [above] {\tiny $c = -b+1$};
			\node[align=center] at (5/6, 0.42) {\(\frac{1}{5} < c \leq \frac{1}{2}\) \\ \(c > -b+1 \)};
			\node[align=center] at (1/2, 1/3) {Sch $\frac{1}{6} < c \leq \frac{1}{2}$};
			
			\draw[thick] (4/5, 1/5) -- (1, 1/5);
			\draw[->] (2.5/3, 1/11) -- (0.8, 0.18);
			\node[circle, fill, inner sep=1.5pt] at (0.8, 0.19) {};
			\node[align=center] at (2.5/3, 1/20) {\(\frac{b}{10} + \frac{1}{10} < c \leq \frac{1}{5}\) \\ \(c < \frac{b}{4}\)};
			
			\draw[very thick] (2/3, 1/6) -- (4/5, 1/5);
			\node at (2/3, 1/6) [below] {\tiny{$(\frac{2}{3}, \frac{1}{6})$}};
			\node[rotate=14.04] at (22/30, 21/120) [above] {\tiny $c=\frac{b}{4}$};
			
			\draw[thick] (1/6, 1/6) -- (2/3, 1/6);
			\node[circle, fill, inner sep=1pt] at (1/6, 1/6) {};
			\node at (7/36, 7/36) [left] {\tiny$(\frac{1}{6}, \frac{1}{6}$)};
			\node at (2/3, 1/6) [below] {\tiny{$(\frac{2}{3}, \frac{1}{6})$}};
			\draw[->] (7/24, 1/11) -- (1/4, 10/72);
			\node[circle, fill, inner sep=1.5pt] at (1/4, 10/72) [above] {};
			\node[align=center] at (7/24, 1/20) { \(\frac{b}{10} + \frac{1}{10} < c \leq \frac{1}{6}\) \\ \(b \neq c\)};
			
			\draw[very thick] (1/9, 1/9) -- (1/6, 1/6);
			\draw[->] (1/11, 1/6) -- (5/36, 5/36);
			\node at (1/10, 1/6) [left] {Sch $\frac{1}{9} < c \leq \frac{1}{6}$};
			
			\draw[dashed] (1/9, 1/9) -- (1, 1/5);
			\node[rotate=5.71] at (5/6, 11/60) [below] {\tiny$c=b/10 + 1/10$};
			\node at (1, 1/5) [right] {\tiny{$c=\frac{1}{5}$}};
			
		\end{tikzpicture}
	\end{center}
\end{pro}

\begin{table}[H]
	\centering
	\begin{tabular}{| c | c | c | c |}
		\hline
		Label & Surface & \# & Eckardt points \\
		\hline
		\hline
		1 & $\widetilde S_{A_1}$ & 1 & 0, 1, 2, 3, 4, or 6 \\
		2 & $\mathrm{Bl}_6\mathbb{F}_0$ & 1 & 0 \\
		3 & $\mathbb{F}_0$ & 6 & 0 \\
		\hline
	\end{tabular}
	\caption{This table is taken from \cite[Table~2]{schock_moduli_2024} The first two columns are the numbered types of irreducible components of the surfaces in the chamber Sch $1/2 < c \leq 2/3$ of type $\widetilde{D}_{A_1}$ pictured in \cite[Fig.~4]{schock_moduli_2024}.  The third column tells us the number of each corresponding component found in the surface. Finally, the last row gives the possible numbers of Eckardt points on each component. For the component of type 1, $\widetilde{S}_{A_1}$ refers to the minimal resolution of a cubic surface with one $A_1$ singularity. For the component of type 2, $\mathrm{Bl}_6\mathbb{F}_0$ refers to the blowup of $\mathbb{F}_0 \cong \mathbb{P}^1 \times \mathbb{P}^1$ at 6 points on the diagonal.}
	\label{tab:DA1eck}
\end{table}

\newpage
\subsection{Remaining wall-and-chamber decompositions for type $\widetilde{E}_{2A_1}$}
\subsubsection{\textbf{$\ell$ is in the smooth locus}}

\begin{pro}\label{E2A1_mult1}
	For type $\widetilde{E}_{2A_1}$ $(b,c)$-weighted stable marked cubic surfaces $(S, (b, c)B)$ where $\ell$ is in the smooth locus, there are six chambers. Additionally, crossing the wall $c = \frac{1}{4}$ introduces type $\widetilde{D}_{A_1} \cap \widetilde{E}_{2A_1}$ surfaces as degerenations of type $\widetilde{E}_{2A_1}$.
	\begin{center}
			\begin{tikzpicture}[scale=14]
					\draw[very thick,-latex] (1/9, 1/9) -- (1.03, 1/9) 
					node[right]{$b$};
					
					\draw[very thick,-latex] (1/9, 1/9) -- (1/9, 1.03) 
					node[left]{$c$};
					
					\draw[solid] (1/9,1/9) -- (1,1);
					\node at (1/9, 1/9) [below] {\tiny $(\frac{1}{9}, \frac{1}{9})$};
					\node at (1, 1) [left] {$(1, 1)$};
					
					\draw[thick] (1, 1/9) -- (1,1);
					\node at (1, 1/9) [below] {$(1, \frac{1}{9})$};
					
					\draw[thick] (2/3, 2/3) -- (1, 2/3) node[right] {\tiny{$c=\frac{2}{3}$}};
					\node[circle, fill, inner sep=1pt] at (2/3, 2/3) {};
					\node at (2/3, 2/3) [left] {\tiny$(\frac{2}{3}, \frac{2}{3})$};
					\node at (8/9, 7/9) {Sch $\frac{2}{3} < c \leq 1$};
					
					\draw[solid] (2/3, 2/3) -- (1, 1/2) node[right] {\color{black} \tiny{$c = \frac{1}{2}$}};
					\node[rotate=-26.6] at (5/6, 57/96) {\tiny wall due to $\ell$ containing Eckardt point(s)};
					
					\draw[thick] (1/2, 1/2) -- (1, 1/2);
					\node[circle, fill, inner sep=1pt] at (1/2, 1/2) {};
					\node at (1/2, 1/2) [left] {\tiny $(\frac{1}{2}, \frac{1}{2})$};
					\node at (0.7, 0.56) {Sch $\frac{1}{2} < c \leq \frac{2}{3}$};
					\node at (0.9, 0.63) {$-\frac{b}{2}+1 < c \leq \frac{2}{3}$};
					\node at (2/3, 1/3) {Sch $\frac{1}{4} < c \leq \frac{1}{2}$};
					
					\draw[thick] (1/4, 1/4) -- (1,1/4);
					\node[circle, fill, inner sep=1pt] at (1/4, 1/4) {};
					\node at (1, 1/4) [right] {\tiny$c=\frac{1}{4}$};
					\node at (1/4, 1/4) [left] {\tiny$(\frac{1}{4}, \frac{1}{4})$};
					\node at (2/3, 5/24) {Sch $\frac{1}{6} < c \leq \frac{1}{4}$};
					
					\draw[thick] (1/6, 1/6) -- (2/3, 1/6);
					\node[circle, fill, inner sep=1pt] at (1/6, 1/6) {};
					\node at (7/36, 7/36) [left] {\tiny$(\frac{1}{6}, \frac{1}{6}$)};
					\node at (2/3, 1/6) [below] {\tiny{$(\frac{2}{3}, \frac{1}{6})$}};
					\draw[->] (7/24, 1/11) -- (1/4, 10/72);
					\node[circle, fill, inner sep=1.5pt] at (1/4, 10/72) [above] {};
					\node[align=center] at (7/24, 1/20) {Sch $\frac{1}{9} < c \leq \frac{1}{6}$};
					
					\draw[dashed] (1/9, 1/9) -- (1, 1/5);
					\node[rotate=5.71] at (5/6, 11/60) [below] {\tiny$c=b/10 + 1/10$};
					\node at (1, 1/5) [right] {\tiny{$c=\frac{1}{5}$}};
					
				\end{tikzpicture}
		\end{center}
\end{pro}

\newpage
\subsubsection{\textbf{$\ell$ passes through exactly one node}}

\begin{pro}\label{E2A1_mult2}
	For type $\widetilde{E}_{2A_1}$ $(b,c)$-weighted stable marked cubic surfaces $(S, (b,c)B)$ where $\ell$ passes through exactly one node in the corresponding $\overline{Y}_{\ell}$ boundary, there are 10 chambers. The line $c = \frac{b}{4}$ is its own chamber. Additionally, crossing the wall $c = \frac{1}{4}$ introduces type $\widetilde{D}_{A_1} \cap \widetilde{E}_{2A_1}$ surfaces as degenerations of type $\widetilde{E}_{2A_1}$.
	\begin{center}
			\begin{tikzpicture}[scale=14]
					\draw[very thick,-latex] (1/9, 1/9) -- (1.03, 1/9) 
					node[right]{$b$};
					
					\draw[very thick,-latex] (1/9, 1/9) -- (1/9, 1.03) 
					node[left]{$c$};
					
					\draw[solid] (1/9,1/9) -- (1,1);
					\node at (1/9, 1/9) [below] {\tiny $(\frac{1}{9}, \frac{1}{9})$};
					\node at (1, 1) [left] {$(1, 1)$};
					
					\draw[thick] (1, 1/9) -- (1,1);
					\node at (1, 1/9) [below] {$(1, \frac{1}{9})$};
					
					\draw[thick] (2/3, 2/3) -- (1, 2/3) node[right] {\tiny{$c=\frac{2}{3}$}};
					\node[circle, fill, inner sep=1pt] at (2/3, 2/3) {};
					\node at (2/3, 2/3) [left] {\tiny$(\frac{2}{3}, \frac{2}{3})$};
					\node at (8/9, 7/9) {Sch $\frac{2}{3} < c \leq 1$};
					
					\draw[thick] (1/2, 1/2) -- (1, 1/2);
					\node[circle, fill, inner sep=1pt] at (1/2, 1/2) {};
					\node at (1/2, 1/2) [left] {\tiny $(\frac{1}{2}, \frac{1}{2})$};
					\node at (7.5/9, 0.57) {Sch $\frac{1}{2} < c \leq \frac{2}{3}$};
					
					\draw[thick] (1/2, 1/2) -- (4/5, 1/5);
					\node[rotate=-45] at (7/12, 5/12) [above] {\tiny $c = -b+1$};
					\node[align=center] at (5/6, 0.42) {\(\frac{1}{4} < c \leq \frac{1}{2}\) \\ \(c > -b+1 \)};
					\node[align=center] at (1/2, 1/3) {Sch $\frac{1}{4} < c \leq \frac{1}{2}$};
					
					\draw[thick] (1/4, 1/4) -- (1,1/4);
					\node[circle, fill, inner sep=1pt] at (1/4, 1/4) {};
					\node at (1, 1/4) [right] {\tiny$c=\frac{1}{4}$};
					\node at (1/4, 1/4) [left] {\tiny$(\frac{1}{4}, \frac{1}{4})$};
					\node at (2.5/6, 5/24) {Sch $\frac{1}{6} < c \leq \frac{1}{4}$};
					\draw[->] (1.1, 1/6) -- (5.2/6, 10.5/48);
					\node[circle, fill, inner sep=1.5pt] at (5.15/6, 10.5/48) [above] {};
					\node[align=left] at (1.16, 1/6) {\(\frac{1}{5} < c \leq \frac{1}{4}\) \\ \(c > -b+1\)};
					
					\draw[thick] (4/5, 1/5) -- (1, 1/5);
					\draw[->] (2.5/3, 1/11) -- (0.8, 0.18);
					\node[circle, fill, inner sep=1.5pt] at (0.8, 0.19) {};
					\node[align=center] at (2.5/3, 1/20) {\(\frac{b}{10} + \frac{1}{10} < c \leq \frac{1}{5}\) \\ \(c < \frac{b}{4}\)};
					
					\draw[very thick] (2/3, 1/6) -- (4/5, 1/5);
					\node at (2/3, 1/6) [below] {\tiny{$(\frac{2}{3}, \frac{1}{6})$}};
					\node[rotate=14.04] at (22/30, 21/120) [above] {\tiny $c=\frac{b}{4}$};
					
					\draw[thick] (1/6, 1/6) -- (2/3, 1/6);
					\node[circle, fill, inner sep=1pt] at (1/6, 1/6) {};
					\node at (7/36, 7/36) [left] {\tiny$(\frac{1}{6}, \frac{1}{6}$)};
					\node at (2/3, 1/6) [below] {\tiny{$(\frac{2}{3}, \frac{1}{6})$}};
					\draw[->] (7/24, 1/11) -- (1/4, 10/72);
					\node[circle, fill, inner sep=1.5pt] at (1/4, 10/72) [above] {};
					\node[align=center] at (7/24, 1/20) { \(\frac{b}{10} + \frac{1}{10} < c \leq \frac{1}{6}\) \\ \(b \neq c\)};
					
					\draw[very thick] (1/9, 1/9) -- (1/6, 1/6);
					\draw[->] (1/11, 1/6) -- (5/36, 5/36);
					\node at (1/10, 1/6) [left] {Sch $\frac{1}{9} < c \leq \frac{1}{6}$};
					
					\draw[dashed] (1/9, 1/9) -- (1, 1/5);
					\node[rotate=5.71] at (5/6, 11/60) [below] {\tiny$c=b/10 + 1/10$};
					\node at (1, 1/5) [right] {\tiny{$c=\frac{1}{5}$}};
					
				\end{tikzpicture}
		\end{center}
\end{pro}

\subsubsection{Type $\widetilde{E}_{2A_1}$: $\ell$ passes through exactly two nodes}
\begin{pro}\label{app:E2A1_mult4degen}
	The following is the wall and chamber decomposition for type $\widetilde{D}_{A_1} \cap \widetilde{E}_{2A_1}$ $(b,c)$-weighted stable marked cubic surfaces $(S, (b,c)B)$ where $\ell$ passes through exactly two nodes in the corresponding $\overline{Y}_{\ell}$ boundary. There are 14 chambers. Along with the chamber $c = \frac{b}{3}$, the lines $c = \frac{1}{4}$ and $c = \frac{b}{4}$ are their own chambers.
\begin{center}
	\begin{tikzpicture}[scale=14]
		\draw[very thick,-latex] (1/9, 1/9) -- (1.03, 1/9) 
		node[right]{$b$};
		
		\draw[very thick,-latex] (1/9, 1/9) -- (1/9, 1.03) 
		node[left]{$c$};
		
		\draw[solid] (1/9,1/9) -- (1,1);
		\node at (1/9, 1/9) [below] {\tiny$(\frac{1}{9}, \frac{1}{9})$};
		\node at (1, 1) [left] {$(1, 1)$};
		
		\draw[thick] (1, 1/9) -- (1,1);
		\node at (1, 1/9) [below] {\tiny$(1, \frac{1}{9})$};
		
		\draw[thick] (2/3, 2/3) -- (1, 2/3);
		\node at (1, 2/3) [right] {\tiny $c=2/3$};
		\node at (2/3, 2/3) [left] {\tiny$(\frac{2}{3}, \frac{2}{3})$};
		\node at (8/9, 7/9) {Sch $\frac{2}{3} < c \leq 1$};
		
		\draw[thick] (1/2, 1/2) -- (1, 1/2);
		\node at (0.8, 0.58) {Sch $\frac{1}{2} < c \leq \frac{2}{3}$};
		
		\draw[thick] (1/2, 1/2) -- (4/5, 1/5);
		\node at (1/2, 1/2) [left] {\tiny$(\frac{1}{2}, \frac{1}{2})$};
		\node at (1, 1/2) [right] {\tiny$c=1/2$};
		\node[rotate=-45] at (2/3, 29/90) [above] {\tiny $c = -b+1$};
		\node[align=center] at (5/6, 0.42) {\(\frac{1}{4} < c \leq \frac{1}{2}\) \\ \(c > -b+1 \)};
		%
		\draw[thick] (1/4, 1/4) -- (7/13, 2/13);
		\node at (1/4, 1/4) [left] {\tiny$(\frac{1}{4}, \frac{1}{4})$};
		\node at (7/13, 2/13) [below] {\tiny{$(\frac{7}{13}, \frac{2}{13})$}};
		\node[rotate=-18.43] at (3/8, 24/120) [above] {\tiny$c=-\frac{b}{3} + \frac{1}{3}$};
		\node[align=center] at (1/2, 1/3) {Sch $\frac{1}{4} < c \leq \frac{1}{2}$};
		\node at (1.75/6, 4.55/24) {Sch $\frac{1}{6} < c \leq \frac{1}{4}$};
		
		%
		\draw[thick] (1/6, 1/6) -- (2/3, 1/6);
		\node at (7/36, 7/36) [left] {\tiny$(\frac{1}{6}, \frac{1}{6}$)};
		\draw[->] (7/24, 1/11) -- (1/4, 10/72);
		\node[circle, fill, inner sep=1pt] at (1/4, 10/72) [above] {};
		\node[align=center] at (7/24, 1/20) {\(\frac{b}{10} + \frac{1}{10} < c \leq \frac{1}{6}\) \\ \(c \leq - \frac{b}{3} + \frac{1}{3}, b \neq c\)};
		%
		\draw[very thick] (3/4, 1/4) -- (1, 1/4);
		\node at (1, 1/4) [right] {\tiny{$c=1/4$}};
		%
		\draw[thick] (4/5, 1/5) -- (1, 1/5);
		\node at (1, 1/5) [right] {\tiny{$c=1/5$}};
		%
		\draw[very thick] (1/2, 1/6) -- (3/4, 1/4);
		\node[rotate=18.43] at (2/3, 17/81) [above] {\tiny$c=\frac{b}{3}$};
		%
		\draw[very thick] (2/3, 1/6) -- (4/5, 1/5);
		\node at (2/3, 1/6) [below] {\tiny{$(\frac{2}{3}, \frac{1}{6})$}};
		\node[rotate=14.04] at (22/30, 21/120) [above] {\tiny $c=\frac{b}{4}$};
		%
		\draw[very thick] (1/9, 1/9) -- (1/6, 1/6);
		\draw[->] (1/11, 1/6) -- (5/36, 5/36);
		\node at (1/10, 1/6) [left] {Sch $\frac{1}{9} < c \leq \frac{1}{6}$};
		
		\draw[->] (1.1, 1/6) -- (5.2/6, 10.5/48);
		\node[circle, fill, inner sep=1.5pt] at (5.15/6, 10.5/48) [above] {};
		\node[align=left] at (1.16, 1/6) {\(\frac{1}{5} < c < \frac{1}{4}\) \\ \(c > -b+1\)};
		
		\draw[->] (1.05, 1/3) -- (8.5/13, 1/5);
		\node[circle, fill, inner sep=1.5pt] at (8.5/13, 1/5) [left] {};
		\node[align=center] at (1.15, 1/3) {\(\frac{1}{6} < c \leq -b+1\) \\ \(\frac{b}{4} < c < \frac{b}{3}\)};
		
		\draw[->] (7/13, 1/11) -- (7.3/13, 9.4/60);
		\node[circle, fill, inner sep=1.5pt] at (7.3/13, 9.4/60) [above] {};
		\node[align=center] at (7/13, 1/20) {\(\frac{b}{10} + \frac{1}{10} < c \leq \frac{1}{6}\) \\ \(c > -\frac{b}{3} + \frac{1}{3}\)};
		
		\draw[->] (10/13, 1/11) -- (10.5/13, 1.1/6);
		\node[circle, fill, inner sep=1.5pt] at (10.5/13, 1.1/6) [above] {};
		\node[align=center] at (10/13, 1/20) {\(\frac{b}{10} + \frac{1}{10} < c \leq \frac{1}{5}\) \\ \(c < \frac{b}{4}\)};
		
		\draw[dashed] (1/9, 1/9) -- (1, 1/5);
		\node[rotate=5.71] at (5/6, 11/60) [below] {\tiny$c=b/10 + 1/10$};
	\end{tikzpicture}
\end{center}
\end{pro}

\newpage

\newpage
\subsection{Type $\widetilde{E}_{3A_1}$}
	\subsubsection{\textbf{$\ell$ is in the smooth locus}}
\begin{pro}
	For type $\widetilde{E}_{3A_1}$ $(b,c)$-weighted stable marked cubic surfaces $(S, (b,c)B)$ where $\ell$ is in the smooth locus, the chambers are the same as in Proposition \ref{E2A1_mult1}. Additionally, crossing the wall $c = \frac{1}{4}$ introduces type 
	$$\widetilde{D}_{A_1} \cap \widetilde{E}_{3A_1}, \hspace{.25cm} \widetilde{E}_{2A_1} \cap \widetilde{E}_{3A_1}, \hspace{.25cm} \widetilde{D}_{A_1} \cap \widetilde{E}_{2A_1} \cap \widetilde{E}_{3A_1}$$
	surfaces as degenerations of type $\widetilde{E}_{3A_1}$ surfaces.
\end{pro}

\subsubsection{\textbf{$\ell$ passes through exactly one node}}
\begin{pro}
	For type $\widetilde{E}_{3A_1}$ $(b,c)$-weighted stable marked cubic surfaces $(S, (b,c)B)$ where $\ell$ passes through exactly one node in the corresponding $\overline{Y}_{\ell}$ boundary, the chambers are the same as in Proposition \ref{E2A1_mult2}. Additionally, crossing the wall $c = \frac{1}{4}$ introduces type 
	$$\widetilde{D}_{A_1} \cap \widetilde{E}_{3A_1}, \hspace{.25cm} \widetilde{E}_{2A_1} \cap \widetilde{E}_{3A_1}, \hspace{.25cm} \widetilde{D}_{A_1} \cap \widetilde{E}_{2A_1} \cap \widetilde{E}_{3A_1}$$
	surfaces as degenerations of type $\widetilde{E}_{3A_1}$ surfaces.
\end{pro}

\subsubsection{\textbf{$\ell$ passes through exactly two nodes}}
\begin{pro}\label{E3A1_mult4}
	For type $\widetilde{E}_{3A_1}$ $(b,c)$-weighted stable marked cubic surfaces $(S, (b,c)B)$ where $\ell$ passes through exactly two nodes in the corresponding $\overline{Y}_{\ell}$ boundary, there are 12 chambers. The line $c = \frac{b}{3}$ is its own chamber. Additionally, crossing the wall $c = -\frac{b}{3} + \frac{1}{3}$ introduces type 
	$$\widetilde{D}_{A_1} \cap \widetilde{E}_{3A_1}, \hspace{.25cm} \widetilde{E}_{2A_1} \cap \widetilde{E}_{3A_1}, \hspace{.25cm} \widetilde{D}_{A_1} \cap \widetilde{E}_{2A_1} \cap \widetilde{E}_{3A_1}$$
	surfaces as degenerations of type $\widetilde{E}_{3A_1}$. In this case, if $\ell$ is in the component that is the degeneration of $\operatorname{Bl}_5 \mathbb{P}^2$, the wall and chamber decomposition has 15 chambers and explained more below. Otherwise, the wall and chamber decomposition is the same as Type $\widetilde{E}_{3A_1}$. 
	\begin{center}
		\begin{tikzpicture}[scale=14]
			\draw[very thick,-latex] (1/9, 1/9) -- (1.03, 1/9) 
			node[right]{$b$};
			
			\draw[very thick,-latex] (1/9, 1/9) -- (1/9, 1.03) 
			node[left]{$c$};
			
			\draw[solid] (1/9,1/9) -- (1,1);
			\node at (1/9, 1/9) [below] {\tiny $(\frac{1}{9}, \frac{1}{9})$};
			\node at (1, 1) [left] {$(1, 1)$};
			
			\draw[thick] (1, 1/9) -- (1,1);
			\node at (1, 1/9) [below] {$(1, \frac{1}{9})$};
			
			\draw[thick] (2/3, 2/3) -- (1, 2/3) node[right] {\tiny{$c=\frac{2}{3}$}};
			\node[circle, fill, inner sep=1pt] at (2/3, 2/3) {};
			\node at (2/3, 2/3) [left] {\tiny$(\frac{2}{3}, \frac{2}{3})$};
			\node at (8/9, 7/9) {Sch $\frac{2}{3} < c \leq 1$};
			\node at (0.9, 0.63) {$-\frac{b}{2}+1 < c \leq \frac{2}{3}$};
			
			\draw[solid] (2/3, 2/3) -- (1, 1/2) node[right] {\color{black} \tiny{$c = \frac{1}{2}$}};
			\node[rotate=-26.6] at (5/6, 57/96) {\tiny wall due to $\ell$ containing Eckardt point(s)};
			
			\draw[thick] (1/2, 1/2) -- (1, 1/2);
			\node[circle, fill, inner sep=1pt] at (1/2, 1/2) {};
			\node at (1/2, 1/2) [left] {\tiny $(\frac{1}{2}, \frac{1}{2})$};
			\node at (0.7, 0.56) {Sch $\frac{1}{2} < c \leq \frac{2}{3}$};
			
			\draw[thick] (1/2, 1/2) -- (3/4, 1/4);
			\node[rotate=-45] at (7/12, 5/12) [above] {\tiny $c = -b+1$};
			\node[align=center] at (5/6, 0.42) {\(\frac{1}{4} < c \leq \frac{1}{2}\) \\ \(c > -b+1 \)};
			\node[align=center] at (1/2, 1/3) {Sch $\frac{1}{4} < c \leq \frac{1}{2}$};
			
			\draw[thick] (1/4, 1/4) -- (7/13, 2/13);
			\node[circle, fill, inner sep=1pt] at (1/4, 1/4) {};
			\node[rotate=-18.43] at (3/8, 24/120) [above] {\tiny$c=-\frac{b}{3} + \frac{1}{3}$};
			\node at (7/13, 2/13) [below] {\tiny{$(\frac{7}{13}, \frac{2}{13})$}};
			\node at (1/4, 1/4) [left] {\tiny$(\frac{1}{4}, \frac{1}{4})$};
			\node at (1.75/6, 4.55/24) {Sch $\frac{1}{6} < c \leq \frac{1}{4}$};
			
			\draw[thick] (1/4, 1/4) -- (1, 1/4) node[right] {\tiny $c = \frac{1}{4}$};
			\node[align=center] at (6.5/13, 2.5/12) {\tiny \(-\frac{b}{3} + \frac{1}{3} < c \leq \frac{1}{4}\) \\ \tiny \(c > \frac{b}{3}\)};
			\node[align=left] at (5.2/6, 10.5/48) {\tiny \(\frac{b}{10} + \frac{1}{10} < c \leq \frac{1}{4}\) \\ \tiny \(\frac{1}{6} < c < \frac{b}{3}\)};
			
			\draw[very thick] (1/2, 1/6) -- (3/4, 1/4);
			\node[rotate=18.43] at (2/3, 17/81) [above] {\tiny$c=\frac{b}{3}$};
			
			\draw[thick] (1/6, 1/6) -- (2/3, 1/6);
			\node[circle, fill, inner sep=1pt] at (1/6, 1/6) {};
			\node at (7/36, 7/36) [left] {\tiny$(\frac{1}{6}, \frac{1}{6}$)};
			\node at (2/3, 1/6) [below] {\tiny{$(\frac{2}{3}, \frac{1}{6})$}};
			\draw[->] (7/24, 1/11) -- (1/4, 10/72);
			\node[circle, fill, inner sep=1.5pt] at (1/4, 10/72) [above] {};
			\node[align=center] at (7/24, 1/20) {\(\frac{b}{10} + \frac{1}{10} < c \leq \frac{1}{6}\) \\ \(c \leq - \frac{b}{3} + \frac{1}{3}, b \neq c\)};
			\draw[->] (7/13, 1/11) -- (7.3/13, 9.4/60);
			\node[circle, fill, inner sep=1.5pt] at (7.3/13, 9.4/60) [above] {};
			\node[align=center] at (7/13, 1/20) {\(\frac{b}{10} + \frac{1}{10} < c \leq \frac{1}{6}\) \\ \(c > -\frac{b}{3} + \frac{1}{3}\)};
			
			\draw[very thick] (1/9, 1/9) -- (1/6, 1/6);
			\draw[->] (1/11, 1/6) -- (5/36, 5/36);
			\node at (1/10, 1/6) [left] {Sch $\frac{1}{9} < c \leq \frac{1}{6}$};
			
			\draw[dashed] (1/9, 1/9) -- (1, 1/5);
			\node[rotate=5.71] at (5/6, 11/60) [below] {\tiny$c=b/10 + 1/10$};
			\node at (1, 1/5) [right] {\tiny{$c=\frac{1}{5}$}};
			
		\end{tikzpicture}
	\end{center}
	The following is the wall and chamber decomposition for type 
	$$\widetilde{D}_{A_1} \cap \widetilde{E}_{3A_1}, \hspace{.25cm} \widetilde{E}_{2A_1} \cap \widetilde{E}_{3A_1}, \hspace{.25cm} \widetilde{D}_{A_1} \cap \widetilde{E}_{2A_1} \cap \widetilde{E}_{3A_1}$$
	$(b,c)$-weighted stable marked cubic surfaces $(S, (b,c)B)$ where $\ell$ passes through exactly two nodes in the corresponding $\overline{Y}_{\ell}$ boundary. There are 15 chambers. Along with the chamber $c = \frac{b}{3}$, the lines $c = \frac{1}{4}$ and $c = \frac{b}{4}$ are their own chambers.
	\begin{center}
		\begin{tikzpicture}[scale=14]
			\draw[very thick,-latex] (1/9, 1/9) -- (1.03, 1/9) 
			node[right]{$b$};
			
			\draw[very thick,-latex] (1/9, 1/9) -- (1/9, 1.03) 
			node[left]{$c$};
			
			\draw[solid] (1/9,1/9) -- (1,1);
			\node at (1/9, 1/9) [below] {\tiny$(\frac{1}{9}, \frac{1}{9})$};
			\node at (1, 1) [left] {$(1, 1)$};
			
			\draw[thick] (1, 1/9) -- (1,1);
			\node at (1, 1/9) [below] {\tiny$(1, \frac{1}{9})$};
			
			\draw[thick] (2/3, 2/3) -- (1, 2/3);
			\node at (1, 2/3) [right] {\tiny $c=2/3$};
			\node at (2/3, 2/3) [left] {\tiny$(\frac{2}{3}, \frac{2}{3})$};
			\node at (8/9, 7/9) {Sch $\frac{2}{3} < c \leq 1$};
			
			\draw[thick] (1/2, 1/2) -- (1, 1/2);
			\node at (0.8, 0.58) {Sch $\frac{1}{2} < c \leq \frac{2}{3}$};
			
			\draw[thick] (1/2, 1/2) -- (4/5, 1/5);
			\node at (1/2, 1/2) [left] {\tiny$(\frac{1}{2}, \frac{1}{2})$};
			\node at (1, 1/2) [right] {\tiny$c=1/2$};
			\node[rotate=-45] at (2/3, 29/90) [above] {\tiny $c = -b+1$};
			\node[align=center] at (5/6, 0.42) {\(\frac{1}{4} < c \leq \frac{1}{2}\) \\ \(c > -b+1 \)};
			%
			\draw[thick] (1/4, 1/4) -- (7/13, 2/13);
			\node at (1/4, 1/4) [left] {\tiny$(\frac{1}{4}, \frac{1}{4})$};
			\node at (7/13, 2/13) [below] {\tiny{$(\frac{7}{13}, \frac{2}{13})$}};
			\node[rotate=-18.43] at (3/8, 24/120) [above] {\tiny$c=-\frac{b}{3} + \frac{1}{3}$};
			\node[align=center] at (1/2, 1/3) {Sch $\frac{1}{6} < c \leq \frac{1}{2}$};
			\node at (1.75/6, 4.55/24) {Sch $\frac{1}{6} < c \leq \frac{1}{4}$};
			
			\draw[thick] (1/6, 1/6) -- (2/3, 1/6);
			\node at (7/36, 7/36) [left] {\tiny$(\frac{1}{6}, \frac{1}{6}$)};
			\draw[->] (7/24, 1/11) -- (1/4, 10/72);
			\node[circle, fill, inner sep=1pt] at (1/4, 10/72) [above] {};
			\node[align=center] at (7/24, 1/20) {\(\frac{b}{10} + \frac{1}{10} < c \leq \frac{1}{6}\) \\ \(c \leq - \frac{b}{3} + \frac{1}{3}, b \neq c\)};
			%
			\draw[thick] (1/4, 1/4) -- (3/4, 1/4);
			\draw[very thick] (3/4, 1/4) -- (1, 1/4);
			\node[align=center] at (6.5/13, 2.5/12) {\tiny \(-\frac{b}{3} + \frac{1}{3} < c \leq \frac{1}{4}\) \\ \tiny \(c > \frac{b}{3}\)};
			\node at (1, 1/4) [right] {\tiny{$c=1/4$}};
			%
			\draw[thick] (4/5, 1/5) -- (1, 1/5);
			\node at (1, 1/5) [right] {\tiny{$c=1/5$}};
			%
			\draw[very thick] (1/2, 1/6) -- (3/4, 1/4);
			\node[rotate=18.43] at (2/3, 17/81) [above] {\tiny$c=\frac{b}{3}$};
			%
			\draw[very thick] (2/3, 1/6) -- (4/5, 1/5);
			\node at (2/3, 1/6) [below] {\tiny{$(\frac{2}{3}, \frac{1}{6})$}};
			\node[rotate=14.04] at (22/30, 21/120) [above] {\tiny $c=\frac{b}{4}$};
			%
			\draw[very thick] (1/9, 1/9) -- (1/6, 1/6);
			\draw[->] (1/11, 1/6) -- (5/36, 5/36);
			\node at (1/10, 1/6) [left] {Sch $\frac{1}{9} < c \leq \frac{1}{6}$};
			
			\draw[->] (1.1, 1/6) -- (5.2/6, 10.5/48);
			\node[circle, fill, inner sep=1.5pt] at (5.15/6, 10.5/48) [above] {};
			\node[align=left] at (1.16, 1/6) {\(\frac{1}{5} < c < \frac{1}{4}\) \\ \(c > -b+1\)};
			
			\draw[->] (1.05, 1/3) -- (8.5/13, 1/5);
			\node[circle, fill, inner sep=1.5pt] at (8.5/13, 1/5) [left] {};
			\node[align=center] at (1.15, 1/3) {\(\frac{1}{6} < c \leq -b+1\) \\ \(\frac{b}{4} < c < \frac{b}{3}\)};
			
			\draw[->] (7/13, 1/11) -- (7.3/13, 9.4/60);
			\node[circle, fill, inner sep=1.5pt] at (7.3/13, 9.4/60) [above] {};
			\node[align=center] at (7/13, 1/20) {\(\frac{b}{10} + \frac{1}{10} < c \leq \frac{1}{6}\) \\ \(c > -\frac{b}{3} + \frac{1}{3}\)};
			
			\draw[->] (10/13, 1/11) -- (10.5/13, 1.1/6);
			\node[circle, fill, inner sep=1.5pt] at (10.5/13, 1.1/6) [above] {};
			\node[align=center] at (10/13, 1/20) {\(\frac{b}{10} + \frac{1}{10} < c \leq \frac{1}{5}\) \\ \(c < \frac{b}{4}\)};
			
			\draw[dashed] (1/9, 1/9) -- (1, 1/5);
			\node[rotate=5.71] at (5/6, 11/60) [below] {\tiny$c=b/10 + 1/10$};
		\end{tikzpicture}
	\end{center}
\end{pro}

\begin{table}[H]
	\centering
	\label{tab:E3A1eck}
	\begin{tabular}{| c | c | c | c |}
		\hline
		Label & Surface & \# & Eckardt points \\
		\hline
		\hline
		1 & $\widetilde{S}_{3A_1}$ & 1 & 0 or 1 \\
		2 & $\mathrm{Bl}_4\mathbb{F}_0$ & 3 & 0 \\
		3 & $\mathrm{Bl}_5\mathbb{P}^2$ & 3 & 0 or 1 \\
		4a & $\mathbb{F}_0$ & 12 & 0 \\
		4b & $\mathbb{F}_0$ & 12 & 0 \\
		5a & $\mathbb{F}_0$ & 6 & 0 \\
		5b & $\mathbb{F}_0$ & 6 & 0 \\
		\hline
	\end{tabular}
	\caption{This table is taken from \cite[Table~4]{schock_moduli_2024} The first two columns are the numbered types of irreducible components of the surfaces in the chamber Sch $1/2 < c \leq 2/3$ of type $\widetilde{E}_{3A_1}$ pictured in \cite[Fig.~15]{schock_moduli_2024}. The third column tells us the number of each corresponding component found in the surface. Finally, the last row gives the possible numbers of Eckardt points on each component. For the component of type 1, $\widetilde{S}_{3A_1}$ refers to the minimal resolution of a cubic surface with three $A_1$ singularities. For the components of type 2, $\mathrm{Bl}_4\mathbb{F}_0$ refers to the blowup of $\mathbb{F}_0 \cong  \mathbb{P}^1 \times \mathbb{P}^1$ at 4 points on
		the diagonal. For the components of type 3, $\mathrm{Bl}_5\mathbb{P}^2$ refers to the special blowup of $\mathbb{P}^2$ at 5 points as in \Cref{pro:bl5p2eck}.}
\end{table}

\newpage
\subsection{Type $\widetilde{E}_{4A_1}$}
\subsubsection{\textbf{$\ell$ is in the smooth locus}}
\begin{pro}
	For type $\widetilde{E}_{4A_1}$ $(b,c)$-weighted stable marked cubic surfaces $(S, (b, c)B)$ where $\ell$ is in the smooth locus, the chambers are the same as in Proposition \ref{E2A1_mult1}. Additionally, crossing the wall $c = \frac{1}{4}$ introduces type 
	\begin{align*}
		\widetilde{D}_{A_1} \cap \widetilde{E}_{4A_1}, \hspace{.25cm} \widetilde{E}_{2A_1} \cap \widetilde{E}_{4A_1}, \hspace{.25cm} \widetilde{E}_{3A_1} \cap \widetilde{E}_{4A_1},& \hspace{.25cm} \widetilde{D}_{A_1} \cap \widetilde{E}_{2A_1} \cap \widetilde{E}_{4A_1}, \hspace{.25cm} \widetilde{D}_{A_1} \cap \widetilde{E}_{3A_1} \cap \widetilde{E}_{4A_1}, \hspace{.25cm} \widetilde{E}_{2A_1} \cap \widetilde{E}_{3A_1} \cap \widetilde{E}_{4A_1},\\
		&\hspace{.25cm} \widetilde{D}_{A_1} \cap \widetilde{E}_{2A_1} \cap \widetilde{E}_{3A_1} \cap \widetilde{E}_{4A_1}
	\end{align*}
	surfaces as degenerations of type $\widetilde{E}_{4A_1}$ surfaces.
\end{pro}

\subsubsection{\textbf{$\ell$ passes through exactly two nodes}}
\begin{pro}
	For type $\widetilde{E}_{4A_1}$ $(b,c)$-weighted stable marked cubic surfaces $(S, (b,c)B)$ where $\ell$ passes through exactly two nodes in the corresponding $\overline{Y}_\ell$ boundary, the chambers are the same as in Proposition \ref{E3A1_mult4}. Additionally, crossing the wall $c = -\frac{b}{3} \cap \frac{b}{3}$ introduces type
	\begin{align*}
		\widetilde{D}_{A_1} \cap \widetilde{E}_{4A_1}, \hspace{.25cm} \widetilde{E}_{2A_1} \cap \widetilde{E}_{4A_1}, \hspace{.25cm} \widetilde{E}_{3A_1} \cap \widetilde{E}_{4A_1},& \hspace{.25cm} \widetilde{D}_{A_1} \cap \widetilde{E}_{2A_1} \cap \widetilde{E}_{4A_1}, \hspace{.25cm} \widetilde{D}_{A_1} \cap \widetilde{E}_{3A_1} \cap \widetilde{E}_{4A_1}, \hspace{.25cm} \widetilde{E}_{2A_1} \cap \widetilde{E}_{3A_1} \cap \widetilde{E}_{4A_1},\\
		&\hspace{.25cm} \widetilde{D}_{A_1} \cap \widetilde{E}_{2A_1} \cap \widetilde{E}_{3A_1} \cap \widetilde{E}_{4A_1}
	\end{align*}
	surfaces as degenerations of type $\widetilde{E}_{4A_1}$ surfaces. If $\ell$ is in the component that is the degeneration of $\mathrm{Bl}_5 \mathbb{P}^2$, the degenerations introduces more chambers that are the same as the chambers of the higher codimension types in Proposition \ref{E3A1_mult4}. Otherwise, the wall and chamber decomposition is the same as type $\widetilde{E}_{4A_1}$.
\end{pro}

\begin{table}[H]
	\centering
	\label{tab:E4A1eck}
	\begin{tabular}{| c | c | c | c |}
		\hline
		Label & Surface & \# & Eckardt points \\
		\hline
		\hline
		1 & $\widetilde{S}_{4A_1}$ & 1 & 0 \\
		2 & $\mathrm{Bl}_3 \mathbb{F}_0$ & 4 & 0 \\
		3 & $\mathrm{Bl}_5 \mathbb{P}^2$ & 6 & 0 or 1 \\
		4a & $\mathbb{F}_0$ & 24 & 0 \\
		4b & $\mathbb{F}_0$ & 24 & 0 \\
		\hline
	\end{tabular}
	\caption{This table is taken from \cite[Table~5]{schock_moduli_2024} The first two columns are the numbered types of irreducible components of the surfaces in the chamber Sch $1/2 < c \leq 2/3$ of type $\widetilde{E}_{4A_1}$ pictured in \cite[Fig.~21]{schock_moduli_2024}.The third column tells us the number of each corresponding component found in the surface. Finally, the last row gives the possible numbers of Eckardt points on each component. For the component of type 1, $\widetilde{S}_{4A_1}$ refers to the minimal resolution of a cubic surface with four $A_1$ singularities. For the components of type 2, $\mathrm{Bl}_3 \mathbb{F}_0$ refers to the blowup of $\mathbb{F}_0 \cong \mathbb{P}^1 \times \mathbb{P}^1$ at three points on the diagonal. For the components of type 3, $\mathrm{Bl}_5 \mathbb{P}^2$ refers to the special blowup of $\mathbb{P}^2$ at 5 points as in \Cref{pro:bl5p2eck}.}
\end{table}

\newpage
\subsection{Type $\widetilde{D}_{3A_2}$}
\subsubsection{\textbf{$\ell$ is in the smooth locus}}
\begin{pro}\label{D3A2_mult1}
	For type $\widetilde{D}_{3A_2}$ $(b,c)$-weighted stable marked cubic surfaces $(S, (b,c)B)$ where $\ell$ is in the smooth locus, there are five chambers.
	\begin{center}
		\begin{tikzpicture}[scale=14]
			\draw[very thick,-latex] (1/9, 1/9) -- (1.03, 1/9) 
			node[right]{$b$};
			
			\draw[very thick,-latex] (1/9, 1/9) -- (1/9, 1.03) 
			node[left]{$c$};
			
			\draw[solid] (1/9,1/9) -- (1,1);
			\node at (1/9, 1/9) [below] {\tiny $(\frac{1}{9}, \frac{1}{9})$};
			\node at (1, 1) [left] {$(1, 1)$};
			
			\draw[thick] (1, 1/9) -- (1,1);
			\node at (1, 1/9) [below] {$(1, \frac{1}{9})$};
			
			\draw[thick] (2/3, 2/3) -- (1, 2/3) node[right] {\tiny{$c=\frac{2}{3}$}};
			\node[circle, fill, inner sep=1pt] at (2/3, 2/3) {};
			\node at (2/3, 2/3) [left] {\tiny$(\frac{2}{3}, \frac{2}{3})$};
			\node at (8/9, 7/9) {Sch $\frac{2}{3} < c \leq 1$};
			\node at (0.9, 0.63) {$-\frac{b}{2}+1 < c \leq \frac{2}{3}$};
			\node at (2/3, 1/2) {Sch $\frac{1}{3} < c \leq \frac{2}{3}$};
			
			\draw[solid] (2/3, 2/3) -- (1, 1/2) node[right] {\color{black} \tiny{$c = \frac{1}{2}$}};
			\node[rotate=-26.6] at (5/6, 57/96) {\tiny wall due to $\ell$ containing Eckardt point(s)};
			
			\draw[thick] (1/3, 1/3) -- (1, 1/3) node[right] {\tiny $c = \frac{1}{3}$};
			\node at (1/3, 1/3) [left] {\tiny $(\frac{1}{3}, \frac{1}{3}$)};
			\node[circle, fill, inner sep=1pt] at (1/3, 1/3) {};
			
			\draw[very thick] (1/9, 1/9) -- (1/3, 1/3);
			\draw[->] (1/11, 1/4) -- (5/24, 5/24);
			\node at (1/10, 1/4) [left] {Sch $\frac{1}{9} < c \leq \frac{1}{3}$};
			\node at (2/3, 1/4) {$\frac{b}{10} + \frac{1}{10} < c \leq \frac{1}{3}$, \hspace{.1in} $b \neq c$};
			
			\draw[dashed] (1/9, 1/9) -- (1, 1/5);
			\node[rotate=5.71] at (5/6, 11/60) [below] {\tiny$c=b/10 + 1/10$};
			\node at (1, 1/5) [right] {\tiny{$c=\frac{1}{5}$}};
			
		\end{tikzpicture}
	\end{center}
\end{pro}

\begin{table}[H]
	\centering
	\label{tab:D3A2eck}
	\begin{tabular}{| c | c | c | c |}
		\hline
		Label & Surface & \# & Eckardt points \\
		\hline
		\hline
		1 & $\mathrm{Bl}_6\mathbb{P}^2$ & 3 & 0, 1, 2, or 3 \\
		2 & $\mathbb{F}_0$ & 9 & 0 \\
		\hline
	\end{tabular}
	\caption{This table is taken from \cite[Table~6]{schock_moduli_2024} The first two columns are the numbered types of irreducible components of the surfaces in the chamber Sch $1/2 < c \leq 2/3$ of type $\widetilde{D}_{3A_2}$ pictured in \cite[Fig.~24]{schock_moduli_2024}. The third column tells us the number of each corresponding component found in the surface. Finally, the last row gives the possible numbers of Eckardt points on each component. For the components of type 1, $\mathrm{Bl}_6 \mathbb{P}^2$ refers to the blowup of $\mathbb{P}^2$ at 3 points on one line and 3 points on another line. Note this is the minimal resolution of a cubic surface with an $A_2$ singularity.}
\end{table}

\newpage
\subsection{Type $\widetilde{D}_{A_1} \cap \widetilde{D}_{3A_2}$}
\subsubsection{\textbf{$\ell$ is in the smooth locus}}
\begin{pro}
	For type $\widetilde{D}_{A_1} \cap \widetilde{D}_{3A_2}$ $(b,c)$-weighted stable marked cubic surfaces $(S, (b,c)B)$ where $\ell$ is in the smooth locus, there are eight chambers.
	\begin{center}
		\begin{tikzpicture}[scale=14]
			\draw[very thick,-latex] (1/9, 1/9) -- (1.03, 1/9) 
			node[right]{$b$};
			
			\draw[very thick,-latex] (1/9, 1/9) -- (1/9, 1.03) 
			node[left]{$c$};
			
			\draw[solid] (1/9,1/9) -- (1,1);
			\node at (1/9, 1/9) [below] {\tiny $(\frac{1}{9}, \frac{1}{9})$};
			\node at (1, 1) [left] {$(1, 1)$};
			
			\draw[thick] (1, 1/9) -- (1,1);
			\node at (1, 1/9) [below] {$(1, \frac{1}{9})$};
			
			\draw[thick] (2/3, 2/3) -- (1, 2/3) node[right] {\tiny{$c=\frac{2}{3}$}};
			\node[circle, fill, inner sep=1pt] at (2/3, 2/3) {};
			\node at (2/3, 2/3) [left] {\tiny$(\frac{2}{3}, \frac{2}{3})$};
			\node at (8/9, 7/9) {Sch $\frac{2}{3} < c \leq 1$};
			\node at (0.9, 0.63) {$-\frac{b}{2}+1 < c \leq \frac{2}{3}$};
			
			\node at (2/3, 2.5/6) {Sch $\frac{1}{3} < c \leq \frac{1}{2}$};
			
			\draw[solid] (2/3, 2/3) -- (1, 1/2) node[right] {\color{black} \tiny{$c = \frac{1}{2}$}};
			\node[rotate=-26.6] at (5/6, 57/96) {\tiny wall due to $\ell$ containing Eckardt point(s)};
			
			\draw[thick] (1/2, 1/2) -- (1, 1/2);
			\node[circle, fill, inner sep=1pt] at (1/2, 1/2) {};
			\node at (1/2, 1/2) [left] {\tiny $(\frac{1}{2}, \frac{1}{2})$};
			\node at (0.7, 0.56) {Sch $\frac{1}{2} < c \leq \frac{2}{3}$};
			
			\draw[thick] (1/3, 1/3) -- (1, 1/3) node[right] {\tiny $c = \frac{1}{3}$};
			\node at (1/3, 1/3) [left] {\tiny $(\frac{1}{3}, \frac{1}{3})$};
			\node[circle, fill, inner sep=1pt] at (1/3, 1/3) {};
			\node at (2/3, 1.5/6)  {$\frac{1}{6} < c \leq \frac{1}{3}$, \hspace{.1in}, $b \neq c$};
			
			\draw[thick] (1/6, 1/6) -- (2/3, 1/6);
			\node[circle, fill, inner sep=1pt] at (1/6, 1/6) {};
			\node at (7/36, 7/36) [left] {\tiny$(\frac{1}{6}, \frac{1}{6}$)};
			\node at (2/3, 1/6) [below] {\tiny{$(\frac{2}{3}, \frac{1}{6})$}};
			\draw[->] (7/24, 1/11) -- (1/4, 10/72);
			\node[circle, fill, inner sep=1.5pt] at (1/4, 10/72) [above] {};
			\node[align=center] at (7/24, 1/20) {\(\frac{b}{10} + \frac{1}{10} < c \leq \frac{1}{6}\) \\ \(c \leq - \frac{b}{3} + \frac{1}{3}, b \neq c\)};
			
			\draw[very thick] (1/6, 1/6) -- (1/3, 1/3);
			\draw[->] (1/11, 1/4) -- (6/24, 6/24);
			\node at (1/10, 1/4) [left] {Sch $\frac{1}{6} < c \leq \frac{1}{3}$};
			
			\draw[very thick] (1/9, 1/9) -- (1/6, 1/6);
			\draw[->] (1/11, 1/6) -- (5/36, 5/36);
			\node at (1/10, 1/6) [left] {Sch $\frac{1}{9} < c \leq \frac{1}{6}$};
			
			\draw[dashed] (1/9, 1/9) -- (1, 1/5);
			\node[rotate=5.71] at (5/6, 11/60) [below] {\tiny$c=b/10 + 1/10$};
			\node at (1, 1/5) [right] {\tiny{$c=\frac{1}{5}$}};
			
		\end{tikzpicture}
	\end{center}
\end{pro}

\newpage
\subsubsection{\textbf{$\ell$ passes through exactly one node}}

\begin{pro}
	For type $\widetilde{D}_{A_1} \cap \widetilde{D}_{3A_2}$ $(b,c)$-weighted stable marked cubic surfaces $(S, (b,c)B)$ where $\ell$ passes through exactly one node in the corresponding $\overline{Y}_\ell$ boundary, there are 11 chambers. The line $c = \frac{b}{4}$ is its own chamber.
	\begin{center}
		\begin{tikzpicture}[scale=14]
			\draw[very thick,-latex] (1/9, 1/9) -- (1.03, 1/9) 
			node[right]{$b$};
			
			\draw[very thick,-latex] (1/9, 1/9) -- (1/9, 1.03) 
			node[left]{$c$};
			
			\draw[solid] (1/9,1/9) -- (1,1);
			\node at (1/9, 1/9) [below] {\tiny $(\frac{1}{9}, \frac{1}{9})$};
			\node at (1, 1) [left] {$(1, 1)$};
			
			\draw[thick] (1, 1/9) -- (1,1);
			\node at (1, 1/9) [below] {$(1, \frac{1}{9})$};
			
			\draw[thick] (2/3, 2/3) -- (1, 2/3) node[right] {\tiny{$c=\frac{2}{3}$}};
			\node[circle, fill, inner sep=1pt] at (2/3, 2/3) {};
			\node at (2/3, 2/3) [left] {\tiny$(\frac{2}{3}, \frac{2}{3})$};
			\node at (8/9, 7/9) {Sch $\frac{2}{3} < c \leq 1$};
			
			\draw[thick] (1/2, 1/2) -- (1, 1/2);
			\node[circle, fill, inner sep=1pt] at (1/2, 1/2) {};
			\node at (1/2, 1/2) [left] {\tiny $(\frac{1}{2}, \frac{1}{2})$};
			\node at (7.5/9, 0.57) {Sch $\frac{1}{2} < c \leq \frac{2}{3}$};
			
			\draw[thick] (1/2, 1/2) -- (4/5, 1/5);
			\node[rotate=-45] at (7/12, 5/12) [above] {\tiny $c = -b+1$};
			\node[align=center] at (5/6, 0.42) {\(\frac{1}{3} < c \leq \frac{1}{2}\) \\ \(c > -b+1 \)};
			\node[align=center] at (1/2, 2.3/6) {Sch $\frac{1}{3} < c \leq \frac{1}{2}$};
			
			\draw[thick] (1/3, 1/3) -- (1, 1/3) node[right] {\tiny $c = \frac{1}{3}$};
			\node at (1/3, 1/3) [left] {\tiny $(\frac{1}{3}, \frac{1}{3}$)};
			\node[align=center] at (1.5/3, 1.5/6)  {\(\frac{1}{6} < c \leq \frac{1}{3}\)\\ \(\frac{b}{4}< c \leq -b+1, b \neq c\)};
			\node[align=center] at (4.3/5, 1.5/6) {\(-b+1 < c \leq \frac{1}{3}\) \\ \(c > \frac{1}{5}\)};
			\node[circle, fill, inner sep=1pt] at (1/3, 1/3) {};
			
			\draw[thick] (4/5, 1/5) -- (1, 1/5);
			\draw[->] (2.5/3, 1/11) -- (0.8, 0.18);
			\node[circle, fill, inner sep=1.5pt] at (0.8, 0.19) {};
			\node[align=center] at (2.5/3, 1/20) {\(\frac{b}{10} + \frac{1}{10} < c \leq \frac{1}{5}\) \\ \(c < \frac{b}{4}\)};
			
			\draw[very thick] (2/3, 1/6) -- (4/5, 1/5);
			\node at (2/3, 1/6) [below] {\tiny{$(\frac{2}{3}, \frac{1}{6})$}};
			\node[rotate=14.04] at (22/30, 21/120) [above] {\tiny $c=\frac{b}{4}$};
			
			\draw[thick] (1/6, 1/6) -- (2/3, 1/6);
			\node[circle, fill, inner sep=1pt] at (1/6, 1/6) {};
			\node at (7/36, 7/36) [left] {\tiny$(\frac{1}{6}, \frac{1}{6}$)};
			\node at (2/3, 1/6) [below] {\tiny{$(\frac{2}{3}, \frac{1}{6})$}};
			\draw[->] (7/24, 1/11) -- (1/4, 10/72);
			\node[circle, fill, inner sep=1.5pt] at (1/4, 10/72) [above] {};
			\node[align=center] at (7/24, 1/20) { \(\frac{b}{10} + \frac{1}{10} < c \leq \frac{1}{6}\) \\ \(b \neq c\)};
			
			\draw[very thick] (1/6, 1/6) -- (1/3, 1/3);
			\draw[->] (1/11, 1/4) -- (6/24, 6/24);
			\node at (1/10, 1/4) [left] {Sch $\frac{1}{6} < c \leq \frac{1}{3}$};
			
			\draw[very thick] (1/9, 1/9) -- (1/6, 1/6);
			\draw[->] (1/11, 1/6) -- (5/36, 5/36);
			\node at (1/10, 1/6) [left] {Sch $\frac{1}{9} < c \leq \frac{1}{6}$};
			
			\draw[dashed] (1/9, 1/9) -- (1, 1/5);
			\node[rotate=5.71] at (5/6, 11/60) [below] {\tiny$c=b/10 + 1/10$};
			\node at (1, 1/5) [right] {\tiny{$c=\frac{1}{5}$}};
			
		\end{tikzpicture}
	\end{center}
\end{pro}

\begin{table}[H]
	\centering
	\label{tab:DA1D3A2eck}
	\begin{tabular}{| c | c | c | c |}
		\hline
		Label & Surface & \# & Eckardt points \\
		\hline
		\hline
		1 & $\mathrm{Bl}_6\mathbb{P}^2$ & 1 & 0, 1, 2, or 3 \\
		2 & $\mathrm{Bl}_5\mathbb{P}^2$ & 2 & 0 \\
		3a & $\mathrm{Bl}_5\mathbb{P}^2$ & 2 & 0 \\
		3b & $\mathbb{F}_0$ & 2 & 0 \\
		4a & $\mathbb{F}_0$ & 6 & 0 \\
		4b & $\mathbb{F}_0$ & 6 & 0 \\
		5 & $\mathbb{F}_0$ & 1 & 0 \\
		6 & $\mathbb{F}_0$ & 6 & 0 \\
		\hline
	\end{tabular}
	\caption{This table is taken from \cite[Table~7]{schock_moduli_2024} The first two columns are the numbered types of irreducible components of the surfaces in the chamber Sch $1/2 < c \leq 2/3$ of type $\widetilde{D}_{A_1} \cap \widetilde{D}_{3A_2}$ pictured in \cite[Fig.~28]{schock_moduli_2024}. The third column tells us the number of each corresponding component found in the surface. Finally, the last row gives the possible numbers of Eckardt points on each component. For the component of type 1, $\mathrm{Bl}_6\mathbb{P}^2$ refers to the blowup of $\mathbb{P}^2$ at 3 points on one line and 3 points on another line, cf. \cref{tab:D3A2eck}. For the components of types 2 and 3a, $\mathrm{Bl}_5\mathbb{P}^2$ refers to the blowup of $\mathbb{P}^2$ at 3 points on one line and 2 points on another line.}
\end{table}

\newpage
\subsection{Type $\widetilde{E}_{2A_1} \cap \widetilde{D}_{3A_2}$}
\subsubsection{\textbf{$\ell$ is in the smooth locus}}
\begin{pro}\label{E2A1D3A2_mult1}
	For type $\widetilde{E}_{2A_1} \cap \widetilde{D}_{3A_2}$ $(b,c)$-weighted stable marked cubic surfaces $(S, (b,c)B)$ where $\ell$ is in the smooth locus, there are 10 chambers. Additionally, crossing the wall $c = \frac{1}{4}$ introduces type $\widetilde{D}_{A_1} \cap \widetilde{E}_{2A_1} \cap \widetilde{D}_{3A_2}$ surfaces as degerenations of type $\widetilde{E}_{2A_1} \cap \widetilde{D}_{3A_2}$.
	\begin{center}
		\begin{tikzpicture}[scale=14]
			\draw[very thick,-latex] (1/9, 1/9) -- (1.03, 1/9) 
			node[right]{$b$};
			
			\draw[very thick,-latex] (1/9, 1/9) -- (1/9, 1.03) 
			node[left]{$c$};
			
			\draw[solid] (1/9,1/9) -- (1,1);
			\node at (1/9, 1/9) [below] {\tiny $(\frac{1}{9}, \frac{1}{9})$};
			\node at (1, 1) [left] {$(1, 1)$};
			
			\draw[thick] (1, 1/9) -- (1,1);
			\node at (1, 1/9) [below] {$(1, \frac{1}{9})$};
			
			\draw[thick] (2/3, 2/3) -- (1, 2/3) node[right] {\tiny{$c=\frac{2}{3}$}};
			\node[circle, fill, inner sep=1pt] at (2/3, 2/3) {};
			\node at (2/3, 2/3) [left] {\tiny$(\frac{2}{3}, \frac{2}{3})$};
			\node at (8/9, 7/9) {Sch $\frac{2}{3} < c \leq 1$};
			
			\node at (2/3, 2.5/6) {Sch $\frac{1}{3} < c \leq \frac{1}{2}$};
			
			\draw[solid] (2/3, 2/3) -- (1, 1/2) node[right] {\color{black} \tiny{$c = \frac{1}{2}$}};
			\node[rotate=-26.6] at (5/6, 57/96) {\tiny wall due to $\ell$ containing Eckardt point(s)};
			
			\draw[thick] (1/2, 1/2) -- (1, 1/2);
			\node[circle, fill, inner sep=1pt] at (1/2, 1/2) {};
			\node at (1/2, 1/2) [left] {\tiny $(\frac{1}{2}, \frac{1}{2})$};
			\node at (0.7, 0.56) {Sch $\frac{1}{2} < c \leq \frac{2}{3}$};
			\node at (0.9, 0.63) {$-\frac{b}{2}+1 < c \leq \frac{2}{3}$};
			
			\draw[thick] (1/3, 1/3) -- (1, 1/3) node[right] {\tiny $c = \frac{1}{3}$};
			\node at (1/3, 1/3) [left] {\tiny $(\frac{1}{3}, \frac{1}{3})$};
			\node[circle, fill, inner sep=1pt] at (1/3, 1/3) {};
			\node at (2/3, 1.8/6)  {$\frac{1}{4} < c \leq \frac{1}{3}$, \hspace{.1in}, $b \neq c$};
			
			\draw[thick] (1/4, 1/4) -- (1,1/4);
			\node at (1, 1/4) [right] {\tiny$c=1/4$};
			\node at (1/4, 1/4) [left] {\tiny$(\frac{1}{4}, \frac{1}{4})$};
			\node at (1.5/3, 1.25/6)  {$\frac{1}{6} < c \leq \frac{1}{4}$, \hspace{.1in}, $b \neq c$};

			\draw[thick] (1/6, 1/6) -- (2/3, 1/6);
			\node[circle, fill, inner sep=1pt] at (1/6, 1/6) {};
			\node at (7/36, 7/36) [left] {\tiny$(\frac{1}{6}, \frac{1}{6}$)};
			\node at (2/3, 1/6) [below] {\tiny{$(\frac{2}{3}, \frac{1}{6})$}};
			\draw[->] (7/24, 1/11) -- (1/4, 10/72);
			\node[circle, fill, inner sep=1.5pt] at (1/4, 10/72) [above] {};
			\node[align=center] at (7/24, 1/20) {\(\frac{b}{10} + \frac{1}{10} < c \leq \frac{1}{6}\) \\ \(c \leq - \frac{b}{3} + \frac{1}{3}, b \neq c\)};
			
			\draw[very thick] (1/4, 1/4) -- (1/3, 1/3);
			\draw[->] (1/11, 1/3) -- (7/24, 7/24);
			\node at (1/10, 1/3) [left] {Sch $\frac{1}{4} < c \leq \frac{1}{3}$};		
			\node[circle, fill, inner sep=1pt] at (1/4, 1/4) {};
			\node[circle, fill, inner sep=1pt] at (1/3, 1/3) {};
			
			\draw[very thick] (1/6, 1/6) -- (1/4, 1/4);
			\draw[->] (1/11, 1/4) -- (5/24, 5/24);
			\node at (1/10, 1/4) [left] {Sch $\frac{1}{6} < c \leq \frac{1}{4}$};
			\node[circle, fill, inner sep=1pt] at (1/6, 1/6) {};
			
			\draw[very thick] (1/9, 1/9) -- (1/6, 1/6);
			\draw[->] (1/11, 1/6) -- (5/36, 5/36);
			\node at (1/10, 1/6) [left] {Sch $\frac{1}{9} < c \leq \frac{1}{6}$};
			
			\draw[dashed] (1/9, 1/9) -- (1, 1/5);
			\node[rotate=5.71] at (5/6, 11/60) [below] {\tiny$c=b/10 + 1/10$};
			\node at (1, 1/5) [right] {\tiny{$c=\frac{1}{5}$}};
			
		\end{tikzpicture}
	\end{center}
\end{pro}

\newpage
\subsubsection{\textbf{$\ell$ passes through exactly one node}}

\begin{pro}\label{E2A1D3A2_mult2}
	For type $\widetilde{E}_{2A_1} \cap \widetilde{D}_{3A_2}$ $(b,c)$-weighted stable marked cubic surfaces $(S, (b,c)B)$ where $\ell$ passes through exactly one node in the corresponding $\overline{Y}_\ell$ boundary, there are 14 chambers. The line $c = \frac{b}{4}$ is its own chamber. Additionally, crossing the wall $c = \frac{1}{4}$ introduces type $\widetilde{D}_{A_1}\cap \widetilde{E}_{2A_1} \cap \widetilde{D}_{3A_2}$ surfaces as degerenations of type $\widetilde{E}_{2A_1} \cap \widetilde{D}_{3A_2}$.
	\begin{center}
		\begin{tikzpicture}[scale=14]
			\draw[very thick,-latex] (1/9, 1/9) -- (1.03, 1/9) 
			node[right]{$b$};
			
			\draw[very thick,-latex] (1/9, 1/9) -- (1/9, 1.03) 
			node[left]{$c$};
			
			\draw[solid] (1/9,1/9) -- (1,1);
			\node at (1/9, 1/9) [below] {\tiny$(\frac{1}{9}, \frac{1}{9})$};
			\node at (1, 1) [left] {$(1, 1)$};
			\node[circle, fill, inner sep=1pt] at (1, 1) {};
			
			\draw[thick] (1, 1/9) -- (1,1);
			\node at (1, 1/9) [below] {\tiny$(1, \frac{1}{9})$};
			
			\draw[thick] (2/3, 2/3) -- (1, 2/3);
			\node at (1, 2/3) [right] {\tiny $c=2/3$};
			\node at (2/3, 2/3) [left] {\tiny$(\frac{2}{3}, \frac{2}{3})$};
			\node[circle, fill, inner sep=1pt] at (2/3, 2/3) {};
			\node at (8/9, 7/9) {Sch $\frac{2}{3} < c \leq 1$};
			
			\draw[thick] (1/2, 1/2) -- (1, 1/2);
			\node at (7.5/9, 0.57) {Sch $\frac{1}{2} < c \leq \frac{2}{3}$};
			
			\draw[thick] (1/2, 1/2) -- (4/5, 1/5);
			\node at (1/2, 1/2) [left] {\tiny$(\frac{1}{2}, \frac{1}{2})$};
			\node[circle, fill, inner sep=1pt] at (1/2, 1/2) {};
			\node at (1, 1/2) [right] {\tiny$c=1/2$};
			\node[rotate=-45] at (7/12, 5/12) [above] {\tiny $c = -b+1$};
			\node[align=center] at (5/6, 0.42) {\(\frac{1}{3} < c \leq \frac{1}{2}\) \\ \(c > -b+1 \)};
			\node[align=center] at (1/2, 4.7/12) {Sch $\frac{1}{3} < c \leq \frac{1}{2}$};
			%
			\draw[thick] (1/3, 1/3) -- (1,1/3);
			\node at (1, 1/3) [right] {\tiny$c=1/3$};
			\node at (1/3, 1/3) [left] {\tiny$(\frac{1}{3}, \frac{1}{3})$};
			\node[align=center] at (1/2, 7/24) {\( \frac{1}{4} < c \leq \frac{1}{3}\)\\ \(c \leq -b+1, b \neq c \)};
			\node[align=center] at (5.2/6, 7/24) {\(\frac{1}{4} < c \leq \frac{1}{3}\)\\ \(c > -b+1\)};
			
			\draw[thick] (1/4, 1/4) -- (1,1/4);
			\node at (1, 1/4) [right] {\tiny$c=1/4$};
			\node at (1/4, 1/4) [left] {\tiny$(\frac{1}{4}, \frac{1}{4})$};
			\node[align=center] at (1/2, 5/24) {\( \frac{1}{6} < c \leq \frac{1}{4}\)\\ \(\frac{b}{4} < c \leq -b+1, b \neq c \)};
			\draw[->] (1.1, 1/6) -- (5.2/6, 10.5/48);
			\node[circle, fill, inner sep=1.5pt] at (5.15/6, 10.5/48) [above] {};
			\node[align=left] at (1.16, 1/6) {\(\frac{1}{5} < c \leq \frac{1}{4}\) \\ \(c > -b+1\)};
			
			\draw[thick] (1/6, 1/6) -- (2/3, 1/6);
			\node at (7/36, 7/36) [left] {\tiny$(\frac{1}{6}, \frac{1}{6}$)};
			\draw[->] (7/24, 1/11) -- (1/4, 10/72);
			\node[circle, fill, inner sep=1.5pt] at (1/4, 10/72) [above] {};
			\node[align=center] at (7/24, 1/20) { \(\frac{b}{10} + \frac{1}{10} < c \leq \frac{1}{6}\) \\ \(b \neq c\)};
			%
			%
			\draw[thick] (4/5, 1/5) -- (1, 1/5);
			\node at (1, 1/5) [right] {\tiny{$c=1/5$}};
			\draw[->] (2.5/3, 1/11) -- (0.8, 0.18);
			\node[circle, fill, inner sep=1.5pt] at (0.8, 0.19) {};
			\node[align=center] at (2.5/3, 1/20) {\(\frac{b}{10} + \frac{1}{10} < c \leq \frac{1}{5}\) \\ \(c < \frac{b}{4}\)};

			\draw[very thick] (2/3, 1/6) -- (4/5, 1/5);
			\node at (2/3, 1/6) [below] {\tiny{$(\frac{2}{3}, \frac{1}{6})$}};
			\node[circle, fill, inner sep=1pt] at (2/3, 1/6) {};
			\node[rotate=14.04] at (22/30, 21/120) [above] {\tiny $c=\frac{b}{4}$};
			%
			
			\draw[very thick] (1/4, 1/4) -- (1/3, 1/3);
			\draw[->] (1/11, 1/3) -- (7/24, 7/24);
			\node at (1/10, 1/3) [left] {Sch $\frac{1}{4} < c \leq \frac{1}{3}$};		
			\node[circle, fill, inner sep=1pt] at (1/4, 1/4) {};
			\node[circle, fill, inner sep=1pt] at (1/3, 1/3) {};
			
			\draw[very thick] (1/6, 1/6) -- (1/4, 1/4);
			\draw[->] (1/11, 1/4) -- (5/24, 5/24);
			\node at (1/10, 1/4) [left] {Sch $\frac{1}{6} < c \leq \frac{1}{4}$};
			\node[circle, fill, inner sep=1pt] at (1/6, 1/6) {};
			
			\draw[very thick] (1/9, 1/9) -- (1/6, 1/6);
			\draw[->] (1/11, 1/6) -- (5/36, 5/36);
			\node at (1/10, 1/6) [left] {Sch $\frac{1}{9} < c \leq \frac{1}{6}$};
			
			\draw[dashed] (1/9, 1/9) -- (1, 1/5);
			\node[rotate=5.71] at (5/6, 11/60) [below] {\tiny$c=b/10 + 1/10$};
		\end{tikzpicture}
	\end{center}
\end{pro}

\newpage
\subsubsection{\textbf{$\ell$ passes through exactly two nodes}}
\begin{pro}
	For type $\widetilde{E}_{2A_1} \cap \widetilde{D}_{3A_2}$ $(b,c)$-weighted stable marked cubic surfaces $(S, (b,c)B)$ where $\ell$ passes through exactly two nodes in the corresponding $\overline{Y}_\ell$ boundary, there are 15 chambers. The line $c = \frac{b}{3}$ is its own chamber. Additionally, crossing the wall $c = -\frac{b}{3} + \frac{1}{3}$ introduces type $\widetilde{D}_{A_1}\cap \widetilde{E}_{2A_1} \cap \widetilde{D}_{3A_2}$ surfaces as degerenations of type $\widetilde{E}_{2A_1}$. In this case, the wall and chamber decomposition has 18 chambers and explained more below.
	\begin{center}
		\begin{tikzpicture}[scale=14]
			\draw[very thick,-latex] (1/9, 1/9) -- (1.03, 1/9) 
			node[right]{$b$};
			
			\draw[very thick,-latex] (1/9, 1/9) -- (1/9, 1.03) 
			node[left]{$c$};
			
			\draw[solid] (1/9,1/9) -- (1,1);
			\node at (1/9, 1/9) [below] {\tiny $(\frac{1}{9}, \frac{1}{9})$};
			\node at (1, 1) [left] {$(1, 1)$};
			
			\draw[thick] (1, 1/9) -- (1,1);
			\node at (1, 1/9) [below] {$(1, \frac{1}{9})$};
			
			\draw[thick] (2/3, 2/3) -- (1, 2/3) node[right] {\tiny{$c=\frac{2}{3}$}};
			\node[circle, fill, inner sep=1pt] at (2/3, 2/3) {};
			\node at (2/3, 2/3) [left] {\tiny$(\frac{2}{3}, \frac{2}{3})$};
			\node at (8/9, 7/9) {Sch $\frac{2}{3} < c \leq 1$};
			\node at (0.9, 0.63) {$-\frac{b}{2}+1 < c \leq \frac{2}{3}$};
			\draw[solid] (2/3, 2/3) -- (1, 1/2) node[right] {\color{black} \tiny{$c = \frac{1}{2}$}};
			\node[rotate=-26.6] at (5/6, 57/96) {\tiny wall due to $\ell$ containing Eckardt point(s)};
			
			\draw[thick] (1/2, 1/2) -- (1, 1/2);
			\node[circle, fill, inner sep=1pt] at (1/2, 1/2) {};
			\node at (1/2, 1/2) [left] {\tiny $(\frac{1}{2}, \frac{1}{2})$};
			\node at (0.7, 0.56) {Sch $\frac{1}{2} < c \leq \frac{2}{3}$};
			
			\draw[thick] (1/2, 1/2) -- (3/4, 1/4);
			\node[rotate=-45] at (7/12, 5/12) [above] {\tiny $c = -b+1$};
			\node[align=center] at (5/6, 0.42) {\(\frac{1}{3} < c \leq \frac{1}{2}\) \\ \(c > -b+1 \)};
			
			\draw[thick] (1/4, 1/4) -- (7/13, 2/13);
			\node[circle, fill, inner sep=1pt] at (1/4, 1/4) {};
			\node[rotate=-18.43] at (3/8, 24/120) [above] {\tiny$c=-\frac{b}{3} + \frac{1}{3}$};
			\node at (7/13, 2/13) [below] {\tiny{$(\frac{7}{13}, \frac{2}{13})$}};
			\node at (1/4, 1/4) [left] {\tiny$(\frac{1}{4}, \frac{1}{4})$};
			\node at (1.77/6, 4.55/24) {$\frac{1}{6} < c \leq \frac{1}{4}$, $b \neq c$};
			
			\draw[thick] (1/3, 1/3) -- (1,1/3);
			\node at (1, 1/3) [right] {\tiny$c=1/3$};
			\node at (1/3, 1/3) [left] {\tiny$(\frac{1}{3}, \frac{1}{3})$};
			\node[align=center] at (1/2, 6.5/24) {\( -\frac{b}{3} + \frac{1}{3} < c \leq \frac{1}{3}\)\\ \(\frac{b}{3} < c \leq -b+1, b \neq c \)};
			\node[align=center] at (5.2/6, 7/24) {\(\frac{1}{4} < c \leq \frac{1}{3}\)\\ \(c > -b+1\)};
			\node[align=center] at (1/2, 4.7/12) {Sch $\frac{1}{3} < c \leq \frac{1}{2}$};
			
			\draw[thick] (3/4, 1/4) -- (1,1/4);
			\node at (1, 1/4) [right] {\tiny$c=1/4$};
			\node at (1/4, 1/4) [left] {\tiny$(\frac{1}{4}, \frac{1}{4})$};
			\node[align=left] at (5.2/6, 10.5/48) {\tiny \(\frac{b}{10} + \frac{1}{10} < c \leq \frac{1}{4}\) \\ \tiny \(\frac{1}{6} < c < \frac{b}{3}\)};
			
			\draw[very thick] (1/2, 1/6) -- (3/4, 1/4);
			\node[rotate=18.43] at (2/3, 17/81) [above] {\tiny$c=\frac{b}{3}$};
			
			\draw[thick] (1/6, 1/6) -- (2/3, 1/6);
			\node[circle, fill, inner sep=1pt] at (1/6, 1/6) {};
			\node at (7/36, 7/36) [left] {\tiny$(\frac{1}{6}, \frac{1}{6}$)};
			\node at (2/3, 1/6) [below] {\tiny{$(\frac{2}{3}, \frac{1}{6})$}};
			\draw[->] (7/24, 1/11) -- (1/4, 10/72);
			\node[circle, fill, inner sep=1.5pt] at (1/4, 10/72) [above] {};
			\node[align=center] at (7/24, 1/20) {\(\frac{b}{10} + \frac{1}{10} < c \leq \frac{1}{6}\) \\ \(c \leq - \frac{b}{3} + \frac{1}{3}, b \neq c\)};
			\draw[->] (7/13, 1/11) -- (7.3/13, 9.4/60);
			\node[circle, fill, inner sep=1.5pt] at (7.3/13, 9.4/60) [above] {};
			\node[align=center] at (7/13, 1/20) {\(\frac{b}{10} + \frac{1}{10} < c \leq \frac{1}{6}\) \\ \(c > -\frac{b}{3} + \frac{1}{3}\)};
			
			\draw[very thick] (1/4, 1/4) -- (1/3, 1/3);
			\draw[->] (1/11, 1/3) -- (7/24, 7/24);
			\node at (1/10, 1/3) [left] {Sch $\frac{1}{4} < c \leq \frac{1}{3}$};		
			\node[circle, fill, inner sep=1pt] at (1/4, 1/4) {};
			\node[circle, fill, inner sep=1pt] at (1/3, 1/3) {};
			
			\draw[very thick] (1/6, 1/6) -- (1/4, 1/4);
			\draw[->] (1/11, 1/4) -- (5/24, 5/24);
			\node at (1/10, 1/4) [left] {Sch $\frac{1}{6} < c \leq \frac{1}{4}$};
			\node[circle, fill, inner sep=1pt] at (1/6, 1/6) {};
			
			\draw[very thick] (1/9, 1/9) -- (1/6, 1/6);
			\draw[->] (1/11, 1/6) -- (5/36, 5/36);
			\node at (1/10, 1/6) [left] {Sch $\frac{1}{9} < c \leq \frac{1}{6}$};
			
			\draw[dashed] (1/9, 1/9) -- (1, 1/5);
			\node[rotate=5.71] at (5/6, 11/60) [below] {\tiny$c=b/10 + 1/10$};
			\node at (1, 1/5) [right] {\tiny{$c=\frac{1}{5}$}};
			
		\end{tikzpicture}
	\end{center}
	The following is the wall and chamber decomposition for type $\widetilde{D}_{A_1}\cap \widetilde{E}_{2A_1} \cap \widetilde{D}_{3A_2}$ $(b,c)$-weighted stable marked cubic surfaces $(S, (b,c)B)$ where $\ell$ passes through exactly two nodes in the corresponding $\overline{Y}_\ell$ boundary. There are 18 walls. Along with the line $c = \frac{b}{3}$, the lines $c = \frac{1}{4}$ and $c = \frac{b}{4}$ are their own chambers.
	\begin{center}
		\begin{tikzpicture}[scale=14]
			\draw[very thick,-latex] (1/9, 1/9) -- (1.03, 1/9) 
			node[right]{$b$};
			
			\draw[very thick,-latex] (1/9, 1/9) -- (1/9, 1.03) 
			node[left]{$c$};
			
			\draw[solid] (1/9,1/9) -- (1,1);
			\node at (1/9, 1/9) [below] {\tiny$(\frac{1}{9}, \frac{1}{9})$};
			\node at (1, 1) [left] {$(1, 1)$};
			
			\draw[thick] (1, 1/9) -- (1,1);
			\node at (1, 1/9) [below] {\tiny$(1, \frac{1}{9})$};
			
			\draw[thick] (2/3, 2/3) -- (1, 2/3);
			\node at (1, 2/3) [right] {\tiny $c=2/3$};
			\node at (2/3, 2/3) [left] {\tiny$(\frac{2}{3}, \frac{2}{3})$};
			\node at (8/9, 7/9) {Sch $\frac{2}{3} < c \leq 1$};
			
			\draw[thick] (1/2, 1/2) -- (1, 1/2);
			\node at (0.8, 0.58) {Sch  $\frac{1}{2} < c \leq \frac{2}{3} $};
			
			\draw[thick] (1/2, 1/2) -- (4/5, 1/5);
			\node at (1/2, 1/2) [left] {\tiny$(\frac{1}{2}, \frac{1}{2})$};
			\node at (1, 1/2) [right] {\tiny$c=1/2$};
			\node[rotate=-45] at (2/3, 29/90) [above] {\tiny $c = -b+1$};
			\node[align=center] at (5/6, 0.42) {\(\frac{1}{3} < c \leq \frac{1}{2}\) \\ \(c > -b+1 \)};
			%
			\draw[thick] (1/4, 1/4) -- (7/13, 2/13);
			\node at (1/4, 1/4) [left] {\tiny$(\frac{1}{4}, \frac{1}{4})$};
			\node at (7/13, 2/13) [below] {\tiny{$(\frac{7}{13}, \frac{2}{13})$}};
			\node[rotate=-18.43] at (3/8, 24/120) [above] {\tiny$c=-\frac{b}{3} + \frac{1}{3}$};
			\node at (1.77/6, 4.55/24) {$\frac{1}{6} < c \leq \frac{1}{4}$, $b \neq c$};
			
			\draw[thick] (1/3, 1/3) -- (1,1/3);
			\node at (1, 1/3) [right] {\tiny$c=1/3$};
			\node at (1/3, 1/3) [left] {\tiny$(\frac{1}{3}, \frac{1}{3})$};
			\node[align=center] at (1/2, 6.5/24) {\( -\frac{b}{3} + \frac{1}{3} < c \leq \frac{1}{3}\)\\ \(\frac{b}{3} < c \leq -b+1, b \neq c \)};
			\node[align=center] at (5.2/6, 7/24) {\(\frac{1}{4} < c \leq \frac{1}{3}\)\\ \(c > -b+1\)};
			\node[align=center] at (1/2, 4.7/12) {Sch $\frac{1}{3} < c \leq \frac{1}{2}$};
			%
			\draw[thick] (1/6, 1/6) -- (2/3, 1/6);
			\node at (7/36, 7/36) [left] {\tiny$(\frac{1}{6}, \frac{1}{6}$)};
			\draw[->] (7/24, 1/11) -- (1/4, 10/72);
			\node[circle, fill, inner sep=1pt] at (1/4, 10/72) [above] {};
			\node[align=center] at (7/24, 1/20) {\(\frac{b}{10} + \frac{1}{10} < c \leq \frac{1}{6}\) \\ \(c \leq - \frac{b}{3} + \frac{1}{3}, b \neq c\)};
			%
			\draw[very thick] (3/4, 1/4) -- (1, 1/4);
			\node at (1, 1/4) [right] {\tiny{$c=1/4$}};
			%
			\draw[thick] (4/5, 1/5) -- (1, 1/5);
			\node at (1, 1/5) [right] {\tiny{$c=1/5$}};
			%
			\draw[very thick] (1/2, 1/6) -- (3/4, 1/4);
			\node[rotate=18.43] at (2/3, 17/81) [above] {\tiny$c=\frac{b}{3}$};
			%
			\draw[very thick] (2/3, 1/6) -- (4/5, 1/5);
			\node at (2/3, 1/6) [below] {\tiny{$(\frac{2}{3}, \frac{1}{6})$}};
			\node[rotate=14.04] at (22/30, 21/120) [above] {\tiny $c=\frac{b}{4}$};
			%
			
			\draw[very thick] (1/4, 1/4) -- (1/3, 1/3);
			\draw[->] (1/11, 1/3) -- (7/24, 7/24);
			\node at (1/10, 1/3) [left] {Sch $\frac{1}{4} < c \leq \frac{1}{3}$};		
			\node[circle, fill, inner sep=1pt] at (1/4, 1/4) {};
			\node[circle, fill, inner sep=1pt] at (1/3, 1/3) {};
			
			\draw[very thick] (1/6, 1/6) -- (1/4, 1/4);
			\draw[->] (1/11, 1/4) -- (5/24, 5/24);
			\node at (1/10, 1/4) [left] {Sch $\frac{1}{6} < c \leq \frac{1}{4}$};
			\node[circle, fill, inner sep=1pt] at (1/6, 1/6) {};
			
			\draw[->] (1.1, 1/6) -- (5.2/6, 10.5/48);
			\node[circle, fill, inner sep=1.5pt] at (5.15/6, 10.5/48) [above] {};
			\node[align=left] at (1.16, 1/6) {\(\frac{1}{5} < c < \frac{1}{4}\) \\ \(c > -b+1\)};
			
			\draw[->] (1.05, .85/3) -- (8.5/13, 1/5);
			\node[circle, fill, inner sep=1.5pt] at (8.5/13, 1/5) [left] {};
			\node[align=center] at (1.15, .85/3) {\(\frac{1}{6} < c \leq -b+1\) \\ \(\frac{b}{4} < c < \frac{b}{3}\)};
			
			\draw[->] (7/13, 1/11) -- (7.3/13, 9.4/60);
			\node[circle, fill, inner sep=1.5pt] at (7.3/13, 9.4/60) [above] {};
			\node[align=center] at (7/13, 1/20) {\(\frac{b}{10} + \frac{1}{10} < c \leq \frac{1}{6}\) \\ \(c > -\frac{b}{3} + \frac{1}{3}\)};
			
			\draw[->] (10/13, 1/11) -- (10.5/13, 1.1/6);
			\node[circle, fill, inner sep=1.5pt] at (10.5/13, 1.1/6) [above] {};
			\node[align=center] at (10/13, 1/20) {\(\frac{b}{10} + \frac{1}{10} < c \leq \frac{1}{5}\) \\ \(c < \frac{b}{4}\)};
			
			\draw[very thick] (1/9, 1/9) -- (1/6, 1/6);
			\draw[->] (1/11, 1/6) -- (5/36, 5/36);
			\node at (1/10, 1/6) [left] {Sch $\frac{1}{9} < c \leq \frac{1}{6}$};
			
			\draw[dashed] (1/9, 1/9) -- (1, 1/5);
			\node[rotate=5.71] at (5/6, 11/60) [below] {\tiny$c=b/10 + 1/10$};
		\end{tikzpicture}
	\end{center}
\end{pro}

\begin{table}[H]
	\centering
	\label{tab:E2A1D3A2eck}
	\begin{tabular}{| c | c | c | c |}
		\hline
		Label & Surface & \# & Eckardt points \\
		\hline
		\hline
		1 & $\mathrm{Bl}_5\mathbb{P}^2$ & 2 & 0 \\
		2 & $\mathrm{Bl}_4\mathbb{P}^2$ & 1 & 0 \\
		3a & $\mathrm{Bl}_5\mathbb{P}^2$ & 2 & 0 \\
		3b & $\mathrm{Bl}_4\mathbb{P}^2$ & 2 & 0 \\
		3c & $\mathbb{F}_0$ & 4 & 0 \\
		3d & $\mathbb{F}_0$ & 4 & 0 \\
		3e & $\mathbb{F}_0$ & 4 & 0 \\
		4 & $\mathrm{Bl}_5 \mathbb{P}^2$ & 1 & 0 or 1 \\
		5a & $\mathbb{F}_0$ & 3 & 0 \\
		5b & $\mathbb{F}_0$ & 6 & 0 \\
		5c & $\mathbb{F}_0$ & 3 & 0 \\
		6a & $\mathbb{F}_0$ & 2 & 0 \\
		6b & $\mathbb{F}_0$ & 2 & 0 \\
		7 & $\mathbb{F}_0$ & 6 & 0 \\
		8a & $\mathbb{F}_0$ & 2 & 0 \\
		8b & $\mathbb{F}_0$ & 1 & 0 \\
		\hline
	\end{tabular}
	\caption{This table is taken from \cite[Table~8]{schock_moduli_2024} The first two columns are the numbered types of irreducible components of the surfaces in the chamber Sch $1/2 < c \leq 2/3$ of type $\widetilde{E}_{2A_1} \cap \widetilde{D}_{3A_2}$ pictured in \cite[Fig.~35]{schock_moduli_2024}. The third column tells us the number of each corresponding component found in the surface. Finally, the last row gives the possible numbers of Eckardt points on each component. For the components of types 1 and 3a, $\mathrm{Bl}_5\mathbb{P}^2$ refers to the blowup of $\mathbb{P}^2$ at 3 points on one line and 2 points on another line. For the components of types 2 and 3b, $\mathrm{Bl}_4 \mathbb{P}^2$ refers to the blowup of $\mathbb{P}^2$ at 2 points on one line and 2 points on another line (i.e., 4 points in general position). For the component of type 4, $\mathrm{Bl}_5 \mathbb{P}^2$ refers to the blowup of $\mathbb{P}^2$ at 2 points on one line, 2 points on another line, and at the intersection point of these two lines.}
\end{table}

\newpage
\subsection{Type $\widetilde{E}_{3A_1} \cap \widetilde{D}_{3A_2}$}
\subsubsection{\textbf{$\ell$ is in the smooth locus}}
\begin{pro}
	For type $\widetilde{E}_{3A_1} \cap \widetilde{D}_{3A_2}$ $(b,c)$-weighted stable marked cubic surfaces $(S, (b,c)B)$ where $\ell$ is in the smooth, the walls are the same as in Proposition \ref{E2A1D3A2_mult1}. Additionally, crossing the wall $c = \frac{1}{4}$ introduces type 
	$$\widetilde{D}_{A_1}\cap \widetilde{E}_{3A_1} \cap \widetilde{D}_{3A_2}, \hspace{.25cm} \widetilde{E}_{2A_1} \cap \widetilde{E}_{3A_1} \cap \widetilde{D}_{3A_2}, \hspace{.25cm} \widetilde{D}_{A_1}\cap \widetilde{E}_{2A_1} \cap \widetilde{E}_{3A_1} \cap \widetilde{D}_{3A_2}$$
	surfaces as degenerations of type $\widetilde{E}_{3A_1} \cap \widetilde{D}_{3A_2}$ surfaces.
\end{pro}

\subsubsection{\textbf{$\ell$ passes through exactly one node}}
\begin{pro}
	For type $\widetilde{E}_{3A_1}\cap\widetilde{D}_{3A_2}$ $(b,c)$-weighted stable marked cubic surfaces $(S, (b,c)B)$ where $\ell$ is passing through exactly one node in the corresponding $\overline{Y}_\ell$ boundary, the walls are the same as in Proposition \ref{E2A1D3A2_mult2}. Additionally, crossing the wall $c = \frac{1}{4}$ introduces type
	$$\widetilde{D}_{A_1}\cap \widetilde{E}_{3A_1} \cap \widetilde{D}_{3A_2}, \hspace{.25cm} \widetilde{E}_{2A_1} \cap \widetilde{E}_{3A_1} \cap \widetilde{D}_{3A_2}, \hspace{.25cm} \widetilde{D}_{A_1}\cap \widetilde{E}_{2A_1} \cap \widetilde{E}_{3A_1} \cap \widetilde{D}_{3A_2}$$ 
	surfaces as degenerations of type $\widetilde{E}_{3A_1} \cap \widetilde{D}_{3A_2}$ surfaces.
\end{pro}

\subsubsection{\textbf{$\ell$ passes through exactly two nodes}}
\begin{pro}
	For type $\widetilde{E}_{3A_1} \cap \widetilde{D}_{3A_2}$ $(b,c)$-weighted stable marked cubic surfaces $(S, (b,c)B)$ where $\ell$ passes through exactly two nodes in the corresponding $\overline{Y}_\ell$ boundary, there are 16 chambers. The line $c = \frac{b}{3}$ is its own chamber. Additionally, crossing the wall $c = -\frac{b}{3} + \frac{1}{3}$ introduces type
	$$\widetilde{D}_{A_1}\cap \widetilde{E}_{3A_1} \cap \widetilde{D}_{3A_2}, \hspace{.25cm} \widetilde{E}_{2A_1} \cap \widetilde{E}_{3A_1} \cap \widetilde{D}_{3A_2}, \hspace{.25cm} \widetilde{D}_{A_1}\cap \widetilde{E}_{2A_1} \cap \widetilde{E}_{3A_1} \cap \widetilde{D}_{3A_2}$$
	surfaces as degerenations of type $\widetilde{E}_{3A_1} \cap \widetilde{D}_{3A_2}$. In this case, if $\ell$ is in the component that is the degeneration of $\mathrm{Bl}_5 \mathbb{P}^2$, the wall and chamber decomposition has 19 chambers and explained more below. Otherwise, the wall and chamber decomposition is the same as type $\widetilde{E}_{3A_1} \cap \widetilde{D}_{3A_2}$.
	\begin{center}
		\begin{tikzpicture}[scale=14]
			\draw[very thick,-latex] (1/9, 1/9) -- (1.03, 1/9) 
			node[right]{$b$};
			
			\draw[very thick,-latex] (1/9, 1/9) -- (1/9, 1.03) 
			node[left]{$c$};
			
			\draw[solid] (1/9,1/9) -- (1,1);
			\node at (1/9, 1/9) [below] {\tiny $(\frac{1}{9}, \frac{1}{9})$};
			\node at (1, 1) [left] {$(1, 1)$};
			
			\draw[thick] (1, 1/9) -- (1,1);
			\node at (1, 1/9) [below] {$(1, \frac{1}{9})$};
			
			\draw[thick] (2/3, 2/3) -- (1, 2/3) node[right] {\tiny{$c=\frac{2}{3}$}};
			\node[circle, fill, inner sep=1pt] at (2/3, 2/3) {};
			\node at (2/3, 2/3) [left] {\tiny$(\frac{2}{3}, \frac{2}{3})$};
			\node at (8/9, 7/9) {Sch $\frac{2}{3} < c \leq 1$};
			\node at (0.9, 0.63) {$-\frac{b}{2}+1 < c \leq \frac{2}{3}$};
			
			\draw[solid] (2/3, 2/3) -- (1, 1/2) node[right] {\color{black} \tiny{$c = \frac{1}{2}$}};
			\node[rotate=-26.6] at (5/6, 57/96) {\tiny wall due to $\ell$ containing Eckardt point(s)};
			
			\draw[thick] (1/2, 1/2) -- (1, 1/2);
			\node[circle, fill, inner sep=1pt] at (1/2, 1/2) {};
			\node at (1/2, 1/2) [left] {\tiny $(\frac{1}{2}, \frac{1}{2})$};
			\node at (0.7, 0.56) {Sch $\frac{1}{2} < c \leq \frac{2}{3}$};
			
			\draw[thick] (1/2, 1/2) -- (3/4, 1/4);
			\node[rotate=-45] at (7/12, 5/12) [above] {\tiny $c = -b+1$};
			\node[align=center] at (5/6, 0.42) {\(\frac{1}{3} < c \leq \frac{1}{2}\) \\ \(c > -b+1 \)};
			
			\draw[thick] (1/4, 1/4) -- (7/13, 2/13);
			\node[circle, fill, inner sep=1pt] at (1/4, 1/4) {};
			\node[rotate=-18.43] at (3/8, 24/120) [above] {\tiny$c=-\frac{b}{3} + \frac{1}{3}$};
			\node at (7/13, 2/13) [below] {\tiny{$(\frac{7}{13}, \frac{2}{13})$}};
			\node at (1/4, 1/4) [left] {\tiny$(\frac{1}{4}, \frac{1}{4})$};
			\node at (1.77/6, 4.55/24) {$\frac{1}{6} < c \leq \frac{1}{4}$, $b \neq c$};
			
			\draw[thick] (1/3, 1/3) -- (1,1/3);
			\node at (1, 1/3) [right] {\tiny$c=1/3$};
			\node at (1/3, 1/3) [left] {\tiny$(\frac{1}{3}, \frac{1}{3})$};
			\node[align=center] at (5.2/6, 7/24) {\(\frac{1}{4} < c \leq \frac{1}{3}\)\\ \(c > -b+1\)};
			\node[align=center] at (1/2, 4.7/12) {Sch $\frac{1}{3} < c \leq \frac{1}{2}$};
			\node[align=center] at (1/2, 7/24) {\( \frac{1}{4} < c \leq \frac{1}{3}\)\\ \(c \leq -b+1, b \neq c \)};
			
			\draw[thick] (1/4, 1/4) -- (1,1/4);
			\node at (1, 1/4) [right] {\tiny$c=1/4$};
			\node at (1/4, 1/4) [left] {\tiny$(\frac{1}{4}, \frac{1}{4})$};
			\node[align=center] at (6.5/13, 2.5/12) {\tiny \(-\frac{b}{3} + \frac{1}{3} < c \leq \frac{1}{4}\) \\ \tiny \(c > \frac{b}{3}\)};
			\node[align=left] at (5.2/6, 10.5/48) {\tiny \(\frac{b}{10} + \frac{1}{10} < c \leq \frac{1}{4}\) \\ \tiny \(\frac{1}{6} < c < \frac{b}{3}\)};
			
			\draw[very thick] (1/2, 1/6) -- (3/4, 1/4);
			\node[rotate=18.43] at (2/3, 17/81) [above] {\tiny$c=\frac{b}{3}$};
			
			\draw[thick] (1/6, 1/6) -- (2/3, 1/6);
			\node[circle, fill, inner sep=1pt] at (1/6, 1/6) {};
			\node at (7/36, 7/36) [left] {\tiny$(\frac{1}{6}, \frac{1}{6}$)};
			\node at (2/3, 1/6) [below] {\tiny{$(\frac{2}{3}, \frac{1}{6})$}};
			\draw[->] (7/24, 1/11) -- (1/4, 10/72);
			\node[circle, fill, inner sep=1.5pt] at (1/4, 10/72) [above] {};
			\node[align=center] at (7/24, 1/20) {\(\frac{b}{10} + \frac{1}{10} < c \leq \frac{1}{6}\) \\ \(c \leq - \frac{b}{3} + \frac{1}{3}, b \neq c\)};
			\draw[->] (7/13, 1/11) -- (7.3/13, 9.4/60);
			\node[circle, fill, inner sep=1.5pt] at (7.3/13, 9.4/60) [above] {};
			\node[align=center] at (7/13, 1/20) {\(\frac{b}{10} + \frac{1}{10} < c \leq \frac{1}{6}\) \\ \(c > -\frac{b}{3} + \frac{1}{3}\)};
			
			\draw[very thick] (1/4, 1/4) -- (1/3, 1/3);
			\draw[->] (1/11, 1/3) -- (7/24, 7/24);
			\node at (1/10, 1/3) [left] {Sch $\frac{1}{4} < c \leq \frac{1}{3}$};		
			\node[circle, fill, inner sep=1pt] at (1/4, 1/4) {};
			\node[circle, fill, inner sep=1pt] at (1/3, 1/3) {};
			
			\draw[very thick] (1/6, 1/6) -- (1/4, 1/4);
			\draw[->] (1/11, 1/4) -- (5/24, 5/24);
			\node at (1/10, 1/4) [left] {Sch $\frac{1}{6} < c \leq \frac{1}{4}$};
			\node[circle, fill, inner sep=1pt] at (1/6, 1/6) {};
			
			\draw[very thick] (1/9, 1/9) -- (1/6, 1/6);
			\draw[->] (1/11, 1/6) -- (5/36, 5/36);
			\node at (1/10, 1/6) [left] {Sch $\frac{1}{9} < c \leq \frac{1}{6}$};
			
			\draw[dashed] (1/9, 1/9) -- (1, 1/5);
			\node[rotate=5.71] at (5/6, 11/60) [below] {\tiny$c=b/10 + 1/10$};
			\node at (1, 1/5) [right] {\tiny{$c=\frac{1}{5}$}};
			
		\end{tikzpicture}
	\end{center}
	The following is the wall and chamber decomposition for type
	$$\widetilde{D}_{A_1}\cap \widetilde{E}_{3A_1} \cap \widetilde{D}_{3A_2}, \hspace{.25cm} \widetilde{E}_{2A_1} \cap \widetilde{E}_{3A_1} \cap \widetilde{D}_{3A_2}, \hspace{.25cm} \widetilde{D}_{A_1}\cap \widetilde{E}_{2A_1} \cap \widetilde{E}_{3A_1} \cap \widetilde{D}_{3A_2}$$ 
	$(b,c)$-weighted stable marked cubic surfaces $(S, (b,c)B)$ where $\ell$ passes through exactly two nodes in the corresponding $\overline{Y}_\ell$ boundary. There are 19 chambers. Along with the line $c = \frac{b}{3}$, the lines $c = \frac{1}{4}$ $(\frac{2}{3} < b \leq 1)$ and $c = \frac{b}{4}$ are their own chambers.
	\begin{center}
		\begin{tikzpicture}[scale=14]
			\draw[very thick,-latex] (1/9, 1/9) -- (1.03, 1/9) 
			node[right]{$b$};
			
			\draw[very thick,-latex] (1/9, 1/9) -- (1/9, 1.03) 
			node[left]{$c$};
			
			\draw[solid] (1/9,1/9) -- (1,1);
			\node at (1/9, 1/9) [below] {\tiny $(\frac{1}{9}, \frac{1}{9})$};
			\node at (1, 1) [left] {$(1, 1)$};
			
			\draw[thick] (1, 1/9) -- (1,1);
			\node at (1, 1/9) [below] {$(1, \frac{1}{9})$};
			
			\draw[thick] (2/3, 2/3) -- (1, 2/3) node[right] {\tiny{$c=\frac{2}{3}$}};
			\node[circle, fill, inner sep=1pt] at (2/3, 2/3) {};
			\node at (2/3, 2/3) [left] {\tiny$(\frac{2}{3}, \frac{2}{3})$};
			\node at (8/9, 7/9) {Sch $\frac{2}{3} < c \leq 1$};
			
			\draw[thick] (1/2, 1/2) -- (1, 1/2);
			\node[circle, fill, inner sep=1pt] at (1/2, 1/2) {};
			\node at (1/2, 1/2) [left] {\tiny $(\frac{1}{2}, \frac{1}{2})$};
			\node at (0.8, 0.58) {Sch $\frac{1}{2} < c \leq \frac{2}{3}$};
			
			\draw[thick] (1/2, 1/2) -- (4/5, 1/5);
			\node[rotate=-45] at (7/12, 5/12) [above] {\tiny $c = -b+1$};
			\node[align=center] at (5/6, 0.42) {\(\frac{1}{3} < c \leq \frac{1}{2}\) \\ \(c > -b+1 \)};
			
			\draw[thick] (1/4, 1/4) -- (7/13, 2/13);
			\node[circle, fill, inner sep=1pt] at (1/4, 1/4) {};
			\node[rotate=-18.43] at (3/8, 24/120) [above] {\tiny$c=-\frac{b}{3} + \frac{1}{3}$};
			\node at (7/13, 2/13) [below] {\tiny{$(\frac{7}{13}, \frac{2}{13})$}};
			\node at (1/4, 1/4) [left] {\tiny$(\frac{1}{4}, \frac{1}{4})$};
			\node at (1.77/6, 4.55/24) {$\frac{1}{6} < c \leq \frac{1}{4}$, $b \neq c$};
			
			\draw[thick] (1/3, 1/3) -- (1,1/3);
			\node at (1, 1/3) [right] {\tiny$c=1/3$};
			\node at (1/3, 1/3) [left] {\tiny$(\frac{1}{3}, \frac{1}{3})$};
			\node[align=center] at (5.2/6, 7/24) {\(\frac{1}{4} < c \leq \frac{1}{3}\)\\ \(c > -b+1\)};
			\node[align=center] at (1/2, 4.7/12) {Sch $\frac{1}{3} < c \leq \frac{1}{2}$};
			\node[align=center] at (1/2, 7/24) {\( \frac{1}{4} < c \leq \frac{1}{3}\)\\ \(c \leq -b+1, b \neq c \)};
			
			\draw[thick] (1/4, 1/4) -- (3/4,1/4);
			\node at (1, 1/4) [right] {\tiny$c=1/4$};
			\node at (1/4, 1/4) [left] {\tiny$(\frac{1}{4}, \frac{1}{4})$};
			\node[align=center] at (6.5/13, 2.5/12) {\tiny \(-\frac{b}{3} + \frac{1}{3} < c \leq \frac{1}{4}\) \\ \tiny \(c > \frac{b}{3}\)};
			
			\draw[thick] (1/6, 1/6) -- (2/3, 1/6);
			\node[circle, fill, inner sep=1pt] at (1/6, 1/6) {};
			\node at (7/36, 7/36) [left] {\tiny$(\frac{1}{6}, \frac{1}{6}$)};
			\node at (2/3, 1/6) [below] {\tiny{$(\frac{2}{3}, \frac{1}{6})$}};
			\draw[->] (7/24, 1/11) -- (1/4, 10/72);
			\node[circle, fill, inner sep=1.5pt] at (1/4, 10/72) [above] {};
			
			\draw[very thick] (1/4, 1/4) -- (1/3, 1/3);
			\draw[->] (1/11, 1/3) -- (7/24, 7/24);
			\node at (1/10, 1/3) [left] {Sch $\frac{1}{4} < c \leq \frac{1}{3}$};		
			\node[circle, fill, inner sep=1pt] at (1/4, 1/4) {};
			\node[circle, fill, inner sep=1pt] at (1/3, 1/3) {};
			
			\draw[very thick] (1/6, 1/6) -- (1/4, 1/4);
			\draw[->] (1/11, 1/4) -- (5/24, 5/24);
			\node at (1/10, 1/4) [left] {Sch $\frac{1}{6} < c \leq \frac{1}{4}$};
			\node[circle, fill, inner sep=1pt] at (1/6, 1/6) {};
			%
			\draw[thick] (1/6, 1/6) -- (2/3, 1/6);
			\node at (7/36, 7/36) [left] {\tiny$(\frac{1}{6}, \frac{1}{6}$)};
			\draw[->] (7/24, 1/11) -- (1/4, 10/72);
			\node[circle, fill, inner sep=1pt] at (1/4, 10/72) [above] {};
			\node[align=center] at (7/24, 1/20) {\(\frac{b}{10} + \frac{1}{10} < c \leq \frac{1}{6}\) \\ \(c \leq - \frac{b}{3} + \frac{1}{3}, b \neq c\)};
			%
			\draw[thick] (1/4, 1/4) -- (3/4, 1/4);
			\draw[very thick] (3/4, 1/4) -- (1, 1/4);
			\node[align=center] at (6.5/13, 2.5/12) {\tiny \(-\frac{b}{3} + \frac{1}{3} < c \leq \frac{1}{4}\) \\ \tiny \(c > \frac{b}{3}\)};
			\node at (1, 1/4) [right] {\tiny{$c=1/4$}};
			%
			\draw[thick] (4/5, 1/5) -- (1, 1/5);
			\node at (1, 1/5) [right] {\tiny{$c=1/5$}};
			%
			\draw[very thick] (1/2, 1/6) -- (3/4, 1/4);
			\node[rotate=18.43] at (2/3, 17/81) [above] {\tiny$c=\frac{b}{3}$};
			%
			\draw[very thick] (2/3, 1/6) -- (4/5, 1/5);
			\node at (2/3, 1/6) [below] {\tiny{$(\frac{2}{3}, \frac{1}{6})$}};
			\node[rotate=14.04] at (22/30, 21/120) [above] {\tiny $c=\frac{b}{4}$};
			
			\draw[->] (1.1, 1/6) -- (5.2/6, 10.5/48);
			\node[circle, fill, inner sep=1.5pt] at (5.15/6, 10.5/48) [above] {};
			\node[align=left] at (1.16, 1/6) {\(\frac{1}{5} < c < \frac{1}{4}\) \\ \(c > -b+1\)};
			
			\draw[->] (1.05, .85/3) -- (8.5/13, 1/5);
			\node[circle, fill, inner sep=1.5pt] at (8.5/13, 1/5) [left] {};
			\node[align=center] at (1.15, .85/3) {\(\frac{1}{6} < c \leq -b+1\) \\ \(\frac{b}{4} < c < \frac{b}{3}\)};
			
			\draw[->] (7/13, 1/11) -- (7.3/13, 9.4/60);
			\node[circle, fill, inner sep=1.5pt] at (7.3/13, 9.4/60) [above] {};
			\node[align=center] at (7/13, 1/20) {\(\frac{b}{10} + \frac{1}{10} < c \leq \frac{1}{6}\) \\ \(c > -\frac{b}{3} + \frac{1}{3}\)};
			
			\draw[->] (10/13, 1/11) -- (10.5/13, 1.1/6);
			\node[circle, fill, inner sep=1.5pt] at (10.5/13, 1.1/6) [above] {};
			\node[align=center] at (10/13, 1/20) {\(\frac{b}{10} + \frac{1}{10} < c \leq \frac{1}{5}\) \\ \(c < \frac{b}{4}\)};
			
			\draw[very thick] (1/9, 1/9) -- (1/6, 1/6);
			\draw[->] (1/11, 1/6) -- (5/36, 5/36);
			\node at (1/10, 1/6) [left] {Sch $\frac{1}{9} < c \leq \frac{1}{6}$};
			
			\draw[dashed] (1/9, 1/9) -- (1, 1/5);
			\node[rotate=5.71] at (5/6, 11/60) [below] {\tiny$c=b/10 + 1/10$};
			\node at (1, 1/5) [right] {\tiny{$c=\frac{1}{5}$}};
			
		\end{tikzpicture}
	\end{center}
\end{pro}

\begin{table}[H]
	\centering
	\begin{tabular}{| c | c | c | c |}
		\hline
		Label & Surface & \# & Eckardt points \\
		\hline
		\hline
		1 & $\mathrm{Bl}_4 \mathbb{P}^2$ & 3 & 0 \\
		2a & $\mathrm{Bl}_4 \mathbb{P}^2$ & 6 & 0 \\
		2b & $\mathbb{F}_0$ & 6 & 0 \\
		2c & $\mathbb{F}_0$ & 12 & 0 \\
		2d & $\mathbb{F}_0$ & 12 & 0 \\
		2e & $\mathbb{F}_0$ & 6 & 0 \\
		3 & $\mathrm{Bl}_5\mathbb{P}^2$ & 3 & 0 or 1 \\
		4a & $\mathbb{F}_0$ & 3 & 0 \\
		4b & $\mathbb{F}_0$ & 6 & 0 \\
		4c & $\mathbb{F}_0$ & 3 & 0 \\
		5a & $\mathbb{F}_0$ & 6 & 0 \\
		5b & $\mathbb{F}_0$ & 3 & 0 \\
		\hline
	\end{tabular}
	\caption{This table is taken from \cite[Table~9]{schock_moduli_2024} The first two columns are the numbered types of irreducible components of the surfaces in the chamber Sch $1/2 < c \leq 2/3$ of type $\widetilde{E}_{3A_1} \cap \widetilde{D}_{3A_2}$ pictured in \cite[Fig.~43]{schock_moduli_2024}. The third column tells us the number of each corresponding component found in the surface. Finally, the last row gives the possible numbers of Eckardt points on each component. For the components of types 1 and 2a, $\mathrm{Bl}_4\mathbb{P}^2$ refers to the blowup of $\mathbb{P}^2$ at 2 points on one line and 2 points on another line (i.e., 4 points in general position). For the components of type 3, $\mathrm{Bl}_5 \mathbb{P}^2$ refers to the blowup of $\mathbb{P}^2$ at 2 points on one line, 2 points on another line, and at the intersection point of these two lines.}
	\label{tab:E3A1D3A2eck}
\end{table}
	
	\else
	
	\fi
	
	\bibliographystyle{amsalpha}
	\bibliography{cubics}

@book{kollar_birational_1998,
	address = {Cambridge},
	series = {Cambridge {Tracts} in {Mathematics}},
	title = {Birational {Geometry} of {Algebraic} {Varieties}},
	isbn = {978-0-521-63277-5},
	url = {https://www.cambridge.org/core/books/birational-geometry-of-algebraic-varieties/FB73FD7FDC65F023521ED6B1BB7F3EC7},
	urldate = {2025-08-21},
	publisher = {Cambridge University Press},
	author = {Kollár, Janos and Mori, Shigefumi},
	year = {1998},
	doi = {10.1017/CBO9780511662560},
}

@misc{casalaina-martin_birational_2024,
	title = {The birational geometry of moduli of cubic surfaces and cubic surfaces with a line},
	url = {http://arxiv.org/abs/2312.15369},
	doi = {10.48550/arXiv.2312.15369},
	urldate = {2024-11-13},
	publisher = {arXiv},
	author = {Casalaina-Martin, Sebastian and Grushevsky, Samuel and Hulek, Klaus},
	month = oct,
	year = {2024},
	note = {arXiv:2312.15369},
}

@article{hacking_stable_2009,
	title = {Stable pair, tropical, and log canonical compactifications of moduli spaces of del {Pezzo} surfaces},
	volume = {178},
	issn = {1432-1297},
	url = {https://doi.org/10.1007/s00222-009-0199-1},
	doi = {10.1007/s00222-009-0199-1},
	language = {en},
	number = {1},
	urldate = {2024-11-13},
	journal = {Inventiones mathematicae},
	author = {Hacking, Paul and Keel, Sean and Tevelev, Jenia},
	month = oct,
	year = {2009},
	pages = {173--227},
}

@article{allcock_complex_2002,
	title = {The complex hyperbolic geometry of the moduli space of cubic surfaces},
	volume = {11},
	issn = {1056-3911, 1534-7486},
	url = {https://www.ams.org/jag/2002-11-04/S1056-3911-02-00314-4/},
	doi = {10.1090/S1056-3911-02-00314-4},
	language = {en},
	number = {4},
	urldate = {2024-11-13},
	journal = {Journal of Algebraic Geometry},
	author = {Allcock, Daniel and Carlson, James and Toledo, Domingo},
	year = {2002},
	pages = {659--724},
}

@article{cayley1849triple,
  title={On the triple tangent planes to a surface of the third order},
  author={Cayley, Arthur},
  journal={Cambridge and Dublin Math},
  volume={4},
  pages={118--132},
  year={1849}
}

@article{naruki_cross_1982,
	title = {Cross {Ratio} {Variety} as a {Moduli} {Space} of {Cubic} {Surfaces}},
	volume = {s3-45},
	copyright = {© 1982 London Mathematical Society},
	issn = {1460-244X},
	url = {https://onlinelibrary.wiley.com/doi/abs/10.1112/plms/s3-45.1.1},
	doi = {10.1112/plms/s3-45.1.1},
	language = {en},
	number = {1},
	urldate = {2025-02-21},
	journal = {Proceedings of the London Mathematical Society},
	author = {Naruki, Isao},
	year = {1982},
	note = {\_eprint: https://onlinelibrary.wiley.com/doi/pdf/10.1112/plms/s3-45.1.1},
	pages = {1--30},
}

@book{kollar_families_2023,
	address = {Cambridge},
	series = {Cambridge {Tracts} in {Mathematics}},
	title = {Families of {Varieties} of {General} {Type}},
	isbn = {978-1-009-34610-8},
	url = {https://www.cambridge.org/core/books/families-of-varieties-of-general-type/423EBD2D90435E5EA8AAC69CE0E6EFD1},
	urldate = {2025-02-21},
	publisher = {Cambridge University Press},
	author = {Kollár, János},
	year = {2023},
	doi = {10.1017/9781009346115},
}

@article{tu_semi-stable_2005,
  title     = {On Semi-stable, Singular Cubic Surfaces},
  author    = {Tu, Nguyen Chanh},
  journal   = {S{\'e}minaires et Congr{\`e}s},
  volume    = {10},
  pages     = {373--389},
  year      = {2005},
  publisher = {Soci{\'e}t{\'e} Math{\'e}matique de France},
  url       = {https://smf.emath.fr/sites/default/files/2017-08/smf_sem-cong_10_373-389__sample.pdf}
}

@article{sekiguchi_cross_2000,
	title = {Cross ratio varieties for root systems {II}: {The} case of the root system of type {E7}},
	volume = {54},
	doi = {10.2206/kyushujm.54.7},
	number = {1},
	journal = {Kyushu Journal of Mathematics},
	author = {Sekiguchi, Jiro},
	year = {2000},
	pages = {7--37},
}

@article{ascher_wall_2023,
	title = {Wall crossing for moduli of stable log pairs},
	volume = {198},
	issn = {0003-486X, 1939-8980},
	url = {https://projecteuclid.org/journals/annals-of-mathematics/volume-198/issue-2/Wall-crossing-for-moduli-of-stable-log-pairs/10.4007/annals.2023.198.2.7.full},
	doi = {10.4007/annals.2023.198.2.7},
	number = {2},
	urldate = {2025-07-29},
	journal = {Annals of Mathematics},
	author = {Ascher, Kenneth and Bejleri, Dori and Inchiostro, Giovanni and Patakfalvi, Zsolt},
	month = sep,
	year = {2023},
	pages = {825--866},
}

@article{gallardo_geometric_2021,
	title = {Geometric interpretation of toroidal compactifications of moduli of points in the line and cubic surfaces},
	volume = {381},
	issn = {0001-8708},
	url = {https://www.sciencedirect.com/science/article/pii/S0001870821000700},
	doi = {10.1016/j.aim.2021.107632},
	urldate = {2025-08-21},
	journal = {Advances in Mathematics},
	author = {Gallardo, Patricio and Kerr, Matt and Schaffler, Luca},
	month = apr,
	year = {2021},
	pages = {107632},
}

@book{kollar_singularities_2013,
	address = {Cambridge},
	series = {Cambridge {Tracts} in {Mathematics}},
	title = {Singularities of the {Minimal} {Model} {Program}},
	isbn = {978-1-107-03534-8},
	url = {https://www.cambridge.org/core/books/singularities-of-the-minimal-model-program/BFD87F262E5174A425F12AD85270E0AA},
	urldate = {2025-08-21},
	publisher = {Cambridge University Press},
	author = {Kollár, János},
	year = {2013},
	doi = {10.1017/CBO9781139547895},
}

@misc{schock_moduli_2024,
	title = {Moduli of weighted stable marked cubic surfaces},
	url = {http://arxiv.org/abs/2305.06922},
	doi = {10.48550/arXiv.2305.06922},
	urldate = {2025-08-21},
	publisher = {arXiv},
	author = {Schock, Nolan},
	month = apr,
	year = {2024},
	note = {arXiv:2305.06922 [math]},
}

@book{oxbury_introduction_2003,
	address = {Cambridge},
	series = {Cambridge {Studies} in {Advanced} {Mathematics}},
	title = {An {Introduction} to {Invariants} and {Moduli}},
	isbn = {978-0-521-80906-1},
	url = {https://www.cambridge.org/core/books/an-introduction-to-invariants-and-moduli/FD47BBB910AA000A98E491D105185928},
	urldate = {2025-02-21},
	publisher = {Cambridge University Press},
	author = {Mukai, Shigeru},
	translator = {Oxbury, W. M.},
	year = {2003},
	doi = {10.1017/CBO9781316257074},
}

@book{manin1986cubic,
  title={Cubic forms: algebra, geometry, arithmetic},
  author={Manin, Yuri},
  volume={4},
  year={1986},
  publisher={Elsevier}
}

@article{dolgachev2005complex,
  title={A complex ball uniformization of the moduli space of cubic surfaces via periods of K3 surfaces},
  author={Dolgachev, Igor and van Geemen, Bert and Kond{\=o}, S},
  year={2005},
  volume = {588},
  journal={J. Reine Angew. Math.},
  pages = {99-148}
}

@book{salmon1912treatise,
  title={A treatise on the analytic geometry of three dimensions},
  author={Salmon, George},
  volume={1},
  year={1912},
  publisher={Longmans, Green and Company}
}

@article{borel1972some,
	title={Some metric properties of arithmetic quotients of symmetric spaces and an extension theorem},
	author={Borel, Armand},
	journal={Journal of Differential Geometry},
	volume={6},
	number={4},
	pages={543--560},
	year={1972},
	publisher={Lehigh University}
}

@article{looijenga1981rational,
  title={Rational surfaces with an anti-canonical cycle},
  author={Looijenga, Eduard},
  journal={Annals of Mathematics},
  volume={114},
  number={2},
  pages={267--322},
  year={1981},
  publisher={JSTOR}
}

@article{ren2016tropicalization,
  title={Tropicalization of del Pezzo surfaces},
  author={Ren, Qingchun and Shaw, Kristin and Sturmfels, Bernd},
  journal={Advances in Mathematics},
  volume={300},
  pages={156--189},
  year={2016},
  publisher={Elsevier}
}

@misc{meng2023mmplocallystablefamilies,
      title={MMP for locally stable families and wall crossing for moduli of stable pairs}, 
      author={Fanjun Meng and Ziquan Zhuang},
      year={2023},
      eprint={2311.01319},
      archivePrefix={arXiv},
      primaryClass={math.AG},
      url={https://arxiv.org/abs/2311.01319}, 
}

@article{gallardo2025explicit,
  title={An explicit wall crossing for the moduli space of hyperplane arrangements},
  author={Gallardo, Patricio and Schaffler, Luca},
  journal={Journal of the London Mathematical Society},
  volume={111},
  number={6},
  pages={e70196},
  year={2025},
  publisher={Wiley Online Library}
}

@book{alexeev2015moduli,
  title={Moduli of Weighted Hyperplane Arrangements},
  author={Alexeev, V. and Bini, G. and Lahoz, M. and Macr{\'\i}, E. and Stellari, P.},
  isbn={9783034809146},
  lccn={2015940007},
  series={Advanced Courses in Mathematics - CRM Barcelona},
  url={https://books.google.com/books?id=NWpbrgEACAAJ},
  year={2015},
  publisher={Springer Basel}
}

@article{doran2004moduli,
  title={Moduli space of cubic surfaces as ball quotient via hypergeometric functions},
  author={Doran, Brent R},
  journal={arXiv preprint math/0404062},
  year={2004}
}

@article{casalaina2021complete,
  title={Complete moduli of cubic threefolds and their intermediate Jacobians},
  author={Casalaina-Martin, Sebastian and Grushevsky, Samuel and Hulek, Klaus and Laza, Radu},
  journal={Proceedings of the London Mathematical Society},
  volume={122},
  number={2},
  pages={259--316},
  year={2021},
  publisher={Wiley Online Library}
}

@article{birkar2010existence,
  title={Existence of minimal models for varieties of log general type},
  author={Birkar, Caucher and Cascini, Paolo and Hacon, Christopher and McKernan, James},
  journal={Journal of the American Mathematical Society},
  volume={23},
  number={2},
  pages={405--468},
  year={2010}
}

@article{cools2011singular,
  title={Singular hypersurfaces possessing infinitely many star points},
  author={Cools, Filip and Coppens, Marc},
  journal={Proceedings of the American Mathematical Society},
  volume={139},
  number={10},
  pages={3413--3422},
  year={2011}
}

@article{derenthal2008nef,
  title={The nef cone volume of generalized del Pezzo surfaces},
  author={Derenthal, Ulrich and Joyce, Michael and Teitler, Zachariah},
  journal={Algebra \& Number Theory},
  volume={2},
  number={2},
  pages={157--182},
  year={2008},
  publisher={Mathematical Sciences Publishers}
}

@article{alexeev2008weighted,
  title={Weighted Grassmannians and stable hyperplane arrangements},
  author={Alexeev, Valery},
  journal={arXiv preprint arXiv:0806.0881},
  year={2008}
}
	
\end{document}